\documentclass{amsart}
\usepackage[dvipsnames]{xcolor}
\input{main.sty}

\title{Affine Higher Bruhat Orders}

\author{Sara C. Billey}
\address{Department of Mathematics, University of Washington, Seattle,
WA, USA}
\email{billey@uw.edu}

\author{Herman Chau}
\address{Department of Mathematics, University of Washington, Seattle,
WA, USA}
\email{beijingherman@gmail.com}

\author{Kevin Liu}
\address{Department of Mathematics and Computer Science, The University of
the South, Sewanee, TN, USA}
\email{keliu@sewanee.edu}

\date{\today}

\subjclass[2020]{Primary 05E16; Secondary 05A05, 06A07, 20F55.}

\keywords{Affine symmetric group, Coxeter group, weak order, reduced
words, reflection orders, inversion sets, higher Bruhat orders}

\begin{document}
\begin{abstract}
The higher Bruhat orders are partial orders originally constructed by
Manin and Schechtman, with an equivalent construction from
Ziegler. These partial orders generalize the weak order on the
symmetric group $\mathfrak{S}_n$, and the second higher Bruhat order
is a poset on commutation classes of reduced words for the longest
element in ${\mathfrak{S}}_n$, where covering relations correspond to
braid relations. Like the weak order, each of the higher Bruhat orders
is a ranked poset with unique minimal and maximal
elements. Constructing analogs of the higher Bruhat orders in other
settings is an area of recent interest, including recent work of
Hothem for permutations in $S_{n}$.  We give a construction that
generalizes any interval $[\id,w]$ in the weak order where $w$ is an
element of the affine symmetric group
$\widetilde{\mathfrak{S}}_n$. Our construction recovers the classical
higher Bruhat orders when $w$ is the longest element in
$\mathfrak{S}_n$. Paralleling the classical case, we show that the
second higher Bruhat order is a ranked poset on commutation classes of
reduced words for $w\in \widetilde{\mathfrak{S}}_n$, allowing us to
prove a conjecture due to Elias.  When $w\in \mathfrak{S}_n$, we also
establish results for our higher Bruhat orders that are direct analogs
of numerous results in the classical higher Bruhat orders.
\end{abstract}

\vspace*{-0.25em}
\maketitle
\vspace{-2em}

\setcounter{tocdepth}{3}

 \tableofcontents

\section{Introduction}

The classical {higher Bruhat orders} $\B(n,k)$ are partial orders
introduced by Manin and Schechtman \cite{manin_schechtman_1989} in
their study of fundamental groups of certain hyperplane
arrangements. They were further characterized by Ziegler
\cite{ziegler_1993} {using single step inclusion of consistent
sets}. The first higher Bruhat order is isomorphic to the well-known
weak Bruhat order on the symmetric group $\S_n$.  The second higher
Bruhat order on $\S_{n}$ is a partial order on commutation classes of
reduced words for the longest permutation.  In the general case, these
partial orders can equivalently be described using hyperplane
arrangements \cite{felsnerweil_sweeps,ziegler_1993}, oriented matroids
\cite{sturmfels-oriented}, and sometimes zonotopal tilings
\cite{felsner-zonotopes}.  The higher Bruhat orders have connections
to areas such as Bott-Samelson bimodules \cite{Elias_2016}, Steenrod
algebras \cite{steenrod}, and the higher Stasheff-Tamari orders
\cite{thomas,williams}.

Generalizing the higher Bruhat orders to other settings has been an
area of recent interest
\cite{danilov2022BC,higherbruhatB}. In 2022, Ben
Elias generalized the second higher Bruhat orders to include
nonreduced words for permutations in order to prove a generalized
Bergman diamond lemma for Hecke-type algebras \cite{Elias.2022}.
Motivated by Elias's work, Daniel Hothem generalized the higher Bruhat
orders beginning from identity intervals $[\mathrm{id},w]$ in the weak
order on the symmetric groups
\cite{hothemthesis,hothem2021extendinghigherbruhatorders}.  Together, they posed
the following problem.

\begin{problem}\label{prob:EliasHothem}(Elias-Hothem, 2020)
    Generalize the higher Bruhat orders to all identity intervals in the affine symmetric groups $\affineS_n$.
\end{problem}

Based on the relationship between reduced words and the second higher Bruhat order,  Elias also proposed orienting edges between commutation classes of reduced words for each $w\in \affineS_n$ according to the directed braid relations $s_is_{i+1}s_i\to s_{i+1}s_is_{i+1}$. This results in a directed graph, denoted $\cG(w)$. 
\begin{conjecture}(Elias, 2021)\label{conj:Elias.n=2}
    For any $w\in \affineS_n$, the directed graph $\cG(w)$ is acyclic with a unique source vertex and a unique sink vertex.
\end{conjecture}

To briefly describe the classical higher Bruhat orders due to
Manin-Schechtman, let $[n]=\{1,2,\ldots,n\}$, and let ${[n]\choose
k}=\{X\subseteq [n]:|X|=k\}$ denote the $k$-subsets of $[n]$.  They
defined the set of \defn{admissible orders} $\A(n,k)$ to be certain
total orders on ${[n]\choose k}$ and defined commutation moves for
admissible orders, giving rise to equivalence classes of admissible
orders. The higher Bruhat order $\B(n,k)$ is the set of commutation
equivalence classes of admissible orders in $\A(n,k)$ with a partial
ordering generated by an operation called a \defn{packet flip}. When
$k=1$, the resulting partial order is isomorphic to the weak Bruhat
order on the symmetric group $\S_n$ ordered by containment of
inversion sets. Manin-Schechtman introduced the concept of a
generalized inversion set for each admissible order that is constant
on equivalence classes and increases exactly by one element with each
packet flip going up in the partial order.  They also showed that the
maximal chains in $\B(n,k)$ are in bijection with $\A(n,k+1)$, giving
rise to a recursive method for constructing $\B(n, k+1)$.
Later work by Ziegler \cite{ziegler_1993} using oriented matroids showed that the generalized inversion sets of admissible orders can be completely characterized by what he called \defn{consistent sets}, without regard to admissibility.  In his formulation, the higher Bruhat orders are partial orders on consistent sets ordered by single step inclusion.

Inspired by \cref{prob:EliasHothem} and the work of Elias-Hothem,
Manin-Schechtman, and Ziegler, we propose a generalization of the
higher Bruhat orders to all identity intervals in the affine symmetric
group, and hence, up to isomorphism, to all weak-order intervals.
We conjecture this generalization maintains all of the
elegant properties of the classical higher Bruhat orders.  Our main
theorems provide evidence for this conjecture.  In particular, we
answer Elias's \cref{conj:Elias.n=2} on the directed graph of
commutation classes of reduced expressions in the affirmative.

\bigskip

Our construction begins with an arbitrary element $w$ in the affine
symmetric group $\affineS_n$. We generalize the notion of
admissibility to define $\A_w(n,k)$, which consists of admissible
orders on the $k$-inversions of $w$. The $k$-inversions of $w$,
denoted $\Inv_k(w)$, are the length $k$ decreasing subsequences in the
one-line notation of $w$, up to a congruence shift.  In particular,
$\Inv_{2}(w)$ represents the familiar inversion set for $w$.  We
extend the notion of commutation equivalence on admissible orders to
$\A_w(n,k)$, denoted $\sim_{w}$.  The corresponding higher Bruhat
order $\B_w(n,k)$ is the poset on equivalence classes
$\A_{w}(n,k)/\hspace{-0.2em}\sim_{w}$, which can again be partially
ordered by packet flips.  Generalizing the classical $k=1$ case, we
prove that $\A_w(n,1)$ is exactly the set of affine permutations in
the interval $[\mathrm{id},w]$, and $\B_w(n,1)$ is isomorphic to the
weak order restricted to $[\mathrm{id},w]$ in \Cref{sub:firstAHB}.  We
associate to each admissible order a generalized inversion set related
to the classical packet flip operations, which is constant on its
equivalence class.  We call the generalized inversion set of an
admissible order the \defn{reversal set} to avoid confusion with
$k$-inversion sets in this context.  The reversal sets of
$k$-admissible orders satisfy certain necessary conditions depending
on $w$.  Using these necessary conditions, we extend Ziegler's
characterization of consistent sets to subsets of $\Inv_{k}(w)$. We
then define $\C_w(n,k)$ to be the poset of consistent subsets of
$\Inv_{k}(w)$ ordered by single step inclusion. For the symmetric
group, we prove a complete analogue of the elegant properties that
Manin-Schechtman and Ziegler proved for the classical higher Bruhat
orders and consistent sets.

\begin{theorem}\label{thm:nonlongest}
    Let $k,n$ be positive integers and $w\in \S_n$. The following
    hold.
    \begin{enumerate}[label=(\alph*)]
    \item The higher Bruhat order
    $\B_w(n,k)=\A_{w}(n,k)/\hspace{-0.2em}\sim_{w}$ partially ordered
    by packet flips is isomorphic 
    to the poset $\C_w(n,k+1)$, and the isomorphism sends an equivalence class of admissible orders to the reversal set of the class.
    \item $\B_w(n,k) \cong \C_w(n,k+1)$ is a ranked poset with 
    unique minimal and  maximal elements corresponding to reversal
    sets  $\emptyset$ and  $\Inv_{k+1}(w)$.  
    \item \label{item:chains}There is a natural bijection between
    maximal chains of $\B_w(n,k) \cong \C_w(n,k+1)$ and the admissible
    orders in $\A_w(n,k+1)$ using the single step partial order.
    \end{enumerate} 
\end{theorem}

In the case where $w$ is the longest element $w_0\in \S_n$, our
construction recovers the classical higher Bruhat orders of
Manin-Schechtman.  Consequently, \Cref{thm:nonlongest} generalizes
several of the classical results in \cite{manin_schechtman_1989} and
\cite{ziegler_1993}. \Cref{thm:nonlongest} also implies that one can
recursively construct the higher Bruhat orders $\B_w(n,k+1)$ using
equivalence classes of maximal chains in $\B_w(n,k)$, analogous to the
recursive construction in the classical case. See \cref{cor:recursive}
for a detailed statement.  Combining this with Hothem's recursive
construction in \cite[Theorem
3.3.1]{hothem2021extendinghigherbruhatorders}, we know that for $w\in
\S_n$, our $\B_w(n,k)$ and Hothem's higher Bruhat orders are
isomorphic posets.  However, our construction differs from his in
several key ways. We introduce a new poset on $\Inv_k(w)$ that we call
the \defn{permanent poset} for $w$, denoted $\P_w(n,k)$.  The
permanent poset is used in our definitions of admissible orders in
$\A_w(n,k)$ and consistent sets in $\C_w(n,k+1)$ for any affine
permutation $w\in \affineS_n$.  

In the process of proving our results, we construct for each $R
\subseteq \Inv_{k+1}(w)$ a directed graph $G_{R}$ on vertex set
$\Inv_{k}(w)$ that generalizes the inversion graph of a permutation,
the inversion multigraph of an affine permutation due to
Bj\"orner-Brenti \cite{bjorner_brenti_1996}, and the heap poset of a
reduced word as described by Stembridge \cite{Stem4,Stem3}.  See
\Cref{def:classical.GR} and \Cref{def:affine.GR}.  Whenever $R$ is the
reversal set of an admissible order in $\A_w(n,k)$, we will show that
all admissible orders with reversal set $R$ form a unique
equivalence class under $\sim_{w}$, and they are exactly the linear
extensions of the heap poset given by the transitive closure on the
graph $G_{R}$.  In particular, representatives of the equivalence
classes with reversal sets $\emptyset$ and $\Inv_{k+1}(w)$ can be
found by considering $G_{\emptyset}$ and $G_{\Inv_{k+1}(w)}$ on
$\Inv_k(w)$.

Though \cref{thm:nonlongest} applies only to $w\in \S_n$, we
conjecture our construction of the affine higher Bruhat orders has all
of the same nice properties.
We have computationally verified the following conjecture for all
$1\leq k\leq n\leq 8$ and affine permutations $w\in \affineS_n$ 
of Coxeter lengths at most the values in \cref{tab:conjecture}.

\begin{conjecture}
    \label{conjecture:main}
  The statements in  \cref{thm:nonlongest} hold for all positive
  integers $k,n$ and $w \in \affineS_n$.
\end{conjecture}

\begin{table}[t]
\begin{tabular}{|c|c|c|c|c|c|} \hline 
    $n$ & 4 & 5 & 6 & 7 & 8 \\ \hline 
    length & 30 & 30 & 24 & 18 & 15 \\ \hline 
\end{tabular}
\caption{For each value of $n$, \cref{conjecture:main} has been verified for all $w\in\affineS_n$ up to and including the indicated length.}
\label{tab:conjecture}
\end{table}

While \cref{conjecture:main} remains open in general, we are able to
prove the affine analogs of most of the results in
\cref{thm:nonlongest} when $k\leq 2$ and when $k\geq n$. See
\Cref{thm:affine} below and \Cref{thm:ultimateAHB}, respectively.  Using these, we will show
\cref{conjecture:main} holds for all $w \in \affineS_n$ with $n\leq
3$ in \Cref{cor:n=3}.  

\begin{theorem}\label{thm:affine}
    For any positive integer $n$ and $w\in \affineS_n$, the following hold.
    \begin{enumerate}[label=(\alph*)]
\item The admissible orders in $\A_{w}(n,1)$ correspond with affine
permutations $u \leq w$ in weak order, and the first higher Bruhat
order $\B_{w}(n,1) \cong \C_w(n,2) \cong [\mathrm{id},w]$.

\item The consistent sets in $\C_w(n,2)$ are exactly the inversion
sets of affine permutations $u \in [\mathrm{id},w]$ ordered by
(single) step inclusion.

\item The admissible orders in $\A_w(n,2)$ are the reflection orders
for $w$, and there are natural bijections between maximal chains of
$\C_w(n,2)$, reduced words for $w$, and reflection orders for $w$.

\item The second higher Bruhat order $\B_w(n,2)$ is a poset on
commutation classes of reflection orders, and its Hasse diagram is
isomorphic to $\cG(w)$ as a directed graph.

\item $\B_w(n,2)\cong \C_w(n,3)$ is a ranked poset with  unique
minimal and maximal elements, and the isomorphism sends an equivalence
class of admissible orders to the reversal set of the class.

\item \label{part:com.equiv.rev.set} Two reduced words for $w$ are
commutation equivalent if and only if the reversal sets of their
associated reflection orders are equal.

\item There is a natural bijection between maximal chains of
$\B_w(n,2) \cong \C_w(n,3)$ and the admissible orders in $\A_w(n,3)$
using the single step partial order.
\end{enumerate}
\end{theorem}

The preceding theorem shows that there are two additional ways to view
the directed graph $\cG(w)$, either as the Hasse diagram of $\B_w(n,2)$ or
$\C_w(n,3)$. \cref{conj:Elias.n=2} and several other properties of
$\cG(w)$ immediately follow from this observation.
\begin{corollary}\label{cor:cGgraph}
For any  $w\in \affineS_n$, the following hold. 
\begin{enumerate}[label=(\alph*)]
	\item $\cG(w)$ is a bipartite directed acyclic graph with a unique source vertex and a unique sink vertex. In particular, $\cG(w)$ is a tree if and only if $\cG(w)$ is a path.
	\item The diameter of $\cG(w)$ as an undirected graph is $|\Inv_3(w)|$, which counts the number of $321$ patterns in $w$ up to a congruence shift. In particular, $\cG(w)$ consists of a single element if and only if $w$ is $321$-avoiding.
\end{enumerate}
\end{corollary}

\subsection*{Related Work}

As mentioned above, Hothem was the first to extend the higher Bruhat
orders to intervals $[\id,w]$ in the symmetric group
\cite{hothem2021extendinghigherbruhatorders}.  This paper shows that
Hothem's extension can also be characterized by the more general
definition of consistent sets given here for the first time. This definition depends 
on our newly constructed permanent poset $\P_w(n,k)$, and connects explicitly to acyclicity
of the reversal graphs $G_R$ mentioned above.  Furthermore, the extension of
higher Bruhat orders to affine permutations, the associated
definitions and statements on $k$-inversion sets, the permanent poset,
consistent sets, and the connection to the reversal graphs have not
been previously studied for higher values of $k$.  

Commutation classes of reduced words for any Coxeter group, including
the symmetric groups and affine symmetric groups, can be represented
using their heap posets due to work of Stembridge \cite{Stem4,Stem3}
building on the theory of commutation monoids and the heap
construction of Cartier-Foata \cite{Cartier.Foata.1969}, Viennot
\cite{Viennot.1986}, and Billey-Jockusch-Stanley \cite[Sec. 2]{BJS93}.
Each commutation class of reduced words, or equivalently reflection
orders, is the set of linear extensions of its heap poset, see
\cite[Prop. 2.2]{Stem4}.  In this paper, the poset induced by a
reversal graph $G_{R}$ extends the heap poset construction to all
$k$-admissible orders, where the case of $k=2$ is isomorphic to the
heap poset of the corresponding commutation class of reduced words.
The permanent poset $\P_{w}(n,k)$ that we define below is permanent in
the sense that it is the intersection of all heap posets for a given
triple $k,n,w$.  See \cref{thm:equiv_class.and.linear.exts} for the
case of classical permutations and
\cref{thm:equiv_class.and.linear.exts.affine} for the case of affine
permutations.

Elnitsky proved that the commutation classes of reduced words for
permutations in $\S_{n}$ are also in bijection with rhombic tilings of
a certain $2n$-sided polygon \cite[Thm 2.2]{elnitsky}.  Thus,
commutation equivalence of reduced words can be witnessed by the heap
poset, the associated Elnitsky tiling, or the reversal set of the
associated reflection order from
\Cref{thm:affine}~\ref{part:com.equiv.rev.set}.  The braid moves on
reduced words correspond with local moves on a hexagon in the tiling
associated to a reduced word, and with packet flips in our
terminology.  Hence, for $w \in \S_{n}$, the second higher Bruhat
order $\B_{w}(n,2)$ can also be viewed as a partial order on rhombic
tilings of the Elnitsky polygon for $w$.  A braid move of the form
$121\to 212$ increases the sum of the numbers in the reduced word for
a classical permutation, though not necessarily for an affine
permutation.  It also decreases the sum of the distances of the
vertices in the corresponding rhombic tiling to the vertex labeled
$0$.  Elnitsky proved that the graphs $\cG(w)$ are connected and
bipartite \cite[Prop. 3.1 and 3.2]{elnitsky}.  \cref{thm:affine} and
\cref{cor:cGgraph} extend this to every $w\in \affineS_n$, as every
ranked poset with  unique minimal and maximal elements has a connected
bipartite Hasse diagram.

Escobar-Pechenik-Tenner-Yong used the Elnitsky tilings to describe and
generalize the Bott-Samelson varieties to a larger class of varieties
that also are desingularizations of Schubert varieties in complete
flag manifolds \cite{EPTY.2018}.  The generalized variation uses the
zonotopal tilings due to Tenner \cite{tenner} as the underlying
geometric structure.  The zonotopes used in these tilings correspond
with $k$-inversions for various $1\leq k\leq n$ in our terminology.
We note that Tenner's poset on zonotopal tilings of the Elnitsky
polytope $E(w)$ for a permutation $w \in \S_{n}$ is not the same as any
of the posets in this paper since it uses a mixture of different size
$k$-inversions.

Our \cref{cor:cGgraph} extends other known results from $\S_n$ to
$\affineS_n$.  For $w\in \S_n$, the equivalence of $\cG(w)$ being a
tree or path was established by Tenner \cite[Theorem 6.12]{tenner},
and \cref{cor:cGgraph}(a) extends this to $w\in \affineS_n$.  We
extend her characterization of path graphs to the affine permutations
in \Cref{thm:path_iff_2common}.  The diameter in \cref{cor:cGgraph}(b)
also extends a result of Gutierres, Mamede, and Santos
\cite{GUTIERRES} from the symmetric group $\S_n$ to the affine
symmetric group $\affineS_n$. The case where $\cG(w)$ is a single
vertex recovers the 321-avoiding characterization of
Billey-Jockusch-Stanley \cite[Theorem 2.1]{BJS93} for $w\in \S_n$ and
Green \cite[Proposition 4.3]{green} for $w\in \affineS_n$.

Assaf \cite{Assaf.2019} defined a directed analog of the
Tits/Matsumoto graph on all reduced words for $w \in \S_n$ and proved
it is the Hasse diagram of a ranked poset with a unique maximal
element \cite[Cor. 2.9]{Assaf.2019}.  By definition, $\cG(w)$ as an
undirected graph is a quotient of the Tits/Matsumoto graph.  A
comparison of our constructions and the bijections in
\cref{thm:affine} shows that $\cG(w)$ is a quotient of Assaf's
graph when $w\in \S_n$.

\subsection*{Outline} We begin in \cref{sec:prelims} by giving
preliminary information on the affine symmetric group and the
classical higher Bruhat orders. In \cref{sec:non-longest}, we study
the $k$-inversion sets of permutations and define $k$-admissible
orders on these sets.  We use the admissible orders to construct our
version of higher Bruhat orders for arbitrary elements $w\in \S_n$,
and we connect these with heap posets and generalized consistent sets in
order to establish the properties in \cref{thm:nonlongest}.  The
constructions for the symmetric group are simpler to understand first
before going on to the affine higher Bruhat orders for all $w\in
\affineS_n$ introduced in \Cref{sec:affine}.  In \Cref{sec:affine}, we
state and prove many of the affine analogs of the results on the
symmetric group. In particular, we study properties of the
$k$-inversion sets for affine permutations and prove that the permanent
posets for any affine permutation are indeed posets in the first two
subsections of \Cref{sec:affine}. 
We spell out the details for the first and second higher Bruhat orders
in \Cref{sub:firstAHB} and \Cref{sub:secondAHB} in order to prove
\cref{thm:affine}.  In \Cref{sub:ultimateAHB}, we address the
admissible orders in $\A_{w}(n,n)$, along with ultimate higher Bruhat
orders $\B_w(n,k)$, $\C_w(n,k)$, and $\P_{w}(n,k)$ for all $k\geq n$.
In \Cref{sec:open}, we conclude with some additional open problems for
affine higher Bruhat orders.

\section{Preliminaries}\label{sec:prelims}

In this section, we review the terminology and theorems we need from the literature. For more background on partial orders, see \cite{StanEC1}.  For more background on affine permutations, see \cite[Sec. 8.3]{BB}.

\subsection{Partially Ordered Sets}

 A \defn{partially ordered set (poset)} $(P, \leq_P)$ is a set of
elements $P$ equipped with a binary relation $\leq_P$ satisfying
reflexivity, antisymmetry, and transitivity. When there is no
confusion, we will also refer to the poset as simply $P$ and use
$\leq$ instead of $\leq_P$. An element $y$ \defn{covers} $x$, denoted
$x \lessdot y$ if $x < y$ and there is no element $z \in P$ such that
$x < z < y$.  A subset $S$ of $P$ is \defn{convex} if $x < y < z$ and
$x,z \in S$ imply $y \in S$.  A set equipped with a binary relation
$\leq_P$ satisfying just reflexivity and transitivity is a
\defn{pre-order}.
 
 There are two directed graphs naturally associated with any poset.
The first is the graph of all strict relations $x<y$ on $P$.  The second is
the \defn{Hasse diagram} or \defn{transitive reduction}, which is the
graph on $P$ with directed edges given by the covering relations in
$P$. Conversely, every directed acyclic graph gives rise to a poset on
its vertex set by taking the \defn{transitive closure}, in which case $x<y$
provided there exists a sequence $(a_1, \ldots, a_\ell)$ such that
$a_1 = x$, $a_\ell = y$, and $a_i \to a_{i+1}$ is an edge in the graph
for all $1\leq i < \ell$. We will often define our posets by
specifying a generating set of relations and taking the transitive
closure, and these generating sets can be viewed as edges of a
directed graph.  However, it is necessary to check that the resulting
relation is, in fact, a partial order. Reflexivity and transitivity
are clear when taking the transitive closure, but one must check the
resulting graph is acyclic to ensure the corresponding binary relation
is antisymmetric.
 
A poset $(P,\leq)$ is a \defn{total order} if every pair of elements $x$
and $y$ is comparable, meaning either $x\leq y$ or $y \leq x$.  A
\defn{linear extension} of $P$ is any total ordering of the elements
in $P$ such that any two comparable elements in $P$ are related in the
same way in the linear extension. We will use $\mathscr{L}(P)$ to denote the set
of all linear extensions of $P$.  If $P$ is finite, we will often
denote a linear extension of $P$ as an ordered list $\rho = (\rho_1,
\rho_2, \ldots, \rho_\ell)$ of all elements in $P$, and we denote the
corresponding total order relation by $\leq_\rho$. For a subset $S
\subseteq P$, the \defn{restriction} $\rho|_S$ is the total order on
$S$ induced from $\rho$, and the restriction of $(P,\leq)$ to $S$ will be
denoted by $(S,\leq_P)$, or just $S$ if the relation is clear from
context.  The nonempty subset $S$ is a \defn{chain} in $P$ if
$(S,\leq_P)$ is a total order. A chain $(c_1 <_P c_2 <_P \cdots <_P c_k)$
has \defn{length} $k-1$, and it is \defn{saturated} if each $c_i <_P
c_{i+1}$ relation is a covering relation. A chain is \defn{maximal} if
it cannot be extended to a longer chain. Note that every maximal chain
is saturated but the converse is not necessarily true.  An
\defn{antichain} is a subset $S$ of $P$ such that no two distinct
elements of $S$ are related in the partial order.

A poset $P$ is \defn{ranked} if there exists a well-defined rank function $r: P
\to \N$ such that $r(x) = 0$ if $x$ is a minimal element in $P$, and
$r(y) = r(x)+1$ whenever $y$ covers $x$ in $P$. If $P$ is a finite poset
with a unique minimal element $\hat{0}$ and a unique maximal element
$\hat{1}$, then it is ranked if and only if every maximal chain in $P$ has
the same length.

If  $x \leq y$ in $P$, the (closed) \defn{interval} $[x,y]$ is the set $\{z \in P : x \le z \le y\}$.  A \defn{lower order ideal} $I$ of a poset $P$ is a subset that satisfies the \defn{downwards closed} condition
\[
\text{if $y \in I$ and $x \leq  y$, then $x \in I$.}
\]
Similarly, an \defn{upper order ideal} $I$ is a subset that satisfies the \defn{upwards closed} condition
\[
\text{if $x \in I$ and $x \leq  y$, then $y \in I$.}
\]

Given two posets $P,Q$, an \defn{order-preserving} map is a map $f: P
\to Q$ such that if $x \le_P y$, then $f(x) \le_Q f(y)$. An
\defn{isomorphism} of posets $P,Q$ is an order-preserving bijection
between $P$ and $Q$ whose inverse is also order-preserving.
Furthermore, the product of two posets $P \times Q$ is the poset on
the pairs $(p,q)$ for all $p\in P, q\in Q$ with the partial order
$(p_{1},q_{1})\leq (p_{2},q_{2})$ if and only if $p_{1}\leq p_{2}$ and
$q_{1}\leq q_{2}$ in $P$ and $Q$ respectively.

\subsection{Symmetric and Affine Symmetric groups}\label{sub:sn} Here
we summarize some known work of Barkley-Speyer, Bj\"orner-Brenti,
Eriksson-Eriksson, Lusztig, and Shi
\cite{Barkley.Speyer.2024,bjorner_brenti_1996,E-E,Lusztig.1983,Shi.1986}
on affine symmetric groups.  A bijection $w: \Z \to \Z$ is called
\defn{$n$-periodic} if for all $x \in \Z$, we have $w(x+n) =
w(x)+n$. Note that an $n$-periodic bijection is uniquely determined by
the list of values it takes on $[n]$.  The \defn{affine symmetric
group} $\affineS_n$ is the group of all $n$-periodic bijections $w: \Z
\to \Z$ that satisfy the additional property $\sum_{i \in [n]}w(i) =
\binom{n+1}{2}$. The group operation on $\affineS_n$ is function composition. The
\defn{symmetric group} $\S_n$ is the subgroup of all $w \in
\affineS_n$ such that $w([n])=[n]$.  Note that for $n=1$, there is
exactly one $n$-periodic bijection in $\affineS_{1}$, which is the
identity map on $ \Z \to \Z$.  Throughout the rest of this section,
assume $n\geq 2$.

An element in $\affineS_n$ is called an \defn{affine permutation}.  If
we need to specify that $w \in \S_n$, we will call $w$ a
\defn{classical permutation}, or just a permutation. Elements in
$\S_n$ and $\affineS_n$ can be represented in multiple ways, and we
will primarily use the window and one-line notations. The \defn{window
notation} for $w$ shows the image of $[n]$ in a tuple
$(w(1),w(2),\ldots,w(n))$.  Since $w \in \affineS_{n}$, every $i \in
[n]$ is congruent to some element in $\{w(1),\dotsc , w(n) \}$ modulo
$n$.  The \defn{one-line notation} extends the window notation by
showing the image of all of $\mathbb{Z}$ as an infinite tuple. For
example, the permutation $w\in \affineS_4$ defined by $w(1)=4$,
$w(2)=7$, $w(3)=-3$, and $w(4)=2$ has window notation $(4, 7, -3, 2)$,
and a portion of its one-line notation is $(\ldots,0,3,-7,-2,4, 7, -3,
2,8,11,1,6,\ldots)$. The $n$-periodic property of affine permutations
can be used to obtain the one-line notation from the window notation,
and the property $\sum_{i \in [n]}w(i) = \binom{n+1}{2}$ can be used
to identify the window notation from the one-line notation by finding
the unique consecutive $n$ elements with a sum of $\binom{n+1}{2}$.

The affine symmetric group $\affineS_n$ is generated by the
\defn{simple transpositions} $s_1, \ldots, s_{n-1}, s_n$ where $s_i$
is the affine permutation that maps $s_i(i) = i+1$, $s_i(i+1) = i$,
and fixes $j$ for all $j \in \Z$ such that $j \not\equiv i, i+1
\pmod{n}$. We will consider the indices of the simple transpositions
modulo $n$. For example, $s_0$ and $s_n$ both denote the same simple
transposition that swaps $n$ with $n+1$. For $n\geq 3$, the simple
transpositions in $\affineS_n$ satisfy the following set of minimal
relations:
\begin{itemize}
    \item $s_i^2 = 1$ for all $i \in \Z$,
    \item $s_is_{i+1}s_i = s_{i+1}s_is_{i+1}$ for all $i \in \Z$, and
    \item $s_is_j = s_js_i$ for all $i,j \in \Z$ with $j \not\equiv i,i+1, i-1 \pmod{n}$.
\end{itemize}
Relations of the second kind are termed \defn{braid relations}, and
relations of the third kind are termed \defn{commutation relations}.
For $n=2$, the generators only satisfy the first relation $s_i^2 = 1$
for all $i \in \Z$.  The symmetric group $\S_n$ is generated by $s_1,
\ldots, s_{n-1}$, omitting the simple transposition $s_n$.

A \defn{reduced decomposition} for an element $w \in \affineS_n$ is a
product $s_{i_1} \cdots s_{i_\ell}$ of minimal length such that
$s_{i_1} \cdots s_{i_\ell} = w$. For brevity, we will consider
\defn{reduced words}, which are the indices $i_1 \cdots i_\ell$ of a
reduced decomposition $s_{i_1} \cdots s_{i_\ell}$.  All reduced words
for an element $w \in \affineS_n$ have the same \defn{length}, which
is denoted $\ell(w)$.  Since $s_{i}=s_{i+n}$ for all $i$, we consider
the reduced words for $w \in \affineS_{n}$ to be elements in
$(\Z/n\Z)^{\ell(w)}$.  There is a unique permutation $w_0 \in \S_n$ of
maximal length, called the \defn{longest permutation} given by $w_0 =
(n, n-1, \ldots, 1)$.  There is no affine permutation of longest
length in $\affineS_{n}$ for $n>1$.

In general, an element $w\in \affineS_n$ can have many reduced words
as elements in $(\Z/n\Z)^{\ell(w)}$. Given any reduced word, one can
generate additional ones using relations among the generators of
$\affineS_n$. If $i_1 \cdots i_\ell$ is a reduced word for $w$ and
$i_{j+1} \not\equiv i_j+1, i_j-1 \pmod{n}$, then swapping $i_j$ and
$i_{j+1}$ is another reduced word for $w$ obtained via
\defn{commutation}. Two reduced words $a=a_1\cdots a_\ell$ and
$b=b_1\cdots b_\ell$ are \defn{commutation equivalent} if they differ
by a series of commutations. If $\mathbf{a}=a_1 \cdots a_{\ell}$ is a
reduced word for $w \in \affineS_{n}$, $n\geq 3$, and $i(i+1)i$
appears in consecutive entries in $\mathbf{a}$, then this triple
$i(i+1)i$ can be replaced with $(i+1)i(i+1)$ to obtain another reduced
word for $w$. Similarly, one can replace $(i+1)i(i+1)$ with $i(i+1)i$,
and both of these operations will be called a \defn{braid}.  There are
no allowed braids on reduced words for $w$ when $n=2$.   

For any $w\in \affineS_n$,  the \defn{graph on reduced words of $w$}, denoted $\cR(w)$, is the undirected graph whose vertex set consists of reduced words for $w$ and edges between reduced words that can be obtained from one another using a commutation or braid. Let $\sim_w$ denote commutation equivalence, and let $\boldsymbol{\cR(w)/\sim_w}$ be the quotient graph with vertices given by the equivalence classes of reduced words. A classical result of Matsumoto \cite{Matsumoto_1964} and Tits \cite{Tits_1968} states that $\cR(w)$ is connected, which also implies that $\cR(w)/\sim_w$ is connected. 

Recall the directed braid graph on commutation classes of reduced
words of $w$, denoted $\cG(w)$ in the introduction, is obtained from
$\cR(w)/\sim_w$ by directing the edges between two commutation classes
as follows.  If there exist representatives $a' \in [a]$ and $b'\in
[b]$ that differ by a braid from $i(i+1)i$ to $(i+1)i(i+1)$ in three
consecutive positions for some $i$, then orient the edge $[a]\to [b]$.
Note, if $w \in \S_n$, the edges point along the direction of
increasing the sum of the entries in the reduced word, but this does
not extend to all affine permutations since we must take the values
modulo $n$.

\begin{example}
    \label{example:affine-reduced-word}
    Consider $w=(1,7,2,0)\in \affineS_4$. A reduced decomposition for $w$ is $s_4s_1s_2s_1s_4s_3s_2$, which corresponds to the reduced word $4121432$. Using commutation and braid relations starting with this reduced word, one can generate  the graphs $\cR(w)$ and $\cG(w)$ shown in \cref{fig:reducedwordgraph}.
    \begin{figure}
    \centering
    \scalebox{0.875}{
    \begin{tikzpicture}
        \node (1) at (0,-3) {$4121432$};
        \node (2) at (2,-3) {$4212432$};
        \node (3) at (4,-3) {$4214232$};
        \node (4) at (6,-3) {$4214323$};
        \node (5) at (2,-1.5) {$2412432$};
        \node (6) at (4,-1.5) {$2414232$};
        \node (7) at (6,-1.5) {$2414323$};
        \node (8) at (4,0) {$2141232$};
        \node (9) at (6,0) {$2141323$};
        \node (10) at (8,0) {$2143123$};
        \path (1) edge node {} (2);
        \path (2) edge node {} (3);
        \path (3) edge node {} (4);
        \path (2) edge node {} (5);
        \path (3) edge node {} (6);
        \path (4) edge node {} (7);
        \path (5) edge node {} (6);
        \path (6) edge node {} (7);
        \path (6) edge node {} (8);
        \path (7) edge node {} (9);
        \path (8) edge node {} (9);
        \path (9) edge node {} (10);
    \end{tikzpicture} \quad 
    \begin{tikzpicture}[every text node part/.style={align=center}]
        \node (1) at (0,-2) {$4121432$};
        \node (2) at (3,-2) {$2414232$ \\ $2412432$ \\ $4212432$ \\ $4214232$  };
        \node (3) at (6,-2) {$2414323$ \\ $4214323$  };
        \node (4) at (3,0) {$2141232$};
        \node (5) at (6,0) {$2141323$ \\ $2143123$};
        \draw[->] (1) edge node {} (2);
        \draw[->] (2) edge node {} (3);
        \path[->] (2) edge node {} (4);
        \path[->] (3) edge node {} (5);
        \path[->] (4) edge node {} (5);
    \end{tikzpicture}
    }
    \caption{$\cR(w)$ (left) and $\cG(w)$ (right) for the permutation $w= (1,7,2,0) \in \affineS_4$.}
    \label{fig:reducedwordgraph}
\end{figure}
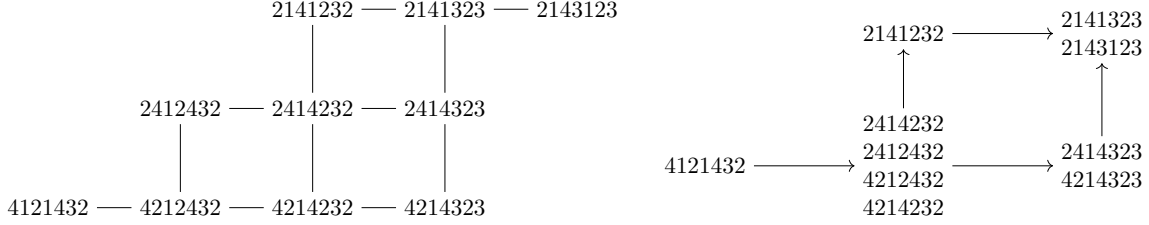
\end{example}

For $w\in \affineS_n$, the \defn{inversion set} of $w$ is
\begin{equation}\label{eq:inversion.set.def}
\Inv(w)=\left\{(x,y)\in [n] \times \mathbb{Z} : x<y \text{ and } w^{-1}(x)> w^{-1}(y) \right\}.
\end{equation}
Observe that if $w(i)<w(i+1)$, then $ws_i$ has exactly one more
inversion than $w$, namely $(w(i),w(i+1))$ where by $n$-periodicity we can
assume $w(i) \in [n]$.  Therefore, affine permutations are uniquely
determined by their inversion set and $\ell(w)=|\Inv(w)|$ by induction
on the length of $w$. By considering each $(i,j)\in \Inv(w)$ under its natural identification in the quotient group $\mathbb{Z}^2/(\mathbb{Z}\cdot (n,n))$, the following statement is a characterization
of inversion sets for affine permutations.  It follows from the
discussion of affine inversions in \cite[Prop
2.1]{Barkley.Speyer.2024} and the more general characterization of
inversion sets for Coxeter groups in \cite[Lemma 4.1(d)]{Dyer.2019},
though stated a bit differently.

\begin{theorem}\label{lem:inversion.sets}
Let $n$ be a positive integer, and let $R$ be a finite subset of
$\mathbb{Z}^2/(\mathbb{Z}\cdot (n,n))$ whose elements have
representatives $(x,y)$ with $x<y$.  Then, $R$ is the inversion set
for some affine permutation in $\affineS_{n}$ if and only if for all
triples of integers $x<y<z$, we have
\begin{enumerate}[label=(\alph*)]
%% MSZ condition requirements
\item \label{part:trapped.value}$(x,z) \in R$ implies  $(x,y) \in R$ or $(y,z) \in R$, 
\item \label{part:inv.packet.gaps.implies.all}$(x,y) \in R$ and $(y,z) \in R$ implies  $(x,z) \in R$, and
%% congruence poset requirements
\item \label{part:inv.congruence}$(x,y) \in R$ implies $x \not \equiv y \pmod{n}$ and $(x,y-en)
\in R$ for all $e \in \Z_{\geq 0}$ such that $ x < y-en<y$.
\end{enumerate}
\end{theorem}

The inversion set of $w \in \affineS_{n}$ can be represented as a
directed acyclic multigraph on $[n]$.  This graph was introduced by
Bj\"orner-Brenti \cite[Sect 5]{bjorner_brenti_1996} and a
characterization was developed by Papi
\cite{papi1997affinepermutations}. See also \cite{bbgraph,shi.1999}
for other characterizations of these graphs.  Recall, $\floor{x}$ is
the largest integer less than or equal to $x \in \mathbb{R}$, even if
$x<0$.

\begin{definition}\label{def:BB.graph}
Let $I_{w}$ be the multigraph on $[n]$ with
\begin{equation}\label{eq:BB.multiplicity}
\left| \left\lfloor \frac{w(j)-w(i)}{n} \right\rfloor \right|
\end{equation}
edges pointing from $i$ to $j$ if  $w(i)<w(j)$  and pointing from
$j$ to $i$ if $w(j)<w(i)$ for each pair $1\leq i<j\leq n$. 
\end{definition}

Examples are shown in \cref{fig:inversion_graph_example}. Since we
have chosen to describe the inversion set on values instead of
positions, it is more natural to consider $I_{w^{-1}}$, which is the
multigraph on $[n]$ with a directed edge from $y \pmod{n} \to x
\pmod{n}$ for each $(x,y) \in \Inv(w)$.  One can recover $\Inv (w)$
from $I_{w^{-1}}$ by Part (c) of \Cref{lem:inversion.sets}.

\begin{figure}[h!]
\centering
\begin{tikzpicture}
        \node (1) at (0,0) {$1$};
        \node (2) at (2,0) {$2$};
        \node (3) at (4,0) {$3$};
        \node (4) at (6,0) {$4$};
        \draw[->] (1) to [out=0,in=180] (2);
        \draw[->] (3) to [out=165,in=15] (2);
        \draw[->] (3) to [out=195,in=345] (2);
        \draw[->] (4) to [out=180,in=0] (3);
        \draw[->] (4) to [out=150,in=30] (2);
        \draw[->] (4) to [out=210,in=330] (2);
        \draw[->] (4) to [out=135,in=45] (1);
\end{tikzpicture}
\qquad 
\begin{tikzpicture}
        \node (1) at (0,0) {$1$};
        \node (2) at (2,0) {$2$};
        \node (3) at (4,0) {$3$};
        \node (4) at (6,0) {$4$};
        \draw[->] (1) to [out=45,in=135] (4);
        \draw[->] (3) to [out=165,in=15] (2);
        \draw[->] (3) to [out=195,in=345] (2);
        \draw[->] (3) to [out=15,in=165] (4);
        \draw[->] (3) to [out=345,in=195] (4);
        \draw[->] (2) to [out=330,in=210] (4);
        \draw[->] (3) to [out=150,in=30] (1);
\end{tikzpicture}
\caption{$I_w$ and $I_{w^{-1}}$ for $w=(1,7,2,0) \in \affineS_{4}$, where $w^{-1}=(1,3,-2,8)$.}
\label{fig:inversion_graph_example}
\end{figure}
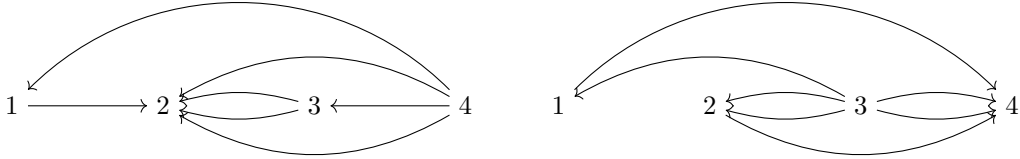

The reflection orders were originally studied by Dyer 
\cite{Dyer.1993}, and now have been applied extensively in the theory
of Coxeter groups and Kazhdan-Lusztig polynomials.  See for example
\cite{Brenti.Casselli.Marietti,cellini1998structure,Dyer.2019,wangwang2024ordertypesreflectionorders}. We
give a definition in the setting of $\affineS_n$.

\begin{definition}\label{def:reflection.order}
Let $w \in \affineS_{n}$.  For each reduced word $\mathbf{i}=i_1
\cdots i_\ell \in \cR(w)$, associate affine permutations
$w^{(j)}=s_{i_{1}}s_{i_{2}}\cdots s_{i_{j-1}}$ for all $1\leq j \leq
\ell$.  Define $\rho_{j}=(w^{(j)}(i_{j}), w^{(j)}(i_{j}+1))$ for all
$1\leq j \leq \ell$. Then, the \defn{reflection order} associated to
$\mathbf{i}$ is the total order on $\Inv (w)$ given by
$\rho(\mathbf{i})=(\rho_{1},\dotsc , \rho_{\ell})$.  The set of all
\defn{reflection orders} associated with $w$ is the set of all such
total orders on $\Inv(w)$ associated to reduced words of $w$.
\end{definition}

The reflection order associated to $i_1 \cdots i_\ell \in \cR(w)$ is
well-defined since $w$ is $n$-periodic, $i_1 \cdots i_\ell \in
(\Z/n\Z)^{\ell}$, and each $s_{i}=s_{i+n}$.  In the notation above, if
$a= w^{(j)}(i_{j})$ and $b=w^{(j)}(i_{j}+1)$, then $a<b$ and $(a,b)
\in \Inv (w)$.  Conversely, each inversion appears exactly once in the
associated list since $i_1 \cdots i_\ell$ is a reduced word for $w$.
Furthermore, any reflection order $\rho = (\rho_{1},\rho_{2},\dotsc ,
\rho_{\ell})$ for $w$ uniquely determines its associated reduced word
for $w$ as an element in $(\Z/n\Z)^{\ell}$. Specifically, $\rho_{1}$
determines $s_{i_{1}}$ up to a shift in index by $n$ and
$(s_{i_{1}}\rho_{2},\dotsc , s_{i_{1}}\rho_{\ell})$ must be a reflection order for
$s_{i_{1}}w$, so it determines a reduced word $(i_{2},\dots ,
i_{\ell})$ by induction on the length of the affine permutation.

\begin{example}
If $w=(3,2,1) \in \S_{3}$, then the reflection order
determined by its reduced word $121$ is $((1,2),(1,3), (2,3))$, while the reflection order determined by $212$ is
$((2,3),(1,3),(1,2))$ in the opposite order.   
\end{example}

\begin{example}
If $w=(1,7,2,0) \in \affineS_{4}$ as in
\Cref{example:affine-reduced-word}, the reflection order for the
reduced word $0121032$ is $((0, 1), (0, 2), (0, 3), (2, 3), (1, 3),
(0, 7), (2, 7))$.  Observe that this reflection order is a total order
on $\Inv(w)$, and $I_{w^{-1}}$ in \Cref{fig:inversion_graph_example} is equal
to the graph obtained from this list of inversions by inserting an
edge from $y  \pmod{n}$ to $x  \pmod{n}$ for each $(x,y) \in \Inv(w)$.
\end{example}

The reflection orders have the property that each ordered prefix must also be
a reflection order since each ordered prefix of a reduced word is still a
reduced word.  Therefore, one can use \Cref{lem:inversion.sets} to
completely characterize the reflection orders for affine permutations
as follows.

\begin{corollary}\label{cor:characterize.ref.order} Let $w \in
\affineS_{n}$.  A total order $\rho = (\rho_{1},\rho_{2},\dotsc ,
\rho_{\ell})$ of $\Inv(w)$ is a reflection order if and only if
\begin{enumerate}[label=(\alph*)]
\item for each triple of integers $x<y<z$ that are distinct mod $n$,
$\rho$ restricted to $\{(x,y),(x,z),(y,z) \}\cap \Inv(w)$ is an
initial sequence of the list $((x,y),(x,z),(y,z))$ or
its reverse $((y,z),(x,z),(x,y))$, and \vspace{.1in}

\item for each pair $(x,y),(x,y+n) \in \Inv (w)$, the pair $(x,y)$ appears before
$(x,y+n)$ in $\rho$.
\end{enumerate}
\end{corollary}

There is a partial order on $\affineS_n$, called the \defn{weak
order}, defined by $v\leq w$ if and only if $\Inv(v)\subseteq
\Inv(w)$.  The covering relations in the weak order are all of the
form $w \lessdot ws_i$ if $\ell(w) = \ell(ws_i) - 1$ for some $i \in
[n]$, so one could also say the weak order is defined by
\defn{single step inclusion} on inversion sets.  Furthermore, every
maximal chain of an identity interval $[\mathrm{id}, w]$ corresponds
with a reduced word for $w$.  A portion of the Hasse diagram for the
weak order on $\affineS_3$ is shown in
\cref{fig:weakbruhatexample}. The weak order is also sometimes called
the \defn{weak Bruhat order}, which is where the higher Bruhat orders
got their name.   See \cite[Chapters 2 and 3]{BB} for details.

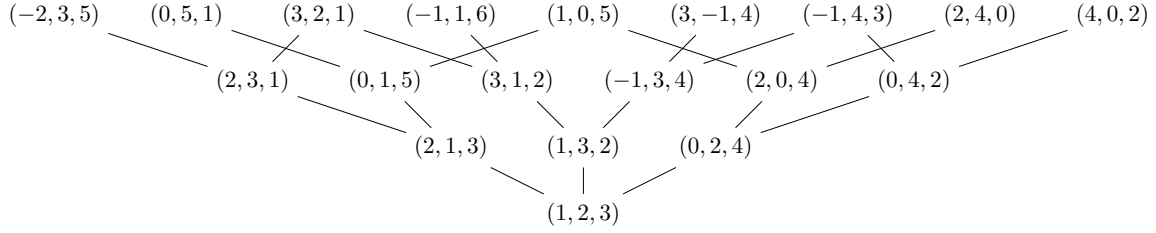
\begin{figure}
\centering 
\scalebox{0.875}{
    \begin{tikzpicture}
        \node (123) at (0,0) {$(1,2,3)$};
        \node (213) at (-2,1) {$(2,1,3)$};
        \node (132) at (0,1) {$(1,3,2)$};
        \node (024) at (2,1) {$(0,2,4)$};
        \path (123) edge node {} (213);
        \path (123) edge node {} (132);
        \path (123) edge node {} (024);
        \node (231) at  (-5,2) {$(2,3,1)$};
        \node (015) at  (-3,2) {$(0,1,5)$};
        \node (312) at  (-1,2) {$(3,1,2)$};
        \node (-134) at  (1,2) {$(-1,3,4)$};
        \node (204) at  (3,2) {$(2,0,4)$};
        \node (042) at  (5,2) {$(0,4,2)$};
        \path (213) edge node {} (231);
        \path (213) edge node {} (015);
        \path (132) edge node {} (312);
        \path (132) edge node {} (-134);
        \path (024) edge node {} (204);
        \path (024) edge node {} (042);
        \node (-235) at  (-8,3) {$(-2,3,5)$};
        \node (051) at  (-6,3) {$(0,5,1)$};
        \node (321) at  (-4,3) {$(3,2,1)$};
        \node (-116) at  (-2,3) {$(-1,1,6)$};
        \node (105) at  (0,3) {$(1,0,5)$};
        \node (3-14) at  (2,3) {$(3,-1,4)$};
        \node (-143) at  (4,3) {$(-1,4,3)$};
        \node (240) at  (6,3) {$(2,4,0)$};
        \node (402) at  (8,3) {$(4,0,2)$};
        \path (231) edge node {} (-235);
        \path (231) edge node {} (321);
        \path (015) edge node {} (051);
        \path (015) edge node {} (105);
        \path (312) edge node {} (321);
        \path (312) edge node {} (-116);
        \path (-134) edge node {} (3-14);
        \path (-134) edge node {} (-143);
        \path (204) edge node {} (105);
        \path (204) edge node {} (240);
        \path (042) edge node {} (-143);
        \path (042) edge node {} (402);
    \end{tikzpicture}}
    \caption{The lower order ideal of the weak order on $\affineS_3$ consisting of permutations with length at most $3$.}
    \label{fig:weakbruhatexample}
\end{figure}

Bj\"orner-Wachs characterized the set of all linear extensions of any poset with $n$ elements in \cite{Bjorner_Wachs_1991} using the Hasse diagram of the weak order on the symmetric group $\S_n$.  The following statement is a direct consequence of their theorem.

\begin{theorem}[{\cite[Thm. 6.6(i)]{Bjorner_Wachs_1991}}]
\label{lem:linear_extensions} Let $P$ be a finite poset. Any two
linear extensions of $P$ can be obtained from one another by a
sequence of adjacent transpositions between elements that are not
comparable in $P$.
\end{theorem}

\subsection{Higher Bruhat Orders}\label{subsection:higher-bruhat-order-intro}

In this subsection, the higher Bruhat
orders are reviewed according to the definitions of Manin-Schechtman
\cite{manin_schechtman_1989} and Ziegler \cite{ziegler_1993}. Given an
unordered set $X = \{x_1, x_2, \ldots, x_n\}$, let $(x_1, x_2, \ldots,
x_n)$ denote $X$ as an ordered list. When $(x_1, x_2, \ldots, x_n)$
are integers in increasing order, the notation $[x_1, x_2, \ldots,
x_n]$ will be used to emphasize that the elements are listed in
increasing order.  When unambiguous in examples, the brackets and
commas are omitted from an ordered set $X=[x_1,\dots,x_n]$ to just
$x_1 \dots x_n$.  In order to avoid ambiguity between $[n]$ as a set
of positive integers and $[n]$ as an ordered set of size one, the
latter will be written as the singleton set $\{n\}$. For any
nonnegative integer $k$, we henceforth revise the notation
$\binom{[n]}{k}$ to mean the collection of ordered subsets of $k$
distinct elements in $[n]$, ordered in increasing order. For example,
$\binom{[4]}{2} = \{[1,2], [1,3], [1,4], [2,3], [2,4], [3,4]\}$. By
convention, $[n] = \emptyset$ if $n = 0$, $\binom{[n]}{0} =
\{\emptyset\}$ for all $n \ge 0$, and if $k$ is a positive integer
such that $k > n$, then $\binom{[n]}{k} = \emptyset$.

 For an ordered set $X = [x_1, \ldots, x_k] \in \binom{[n]}{k}$ and a sequence of positive integers $1 \le i_1 < i_2 < \cdots < i_j \le k$, the notation $X_{i_1, \ldots, i_j}$ denotes the ordered set $X \setminus \{x_{i_1}, \ldots, x_{i_j}\}$. For example, if $X = [1, 5, 6, 8, 9]$, then $X_{3,5} = [1,5,8]$. The \defn{packet} of an ordered set $X$ is the unordered set of all size $(k-1)$ subsets of $X$, denoted $P(X) = \{X_1, X_2, \ldots, X_k\}$. For a total order $\rho$ on a set $\mathcal{S}$ and a subset $I \subseteq \mathcal{S}$, the \defn{restriction} $\rho|_I$ is the total order on $I$ inherited from $\rho$.
 
 The \defn{lexicographic} (lex) order on $\binom{[n]}{k}$ is the total
order on $\binom{[n]}{k}$ where, for any two distinct elements $[i_1,
\ldots, i_k]$, $[j_1, \ldots, j_k] \in \binom{[n]}{k}$, $[i_1, \ldots,
i_k] < [j_1, \ldots, j_k]$ if and only if there exists $r \in [k]$
such that $i_r < j_r$, and $i_1 = j_1$, $i_2 = j_2$, $\cdots$,
$i_{r-1} = j_{r-1}$. The \defn{antilexicographic} (antilex) order is
the total order on $\binom{[n]}{k}$ obtained by reversing the lex
order. For a packet $P(X)$, the lex order is $(X_k, X_{k-1}, \ldots,
X_1)$, and the antilex order is $(X_1, X_2, \ldots, X_k)$. Notice that
the indices of $X_k, X_{k-1}, \ldots, X_1$ are in decreasing order for
the lex order on $P(X)$ since omitting the largest element $x_k$
corresponds to the lexicographically first element $X_k = [x_1, x_2,
\ldots, x_{k-1}]$. A \defn{prefix} of $P(X)$ is an unordered subset of
the form $\{X_k, X_{k-1}, \ldots, X_i \}$, and a \defn{suffix} of
$P(X)$ is an unordered subset of the form $\{X_i, X_{i-1}, \ldots, X_1
\}$.  The empty set and $P(X)$ are both prefixes and suffixes of
$P(X)$. 

\begin{example}
    The lex order on $\binom{[4]}{2}$ is the total order
    \[
    [1,2] < [1,3] < [1,4] < [2,3] < [2,4] < [3,4],
    \]
    which is identified with the ordered list 
    \[
    ([1,2], [1,3], [1,4], [2,3], [2,4], [3,4]).
    \]
 Thus, the lex order on $\binom{[4]}{2}$ may be written compactly as $(12, 13, 14, 23, 24, 34)$. The antilex order on $\binom{[4]}{2}$ is $(34, 24, 23, 14, 13, 12)$. The restriction of the antilex order on $\binom{[4]}{2}$ to the packet $P(123)$ is $(23, 13, 12)$.    
\end{example}

\begin{definition}[\cite{manin_schechtman_1989}]\label{def:admissible.MS}
A total order $\rho$ on $\binom{[n]}{k}$ is \defn{admissible} if the
restriction $\rho|_{P(X)}$ is either the lex or antilex order on
$P(X)$ for every $X \in \binom{[n]}{k+1}$.  The set of admissible
orders on $\binom{[n]}{k}$ is denoted $\boldsymbol{\A(n,k)}$. Given an
admissible order $\rho \in \A(n,k)$, its \defn{reversal set} is
determined by the packets in antilex order, namely

\begin{equation}
\boldsymbol{\Rev_{n,k}(\rho)} = \left\{X \in \binom{[n]}{k+1} : \rho|_{P(X)}=(X_1,\dots,X_{k+1})\right\}.
\end{equation}

\end{definition}

The subscript in $\Rev_{n,k}$ is dropped when $k$ and $n$ are clear
from context. For integers $1 \le k \le n$, the lex (resp. antilex)
order on $\binom{[n]}{k}$ is always an admissible order in $\A(n,k)$
since the restriction to any packet is always the lex (resp. antilex)
order on the packet. The reversal set of the lex order on
$\binom{[n]}{k}$ is the empty set since there are no packets in
antilex order. The reversal set of the antilex order on
$\binom{[n]}{k}$ is $\binom{[n]}{k+1}$ since the restriction to every
packet is the antilex order.

\begin{remark}\label{rem:MS.inv.set}
Note, what we call the reversal set here is the same as what
Manin-Schechtman called the inversion set of an admissible order.  We
changed the name to differentiate this concept from the $k$-inversion
sets we will define for classical and affine permutations in
\Cref{def:k.inversion.Sn} and \Cref{def:k.inversion.affine}.  Both
concepts generalize the inversion set of the longest permutation
$w_{0}$ in $S_{n}$.
\end{remark}

\begin{example}
    Consider the total order
    \begin{equation}
    \label{equation:example-admissible-order}
    \rho = (23, 24, 25, 45, 13, 15, 35, 14, 34, 12).
    \end{equation}
    The restriction of $\rho$ to the packet $P(145) = \{14, 15, 45\}$ is $\rho|_{P(145)} = (45, 15, 14)$, which is the antilex order on $P(145)$. One can check that for every $X \in \binom{[5]}{3}$, the restriction of $\rho$ to $P(X)$ is either the lex or antilex order on $P(X)$. Thus, $\rho$ is in $\A(5,2)$. Its reversal set is
    \begin{equation}
        \label{equation:example-reversal-set}
        \Rev(\rho) = \{123, 124, 125, 145, 345\}.
    \end{equation}
\end{example}

\begin{example}
    As a nonexample, consider transposing $13$ and $15$ in the total order $\rho$ of \eqref{equation:example-admissible-order}. This yields the total order
    \begin{equation}
        \tau = (23, 24, 25, 45, 15, 13, 35, 14, 34, 12).
    \end{equation}
    The restriction of $\tau$ to the packet $P(135) = \{13, 15, 35\}$ is $\tau|_{P(135)} = (15, 13, 35)$, which is neither the lex order nor the antilex order on $P(135)$. Thus, $\tau$ is \emph{not} an admissible order.

    Note that $P(13) \cap P(15) \neq \emptyset$. However, $P(45) \cap P(13) = \emptyset$ and one can check that transposing $45$ and $13$ in $\rho$ does yield an admissible total order. 
\end{example}

Subsets $X,Y \in \binom{[n]}{k}$ are said to \defn{commute} if they do
not belong to a common packet $P(Z)$ for any $Z \in \binom{[n]}{k+1}$.
If $\rho = (\rho_1, \ldots, \rho_{\binom{n}{k}}) \in \A(n,k)$ and
$\rho_i, \rho_{i+1}$ commute, then the total order obtained by
transposing $\rho_i$ and $\rho_{i+1}$ is also an admissible order. Two
admissible orders $\rho, \rho' \in \A(n,k)$ are \defn{commutation
equivalent}, denoted $\boldsymbol{\rho \sim \rho'}$, if $\rho'$ can be
obtained from $\rho$ by a sequence of commutations. The
\defn{commutation class} of an admissible order $\rho \in \A(n,k)$,
denoted $\boldsymbol{[\rho]}$, is the set of admissible orders that
are commutation equivalent to $\rho$. Note that reversal sets are
invariant under commutations, so $\Rev([\rho])$ is well-defined for a
commutation class $[\rho]$. 

Let $\rho = (\rho_1, \ldots, \rho_{\binom{n}{k}}) \in \A(n,k)$ and $X
\in \binom{[n]}{k+1}$ be chosen such that $\rho|_{P(X)}$ is a
saturated chain $(\rho_i, \rho_{i+1}, \ldots, \rho_{i+k})$ of
$\rho$. Reversing the order only on the saturated chain yields a new total order 
\begin{equation}
    \sigma = (\rho_1, \ldots, \rho_{i-1}, \rho_{i+k}, \ldots, \rho_i, \rho_{i+k+1}, \ldots, \rho_{\binom{n}{k}}).
\end{equation} If $\rho|_{P(X)}$ is the lex (resp. antilex) order on $P(X)$, then $\sigma$ is said to be obtained from $\rho$ by a \defn{packet flip of $P(X)$ from lex-to-antilex} (resp. antilex-to-lex). One can observe that the new total order $\sigma$ is also an admissible order.
\begin{definition}[\cite{manin_schechtman_1989}]
    \label{definition:higher-bruhat-order-1}
    For integers $1 \le k \le n$, the \defn{higher Bruhat order} $\boldsymbol{\B(n,k)}$ is the poset on commutation classes $\A(n,k)/\hspace{-0.2em}\sim$ obtained by taking the transitive closure of the cover relations where $[\rho] \lessdot [\sigma]$ if and only if there exist $\rho' \in [\rho]$ and $\sigma' \in [\sigma]$ such that $\sigma'$ is obtained from $\rho'$ by a packet flip from lex-to-antilex.
\end{definition}

\begin{example}
\label{example:total-order-reversal-set}
 The total order $\rho = (23, 13, 24, 14, 12, 34)$ is an admissible order in $\A(4,2)$ with reversal set $\Rev(\rho) = \{123, 124\}$. The commutation class $[\rho]$ is
 \[
    \{(23, 24, 13, 14, 12, 34), (23, 13, 24, 14, 34, 12),
    (23, 24, 13, 14, 34, 12), (23, 13, 24, 14, 12, 34)\},
 \]
 obtained by commuting only $(13, 24)$, only $(12, 34)$, both, or
 neither. The admissible order \[\sigma = (23, 24, 34, 14, 13, 12)\]is
 obtained from $(23, 24, 13, 14, 34, 12) \in [\rho]$ by a packet flip
 of $P(134) = \{13, 14, 34\}$ from lex-to-antilex. Thus, $[\rho]
 \lessdot [\sigma]$ in $\B(4,2)$. The Hasse diagram of $\B(4,2)$ is
 depicted on the left in \cref{fig:b-4-2} with the corresponding
 reversal sets on the right, ordered by single step inclusion.
\end{example}

 \begin{theorem}[(Manin-Schechtman)]\cite[Theorem 3]{manin_schechtman_1989}\label{thm:MS}
    Let $1\leq k\leq n$.
    \begin{enumerate}[label=(\alph*)]
        \item $\B(n,k)$ is a ranked poset with rank function
        $\rev([\rho])=|\Rev(\rho)|.$
        \item $\B(n,k)$ has a unique minimal element with reversal set
	$\emptyset$ and a unique maximal element with reversal set
	$\binom{[n]}{k+1}$.
        \item \label{item:unique.rev.set} Every element in $\B(n,k)$ is uniquely determined by its reversal set. 
        \item \label{item:bij} There is a natural bijection between the maximal chains
	of $\B(n,k)$ and the elements of $\A(n,k+1)$ given by mapping
	a maximal chain $([\rho_{0}],[\rho_{1}],\dots, [\rho_{\binom{n}{k+1}}])$ of
	$\B(n,k)$ to the total order on 	$\binom{[n]}{k+1}$
	given by the unique elements in the list of sets 
\[
(\Rev (\rho_{1}) \setminus \Rev (\rho_{0}) , \Rev
(\rho_{2})\setminus \Rev (\rho_{1}), \dots , \Rev (\rho_{\binom{n
}{k+1}})\setminus \Rev (\rho_{\binom{n}{k+1}-1})).
\]
    \end{enumerate}
\end{theorem}

By \Cref{thm:MS} Part~\ref{item:bij}, observe that $\Rev(\rho) \cap
P(X)$ must be a prefix or suffix of lex order on $P(X)$ for all $\rho
\in \A(n,k)$ and $X\in \binom{[n] }{k+2}$ by the definition of
admissible.  This observation led Ziegler to make the following
definition. Motivated by the Manin-Schechtman \cref{thm:MS} and
Ziegler's definition below, we will refer to this property as the
\defn{MSZ Condition}.

 \begin{figure}[b]
     \centering
\scalebox{1.0}{
            \begin{tikzpicture}[every node/.style={font=\small}]
            % Left hexagonal structure
            \node (a1) at (0, 2.5) {$[(34, 24, 23, 14, 13, 12)]$};
            \node (a2) at (-2,1.25) {$[(23, 24, 34, 14, 13, 12)]$};
            \node (a3) at (2,1.25) {$[(34, 24, 14, 12, 13, 23)]$};
            \node (a4) at (-2,0) {$[(23, 13, 24, 14, 12, 34)]$};
            \node (a5) at (2,0) {$[(12, 34, 14, 13, 24, 23)]$};
            \node (a6) at (-2,-1.25) {$[(23, 13, 12, 14, 24, 34)]$};
            \node (a7) at (2,-1.25) {$[(12, 13, 14, 34, 24, 23)]$};
            \node (a8) at (0,-2.5) {$[(12, 13, 14, 23, 24, 34)]$};
            
            % Draw lines for the left hexagon
            \draw (a1) -- (a2);
            \draw (a1) -- (a3);
            \draw (a2) -- (a4);
            \draw (a3) -- (a5);
            \draw (a4) -- (a6);
            \draw (a5) -- (a7);
            \draw (a6) -- (a8);
            \draw (a7) -- (a8);
            \end{tikzpicture}
        }
        \scalebox{1.0}{
            \begin{tikzpicture}[every node/.style={font=\small}]
            % Left hexagonal structure
            \node (a1) at (0, 2.5) {$\{123, 124, 134, 234\}$};
            \node (a2) at (-2,1.25) {$\{123, 124, 134\}$};
            \node (a3) at (2,1.25) {$\{124, 134, 234\}$};
            \node (a4) at (-2,0) {$\{123, 124\}$};
            \node (a5) at (2,0) {$\{134, 234\}$};
            \node (a6) at (-2,-1.25) {$\{123\}$};
            \node (a7) at (2,-1.25) {$\{234\}$};
            \node (a8) at (0,-2.5) {$\emptyset$};
            
            % Draw lines for the left hexagon
            \draw (a1) -- (a2);
            \draw (a1) -- (a3);
            \draw (a2) -- (a4);
            \draw (a3) -- (a5);
            \draw (a4) -- (a6);
            \draw (a5) -- (a7);
            \draw (a6) -- (a8);
            \draw (a7) -- (a8);
            \end{tikzpicture}
        }

    \caption{The partial orders $\B(4,2)$ on the left and $\C(4,3)$ on the right.}
     \label{fig:b-4-2}
 \end{figure}
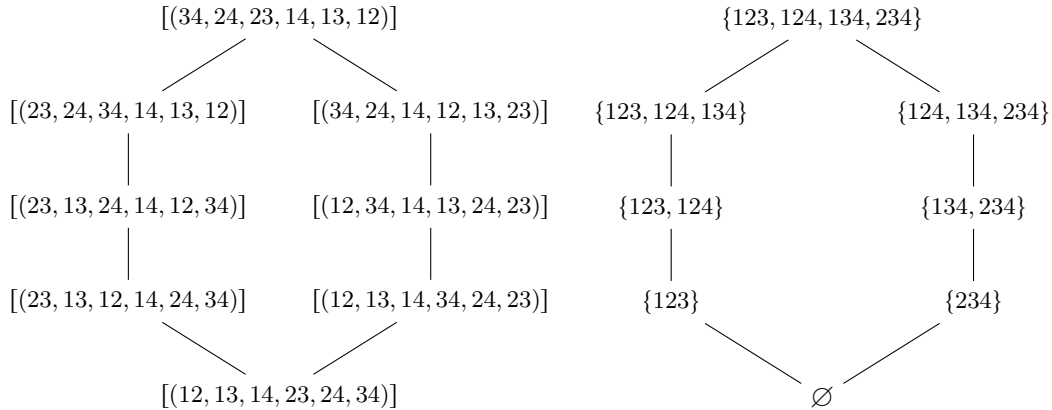

\begin{definition}[\cite{ziegler_1993}]\label{def:consistent.Z}
For positive integers $k,n$, a subset $I \subseteq \binom{[n]}{k}$ is
\defn{consistent} if the \defn{MSZ condition} holds:
\begin{quote}
     $P(X) \cap I$ is a prefix or a suffix of $P(X)$ for every $X \in \binom{[n]}{k+1}$.
\end{quote}
The poset of all consistent subsets of $\binom{[n]}{k}$ ordered by single step inclusion is denoted $\boldsymbol{\C(n,k)}$.
\end{definition}

\begin{example}
    One can check that the reversal set $\Rev(\rho)$ in \eqref{equation:example-reversal-set} is a consistent subset in $\C(5,3)$ by considering $\Rev(\rho) \cap P(X)$ for all $X \in \binom{[5]}{4}$. For example, the intersection $\Rev(\rho) \cap P(1234) = \{123, 124 \}$ is a prefix of $P(1234)$, and the intersection $\Rev(\rho) \cap P(1345) = \{145, 345 \}$ is a suffix of $P(1345)$ in lex order. The remaining 3 subsets are left to the reader.
\end{example}

Ziegler also gave an alternative method of checking the MSZ
condition. While not as efficient, this alternative method will be
useful for theoretical purposes.  The following lemma is closely
related to \Cref{lem:inversion.sets}.

\begin{lemma}\cite[Lemma 2.4]{ziegler_1993}\label{lem:prefix_suffix}
For any $R\subseteq \binom{[n]}{k}$ and $X\in \binom{[n]}{k+1}$, the following are equivalent.
\begin{enumerate}[label=(\alph*)]
	\item The intersection $R\cap P(X)$ is a prefix or suffix of $P(X)$. 
	\item For any  $1\leq h<i<j\leq k+1$, the intersection of $R$ with $\{X_h,X_i,X_j\}$ is neither $\{X_h,X_j\}$ nor $\{X_i\}$.
\end{enumerate}
\end{lemma}

Total orders on $\binom{[n]}{1}$ are the admissible orders
$\A(n,1)$. Such total orders are easily identified with permutations
in $\S_n$. Under this identification, each permutation is in its own
commutation class, and packet flips correspond to transposition of
adjacent values in the permutation.  Hence, $\B(n,1)$ is isomorphic as
a poset to $\S_n$ under the weak order.  Furthermore, reversal sets of
admissible orders in $\A(n,1)$ correspond to inversion sets of
permutations in $\S_n$, which are the consistent subsets of
$\binom{[n]}{2}$ by \Cref{lem:inversion.sets}. Since the weak order is
given by single step inclusion on inversion sets, $\C(n,2)$ is also
isomorphic to the weak order on $\S_{n}$ and $\B(n,1)$.  The
isomorphism between $\B(n,1)$ and $\C(n,2)$ motivated the following
theorem due to Ziegler.

\begin{theorem}\cite[Theorem~4.1]{ziegler_1993}\label{thm:ziegler}
    For integers  $1\leq k\leq n$, the map $[\rho] \mapsto \Rev(\rho)$ is a poset isomorphism between $\B(n,k)$ and $\C(n,k+1)$.
\end{theorem}

\begin{example}
    The Hasse diagram of $\C(4,3)$ is drawn on the right in \cref{fig:b-4-2}. Per \cref{thm:ziegler}, $\B(4,2)$ and $\C(4,3)$ are isomorphic as posets, with the isomorphism sending $[\rho] \mapsto \Rev([\rho])$, which can be seen from the figure.
\end{example}

\begin{remark}
Ziegler \cite[Theorem 4.5]{ziegler_1993} showed that the partial order
on consistent sets in $\C(n,k)$ is equivalent to inclusion order when
$k\leq 2$ or $n-k\leq 3$. Subsequent work by Felsner and Weil
\cite[Theorem 1]{felsnerweil} showed that the partial order on
$\C(n,3)$ is the same as the poset ordered by inclusion. However,
Ziegler also showed the partial order on $\C(8,4)$ is not the same as
inclusion order.  Thus, inclusion and single step inclusion
do not generate the same partial order on consistent sets in all
cases.  
\end{remark}

\section{Higher Bruhat Orders for Permutations in Symmetric Groups}\label{sec:non-longest}

In this section, we generalize the admissible orders and higher
Bruhat orders given in \cref{subsection:higher-bruhat-order-intro} by defining
$\A_w(n,k)$, $\B_w(n,k)$, and $\C_w(n,k)$ for arbitrary $w\in \S_n$.
The classical $\A(n,k)$ and $\B(n,k)$ are constructed using total
orders on ${[n]\choose k}$ in
\Cref{subsection:higher-bruhat-order-intro}.  The set ${[n]\choose k}$
can also be described as the set of ordered $k$-tuples $[x_1,\cdots
,x_k]$ satisfying $w_0^{-1}(x_1)>\ldots>w_0^{-1}(x_k)$, where $w_0\in
\S_n$ is the longest permutation. This second description motivates us
to generalize the notion of an inversion for $w \in \S_{n}$ to larger
subsets of $[n]$ which appear in reverse order in the window 
notation for $w$, which we call $k$-inversions for $w$.  We will
define a poset on the $k$-inversions for $w$ called the permanent poset $\P_{w}(n,k)$. We then define the admissible orders for $w$ to be linear extensions of $\P_{w}(n,k)$ such that
every full-packet must again be in lex or antilex order.  The notions
of higher Bruhat orders and consistent sets also depend on $\P_{w}(n,k)$.  Once the definitions
and basic properties are presented, we prove \cref{thm:nonlongest}.
Throughout this section assume $k,n\in \mathbb{Z}$ are positive integers and $w \in \S_{n}$,
unless otherwise specified.  

\subsection{Higher Order Inversions and Quasi-inversions}

Recall from \cref{subsection:higher-bruhat-order-intro} that $\binom{[n]}{k}$
denotes the collection of ordered subsets $[x_1,\dots , x_k]$ of $k$
distinct elements in $[n]$ listed in increasing order.  In
this notation, $[i,j] \in \binom{[n]}{2}$ is an inversion for $w$
whenever $w^{-1}(i)>w^{-1}(j)$.  More generally, we consider the size
$k$-subsets appearing in reverse order for $w$.  

\begin{definition}\label{def:k.inversion.Sn}
We say $X=[x_1,\ldots,x_k] \in
\binom{[n]}{k}$ is a $\boldsymbol{k}$-\defn{inversion} for
$w\in \S_n$ provided $w^{-1}(x_1)>\cdots >w^{-1}(x_k)$.  The
\defn{$k$-inversion set} of $w$ is defined to 
be
\[\boldsymbol{\Inv_k(w)}=\left\{[x_1,\ldots,x_k] \in \binom{[n]}{k}:
w^{-1}(x_1) > \cdots > w^{-1}(x_k) \right\}.
\]
\end{definition}

Note that for $k>n$, we have that $\Inv_{k}(w)=\emptyset$.  For $k=0$,
$\Inv_{0}(w)=\{\emptyset \}$.  For $k=1$, every $\{i\} \in \binom{[n]
}{1}$ is a 1-inversion of $w \in \S_{n}$.  For $k>1$, the next lemma
follows directly from the definition.  We will use it later on with no
further comment.   

\begin{lemma}\label{lem:adjacent.pairs}
If $X=[x_{1},\dotsc, x_{k}] \in \binom{[n]}{k}$ and $k>1$, then the following are equivalent.
\begin{enumerate}[label=(\alph*)]
\item $X \in \Inv_k(w)$.
\item For every subsequence $1\leq i_{1}<i_{2}<\cdots <i_{j}\leq k$,
the subset $[x_{i_{1}},x_{i_{2}},\dotsc , x_{i_{j}}] \in
\Inv_{j}(w)$.  
\item  For all $1\leq i<k$, the adjacent pair $[x_{i},x_{i+1}] \in \Inv_{2}(w)$. 
\end{enumerate}
\end{lemma}

If $X \in \Inv_{k+1}(w)$, then $P(X) \subseteq \Inv_{k}(w)$, so we
call $P(X)$ a \defn{full-packet} for $w$.  For $k>1$ and $X\in
{[n]\choose k+1}$ the intersection $P(X)\cap \Inv_k(w)$ does not need
to be the entire packet. Instead, $P(X)\cap \Inv_k(w)$ can be
characterized using one of the following four cases.

\begin{lemma}\label{lem:intersection_size}
     For any $X \in \binom{[n]}{k+1}$, the intersection $P(X)\cap
\Inv_k(w)$ is one of the
following: 
\begin{enumerate}[label=(\alph*)] 
    \item $\emptyset$, 
    \item $\{X_i\}$ for some $i\in [k+1]$, 
    \item $\{X_i,X_{i+1}\}$ for some
    $1\leq i\leq k$, or 
    \item $P(X)$. 
\end{enumerate}
Thus, the possible intersection sizes are $\{0,1,2,k+1 \}$. 
\end{lemma}

\begin{proof}
If $P(X)\cap \Inv_k(w)$ is among the cases in (a),(b), or (c), then
the result holds.  So, assume there exist two $X_i, X_j \in \Inv_k(w)$
with $j>i+1$.  Then, every adjacent pair in $X$ forms an inversion for
$w$, so $X \in \Inv_{k+1}(w)$ by \Cref{lem:adjacent.pairs}.
\end{proof}

We need to consider carefully which $(k+1)$-subsets are of type (c) in
\Cref{lem:intersection_size}.  This leads us to consider the following
definition and some straightforward consequences of
\Cref{lem:adjacent.pairs} and \Cref{lem:intersection_size}.

\begin{definition}\label{def:quasi.inversion}
  We say $X=[x_1,\ldots, x_k] \in \binom{[n]}{k}$ is a
\defn{$k$-quasi-inversion} for $w \in \S_n$ if all but exactly one of
the pairs $\{[x_i, x_j] : 1 \le i < j \le k\}$ is a 2-inversion of
$w$.  We use $\boldsymbol{\Qinv_{k}(w)}$ to denote the set of $k$-quasi-inversions for $w$.
\end{definition}

We note that sometimes the $2$-quasi-inversions are called
co-inversions or non-inversions of $w$.  The following corollary
describes $k$-quasi-inversions for all positive integers $k$.  

\begin{corollary}\label{lem:types.of.packets}
The following hold.
\begin{enumerate}[label=(\alph*)]
    \item There are no 1-quasi-inversions for
$w$. 
    \item The set of 2-quasi-inversions for $w$ is the
complement of $\Inv_{2}(w)$ in $\binom{[n]}{2}$.  

\item For $k>2$, the $k$-quasi-inversions are exactly the $k$-subsets
$X=[x_{1},\dots ,x_{k}] \in \binom{[n]}{k}$ such that $P(X)\cap
\Inv_{k-1}(w) = \{X_i,X_{i+1}\}$ for some $i \in [k-1]$.  In this 
case, $[x_{i},x_{i+1}] \not \in \Inv_{2}(w)$ and 
\[
w^{-1}(x_1)>\cdots > w^{-1}(x_{i-1})> w^{-1}(x_{i+1})>w^{-1}(x_i) >
w^{-1}(x_{i+2})> \cdots  > w^{-1}(x_{k}).
\]
\end{enumerate}

\end{corollary}

\begin{example}\label{ex:k.inv}
 Consider $w=(6,4,5,2,3,1)\in \S_6$. The $4$-inversions of $w$ are
given by \[\Inv_4(w)=\{[1,2,4,6], [1,2,5,6],[1,3,4,6],[1,3,5,6]\}.\]
Observe that
$P([1,2,3,4]) \cap \Inv_{3}(w)=\{[1,2,4],[1,3,4]\}$, so $[1,2,3,4]$ is
a $4$-quasi-inversion of $w$. The entire set of $4$-quasi-inversions is \[\{[1,2,3,4],[1,2,3,5],[1,2,3,6],[1,2,4,5],[1,3,4,5],[1,4,5,6],[2,3,4,6],[2,3,5,6],
[2,4,5,6], [3,4,5,6]\}.\] 
The subset $[2,3,4,5]$ is neither an
inversion nor a quasi-inversion for $w$. 
\end{example}

\begin{remark}\label{ex:nonprefixsuffixintersection}
While the size 1 or 2 intersections $\Inv_{2}(w)\cap P(X)$ must be a
prefix or suffix for all $X \in \binom{[n] }{3}$ by
\Cref{lem:inversion.sets}, the same requirement does not need to hold
for the size 1 or 2 intersections $\Inv_{k}(w)\cap P(X)$ for $k>2$.
For example, if $w=(5,2,3,1,4)$, then
$\Inv_{3}(w)=\{[1,2,5],[1,3,5]\}$.  For $X=[1,2,4,5]$,\
$P(X)\cap \Inv_{3}(w)=\{X_{3} \}$.  For $Y=[1,2,3,5]$,\ $P(Y)\cap
\Inv_{3}(w)=\{Y_{2},Y_{3} \}$.   
\end{remark}

\subsection{Permanent Posets for Permutations}\label{sub:permanent.Sn}

Recall from \Cref{sub:sn} that the weak order on $\S_{n}$ is
determined by (single step) inclusion of inversion sets. Furthermore,
the 2-quasi-inversions of $w$ are exactly the pairs in $\binom{[n]
}{2} \setminus \Inv_{2}(w)$. Therefore, if $[i,j]$ is a
2-quasi-inversion for $w$, every permutation $v$ in the interval
$[\mathrm{id},w]$ in weak order has the value $i$ appearing before $j$
in its window notation. Furthermore, every permutation satisfying
these conditions is in $[\mathrm{id},w]$, since their inversion sets
must be contained in $\Inv_{2}(w)$.  Thus, we can identify the
permutations in $[\mathrm{id},w]$ with the total orders of $[n]\cong
\binom{[n] }{1}$ satisfying $i<j$ whenever $[i,j]$ is a
2-quasi-inversion for $w$.  Similarly, reflection orders of $w \in
\S_{n}$ are certain total orders on $\Inv_{2}(w)$ as defined in
\Cref{def:reflection.order}.  By the prefix/suffix condition in
\Cref{cor:characterize.ref.order} Part (a), we know that for each
3-quasi-inversion $X$ for $w$, $P(X)\cap\Inv_2(w)=\{X_{2},X_{3}\}$
implies $X_{3}<X_{2}$ in every reflection order for $w$, and
$P(X)\cap\Inv_2(w)=\{X_{1},X_{2} \}$ implies $X_{1}<X_{2}$ in every
reflection order for $w$.  Note, $X_{2}$ is greater in either case.
The relations imposed by the 2-quasi-inversions and the
3-quasi-inversions are generalized by the following family of posets
and the corresponding parity condition.

\begin{definition}\label{def:permanent.poset.}
Let $\boldsymbol{\P_w(n,k)}$ denote the set $\Inv_{k}(w)$ with order
relations given by the transitive closure of the \defn{quasi-inversion
relations}: if $X \in \binom{[n]}{k+1}$ is a quasi-inversion for $w$
such that $P(X)\cap \Inv_{k}(w) = \{X_i,X_{i+1}\}$, then $X_{i}$ and
$X_{i+1}$ are ordered in $\P_w(n,k)$ depending on the following
\defn{parity condition}:
\begin{itemize}
\item antilex order $X_{i}<_{\P_w(n,k)} X_{i+1}$ if $(k-i)$ is odd, and
\item lex order $X_{i+1}<_{\P_w(n,k)} X_{i}$ if $(k-i)$ is even.
\end{itemize}
Let $\P_{w}(n,0)$ be $\Inv_{0}(w)=\{\emptyset\}$ with no relations.
We call $\boldsymbol{\P_w(n,k)}$ the \defn{permanent poset} for $n$,
$k$, and $w$.
\end{definition}

\begin{remark}\label{rem:easy.pnks}
The poset $\P_w(n,1)$ is generated by the relations $\{i \}<\{j \} $
for all 2-quasi-inversions $[i,j] \in \binom{[n] }{2}\setminus
\Inv_{2}(w)$.  Hence, $\P_w(n,1)$ is isomorphic to the poset on $[n]$
generated by the natural order relations $i<j$ whenever $[i,j]$ is not
an inversion of $w$.  Furthermore, when $k>n$, we know that
$\Inv_k(w)=\emptyset$, so $\P_w(n,k)$ is the empty poset.
\end{remark}

\begin{example}
For $w=(3,4,2,1) \in \S_{4}$ and $k=1$, $\P_{w}(4,1)$ is the poset on
vertices $\{\{1 \},\{2 \},\{3 \},\{4 \} \}$ with the single
quasi-inversion relation $\{3 \}<_{\P_{w}(4,1)}\{4 \}$.  The linear
extensions of $\P_{w}(4,1)$, denoted $\mathscr{L}(\P_{w}(4,1))$,
include $(\{3 \},\{1 \},\{2 \},\{4 \})$, corresponding with the
permutation $(3,1,2,4)$.  In fact, $\mathscr{L}(\P_{w}(4,1))$ is in
natural bijection with the permutations in $\S_{4}$ such that $3$
appears before $4$ in window notation.  Such permutations are exactly
the ones in the weak order interval $[\mathrm{id},w]$ defined by
containment of inversion sets.
\end{example}

We will give some nontrivial examples next.  We will then establish
properties in the case where $1< k\leq n$.

\begin{example}\label{ex:permanent poset}
    Consider $w=(6,4,5,2,3,1)\in \S_6$ from \cref{ex:k.inv}.  The permanent poset on
$\Inv_3(w)$ is determined by the quasi-inversion relations by
\cref{def:permanent.poset.}. The Hasse diagram of $\P_w(6,3)$ is shown
in \cref{fig:permanent.example}. Observe $|\Inv_3(w)|=12$ and
$\P_w(6,3)$ has 4 connected components.  Each of the 10
quasi-inversions of size 4 given in \Cref{ex:k.inv} determines a
covering relation in $\P_w(6,3)$. However, the quasi-inversion
relations are not always covering relations.
    
    \begin{figure}[h!]
\centering 
    \begin{tikzpicture}
        \node (134) at (0,0) {$[1,3,4]$};
        \node (135) at (-1,1) {$[1,3,5]$};
        \node (124) at (1,1) {$[1,2,4]$};
        \node (125) at (0,2) {$[1,2,5]$};
        \node (136) at (3,0) {$[1,3,6]$};
        \node (126) at (3,1) {$[1,2,6]$};
        \node (156) at (5,0) {$[1,5,6]$};
        \node (146) at (5,1) {$[1,4,6]$};
        \node (256) at (8,0) {$[2,5,6]$};
        \node (246) at (7,1) {$[2,4,6]$};
        \node (356) at (9,1) {$[3,5,6]$};
        \node (346) at (8,2) {$[3,4,6]$};
        \path (134) edge node {} (135);
        \path (134) edge node {} (124);
        \path (135) edge node {} (125);
        \path (124) edge node {} (125);
        \path (136) edge node {} (126);
        \path (156) edge node {} (146);
        \path (256) edge node {} (246);
        \path (246) edge node {} (346);
        \path (256) edge node {} (356);
        \path (356) edge node {} (346);
    \end{tikzpicture}
    \caption{The Hasse diagram of $\P_w(6,3)$ for $w=(6,4,5,2,3,1)\in \S_6$.}
    \label{fig:permanent.example}
\end{figure}
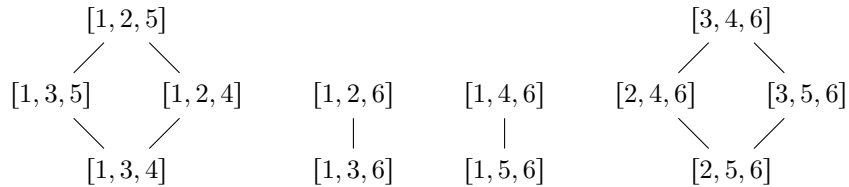
\end{example}

While reflexivity and transitivity of $\leq_{\P_w(n,k)}$ follow by
construction, antisymmetry does not. We will show that antisymmetry
does hold, so that $\P_w(n,k)$ is a poset. Additionally, we will
establish that any full-packet $P(X)\subseteq \Inv_k(w)$ is an
antichain in $\P_w(n,k)$. Our general approach will be to associate a
matrix to each $Y\in \Inv_k(w)$ and show that quasi-inversion
relations imply certain monotonicity properties.

Let $\mathbb{Z}^{2\times k}$ be the $2\times k$ matrices with entries in $\mathbb{Z}$, which we consider as a group under addition. Define $M:\Inv_k(w) \to \mathbb{Z}^{2\times k}$ on each element $Y=[y_1,\ldots,y_k]\in \Inv_k(w)$ by
\begin{equation}\label{eq:inversion_matrix}
    M(Y) = \begin{bmatrix}
    w^{-1}(y_1) & w^{-1}(y_2) & \ldots & w^{-1}(y_{k-1}) & w^{-1}(y_k) \\
    y_1 & y_2 & \ldots &  y_{k-1} &
y_k \end{bmatrix}.
\end{equation}
Note, the top row of $M(Y)$ denotes the positions where the values
$y_{1}<\dots <y_{k}$ occur in the window notation for $w$.  Since
$Y \in \Inv_k(w)$, the top row must be decreasing from left to right.
We will use $A_{i,j}$ to denote the entry in row $i\in [2]$ and column
$j\in [k]$ of $A\in \mathbb{Z}^{2\times k}$. For each $2\leq j\leq k$,
define
$F_j:\mathbb{Z}^{2\times k} \to \mathbb{Z}$ and
$G_j:\mathbb{Z}^{2\times k} \to \mathbb{Z}$ by
\begin{equation}\label{eq:Fi_Gi}
    F_j(A)=(-1)^{k-j}(A_{1,j}-A_{2,j-1})\quad \text{ and }\quad G_j(A)=(-1)^{k-j}(A_{2,j}-A_{1,j-1}).
\end{equation}
These functions calculate the difference between diagonal entries
within two consecutive columns. It is straightforward to show that
these functions are additive group homomorphisms.  We emphasize that $k$ is
used in the definition of $F_{j}$ and $G_{j}$ so it must be understood
from context or determined by the input.  Furthermore, the definition
of $M$ depends on the fixed permutation $w$.   

\begin{example}\label{example:FG_functions}
We revisit $w=(6,4,5,2,3,1)\in \S_6$ from \cref{ex:permanent poset} and $[2,5,6]\in \Inv_3(w)$. The corresponding matrix is
\[M([2,5,6])=\begin{bmatrix}
4 & 3 & 1 \\
2 & 5 & 6
\end{bmatrix}.\] For the functions in \eqref{eq:Fi_Gi} with $k=3$, we
find
\begin{equation*}
    \begin{split}
        F_2\circ M([2,5,6]) & = (-1)^{3-2}(3-2) = -1,\\
        F_3\circ M([2,5,6]) & = (-1)^{3-3}(1-5) = -4,\\
        G_2\circ M([2,5,6]) & = (-1)^{3-2}(5-4) = -1,\\
        G_3\circ M([2,5,6]) & = (-1)^{3-3}(6-3) = 3.
    \end{split}
\end{equation*}
For $k=2$, we observe that $F_{2}$ is decreasing and $G_{2}$ is
increasing on the packet of $[2,5,6]$ in lex order:
\[
F_{2}\circ M([2,5])=1>F_{2}\circ M([2,6])=-1>F_{2}\circ M([5,6])=-4,
\]
\[
G_{2} \circ M ([2,5])=1<G_{2} \circ M([2,6])=2<G_{2}\circ M([5,6])=3.
\]
\end{example}

\begin{lemma}\label{lem:diagonal_differences}
Suppose $2 \le k \le n$. If $X <_{\P_w(n,k)} Y$, then for any $j\in
[k]\setminus \{1\}$, we have that $F_j\circ M(X)\leq F_j\circ M(Y)$
and $G_j\circ M(X)\leq G_j\circ M(Y)$. Furthermore, there exists $j\in
[k]\setminus \{1\}$ such that the inequality is strict for both $F_j$
and $G_j$.
\end{lemma}

\begin{proof}
It suffices to establish the claim for the
quasi-inversion relations that generate $\P_w(n,k)$. Let $X\in \binom{[n]}{k+1}$ be a quasi-inversion for $w$ where
$P(X)\cap \Inv_{k}(w)=\{X_i,X_{i+1}\}$.  Then,
\begin{equation}\label{eq:quasi_matrix_difference}
M(X_i)-M(X_{i+1})=\begin{bmatrix}
    0 & \ldots & 0 & w^{-1}(x_{i+1})-w^{-1}(x_{i}) & 0 & \ldots &0 \\
    0 & \ldots & 0 & x_{i+1}-x_{i}                  & 0 & \ldots &0
    \end{bmatrix},
    \end{equation}
    where the $i^{th}$ column is necessarily nonzero. As $X$ is a quasi-inversion, \cref{lem:types.of.packets} implies that $x_i<x_{i+1}$ and $w^{-1}(x_i)<w^{-1}(x_{i+1})$, so both of the nonzero terms are positive. 
    
    If $i\geq 2$ and $k-i$ is even, we have the quasi-inversion
relation $X_{i+1}<_{\P_w(n,k)} X_{i}$.  Since each $F_{j}$ is an
additive group homomorphism, $F_j\circ M(X_i)-F_j\circ M (X_{i+1})
=F_{j}(M(X_i)-M(X_{i+1}))=0$ for all $j \in [k] \setminus \{1,i, i+1\}$,
\begin{equation}\label{eq:even_quasiinversion} F_i\circ
M(X_i)-F_i\circ M(X_{i+1})= (-1)^{k-i}\cdot
(w^{-1}(x_{i+1})-w^{-1}(x_{i}) - 0)>0,
\end{equation}
and when $i< k$,
\begin{equation}
F_{i+1}\circ M(X_i)-F_{i+1}\circ M(X_{i+1})=(-1)^{k-(i+1)}\cdot
(0-(x_{i+1}-x_{i}))>0.
\end{equation}
 In the case where $i=1$ and
$k-1$ is even, then we omit \eqref{eq:even_quasiinversion}, but the
remaining statements hold. Similarly, when $k-i$ is odd, we have that
$X_{i}<_{\P_w(n,k)}
X_{i+1}$,
\begin{equation}\label{eq:odd_quasiinversion}
F_i\circ M(X_i)-F_i\circ M(X_{i+1})
    =(-1)^{k-i}(w^{-1}(x_{i+1})-w^{-1}(x_{i}) - 0)<0,
\end{equation}
\begin{equation}\label{eq:odd_quasiinversion2}
F_{i+1}\circ M(X_i)-F_{i+1}\circ
M(X_{i+1})=(-1)^{k-(i+1)}(0-(x_{i+1}-x_{i}))<0,
\end{equation}
and $F_j\circ M(X_i)-F_j\circ M(X_{i+1})=0$ for all remaining $j \in
[k] \setminus \{1,i, i+1\}$, where we omit \eqref{eq:odd_quasiinversion}
when $i=1$ and \eqref{eq:odd_quasiinversion2} when $i=k$.  Combined,
we see that each $F_j\circ M$ is monotone increasing with respect to
the quasi-inversion relations in $\P_w(n,k)$, and the values are
defined and strictly increasing for at least one of $j=i$ or $j = i+1$.  Hence,
the statements in the lemma for the $F_{j}$'s both hold.

A similar argument with casework and \eqref{eq:quasi_matrix_difference} establishes the result for the $G_j\circ M$
functions. The details in this case are left to the reader.
\end{proof}

We will next consider properties of the functions $F_{j}\circ M$ and
$G_{j}\circ M$ on full-packets in $\Inv_k(w)$. \cref{example:FG_functions} suggests that these behave differently on full packets. In general, the parity
of $k-j$ determines which is increasing or decreasing on every full-packet in lex order.

\begin{lemma}\label{lem:diagonal_packets}
   Suppose $X\in \Inv_{k+1}(w)$ and $2 \leq j\leq k$. If $k-j$ is
even, then
   \begin{equation}\label{eq:FG_on_packets} \begin{split}
            F_{j}\circ M(X_1)=\ldots=F_{j}\circ M(X_{j-1})<F_{j}\circ M(X_j)<F_{j}\circ M(X_{j+1})=\ldots = F_{j}\circ M(X_{k+1}), \text{ and}\\
    G_{j}\circ M(X_1)=\ldots=G_{j}\circ M(X_{j-1})>G_{j}\circ
    M(X_j)>G_{j}\circ M(X_{j+1})=\ldots = G_{j}\circ M(X_{k+1}). \hspace{.15in}
        \end{split}
    \end{equation}
    If $k-j$ is odd, then the above inequalities are reversed.
\end{lemma}

\begin{proof}
    This follows from a similar argument as in \cref{lem:diagonal_differences}. For $F_{j}\circ M$, we first note that when $i \in [k] \setminus \{j-1,j\}$, the difference $M(X_i)-M(X_{i+1})$ has 0's in columns $j-1$ and $j$, so $F_{j}\circ M(X_i)-F_{j}\circ M(X_{i+1})=0$. 
    
    For the case $i=j-1$, we first express
    \begin{equation}\label{eq:full_packet_difference1}
M(X_{j-1})-M(X_{j})=\begin{bmatrix}
    0 & \ldots & 0 & w^{-1}(x_{j})-w^{-1}(x_{j-1}) & 0 & \ldots &0 \\
    0 & \ldots & 0 & x_{j}-x_{j-1}                  & 0 & \ldots &0
    \end{bmatrix},
    \end{equation}
    where column $j-1$ is nonzero, $w^{-1}(x_{j})-w^{-1}(x_{j-1})<0$,
    and $x_{j}-x_{j-1}>0$ since $X \in \Inv_{k+1}(w)$. From this, we find
    \[F_{j}\circ M(X_{j-1})-F_{j}\circ M(X_{j})=
    F_{j}(M(X_{j-1})- M(X_{j}))= (-1)^{k-j}(0-(x_j-x_{j-1})).\]
    When $k-j$ is even, this is negative, so $F_{j}\circ M(X_{j-1})<F_{j}\circ M(X_j)$. When $k-j$ is odd, we find the opposite inequalities. 
    
    For the case $i=j$, we similarly express 
    \begin{equation}\label{eq:full_packet_difference2}
M(X_j)-M(X_{j+1})=\begin{bmatrix}
    0 & \ldots & 0 & w^{-1}(x_{j+1})-w^{-1}(x_{j}) & 0 & \ldots &0 \\
    0 & \ldots & 0 & x_{j+1}-x_{j}                  & 0 & \ldots& 0
    \end{bmatrix},
    \end{equation}
    where column $j$ is nonzero, $w^{-1}(x_{j+1})-w^{-1}(x_{j})<0$, and $x_{j+1}-x_{j}>0$. We similarly calculate 
    \[F_{j}\circ M(X_{j})-F_{j}\circ M(X_{j+1})=(-1)^{k-j}(w^{-1}(x_{j+1})-w^{-1}(x_{j})-0).\]
    When $k-j$ is even, this is negative, so $F_{j}\circ M(X_j)<F_{j}\circ M(X_{j+1})$. When $k-j$ is odd, this is instead positive. 
    
    Combined, we conclude that when $k-j$ is even, \eqref{eq:FG_on_packets} holds for $F_{j}\circ M$, and when $k-j$ is odd, the reverse relations hold. A similar argument establishes the corresponding results for $G_{j}\circ M$.
\end{proof}

\begin{theorem}\label{lem:poset.packet.antichain}
For any integers $1\leq k\leq n$ and $w\in \S_n$, we have that $\P_{w}(n,k)$ is a poset.  Furthermore, if $X\in
\Inv_{k+1}(w)$, then the elements in $P(X)$ form an antichain in
$\P_w(n,k)$.
\end{theorem}

\begin{proof}
By \Cref{rem:easy.pnks}, $\P_{w}(n,1)$ is a poset on $[n]$ and both the identity
permutation and $w$ are linear extensions of $\P_w(n,1)$.  Therefore,
if $[i,j] \in \Inv_2(w)$, then $\{i\}$ and $\{j\}$ must be
incomparable in $\P_w(n,1)$ and $P([i,j])=\{\{i \},\{j \} \}$ is an antichain. Thus, it suffices to consider the case $2
\le k \le n$.

As mentioned previously, reflexivity and transitivity hold on
$\P_w(n,k)$ by construction. For antisymmetry, define the function
$F=\sum_{j=2}^k F_j\circ M$. By \cref{lem:diagonal_differences},
$Y<_{\P_w(n,k)}Z$ implies $F(Y)<F(Z)$, so $Z\not<_{\P_w(n,k)}Y$.  We
conclude that antisymmetry holds, so $\P_w(n,k)$ is a poset.  Observe
the function $G=\sum_{j=2}^k G_j\circ M$ could have been used to prove
antisymmetry as well since it also has the property
$Y<_{\P_w(n,k)}Z$ implies $G(Y)<G(Z)$.

    For the antichain claim, we let $1\leq i<j\leq k+1$, and consider
$X_i,X_j\in P(X)$.  When $j\leq k$, \cref{lem:diagonal_packets} implies that depending on the parity of $k-j$, either
\[
F_{j}\circ M(X_{i})> F_{j}\circ M(X_{j})\text{ and } G_{j}\circ M(X_{i})<G_{j}\circ M(X_{j}) 
\]
or
\[
  G_{j}\circ M(X_{i})>G_{j}\circ M(X_{j}) \text{ and }F_{j}\circ M(X_{i})<F_{j}\circ M(X_{j}).
\]
In either case, these conditions imply
neither $X_i<_{\P_w(n,k)} X_j$ nor $X_i>_{\P_w(n,k)} X_j$ is 
possible, since $F_{j}\circ M$ and $G_{j}\circ M$ are weakly order
preserving on $\P_w(n,k)$ by \cref{lem:diagonal_differences}. When $j=k+1$, 
\[
G_{k}\circ M(X_{i})>G_{k}\circ M(X_{j})\text{ and }  F_{k}\circ M(X_{i})<F_{k}\circ M(X_{j}),
\]
and we similarly conclude neither $X_i<_{\P_w(n,k)} X_j$ nor $X_i>_{\P_w(n,k)} X_j$ are
possible. As $i,j\in [k+1]$ are arbitrary distinct elements, we conclude that 
$P(X)$ is an antichain in $\P_w(n,k)$. 
\end{proof}

\begin{lemma}\label{lem:permanent_on_semipackets}
Let $Z\in \Qinv_{k+2}(w)$ with $P(Z)\cap
\Inv_{k+1}(w)=\{Z_i,Z_{i+1}\}$, and let $h\in [k+2]\setminus
\{i,i+1\}$.  If $k-i$ is even, then we have the quasi-inversion
relations
\[
\begin{array}{lllll}
  (Z_{i})_h   & <_{\P_w(n,k)} &  (Z_{i+1})_{h} & \text{if} & h<i, \text{ and}\\ 

  (Z_{i+1})_{h-1} & <_{\P_w(n,k)} & (Z_i)_{h-1} & \text{if} & h>i+1.
\end{array}
\]
If $k-i$ is odd, then the relations above are reversed. 
\end{lemma}
\begin{proof}
Since $Z$ is a quasi-inversion of $w$ with $P(Z)\cap
\Inv_{k+1}(w)=\{Z_i,Z_{i+1}\}$ and $h\in [k+2]\setminus \{i,i+1\}$, we
know $Z_h$ is a quasi-inversion of $w$ by \cref{lem:types.of.packets}.
If $h<i$, then $(Z_i)_h=(Z_h)_{i-1}$, \ $(Z_{i+1})_h=(Z_h)_{i}$,  and
both are $k$-inversions for $w$ since $Z_{i}$ and $Z_{i+1}$ are
$(k+1)$-inversions for $w$.  If in addition $k-i$ is even, then
$k-(i-1)$ is odd and $(Z_h)_{i-1}<_{\P_w(n,k)} (Z_h)_{i}$ by the
quasi-inversion relations from \Cref{def:permanent.poset.}, so
\[
(Z_i)_h= (Z_h)_{i-1} <_{\P_w(n,k)} (Z_h)_{i}=(Z_{i+1})_h.
\]
Similarly, if $h>i+1$ and $k-i$ is even, then
$(Z_{i+1})_{h-1}=(Z_{h})_{i+1},\ (Z_i)_{h-1}=(Z_{h})_i \in \Inv_{k}(w)$,
and
\[
(Z_{i+1})_{h-1}=(Z_{h})_{i+1}<_{\P_w(n,k)} (Z_{h})_{i}=(Z_i)_{h-1}.
\]
If $k-i$ is odd, the quasi-inversion relations are reversed from
those above because of the parity constraint in
\Cref{def:permanent.poset.}.  Hence, $(Z_{i})_h>_{\P_w(n,k)}
(Z_{i+1})_{h}$ for $h<i$ and $(Z_{i+1})_{h-1}>_{\P_w(n,k)} (Z_i)_{h-1}$ for $h>i+1$.
\end{proof}

\subsection{Admissible Orders for Permutations}\label{sub:admissible.orders.Sn}

We define admissible orders on all $w \in \S_n$ by extending the
full-packet condition from \Cref{def:admissible.MS}, and highlight some of
their properties.  The primary modification from the classical
admissible orders in $\A(n,k)$ is a separate condition for
quasi-inversions, which is encoded in the permanent posets.

\begin{definition}\label{def:admissible.BCL}
A total
order $\rho$ on $\Inv_k(w)$ is a $\boldsymbol{k}$-\defn{admissible
order} for $w$ if $\rho$ is a linear extension of $\P_w(n,k)$  and
satisfies the \defn{full-packet condition}:
\begin{quote}
 $\rho|_{P(X)}$ is the full-packet in
	lex order $(X_{k+1},\ldots,X_{2},X_1)$ or in antilex order
	$(X_1,X_2,\ldots,X_{k+1})$ for every $X\in \Inv_{k+1}(w)$.
\end{quote}
    Let $\boldsymbol{\A_w(n,k)}$ be the set of $k$-admissible orders
    on $\Inv_k(w)$.     The \defn{reversal set} of $\rho \in
    \A_w(n,k)$ is
    \[
\boldsymbol{\Rev_{n,k,w}(\rho)}=
\{X \in \Inv_{k+1}(w):  \rho|_{P(X)}=(X_1,X_2,\ldots,X_{k+1})\},
    \]
and the \defn{reversal statistic} is $\boldsymbol{\rev_{n,k,w}(\rho)}$ $
=|\Rev_{n,k,w}(\rho)|$.  Again, we write $\boldsymbol{\Rev(\rho)}$ and
$\boldsymbol{\rev(\rho)}$ when $n,k,w$ are clear from the context.
\end{definition}

When $k>n$, we know that $\Inv_k(w)=\emptyset$, so $\A_w(n,k)$
consists of a single element, the empty list $()$,  with reversal set $\emptyset$. We
will focus on the cases where $1\leq k\leq n$.

\begin{example}\label{ex:admissible order}
    Consider $w=(6,4,5,2,3,1)\in \S_6$ from \cref{ex:permanent poset}.
There are $831600=2\cdot2\cdot\binom{12}{2,2,4,4}$ linear extensions
of $\P_w(6,3)$, but only 1228 of them also satisfy the requirement
that every full-packet must be in lex or antilex order.  Suppressing
the square brackets and internal commas, six of the linear extensions
of $\P_w(6,3)$ are given in \Cref{table:admissible} along with their
reversal sets. Checking the restriction of
each $\rho^i$ to the packet of each inversion in $\Inv_4(w)=\{1246,
1256, 1346, 1356\}$ shows
that all are admissible orders. 

\begin{table}[h]\centering 
    \begin{tabular}{|c|c|}  \hline 
    Admissible order & Reversal Set \\ \hline 
        $\rho^1 = (134, 124, 135, 136, 125, 126, 156, 256, 146, 246, 356, 346)$ & $\emptyset$\\ \hline 
        $\rho^2 = (134, 256, 135, 136, 156, 356, 124, 126, 146, 246, 346, 125)$ & $\{1256\}$\\ \hline 
        $\rho^3 = (134, 256, 135, 136, 156, 356, 246, 146, 126, 124, 346, 125)$ & $\{1246, 1256\}$\\\hline 
        $\rho^4 = (134, 256, 356, 156, 136, 135, 124, 126, 146, 246, 346, 125)$ & $\{1256, 1356\}$\\ \hline 
        $\rho^5 = (134, 256, 356, 156, 136, 135, 246, 146, 126, 124, 346, 125)$ & $\{1246, 1256, 1356\}$\\\hline 
        $\rho^6 = (256, 246, 356, 156, 346, 146, 136, 134, 126, 124, 135, 125)$ & $\Inv_4(w)$ \\\hline 
    \end{tabular}
    \caption{Six $3$-admissible orders for the permutation $w=(6,4,5,2,3,1)\in \S_6$ from \cref{fig:permanent.example}.  Square brackets and internal commas have been suppressed for clarity.}
\label{table:admissible}
\end{table}
\end{example}

\begin{lemma}\label{lem:An1.weak.order}
The admissible orders in $\A_{w}(n,1)$ are
exactly the permutations in the interval $[\mathrm{id},w]$ in the weak
order on $\S_{n}$ if one equates the singleton sets $\{i \}$ with the
value $i$.   
\end{lemma}

\begin{proof}
By \Cref{rem:easy.pnks}, the linear extensions of $\P_{w}(n,1)$ are in
natural bijection with the permutations in $[\mathrm{id},w]$.  If
$\rho$ is a linear extension of $\P_{w}(n,1)$, then $\rho|_{P(X)}$ is
either in lex or antilex order for every $X=[x_{1},x_{2}] \in
\Inv_{2}(w)$ automatically since both $\{x_{1} \}$ and $\{x_{2} \}$
must appear in $\rho$.  Therefore, linear extensions of $\P_{w}(n,1)$
are exactly the admissible orders in $\A_{w}(n,1)$.
\end{proof}

\begin{lemma}\label{lem:An2.weak.order}
The admissible orders in $\A_{w}(n,2)$ are the
reflection orders for $w$.
\end{lemma}

\begin{proof}
When $k=2$, \Cref{def:admissible.BCL} can be restated as follows. Each $\rho\in \A_w(n,2)$ is a total order on $\Inv_2(w)$ with the following conditions for any $1\leq x<y<z \leq n$.
    \begin{itemize}
        \item If $[x,y,z]\in \Inv_3(w)$, then the restriction of $\rho$ to $P([x,y,z])$ is the lex or antilex order.
        \item If $P([x,y,z])\cap \Inv_2(w)=\{[x,y],[x,z]\}$, then $[x,y]<_{\rho} [x,z]$. 
        \item If $P([x,y,z])\cap \Inv_2(w)=\{[x,z],[y,z]\}$, then
	$[y,z]<_{\rho} [x,z]$.
    \end{itemize}
    The proof now follows by comparing this with
\Cref{cor:characterize.ref.order}(a).
\end{proof}

Recall from \Cref{subsection:higher-bruhat-order-intro} that any two
elements $X,Y\in {[n]\choose k}$ commute if they do not belong to a
common packet $P(Z)$ for any $Z \in \binom{[n]}{k+1}$.  Note that
$P(X) \cap P(Y)=\emptyset$ if and only if $X$ and $Y$ do not belong to
a common packet.  This is straightforward to check for $k>1$.  When
$k=1$ and $x,y \in [n]$ are distinct, they must belong to a common
packet and $P(\{x \})\cap P(\{y \})=\{\emptyset \} \neq \emptyset$.
This concept generalizes to $k$-inversions of $w \in \S_{n}$ as
follows.

\begin{definition}\label{def:commutes.Sn}
 For $X,Y \in \Inv_{k}(w)$, we say $X,Y$ \defn{commute with respect
to} $\boldsymbol{w}$ provided they do not belong to a common packet
and $X,Y$ are incomparable in $\P_{w}(n,k)$. Two total orders on
$\Inv_{k}(w)$ are related by a \defn{commutation with respect to}
$\boldsymbol{w}$ provided that, as lists, they differ by a single
adjacent transposition between two elements that commute with respect
to $w$.  Let $\sim_{w}$ denote the equivalence relation on total
orders of $\Inv_{k}(w)$ generated by commutations with respect to $w$.
\end{definition}

Since the permutation $w$ is fixed throughout this section, all
commutations discussed below are with respect to the fixed permutation
$w$.  The commutation class of a total order $\rho$ on $\Inv_{k}(w)$
will be denoted $\boldsymbol{[\rho]}$ as in the classical case.  The
next lemma shows $\A_w(n,k)$ is closed under commutations.

\begin{lemma}\label{lem:commutations}
 If $\rho \in \A_w(n,k)$ and $\sigma \sim_{w} \rho$, then $\sigma \in
\A_w(n,k)$ and $\Rev(\rho) = \Rev(\sigma)$.
\end{lemma}

\begin{proof}
Since all total orders in the same equivalence class with $\rho$ are
connected by commutations, it suffices to assume $\sigma$ is obtained
from $\rho$ by interchanging an adjacent pair $X,Y \in \Inv_{k}(w)$
that commute with respect to $w$. Since $X$ and $Y$ commute, we know
they are incomparable elements in $\P_{w}(n,k)$.  Hence, $\sigma$ is
also a linear extension of $\P_w(n,k)$.  Furthermore, $X$ and $Y$
cannot belong to the same packet $P(Z)$ for any $Z\in \Inv_{k+1}(w)$, so
$\rho|_{P(Z)}=\sigma|_{P(Z)}$.  Therefore, $\rho \in \A_w(n,k)$
implies $\sigma \in \A_w(n,k)$ and $\Rev(\rho) = \Rev(\sigma)$ by
\cref{def:admissible.BCL}.
\end{proof}

The concept of a packet flip on an admissible order from
\Cref{subsection:higher-bruhat-order-intro} extends easily to all
admissible orders of $\Inv_{k}(w)$ and their equivalence classes.  We
will need to add more vocabulary pertaining to packet flips for the
proof of \Cref{thm:nonlongest}.

\begin{definition}\label{def:flippable}
     For any $\rho\in \A_w(n,k)$ and $Z\in \Inv_{k+1}(w)$, the packet
$P(Z)$ is \defn{flippable} in $\rho$ if $P(Z)$ forms a saturated chain in
$\rho$.  If $\rho|_{P(Z)}$ is in lex (resp. antilex) order, then a
\defn{lex-to-antilex packet flip} (resp. \defn{antilex-to-lex packet
flip}) at $P(Z)$ is the operation that reverses the order of $P(Z)$ in
$\rho$.  Similarly, the packet $P(Z)$ is \defn{flippable} for the
equivalence class $[\rho]\in \A_w(n,k)/\hspace{-0.2em}\sim_{w}$ if there exists some
representative $\sigma\in [\rho]$ such that $P(Z)$ forms a saturated
chain in $\sigma$.
\end{definition}

\begin{example}\label{ex:flips}
    In the admissible order $\rho^2$ from \cref{ex:admissible order},
observe that $P(1246)$ and $P(1356)$ each form a saturated chain in
lex order, so both are lex-to-antilex flippable. Flipping these
respectively results in $\rho^3$ or $\rho^4$. Additionally, $\rho^2$
is commutation equivalent to
\[
\sigma = (134, 124, 135, 136, 256,156, 126, 125, 146, 246, 356, 346),
\]
where $P(1256)$ forms a saturated chain in antilex order in
$\sigma$. Hence, $P(1256)$ is antilex-to-lex flippable for $\sigma$,
and flipping this packet in $\sigma$ results in $\rho^1$.  Therefore,
we say the equivalence class $[\sigma]=[\rho^{2}]$ can be obtained
from $[\rho^{1}]$ by a lex-to-antilex packet flip.
\end{example}

\begin{lemma}\label{lem:packet.flips.Sn}
 If $\sigma$ is obtained from $\rho\in \A_w(n,k)$ by a lex-to-antilex
(resp. antilex-to-lex) packet flip at $P(Z)$ for $Z \in
\Inv_{k+1}(w)$, then $\sigma \in \A_w(n,k)$ and
$\Rev(\sigma)=\Rev(\rho)\cup \{Z\}$
(resp. $\Rev(\sigma)=\Rev(\rho)\setminus \{Z \}$).
\end{lemma}

\begin{proof}
Since $\rho \in \A_w(n,k)$, we know $\rho$ is a linear extension of
$\P_{w}(n,k)$ and its restriction to any full-packet is either in lex or
antilex order.  Assuming $\sigma$ is obtained from $\rho$ by a
lex-to-antilex packet flip at $P(Z)$, we know $P(Z)$ forms a saturated chain
in $\rho$ and in $\sigma$.  By \Cref{lem:poset.packet.antichain}, we
know $P(Z)$ forms an antichain in $\P_{w}(n,k)$, so permuting the
elements of $P(Z)$ in $\rho$ results in another linear extension of
$\P_{w}(n,k)$. The restriction of $\sigma$ to any full-packet that is
not $P(Z)$ will match the restriction of $\rho$, while their
restrictions to $P(Z)$ will be reverses of one another. Hence,
$\sigma$ is a linear extension of $\P_w(n,k)$ that restricts to the
lex or antilex order on any full-packet, and
$\Rev(\sigma)=\Rev(\rho)\cup \{Z\}$.

The case where $\sigma$ is obtained from $\rho$ by an antilex-to-lex flip at
$P(Z)$ is proved similarly.  Such a flip removes only $Z$ from the
reversal set.
\end{proof}

\begin{lemma}\label{lem:reversal_is_consistent}
    For any $\rho\in \A_w(n,k)$, the reversal set $\Rev(\rho) \subseteq
\Inv_{k+1}(w)$ satisfies

\begin{itemize}
\item $\Rev(\rho)$ is a lower-order ideal of $\P_w(n,k+1)$, and
\item  the MSZ Condition holds:
$P(Z) \cap \Rev(\rho)$ is a prefix or a suffix of $P(Z)$ for every $Z
\in \Inv_{k+2}(w)$.
\end{itemize}
\end{lemma}

\begin{proof}
First, let $Z \in \binom{[n] }{k+2}$ be a quasi-inversion for $w$ such
that $P(Z)\cap \Inv_{k+1}(w)=\{Z_{i},Z_{i+1}\}$.  Assume
$Z_{i}<_{\P_w(n,k+1)} Z_{i+1}$ and so $k+1-i$ is odd.  If $Z_{i+1} \in
\Rev(\rho)$, then $(Z_{i+1})_1$ and $(Z_{i+1})_{k+1}$ are distinct
since $k\geq 1$ and $(Z_{i+1})_1<_{\rho} (Z_{i+1})_{k+1}$.  Since
$\rho\in \A_w(n,k)$, $\rho$ must be a linear extension of $\P_w(n,k)$.
If $i=1$, then $(Z_{i})_{1}=(Z_{i+1})_{1}$.  Otherwise,
$(Z_{i})_{1}=(Z_{1})_{i-1}<_{\P_w(n,k)}(Z_{1})_{i}=(Z_{i+1})_{1}$
because the parity condition is the same when deleting the smallest
entry first.  Similarly, if $i=k+1$, then
$(Z_{i})_{k+1}=(Z_{i+1})_{k+1}$. Otherwise,
$(Z_{i+1})_{k+1}=(Z_{k+2})_{i+1}<_{\P_w(n,k)}(Z_{k+2})_{i}=(Z_{i})_{k+1}$
because the parity condition flips when deleting the largest entry.
Combining these relations with \cref{lem:permanent_on_semipackets} and
the fact that $\rho\in \A_w(n,k)$ is a linear extension of
$\P_w(n,k)$, we have
\begin{equation}\label{eq:nesting.1}
(Z_{i})_1\leq _{\rho} (Z_{i+1})_1<_{\rho} (Z_{i+1})_{k+1} \leq_{\rho}
(Z_{i})_{k+1}.
\end{equation}
Therefore, $Z_{i}<_{\P_w(n,k+1)} Z_{i+1}$ and $Z_{i+1}\in \Rev(\rho)$
implies $Z_{i}\in \Rev(\rho)$. Similarly, $Z_{i+1}<_{\P_w(n,k+1)}
Z_{i}$ and $Z_{i}\in \Rev(\rho)$ implies $k+1-i$ is even and
\begin{equation}\label{eq:nesting.2}
(Z_{i+1})_1\leq _{\rho} (Z_{i})_1<_{\rho} (Z_{i})_{k+1} \leq_{\rho}
(Z_{i+1})_{k+1}, 
\end{equation}
so $Z_{i+1}\in \Rev(\rho)$.  Hence, $\Rev(\rho)$ is a lower order
ideal of $\P_w(n,k+1)$.

Second, let $Z\in \Inv_{k+2}(w)$, and let $P^2(Z)=\{Z_{i,j}:1\leq
i<j\leq k+2\}$ be the second packet of $Z$.
Observe that
\[\Rev(\rho)\cap P(Z) = \Rev(\rho|_{P^2(Z)})\cap P(Z).
\]
Furthermore, only the relative order of the values in $Z$ are relevant
to showing whether or not $\Rev(\rho|_{P^2(Z)})\cap P(Z)$ is a prefix
or suffix of $P(Z)$ since $Z\in \Inv_{k+2}(w)$. Hence, without loss of
generality the restriction $\rho|_{P^2(Z)}$ can be viewed as an
admissible order on $\binom{[k+2]}{k}$ in the classical sense by
comparing \cref{def:admissible.MS} and \Cref{def:admissible.BCL}.
Therefore, $\Rev(\rho|_{P^2(Z)})\cap P(Z)$ must be a prefix or suffix
of $P(Z)$ by \cref{thm:ziegler}, so $\rho$ satisfies the MSZ
Condition.
\end{proof}

\begin{remark}\label{rem:nested}
The parity condition in \Cref{def:permanent.poset.} is critical to
the proof above.  Observe from \eqref{eq:nesting.1} and
\eqref{eq:nesting.2} that if $X<_{\P_w(n,k+1)} Y$, $\rho\in
\A_w(n,k)$, and $Y \in \Rev(\rho)$, then $P(Y)$ is in antilex order
and nested between $X_{1}$ and $X_{k+1}$ in $\rho$.  In particular,
$P(X)$ is not flippable in $\rho$ whenever $Y \in \Rev(\rho)$.  
\end{remark}

\subsection{Higher Bruhat Orders for Permutations}\label{sub:Higher.for.Sn}

By \Cref{lem:commutations} and \Cref{lem:packet.flips.Sn}, the set of
admissible orders $\A_w(n,k)$ is closed under commutations and
packet flips.  Hence, the classical definition of the $k^{th}$ higher
Bruhat orders from \Cref{subsection:higher-bruhat-order-intro} extends
to the case of arbitrary permutations.

\begin{definition}\label{def:higher.B.Sn} The
$\boldsymbol{k^{th}}$\defn{ higher Bruhat order for $\boldsymbol{w}$,}
denoted $\boldsymbol{\B_w(n,k)}$, is the poset
$(\A_w(n,k)/\hspace{-0.2em}\sim_{w}, \leq)$ where $[\rho]\leq
[\sigma]$ if $\sigma$ can be obtained from $\rho$ by some sequence of
commutations and lex-to-antilex packet flips.
\end{definition}

By \Cref{lem:packet.flips.Sn}, the covering relations in $\B_w(n,k)$
are exactly the lex-to-antilex packet flips on commutation classes in
$\A_w(n,k)/\hspace{-0.2em}\sim_{w}$.  Since each lex-to-antilex packet
flip increases the size of the corresponding reversal set by exactly
one, $\rev([\rho])= \rev([\sigma])-1$ if $[\rho] \lessdot [\sigma]$.
Furthermore, this proves the relations in $\B_w(n,k)$ are
antisymmetric, so $\B_w(n,k)$ is indeed a poset. 

When $k>n$, we know that $\A_w(n,k)$ contains a single element, so $\B_w(n,k)$ is a poset on a single element. We will focus on the cases where $1\leq k\leq n$.

\begin{example}\label{ex:B.645231.6.3}
Consider the admissible orders for $w=(6,4,5,2,3,1)\in \S_6$ from
\cref{ex:admissible order}. By computer, one can verify that
$|\A_w(6,3)/\hspace{-0.2em}\sim_{w}\hspace{-0.2em}|=6$ and that the
six admissible orders give a representative from each of the six
commutation classes of admissible orders in $\A_w(6,3)$. By applying
all possible commutations and packet flips, we find that $\B_w(6,3)$
is the poset on the left in \cref{fig:BnkCnk}. Additionally,
$\P_w(6,4)$ is shown in the middle of \cref{fig:BnkCnk}.  The poset on
reversal sets for each element in $\B_w(6,3)$ given by either single
step inclusion or inclusion is shown on the right of
\cref{fig:BnkCnk}.  Observe that the reversal sets in this case are
always lower order ideals of $\P_w(6,4)$, and they satisfy the MSZ
condition by \Cref{lem:reversal_is_consistent}.
\end{example}

\begin{figure}[h!]
    \centering 
    \begin{tikzpicture}
        %%% Bnk
        \node (rho1) at (0,0) {$[\rho^1]$};
        \node (rho2) at (0,1) {$[\rho^2]$};
        \node (rho4) at (1,2) {$[\rho^4]$};
        \node (rho3) at (-1,2) {$[\rho^3]$};
        \node (rho5) at (0,3) {$[\rho^5]$};
        \node (rho6) at (0,4) {$[\rho^6]$};
        \path (rho1) edge node {} (rho2);
        \path (rho2) edge node {} (rho3);
        \path (rho2) edge node {} (rho4);
        \path (rho3) edge node {} (rho5);
        \path (rho4) edge node {} (rho5);
        \path (rho5) edge node {} (rho6);
        %%% Pnk
        \node (1256) at (4,1) {$[1,2,5,6]$};
        \node (1246) at (3,2) {$[1,2,4,6]$};
        \node (1356) at (5,2) {$[1,3,5,6]$};
        \node (1346) at (4,3) {$[1,3,4,6]$};
        \path (1256) edge node {} (1246);
        \path (1256) edge node {} (1356);
        \path (1246) edge node {} (1346);
        \path (1356) edge node {} (1346);
        %%% Cnk
        \node (r1) at (10,0) {$\emptyset$};
        \node (r2) at (10,1) {$\{[1,2,5,6]\}$};
        \node (r4) at (8,2) {$\{[1,2,4,6],[1,2,5,6]\}$};
        \node (r3) at (12,2) {$\{[1,2,5,6],[1,3,5,6]\}$};
        \node (r5) at (10,3) {$\{[1,2,4,6],[1,2,5,6],[1,3,5,6]\}$};
        \node (r6) at (10,4) {$\Inv_4(w)$};
        \path (r1) edge node {} (r2);
        \path (r2) edge node {} (r3);
        \path (r2) edge node {} (r4);
        \path (r3) edge node {} (r5);
        \path (r4) edge node {} (r5);
        \path (r5) edge node {} (r6);
    \end{tikzpicture}
    \caption{The Hasse diagrams for $\B_w(6,3)$ (left), $\P_w(6,4)$ (middle), and $\C_w(6,4)$ (right) from \cref{ex:BnkCnk}. }
    \label{fig:BnkCnk}
\end{figure}
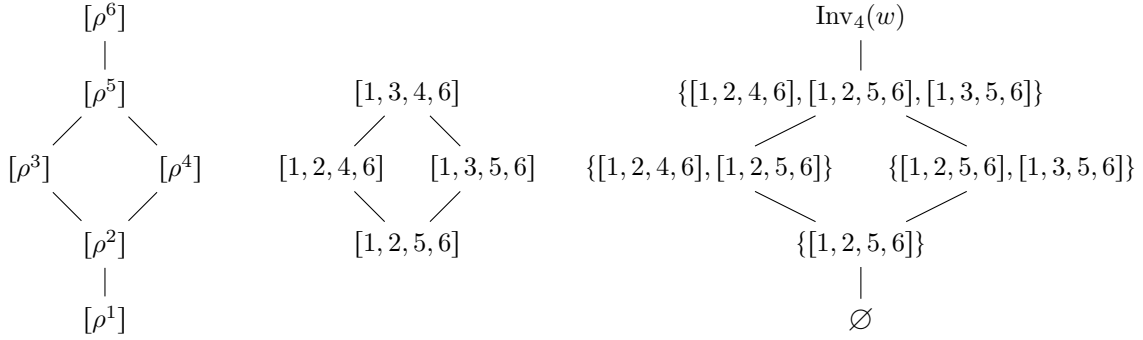

\begin{lemma}\label{lem:Bn1.weak.order}
The poset $\B_{w}(n,1)$ is isomorphic to weak order
on the interval $[\mathrm{id},w]$.  
\end{lemma}

\begin{proof}
By \Cref{lem:An1.weak.order} and an abuse of notation equating
singleton sets with values, $\A_{w}(n,1)=[\mathrm{id},w]$.  Since every
pair of numbers $i<j$ corresponds with the packet of some $X \in
\binom{[n] }{2}$, no two singletons commute, so
$\A_{w}(n,1)=\B_{w}(n,1)$ as sets.  Thus, since lex-to-antilex packet flips on
$\B_{w}(n,1)$ correspond with adjacent transpositions increasing
length, $\B_{w}(n,1)$ is isomorphic to weak order on the interval
$[\mathrm{id},w]$ by the definition of weak order.  
\end{proof}

\begin{lemma}\label{lem:Bn2.weak.order}
The Hasse diagram of the poset $\B_w(n,2)$ is isomorphic to
$\cG(w)$, the directed graph on commutation classes of reduced words
for $w$.
\end{lemma}

\begin{proof}
If $n \leq 2$, then there are no braid moves and each $\B_w(n,2)$ has
a single element, so the statement is immediate.  Hence, assume
$n \geq 3$.  By \Cref{lem:An2.weak.order}, the admissible orders in
$\A_w(n,2)$ are exactly the reflection orders for $w$.  We write
$\rho(\mathbf{j})$ for the reflection order associated to a reduced word
$\mathbf{j}=j_1\cdots j_\ell \in \cR(w)$. We also write $[\mathbf{j}]$ 
for the commutation class of $\mathbf{j}$ as an element in $\cR(w)/\sim_w$, 
which is the vertex set of $\cG(w)$. 

To compare commutation relations on reduced words and admissible
orders, observe that the case $n=3$ has no nontrivial commutations, so
the commutation comparison below is only needed for $n\geq 4$.  Let
$\rho(\mathbf{j})=(\rho_1,\dots,\rho_\ell)$, fix a position $r$, and
set
\[
    u=s_{j_1}\cdots s_{j_{r-1}}.
\]
Write
\[
    X=\rho_r=[a,b]=[u(j_r),u(j_r+1)]
\]
and
\[
    Y=\rho_{r+1}=[c,d]
    =[u s_{j_r}(j_{r+1}),u s_{j_r}(j_{r+1}+1)].
\]

Suppose first that $j_{r+1}\not\equiv j_{r}, j_{r}+1, j_{r}-1
\pmod{n}$, so $s_{j_{r}}$ and $s_{j_{r+1}}$ commute as simple
transpositions.  Then, $s_{j_r}$ fixes $j_{r+1}$ and $j_{r+1}+1$, so
\[
    Y=[u(j_{r+1}),u(j_{r+1}+1)].
\]
The four positions $j_r,j_r+1,j_{r+1},j_{r+1}+1$ are distinct, and
therefore $|\{a,b,c,d\}|=4$.  Thus $X$ and $Y$ do not belong to a
common packet.

If $\mathbf{j}'$ is obtained from $\mathbf{j}$ by a commutation
swapping $j_r$ and $j_{r+1}$, then $\mathbf{j}'\in \cR(w)$, and
$\rho(\mathbf{j}')$ is obtained from $\rho(\mathbf{j})$ by swapping
the adjacent entries $X$ and $Y$.  Since both $\rho(\mathbf{j})$ and
$\rho(\mathbf{j}')$ are admissible orders, both are linear extensions
of $\P_w(n,2)$, so $X$ and $Y$ must be incomparable in $\P_w(n,2)$.
Hence $X$ and $Y$ commute with respect to $w$ in the sense of
\Cref{def:commutes.Sn}.

Conversely, suppose adjacent entries $X=\rho_r$ and $Y=\rho_{r+1}$
commute with respect to $w$.  Since $\mathbf{j}$ is reduced, we cannot
have $j_r\equiv j_{r+1}$.  If $j_{r+1}\equiv j_{r}\pm 1 \pmod{n}$,
then $X$ and $Y$ share a value.  Indeed, if $j_{r+1}=j_r+1$, then $Y$
contains $u(j_r)$, while if $j_{r+1}=j_r-1$, then $Y$ contains
$u(j_r+1)$.  Thus $X$ and $Y$ would lie in the common packet of their
union, contradicting that they commute with respect to $w$.  Therefore
$j_{r+1}\not\equiv j_{r}, j_{r}+1, j_{r}-1 \pmod{n}$, so the adjacent
values $j_r$ and $j_{r+1}$ commute as simple transpositions.  This
proves that commutations of reduced words correspond exactly to
commutations of the associated admissible orders.  Consequently, the
map
\begin{equation}\label{eq:Gw_Bn2}
 \Phi:\cR(w)/\sim_w \longrightarrow \A_w(n,2)/\hspace{-0.2em}\sim_w,
\end{equation}
sending $[\mathbf{j}]\longmapsto [\rho(\mathbf{j})]$ is a well-defined
bijection.

It remains to compare the directed edges relating braids to packet
reversals.  Suppose $[\rho]\lessdot[\sigma]$ in $\B_w(n,2)$.  By the
definition of the covering relations in $\B_w(n,2)$, after replacing
$\rho$ and $\sigma$ by representatives of their commutation classes,
we may assume that $\sigma$ is obtained from $\rho$ by a
lex-to-antilex packet flip at $P(Z)$ for some $Z=[z_1,z_2,z_3]\in
\Inv_3(w)$.  Thus, for some position $r$,
\[
    (\rho_r,\rho_{r+1},\rho_{r+2})
    =(Z_3,Z_2,Z_1)
    =([z_1,z_2],[z_1,z_3],[z_2,z_3])
\]
and
\[
    (\sigma_r,\sigma_{r+1},\sigma_{r+2})
    =(Z_1,Z_2,Z_3).
\]
Let $\mathbf{j}=j_1\cdots j_\ell$ be the reduced word with
$\rho=\rho(\mathbf{j})$, let $u=s_{j_1}\cdots s_{j_{r-1}}$, and set
$p=j_r$.  Since $\rho_r=[z_1,z_2]$, we have $u(p)=z_1$ and
$u(p+1)=z_2$.  The next inversion in the admissible order is
$\rho_{r+1}=[z_1,z_3]$, which shares the smaller value $z_1$ with
$\rho_r$, so it must be that $j_{r+1}=p+1$ and $u(p+2)=z_3$.  In
$us_{p}s_{p+1}$, the values in positions $p,p+1,p+2$ are
$z_2,z_3,z_1$, so the entry $\rho_{r+2}=[z_2,z_3]$ forces $j_{r+2}=p$.
Hence $\mathbf{j}$ has the form
\[
    \mathbf{j}
    =
    j_1\cdots j_{r-1}\, p(p+1)p\, j_{r+3}\cdots j_\ell .
\]
Replacing the subword $p(p+1)p$ by $(p+1)p(p+1)$ gives another reduced
word for $w$ via a braid relation 
\[
    \mathbf{j}'
    =
    j_1\cdots j_{r-1}\, (p+1)p(p+1)\, j_{r+3}\cdots j_\ell .
\]
A direct calculation of the three reflections in the braid shows that
$\rho(\mathbf{j}')=\sigma$.  Therefore the packet flip defining the cover relation
$[\rho]\lessdot[\sigma]$ corresponds under $\Phi^{-1}$ to the directed
braid
\[
    p(p+1)p \longrightarrow (p+1)p(p+1)
\]
defining the edge $\Phi^{-1}([\rho])=[\mathbf{j}]\to [\mathbf{j}']=\Phi^{-1}([\sigma])$ in $\cG(w)$.

Conversely, suppose $[\mathbf{j}]\to [\mathbf{j}']$ as a 
directed edge in $\cG(w)$.  Without loss of generality, we can assume
\[
    \mathbf{j}
    =
    j_1\cdots j_{r-1}\, p(p+1)p\, j_{r+3}\cdots j_\ell
\]
and
\[
    \mathbf{j}'
    =
    j_1\cdots j_{r-1}\, (p+1)p(p+1)\, j_{r+3}\cdots j_\ell
\]
differ by a directed braid.  Let
$u=s_{j_1}\cdots s_{j_{r-1}}$.  Since $\mathbf{j}$ is reduced, the
three transpositions in the braid all increase the length of the permutation,
 which implies
\[
    u(p)<u(p+1)<u(p+2).
\]
Set $Z=[u(p),u(p+1),u(p+2)]$.  The contribution of the subword
$p(p+1)p$ to $\rho(\mathbf{j})$ is
\[
    (Z_3,Z_2,Z_1),
\]
while the contribution of $(p+1)p(p+1)$ to $\rho(\mathbf{j}')$ is
\[
    (Z_1,Z_2,Z_3).
\]
The prefix and suffix contributions agree outside these three positions.
Thus $\rho(\mathbf{j}')$ is obtained from $\rho(\mathbf{j})$ by a
lex-to-antilex packet flip at $P(Z)$, so 
$\Phi([\mathbf{j}])=[\rho(\mathbf{j})]\lessdot[\rho(\mathbf{j}')]=\Phi([\mathbf{j}'])$
is a cover relation in $\B_w(n,2)$.

The bijection $\Phi$ therefore sends directed edges of $\cG(w)$ exactly
to cover relations in the Hasse diagram of $\B_w(n,2)$, preserving
orientation.  Hence the Hasse diagram of $\B_w(n,2)$ is isomorphic to
$\cG(w)$ as a directed graph.
\end{proof}

Our next goal is to show that each equivalence class $[\rho ] \in
\B_{w}(n,k)$ is exactly the set of linear extensions of a poset.  We
start by introducing directed graphs that generalize the inversion
graphs of permutations and encode a sufficient set of inequalities
satisfied by admissible sequences with a given reversal set.

\begin{definition}\label{def:classical.GR}
For any $R \subseteq \Inv_{k+1}(w)$, define the \defn{reversal graph}, denoted
$G_R$, to be the graph with vertex set $\Inv_{k}(w)$ and directed
edges
\begin{itemize}
\item $X \to Y$ if $X<_{{\P_{w}(n,k)}}Y$ is a quasi-inversion
relation,

\item $X_i\color{red}\to\color{black} X_{i+1}$ for all $1\leq i\leq  k$
if $P(X)$ is a full-packet for $w$ and $X\in R$, and

\item $X_{i+1}\color{blue}\to\color{black} X_i$ for all $1\leq i\leq  k$ if $P(X)$ is a full-packet for $w$ and $X\notin R$. 
\end{itemize}
We refer to the first type of edges as \defn{quasi-inversion
edges}. We refer to the latter two types of edges as \defn{full-packet
edges}, more specifically as \defn{full-antilex} (red) or
\defn{full-lex edges} (blue).  In all cases, any two vertices
connected by an edge differ in exactly one position.  Note, the color
on the edges is only there as a visual reminder of the reason each
edge was added.
\end{definition}

Any directed graph induces a preorder on its vertices by considering
edges $X\to Y$ to imply $X<Y$ and taking the transitive closure.  If
the graph is acyclic then the associated preorder is
antisymmetric. Hence, every acyclic directed graph induces a
corresponding partial order.  A reader familiar with heaps of reduced
words from \cite[Sec. 2]{BJS93} and \cite[Prop. 2.2]{Stem4} may
recognize that when $R\subseteq \binom{[n] }{3}$ and $G_{R}$ is
acyclic, then the associated preorder $(\Inv_2(w),\le_R)$ is
isomorphic to the heap poset of a reduced word for $w$, as mentioned
in the introduction.  We extend the notion of heap posets to all
$k,n,w$.

\begin{definition}\label{def:heap.poset.preorder}
For $R \subseteq \Inv_{k+1}(w)$, the \defn{reversal preorder for
$\mathbf{R}$}, denoted $\boldsymbol{(\Inv_{k}(w),\leq_R)}$, is the
transitive closure of the relations $X<_{R}Y$ for all
(quasi-inversion, full lex, or full antilex) edges $X\to Y$ in
$G_{R}$.  If $G_{R}$ is acyclic, then we refer to $(\Inv_k(w),\le_R)$
as the \defn{heap poset} associated to $R$.
\end{definition}

\begin{remark}\label{rem:k=2graphs}
Let $R\subseteq \Inv_{2}(w)$ be the inversion set of some permutation
$u\in [\mathrm{id},w]$. The corresponding reversal graph $G_{R}$ is
the acyclic oriented complete graph on $\Inv_{1}(w)=\binom{[n]
}{1}=\{\{i \}: i \in [n] \}$ with an edge $\{j \} \to \{i \}$ for all
$[i,j] \in R$ and $\{i \} \to \{j \}$ for all $[i,j] \not \in R$,
which induces a partial order with the unique linear extension $u$.
Observe the graph $I_{u^{-1}}$ defined in \Cref{def:BB.graph} is the
same as $G_{R}$ minus the edges corresponding with $[i,j] \not \in R$,
or equivalently minus edges indexed by $i<j$ with $\{i \} \to \{j \}$.
\end{remark}

\begin{example}\label{ex:GR example}
Consider $w=(6,4,5,2,3,1)\in \S_6$ from \cref{ex:B.645231.6.3}. The
$4$-inversions of $w$
are \[\Inv_4(w)=\{[1,2,4,6],[1,2,5,6],[1,3,4,6],[1,3,5,6]\}.\]
For $R=\{[1,2,5,6],[1,3,5,6]\}\subseteq \Inv_4(w)$, the corresponding 
graph $G_R$ is shown in \cref{fig:Gr
example}.  As one can see from the figure, the structure of $G_R$ as a
directed graph can be complex.  

\begin{figure}[h!]  \centering \begin{tikzpicture} \node
(134) at (0,0) {$[1,3,4]$}; \node (135) at (-1,1) {$[1,3,5]$}; \node
(124) at (1,1) {$[1,2,4]$}; \node (125) at (0,2) {$[1,2,5]$}; \node
(136) at (3,0) {$[1,3,6]$}; \node (126) at (3,1) {$[1,2,6]$}; \node
(156) at (5,0) {$[1,5,6]$}; \node (146) at (5,1) {$[1,4,6]$}; \node
(256) at (8,0) {$[2,5,6]$}; \node (246) at (7,1) {$[2,4,6]$}; \node
(356) at (9,1) {$[3,5,6]$}; \node (346) at (8,2) {$[3,4,6]$};
\draw[->] (134) edge node {} (135); \draw[->] (134) edge node {}
(124); \draw[->] (135) edge node {} (125); \draw[->] (124) edge node
{} (125); \draw[->] (136) edge node {} (126); \draw[->] (156) edge
node {} (146); \draw[->] (256) edge node {} (246); \draw[->] (246)
edge node {} (346); \draw[->] (256) edge node {} (356); \draw[->]
(356) edge node {} (346); \draw[<-,color=red] (125) edge node {}
(126); \draw[<-,color=red] (126) edge node {} (156);
\draw[<-,color=red] (156) edge node {} (256); \draw[<-,color=red]
(135) edge node {} (136); \draw[<-,color=red] (136) edge node {}
(156); \draw[<-,color=red] (156) edge node {} (356);
\draw[->,color=blue] (124) edge node {} (126); \draw[->,color=blue]
(126) edge node {} (146); \draw[->,color=blue] (146) edge node {}
(246); \draw[->,color=blue] (134) edge node {} (136);
\draw[->,color=blue] (136) edge node {} (146); \draw[->,color=blue]
(146) edge node {} (346); \end{tikzpicture} \caption{An example of a
graph $G_R$, with quasi-inversion edges shown in black, full lex edges
in blue, and full antilex edges in red.}  \label{fig:Gr example}
\end{figure}
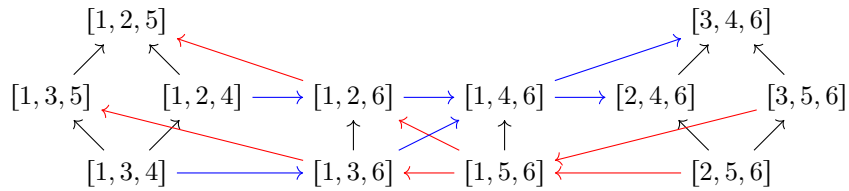
\end{example}

We say a total order $\rho$ on the vertices of a directed graph $G$
\defn{respects} the directed edge $X\to Y$ in $G$ if $X$ comes before
$Y$ in $\rho$.  If there exists a total order $\rho$ respecting every
directed edge in $G$, then $G$ must be acyclic. In this case, its
transitive closure is a partial order, and the total orders respecting
every edge of $G$ are exactly the linear extensions of this partial
order.  Recall that $\mathscr{L}(P, \leq)$ denotes the set of linear
extensions of the poset $(P,\leq )$.  Using these concepts, we can
relate the acyclic reversal graphs to admissible orders as follows.

\begin{lemma}\label{lem:GR_to_admissible}
Let $R\subseteq \Inv_{k+1}(w)$.  If $\rho\in \A_w(n,k)$ and
$\Rev(\rho)=R$, then $G_R$ is acyclic. Conversely, if $G_R$ is
acyclic, then
\[
    \{\rho \in \A_w(n,k):\Rev(\rho)=R\} =
\mathscr{L}(\Inv_k(w),\le_R).
\]
\end{lemma}

\begin{proof}
If $\rho\in \A_w(n,k)$ with $\Rev(\rho)=R$, then $\rho$ is a linear
extension of $\P_w(n,k)$ by definition.  Additionally, $\rho$
restricts to the antilex or lex order on every packet $P(X)\subseteq
\Inv_k(w)$ depending on $X \in \Rev(\rho)$ or $X \not \in
\Rev(\rho)$, respectively. These two possibilities align with the relations $X_1<_R
X_2<_R \ldots <_R X_{k+1}$ or $X_{k+1}<_R X_k <_R \ldots
<_R X_1$ in $(\Inv_k(w),\leq_R)$ depending on whether or not $X
\in R$. Hence, $\rho$ respects every directed edge of $G_R$, so
$G_{R}$ is acyclic and $\rho \in \mathscr{L}(\Inv_k(w),\le_R).$

Conversely, assume $G_{R}$ is acyclic, so $(\Inv_k(w),\leq_R)$ is a
partial order.  By construction, $(\Inv_k(w),\leq_R)$ contains the
permanent poset $\P_w(n,k)$ as a subposet on the same vertex set.
Hence, if $\rho \in \mathscr{L}(\Inv_k(w),\leq_R)$, then $\rho$ is
also a linear extension of $\P_w(n,k)$.  Furthermore, every
full-packet $P(X)\subseteq \Inv_{k}(w)$ must appear in lex or antilex
order since its elements were specifically ordered that way in
$G_{R}$.  In particular, $P(X)$ must be in antilex order if and only
if $X \in R$.  We conclude that $\rho\in \A_w(n,k)$ and $\Rev(\rho)=R$
by \cref{def:admissible.BCL} to complete the proof.
\end{proof}

In the next theorem, we show the elements of $\B_w(n,k)$ are uniquely
determined by their reversal set, as in
\Cref{thm:MS}~\ref{item:unique.rev.set}.  Each commutation class is
the set of linear extensions of the corresponding poset.

\begin{theorem}\label{thm:equiv_class.and.linear.exts}
If $\rho \in \A_{w}(n,k)$ and $R=\Rev(\rho)$, then $G_{R}$ is acyclic
and its commutation equivalence class $[\rho]\in \B_w(n,k)$ is given
by
\begin{equation}\label{eq:rev_inverse.classical}
    [\rho]=\{\sigma \in \A_{w}(n,k): \Rev (\sigma )=R\}
          =\mathscr{L}(\Inv_k(w),\leq_R).
\end{equation}
\end{theorem}

\begin{proof}
By \Cref{lem:commutations}, we know that if $\sigma \in [\rho]$, 
then $\sigma \in \A_{w}(n,k)$ and $\Rev(\rho)=\Rev(\sigma)$.  Hence,
if $\rho \in \A_{w}(n,k)$ and $R=\Rev(\rho)$, then $G_{R}$ is acyclic
and $[\rho] \subseteq\{\sigma \in \A_{w}(n,k): \Rev (\sigma )=R \}
= \mathscr{L}(\Inv_k(w),\leq_R)$ by \Cref{lem:GR_to_admissible}.

It remains to show that every element of
$\mathscr{L}(\Inv_k(w),\le_R)$ lies in $[\rho]$. By
\cref{lem:linear_extensions}, any two linear extensions of
$(\Inv_k(w),\le_R)$ can be obtained from one another by a sequence of
adjacent transpositions of incomparable elements. If two elements
$X,Y$ are incomparable in $(\Inv_k(w),\le_R)$, then they are
incomparable in $\P_w(n,k)$.  By \cref{lem:intersection_size}, if
$X,Y$ lie in a common packet, then either the packet is full, in which
case the full-packet edges of $G_R$ totally order them, or they form
the two-element packet intersection coming from a quasi-inversion, in
which case they are ordered by a quasi-inversion edge. Thus, $X,Y$
commute with respect to $w$. Hence every linear extension lies in the
commutation equivalence class $[\rho]$.
\end{proof}

\begin{corollary}\label{cor:reversals_unique}
Let $\rho,\sigma \in \A_{w}(n,k)$. Then, $\Rev(\rho)=\Rev(\sigma)$ if
and only if $[\rho]=[\sigma]$.
\end{corollary}

\begin{corollary}\label{cor:B.equiv.acyclic}
The reversal map induces a bijection
\[
Rev: \B_{w}(n,k) \to \{R \subseteq \Inv_{k+1}(w): G_{R} \text{ is
acyclic} \}, 
\]
where the inverse sends $R$ to the commutation class given by linear
extensions 
\[
    \mathscr{L}(\Inv_k(w),\le_R).
\]
\end{corollary}

The goal of the next subsection is to replace the global acyclicity
condition on $G_R$ by local conditions on $R$. These local conditions
are the generalized consistent set conditions defining \(\C_w(n,k+1)\).

\subsection{Consistent Sets for Permutations}\label{sub:consistent.Sn}

The following notion generalizes the consistent sets defined in
\Cref{def:consistent.Z}. Similar to the definition of admissible
orders, the permanent poset will play an essential role in this
generalization.  Here the goal is to define consistent sets in such a
way that the reversal sets of admissible orders for $w$ are
automatically consistent.  Recall we already know two necessary
conditions by \Cref{lem:reversal_is_consistent}.  It will turn out
that these conditions are also sufficient.  Hence, the main theorem of
this subsection shows that every consistent set for $w \in \S_{n}$ is
also the reversal set for a unique equivalence class of admissible
orders for $w$.  Using this characterization, we will complete the
proof of \Cref{thm:nonlongest}.

Note that we momentarily shift indices and define consistency for subsets
of $\Inv_k(w)$.  With this shifted indexing, if $R \subseteq
\Inv_k(w)$, then $G_R$ is a directed graph on $\Inv_{k-1}(w)$
by definition.

\begin{definition}\label{def:consistent.Sn}
    A subset $R\subseteq \Inv_{k}(w)$ is \defn{consistent with respect to $\boldsymbol{w}$} if 
    \begin{itemize}
        \item $R$ is a lower-order ideal of $\P_w(n,k)$, and
	\item  the MSZ Condition holds: $P(X) \cap R$ is a prefix or a
	suffix of $P(X)$ for every $X \in \Inv_{k+1}(w)$.
    \end{itemize}
    When the permutation $w$ is clear from context, we will simply refer to $R$ as
\defn{consistent}.  Define $\boldsymbol{\C_w(n,k)}$ to be the poset on consistent subsets
of $\Inv_k(w)$ with partial order $\leq_1$ generated by single step
inclusion.
\end{definition}

If $k=1$, recall the permanent poset $\P_{w}(n,1)$ is the poset on
$\binom{[n] }{1}=\Inv_{1}(w)$ generated by the 2-quasi-inversion
relations, and its linear extensions are in bijection with
$[\mathrm{id},w]$.  Therefore, the lower order ideals of $\P_{w}(n,1)$
are in bijection with the prefixes of permutations $u \in
[\mathrm{id},w]$. Every lower order ideal of $\P_{w}(n,1)$ satisfies
the MSZ condition for every $X \in \Inv_{2}(w)$, so
\begin{equation}\label{eq:1consistent}
\C_{w}(n,1) = \{\{\{u(1) \},\{u(2) \},\dots
, \{u(j)  \}\}:  0\leq j\leq n,\  u=(u(1),\dots , u(n)) \in
[\id,w] \}.
\end{equation}

\begin{lemma}\label{lem:consistent.k=2}
The weak Bruhat order on $[\mathrm{id},w]$ is isomorphic to
$\C_{w}(n,2)$, and the isomorphism is given by mapping a permutation
to its inversion set.
\end{lemma}

\begin{proof}
By \Cref{lem:inversion.sets}, every inversion set $\Inv_{2}(u)$ for a
permutation $ u \in [\mathrm{id},w]$ is a subset of $\Inv_{2}(w)$ such
that $\Inv_{2}(u)\cap P(X)$ is a prefix or a suffix of $P(X)$.  Note
that if $|\Inv_{2}(w)\cap P(X)|=2$, then the corresponding
quasi-inversion relation in $\P_{w}(n,2)$ encodes exactly the same
requirement on the lower order ideals of $\P_{w}(n,2)$ by the parity
condition. Thus, for all $u\in [\mathrm{id},w]$, both conditions for
consistent sets for $w$ in \Cref{def:consistent.Sn} hold and
$\Inv_{2}(u) \in \C_{w}(n,2)$.  Conversely, a consistent set $R \in
\C_{w}(n,2)$ satisfies the  conditions of \Cref{lem:inversion.sets} so
$R$ is the inversion set of some permutation $u$.  Furthermore, since
$R \subseteq \Inv_{2}(w)$, we know $u\in [\mathrm{id},w]$ in weak
order.  Therefore, $\Inv_{2}: [\mathrm{id},w] \to \C_{w}(n,2)$ is a
bijection.  Since the covering relations for both posets are given by
single step inclusion, the bijection is a poset isomorphism.
\end{proof}

\begin{lemma}\label{lem:reversal_is_consistent.Sn}
For any $\rho\in \A_w(n,k)$, the reversal set $\Rev(\rho) \in \C_w(n,k+1)$. Furthermore, $\Rev(\rho)$ is constant on the equivalence class $[\rho]$, so $\Rev:\B_w(n,k) \to \C_w(n,k+1)$ is a well-defined order-preserving map.
\end{lemma}

\begin{proof}
The first claim is proved by comparing
\Cref{lem:reversal_is_consistent} and \Cref{def:consistent.Sn}.  The
set $\Rev(\rho)$ is constant on the equivalence class $[\rho]$ by
\Cref{lem:commutations}, so the map $\Rev_{n,k,w}: \B_w(n,k) \to
\C_w(n,k+1)$ is well-defined.  The map is order-preserving by
\Cref{lem:packet.flips.Sn} and \Cref{def:higher.B.Sn}.
\end{proof}

\begin{example}\label{ex:empty.full}
Both $\emptyset$ and $\Inv_{k}(w)$ are consistent sets in $\C_w(n,k)$.
Indeed, both are naturally lower order ideals of $\P_{w}(n,k)$.
Furthermore, for each $X \in \Inv_{k+1}(w)$, $P(X)\cap
\Inv_{k}(w)=P(X)$ and both $\emptyset$ and $P(X)$ are prefixes of $P(X)$,
so the $MSZ$ condition holds.  In particular, $\C_w(n,k)=\{\emptyset
\}$ for all $k>n$.  However, it is not clear directly from the
definition that $\emptyset$ and $\Inv_{k}(w)$ are comparable in
$\C_w(n,k)$ using single step inclusion order for $1<k<n$, but it will
follow from \Cref{thm:nonlongest}.
\end{example}

\begin{example}
For $w=(3,4,2,1) \in \S_{4}$, recall $\P_{w}(4,1)$ is the poset on vertices
$\{\{1 \},\{2 \},\{3 \},\{4 \} \}$ with the single quasi-inversion
relation $\{3 \}<_{\P_{w}(4,1)}\{4 \}$.  The lower-order ideals of
$\P_{w}(4,1)$ are exactly the subsets $I \subseteq \{\{1 \},\{2 \},\{3
\},\{4 \} \}$ such that $\{4 \} \in I$ implies $\{3 \} \in I$.  The
MSZ Condition is trivially satisfied in all cases when $k=1$, so
$\C_{w}(4,1)$ is the set of all lower order ideals of $\P_{w}(4,1)$
ordered by single step inclusion. Such lower order ideals correspond
with subsets of $[4]$ that occur as prefixes of permutations $u \in
[\mathrm{id},w]$ in window notation.
\end{example}

\begin{example}
Consider $w_0\in \S_n$ and any $1\leq k\leq n$.  We know $w_{0}$ has
no $k$-quasi-inversions, so the permanent poset $\P_{w_{0}}(n,k)$ is
an antichain on the set $\Inv_{k}(w_{0})=\binom{[n] }{k}$.  Therefore,
every subset of $\Inv_{k}(w_{0})$ is a lower-order ideal of
$\P_{w_{0}}(n,k)$.  Consequently, the definitions for $\A_{w_0}(n,k)$,
$\B_{w_0}(n,k)$, and $\C_{w_0}(n,k)$ reduce to the ones in
\cref{subsection:higher-bruhat-order-intro} requiring only the
full-packet condition or the MSZ condition accordingly.
\end{example}

\begin{example}\label{ex:BnkCnk}
Continuing with $w=(6,4,5,2,3,1)\in \S_6$ from \cref{ex:B.645231.6.3},
the poset $\C_w(6,4)$ is exactly the poset on the right in
\cref{fig:BnkCnk}, which is also the poset on reversal sets for each
element in $\B_w(6,3)$ ordered by single step inclusion.  To verify
the Hasse diagram, observe that all of the lower-order ideals of
$\P_w(6,4)$ are consistent since $\Inv_5(w)=\emptyset$ and the MSZ
condition is trivially satisfied since there are no prefix or suffix
requirements in this case.  Observe that $\B_w(6,3)$ and $\C_w(6,4)$
are isomorphic, and an explicit isomorphism is given by
$\Rev_{6,3,w}$.
\end{example}

\begin{example}
    For a more complex example of a consistent set, consider $w=(6,4,5,2,3,1)$ with $k=3$. From \cref{fig:permanent.example}, one can observe that 
    \[R=\{[1,3,4],[1,3,5],[1,3,6],[1,5,6],[2,5,6],[3,5,6]\}\]
    is a lower-order ideal of $\P_w(6,3)$. Its intersections with the packets of the four 4-inversions of $w$ are
    \begin{equation}
        \begin{split}
            R\cap P([1,2,4,6]) & = \emptyset, \\
            R\cap P([1,2,5,6]) & = \{[1,5,6],[2,5,6]\}, \\
            R\cap P([1,3,4,6]) & = \{[1,3,4],[1,3,6]\}, \\
            R\cap P([1,3,5,6]) & = \{[1,3,5],[1,3,6],[1,5,6],[3,5,6]\}.
        \end{split}
    \end{equation}
    Each of the above is a prefix or suffix of the respective packet, so $R\in \C_w(n,3)$. 
\end{example}

In the next lemma, we will show that $R \in \C_w(n,k)$ is necessary
for $G_{R}$ to be acyclic for $k>1$.  In the case $k=1$, we observe
that $G_{R}$ is the single vertex graph with no edges for all $R
\subseteq \binom{[n] }{1}$ since $\Inv_{0}(w)=\{\emptyset \}$.  Hence,
we exclude the $k=1$ case below.

\begin{lemma}\label{lem:not_consistent_cycle.Sn} Assume $k>1$. If
$R\subseteq \Inv_k(w)$ is not consistent, then $G_R$ contains a cycle.
\end{lemma}

\begin{proof}
By \Cref{lem:GR_to_admissible}, we know that if $G_R$ is acyclic, then
there exists an admissible order $\rho \in \A_w(n,k-1)$ such that
$\Rev(\rho)=R$.  By \Cref{lem:reversal_is_consistent.Sn}, the reversal
sets of admissible orders are consistent.  Thus, $G_R$ is acyclic
implies $R$ is consistent, and the proof follows using the
contrapositive statement.
\end{proof}

By \Cref{cor:B.equiv.acyclic} and
\Cref{lem:reversal_is_consistent.Sn}, we know $\Rev_{n,k,w}: \B_w(n,k)
\to \C_w(n,k+1)$ is a well-defined injective map.  In order to prove
this is a bijection as required for \Cref{thm:nonlongest}, it remains
to show that every consistent set in $\C_w(n,k+1)$ is the reversal set
for some admissible order in $\A_w(n,k)$.  The next theorem shows that
consistent sets characterize the acyclic reversal graphs for $w$.

\begin{theorem}\label{lem:GRacyclic}
Let $k,n\in\mathbb{Z}$ with $k>1$ and $n>0$, let $w\in\S_n$, and
let $R\subseteq\Inv_k(w)$. The graph $G_R$ is acyclic if and only if
$R\in\C_w(n,k)$.
\end{theorem}

We have shown $R$ must necessarily be consistent in order for $G_{R}$
to be acyclic in \Cref{lem:not_consistent_cycle.Sn}.  The proof that
$R$ being consistent is sufficient will be postponed until
\Cref{sub:G_R}, as it requires induction with several additional
tools.  For now, assume $G_{R}$ is a directed acyclic graph on
$\Inv_{k}(w)$ for all $R\in \C_w(n,k+1)$, so $(\Inv_{k}(w),\leq_{R})$
is a partial order.    The following statement is a generalization of
Ziegler's \cref{thm:ziegler}.

\begin{theorem}\label{thm:Bnk_iso_Cnk}
    Let $k,n\in \mathbb{Z}$ be positive integers and $w\in \S_n$.  The
map $\Rev_{n,k,w}: \B_w(n,k) \to \C_w(n,k+1)$ is a poset isomorphism
with inverse given by
\begin{equation}\label{eq:rev_inverse.iso.inv.class}
    \Rev_{n,k,w}^{-1}(R)=\mathscr{L}(\Inv_k(w),\leq_R).
\end{equation}
In particular, there exists a commutation equivalence class of
$k$-admissible orders $[\rho]\in \B_w(n,k)$  with $\Rev([\rho])=R$
for every $R \in \C_w(n,k+1)$.
\end{theorem}

\begin{proof}
\Cref{cor:B.equiv.acyclic} and \Cref{lem:GRacyclic} imply that $\Rev =
\Rev_{n,k,w}$ is a well-defined bijection with inverse given by
\eqref{eq:rev_inverse.classical}.  Furthermore, $\Rev$ is order-preserving since
the covering relations in $\B_{w}(n,k)$ are given by lex-to-antilex
packet flips, and these flips result in single step inclusion of reversal sets.

It remains to show that $\Rev^{-1}$ is also order-preserving.  Assume
$R, R' \in \C_w(n,k+1)$ and $R\lessdot_{1} R'$. By definition of
single step inclusion, $R'=R\cup \{X \}$ for some $X\in
\Inv_{k+1}(w)$.

By construction, $G_R$ and $G_{R'}$ differ only by reversing the
full-packet edges along $P(X)$. If a directed path in $G_{R}$ between
two elements of $P(X)$ left $P(X)$, then some segment with endpoints
in $P(X)$ and internal vertices outside $P(X)$ would survive the
reversal of edges along $P(X)$.  This would result in a directed cycle
in $G_{R'}$ with the reversed chain on $P(X)$.  The same argument
applies if $G_{R}$ and $G_{R'}$ are interchanged.  However, since both
$R,R' \in \C_w(n,k+1)$, both $G_{R}$ and $G_{R'}$ are acyclic by
\Cref{lem:GRacyclic}. Therefore, no such directed path including
vertices of $P(X)$ can leave $P(X)$, so $P(X)$ is a convex saturated
chain in both heap posets $(\Inv_k(w),\leq_R)$ and
$(\Inv_k(w),\leq_{R'})$.  Consequently, there are linear extensions
$\rho\in \mathscr{L}(\Inv_k(w),\leq_R)$ and $\sigma\in
\mathscr{L}(\Inv_k(w),\leq_{R'})$ that differ by a packet flip at
$P(X)$. We conclude that $\Rev^{-1}(R)=[\rho] \lessdot_{\B_w(n,k)}
[\sigma ] = \Rev^{-1}(R')$.
\end{proof}

\begin{corollary}\label{cor:flippable}
 For any $[\rho]\in \B_w(n,k)$ and $X\in \Inv_{k+1}(w)$, the following hold:
    \begin{enumerate}[label=(\alph*)]
        \item $P(X)$ is lex-to-antilex flippable in $[\rho]$ if and only if $X\notin \Rev(\rho)$ and $\Rev(\rho)\cup \{X \}\in \C_w(n,k+1)$, and
        \item $P(X)$ is antilex-to-lex flippable in $[\rho]$ if and only if $X\in \Rev(\rho)$ and $\Rev(\rho)\setminus \{X \}\in \C_w(n,k+1)$.
    \end{enumerate}
\end{corollary}

We next establish poset properties of $\B_w(n,k)\cong \C_w(n,k+1)$,
namely that they are ranked with unique minimal and maximal elements.
We will focus on $\C_w(n,k)$ where the minimal element must be
$\emptyset$ and maximal element must be $\Inv_{k}(w)$ by
\Cref{ex:empty.full}. From \cref{cor:flippable}, this can be reduced
to showing that an element can be removed or added whenever $R\in
\C_w(n,k)$ is not one of these two sets. Our approach has some
connections to the \emph{boundary} of a consistent set defined by
Ziegler in \cite[p. 262]{ziegler_1993}, though it has several key
differences to account for our more general setting.  There is also a
similarity to the sets $J_{s}$ and $J_{F}$ defined by Hothem in
\cite[Def. 2.5.4]{hothem2021extendinghigherbruhatorders}. However,
what Hothem calls a realizable $k$-set is a subset $J
\subseteq\binom{[n] }{k}$ such that $P(X) \cap J$ is a prefix or
suffix for all $X \in \binom{[n] }{k+1}$ in our notation.  This notion
is not comparable to the consistent sets in $\C_{w}(n,k)$.

\begin{definition}\label{def:suffix_set}
For $R\in \C_w(n,k)$, define the \defn{suffix set} for $R$ to
be
\[S(R) = \{Y \in \Inv_{k+1}(w): Y_1\in R\}.\]
\end{definition}

\begin{lemma}\label{lem:suffix_set_consistent}
  The following statements hold for any consistent set $R\in \C_w(n,k)$.
\begin{enumerate}[label=(\alph*)]
\item The suffix set $S(R) \in \C_w(n,k+1)$.
\item We have that $R$ is a lower order ideal in the corresponding
heap poset
$(\Inv_{k}(w), \leq_{S(R)})$.
\item Every lower order ideal of $(\Inv_{k}(w), \leq_{S(R)})$ is in
$\C_w(n,k)$.
\item Every lower order ideal of $(\Inv_{k-1}(w), \leq_{R})$ is in
$\C_w(n,k-1)$ provided that $k>1$.  
\end{enumerate}
\end{lemma}

\begin{proof}
To prove $S(R) \in \C_w(n,k+1)$, we need to show that $S(R)$ is a
lower-order ideal of $\P_w(n,k+1)$ that satisfies the MSZ condition.
Let $X\in \binom{[n]}{k+2}$ be a quasi-inversion of $w$ with $P(X)\cap
\Inv_{k+1}(w)=\{X_i,X_{i+1}\}$.  If $k-i$ is odd and $i>1$, then
\[
 (X_{i+1})_{1}=(X_{1})_{i}<_{\P_w(n,k)} (X_{1})_{i-1}=(X_i)_1
\]
by \Cref{def:permanent.poset.}.  Since $R \in \C_w(n,k)$ implies $R$
is a lower-order ideal of $\P_w(n,k)$, we know $(X_i)_1\in R$ implies
$(X_{i+1})_1\in R$. Thus, $X_{i} \in S(R)$ implies $X_{i+1}\in
S(R)$. Similarly, when $k-i$ is even and $i>1$, then $X_{i+1}\in S(R)$
implies $X_{i}\in S(R)$.  If $i=1$, note $(X_{i+1})_{1}=(X_{i})_{1}$
so the corresponding statements for $k-i$ odd or even trivially hold.
Since the quasi-inversion relations generate $\P_w(n,k+1)$ by
definition, we conclude that $S(R)$ is a lower-order ideal of
$\P_w(n,k+1)$.

Next, suppose $X\in \Inv_{k+2}(w)$, so $X_{1} \in \Inv_{k+1}(w)$.
Consider any $1\leq h<i<j\leq k+2$.  If $X_h,X_j\in S(R)$, we claim
$X_i\in S(R)$ as well.  The claim follows because
$(X_h)_1,(X_i)_1,(X_j)_1 \in P(X_{1})$ and $R \in \C_w(n,k)$, so
$(X_h)_1,(X_j)_1\in R$ implies $(X_i)_1\in R$.  Similarly, if
$X_h,X_j\notin S(R)$, we can show $X_i\notin S(R)$ since $R \in
\C_w(n,k)$. We conclude that $S(R) \cap P(X)$ is a prefix or
suffix of $P(X)$ by \cref{lem:prefix_suffix}, so $S(R)$ satisfies the MSZ condition. Combined, we conclude that Part (a) holds. 

To see that $R$ is a lower order ideal of $(\Inv_{k}(w),
\leq_{S(R)})$, recall $R$ is a lower order ideal of $\P_{w}(n,k)$ by
\Cref{def:consistent.Sn}.  Moreover, by definition of $S(R)$, we know
$R \cap P(X)$ is a prefix for all $X \in \Inv_{k+1}(w) \setminus S(R)$
and $R \cap P(X)$ is a suffix for all $X \in S(R)$.  The covering
relations in $(\Inv_{k}(w), \leq_{S(R)})$ all come from covering
relations in $\P_{w}(n,k)$ or the required lex/antilex order on
full-packets.  Hence, if $X\leq_{S(R)} Y$ and $Y \in R$, then $X \in
R$ so $R$ is a lower order ideal of $(\Inv_{k}(w), \leq_{S(R)})$,
proving Part (b).

Note, Part (c) follows from Parts (a) and (d).  To prove Part (d), let
$Q$ be a lower order ideal of $(\Inv_{k-1}(w), \leq_{R})$.  Then $Q$
is also a lower order ideal of $\P_w(n,k-1)$ by definition of
$\leq_{R}$.  For any $X \in \Inv_{k}(w)$, note that $P(X) \cap Q$ is a
prefix or suffix of $P(X)$ since $P(X)$ must be a chain in
$(\Inv_{k-1}(w), \leq_{R})$ in lex or antilex order by definition.
\end{proof}

The analog of flippable packets in $\B_{w}(n,k)$ from
\Cref{cor:flippable} in $\C_w(n,k)$ is described in the following
lemma.  Note, the lower order ideals of any poset $(P,\leq )$
partially ordered by inclusion form another poset $J(P)$ called the
lattice of order ideals of $P$.  This lattice is ranked by size
\cite[Prop. 3.4.5]{StanEC1}.  Thus, $J(P)$ could also be defined by
single step inclusion since the covering relations are determined by
adding to a lower order ideal $I$ a single minimal element in $P
\setminus I$.  Similarly, consistent sets are lower order ideals of
the permanent poset ordered by single step inclusion, but not all
lower order ideals of $\P_{w}(n,k)$ are consistent.  Instead,
consistent sets are lower order ideals of certain heap posets.

\begin{lemma}\label{lem:suffix_set.build.up}
For $R\in\C_w(n,k)$ and $X\in\Inv_k(w)$, the following properties hold.
\begin{enumerate}[label=(\alph*)]
\item If $X$ is minimal with respect to the heap poset
$(\Inv_k(w),\leq_{S(R)})$ restricted to $\Inv_k(w)\setminus R$,
then $R\cup\{X\}\in\C_w(n,k)$.
\item If $X$ is maximal with respect to $\leq_{S(R)}$ restricted
to $R$, then $R\setminus\{X\}\in\C_w(n,k)$.
\end{enumerate}
\end{lemma}

\begin{proof}
By \cref{lem:suffix_set_consistent}(b), $R$ is a lower order ideal
of $(\Inv_k(w),\leq_{S(R)})$. Adding a minimal element of its
complement or removing a maximal element of $R$ gives another lower
order ideal of this heap poset. By
\cref{lem:suffix_set_consistent}(c), each resulting set is in
$\C_w(n,k)$.
\end{proof}

\begin{theorem}\label{thm:Cnk_min_max}
 Let $k,n\in \mathbb{Z}$ be positive integers and $w\in \S_n$. The poset $\C_w(n,k)$
is a ranked poset with unique minimal element $\emptyset$ and a unique
maximal element $\Inv_{k}(w)$.  The rank of $R \in \C_w(n,k)$ is
$|R|$.
\end{theorem}

\begin{proof}
By \Cref{ex:empty.full}, we know $\emptyset$ is a minimal element and
$\Inv_{k}(w)$ is a maximal element in $\C_w(n,k)$ for all $k$ and $n$.
If $k>n$, then $\Inv_{k}(w)=\emptyset$ and $\C_w(n,k) = \{\emptyset \}$, so the
statement is true. The case $k=1$ follows from the fact that the lower
order ideals of $\P_w(n,1)$ form a ranked poset under single step
inclusion with unique minimal element $\emptyset$ and a unique maximal
element $\Inv_{1}(w)=[n]$.  Furthermore, if $k=1$ the MSZ condition is
trivially satisfied since every subset of a two element packet is a
prefix or suffix.  Therefore, assume $2\leq k\leq n$.

Suppose $R\in \C_w(n,k)$ is nonempty.  By \Cref{lem:suffix_set_consistent}, the suffix set $S(R) \in
\C_w(n,k+1)$, and the corresponding poset
$(\Inv_k(w),\leq_{S(R)})$ contains $R$ as a lower order ideal.  Since $R$ is nonempty, there is
some $Z\in R$ that is maximal with respect to $\leq_{S(R)}$ restricted
to $R$. \Cref{lem:suffix_set.build.up} implies $R\setminus \{Z \} \in
\C_w(n,k)$. Since the poset structure on $\C_w(n,k)$ is defined to be
single step inclusion, $R\setminus \{Z \} \lessdot_{1} R$.
Therefore, $R\neq \emptyset$ implies $R$ is not minimal.  Furthermore,
by induction we can assume there exists a saturated chain in
$\C_w(n,k)$ from $\emptyset$ to $R\setminus \{Z\}$ of length $|R\setminus
\{Z \}|$, which can be extended to a chain from $\emptyset$ to $R$
through $R\setminus \{Z\}$ of length $|R|$.  The  length of any
saturated chain from $\emptyset$ to $R$ in a poset defined by
single step inclusion is $|R|$.  Hence, $\emptyset$ is the unique
minimal element in $\C_w(n,k)$ and the rank of $R$ is its size for all
$R\in \C_w(n,k)$.

Similarly, if $R\neq \Inv_k(w)$, then there is some $Z\in
\Inv_k(w)\setminus R$ that is minimal with respect to $\leq_{S(R)}$
restricted to $\Inv_k(w)\setminus R$.  By
\Cref{lem:suffix_set.build.up}, $R \cup \{Z \} \in \C_w(n,k)$.    Thus,
$R \lessdot_{1} R\cup \{Z\}$ so $R$ is not maximal in
$\C_w(n,k)$.  By induction on $|\Inv_{k}(w)|-|R|$, there exists a
saturated chain from $R$ to $\Inv_{k}(w)$ as well.  Hence,
$\Inv_{k}(w)$ is the unique maximal element in $\C_w(n,k)$.
\end{proof}

\begin{corollary}\label{cor:Bwnk.ranked}
The higher Bruhat order $\B_w(n,k)$ is a ranked poset with unique
minimal and maximal elements.
\end{corollary}

\begin{proof}
This follows from \cref{thm:Bnk_iso_Cnk} and \cref{thm:Cnk_min_max}.
\end{proof}

\begin{corollary}\label{cor:connected}
 Any two admissible orders $\rho,\sigma\in \A_w(n,k)$ can be obtained
from one another using a sequence of commutations and packet flips.
\end{corollary}

\begin{proof}
By \Cref{cor:Bwnk.ranked}, the Hasse diagram of $\B_w(n,k)$ is
connected, and some sequence of packet flips on representative
elements allows us to obtain $[\sigma]$ from $[\rho]$. We denote the
resulting sequence of elements in $\B_w(n,k)$
as \begin{equation}\label{eq:flip_sequence}
[\rho]=[\rho_0],[\rho_1],[\rho_2],\ldots,[\rho_{\ell-1}],[\rho_{\ell}]=[\sigma].
\end{equation}
The admissible orders in each commutation equivalence class $[\rho_i]$
are all connected by commutations by definition. Therefore, the result
now follows by combining these commutations within each $[\rho_i]$
with the packet flips used in \eqref{eq:flip_sequence}.
\end{proof}

We next consider maximal chains in $\C_w(n,k)$ and their relationship
to admissible orders in $\A_w(n,k)$.  The following result shows that
every prefix of every admissible order $\rho\in \A_w(n,k)$ is a
consistent set in $\C_w(n,k)$, generalizing \eqref{eq:1consistent}.  

\begin{lemma}\label{lem:admissible_to_consistent_set}
Suppose $\rho=(\rho_1,\rho_2,\ldots,\rho_{|\Inv_k(w)|})\in \A_w(n,k)$. 
Then for any $0\leq i\leq |\Inv_k(w)|$, \[\{\rho_1,\rho_2,\ldots,\rho_i\}\in \C_w(n,k).\]
\end{lemma}

\begin{proof}
Note that the case of $i=0$ results in $\emptyset$, and this is a
consistent set. For $i\geq 1$, let
$R=\{\rho_1,\rho_2,\ldots,\rho_i\}$. We must show that $R$ is a
lower-order ideal of $\P_w(n,k)$ that satisfies the MSZ condition.

Let $Y\in \Inv_k(w)$, and suppose $Y\in R$. As $\rho$ is a linear extension of $\P_w(n,k)$,
 every $X<_{\P_w(n,k)} Y$ must appear to the left of $Y$ in $\rho$. As $R$ is constructed using the
first $i$ elements in $\rho$, we conclude $X\in R$, so $R$ is a lower-order ideal of 
$\P_w(n,k)$. 

Recall $P(X)\subseteq \Inv_k(w)$ whenever $X\in
\Inv_{k+1}(w)$. Since $\rho \in \A_w(n,k)$, the restriction
$\rho|_{P(X)}$ is either the lex or antilex order on $P(X)$ for all
$X\in \Inv_{k+1}(w)$. As $R$ is constructed using the first $i$
elements in $\rho$, we conclude $R\cap P(X)$ is a prefix or suffix of
$P(X)$, respectively, for all $X\in \Inv_{k+1}(w)$.  Hence, $R$
satisfies the MSZ condition, so $R\in \C_w(n,k)$.
\end{proof}

By \Cref{thm:Cnk_min_max}, we know the poset $\C_w(n,k)$ has a unique
minimal element $\emptyset$ and a unique maximal element
$\Inv_{k}(w)$, so each maximal chain in $\C_w(n,k)$ must begin at $\emptyset$ and end at 
$\Inv_k(w)$. Additionally, the partial order on $\C_w(n,k)$ is given by single step
inclusion, so \cref{lem:admissible_to_consistent_set} gives a natural method for
 obtaining such a maximal chain when given an 
admissible order $\rho\in \A_w(n,k)$. Conversely, each maximal chain determines a total order
on $\Inv_{k}(w)$ given by recording the order in which $k$-inversions
are added along the chain. We now formalize this correspondence
between maximal chains in $\C_w(n,k)$ and admissible orders in $\A_w(n,k)$.

\begin{definition}\label{def:consistentchaintoadmissibleseq}
Let $\rho=(\rho_1,\rho_2,\ldots,\rho_{|\Inv_k(w)|})\in \A_w(n,k)$. For
all $0\leq i\leq |\Inv_k(w)|$, define the \defn{prefix maps} by
$R_i(\rho)=\{\rho_1,\ldots,\rho_i\}$.  Define the \defn{successive
prefix map}
\[\mathcal{R}: \A_w(n,k) \to \{\text{maximal chains of } \C_w(n,k)
\} \] by
\[\mathcal{R}(\rho)=(\emptyset=R_0(\rho)\lessdot_1 R_1(\rho)\lessdot_1 R_2(\rho) \lessdot_1 \dots \lessdot_1 R_{|\Inv_{k}(w)|}(\rho)).\]
\end{definition}

\begin{theorem}\label{thm:chains_to_Ank}
    Let $k,n\in \mathbb{Z}$ be positive integers and $w\in \S_n$. The
    successive prefix map
\[\mathcal{R}:
\A_w(n,k) \to \{\text{maximal chains of } \C_w(n,k) \} \]
is a bijection.
\end{theorem}

\begin{proof}
\cref{lem:admissible_to_consistent_set} shows that $\mathcal{R}$ is well-defined, and it is clear 
from its construction that $\mathcal{R}$ is injective. It remains to show surjectivity. 
For a maximal chain
\begin{equation}\label{eq:chain}
C= (\emptyset = R_{0}\lessdot_{1} R_1 \lessdot_{1} R_2 \lessdot_{1} \dots
\lessdot_{1} R_{|\Inv_{k}(w)|})
\end{equation}
in $\C_w(n,k)$, its pre-image under $\mathcal{R}$ must be
$\rho=(\rho_{1},\rho_{2},\dots, \rho_{|\Inv_{k}(w)|})$, where $R_{i}\setminus R_{i-1}=\{\rho_{i}\}$. 
We must show that this is an admissible order.  By \Cref{thm:Cnk_min_max},
we know $\rho$ is a total order on $\Inv_{k}(w)$.  Since each
$R_{i} \in \C_{w}(n,k)$, we know $R_{i}$ is a lower order ideal of
$\P_w(n,k)$.  Consequently, if $\rho_{i} <_{\P_{w}(n,k)} \rho_{j}$, then 
$\rho_{i} \in R_{j}$, so $i<j$.  Hence,
$\rho$ is a linear extension of $\P_w(n,k)$. Furthermore, for $X
\in \Inv_{k+1}(w)$, the restriction of $\rho$ to $P(X)$ records the
order in which elements of $P(X)$ are added into \eqref{eq:chain}.
Since each $R_{i} \in \C_{w}(n,k)$, we conclude $R_{i}\cap P(X)$ is a
prefix or suffix of $P(X)$ for all $i$, so $\rho|_{P(X)}$ must be the lex or
antilex order on $P(X)$. We conclude $\rho\in \A_{w}(n,k)$ and $\mathcal{R}(\rho)=C$,
 so $\mathcal{R}$ is surjective.  
\end{proof}

\begin{proof}[Proof of \cref{thm:nonlongest}] Recall that we will prove
\Cref{lem:GRacyclic} in \cref{sub:G_R}, which shows that $G_{R}$ is acyclic for all $R
\in \C_w(n,k)$.  Modulo that proof, we can complete the proof of
\Cref{thm:nonlongest}. 
Part (a) is given in \cref{thm:Bnk_iso_Cnk}. Part (b) is given in
\cref{thm:Cnk_min_max}. Part (c) is given in \cref{thm:chains_to_Ank}.
\end{proof}

\begin{remark}\label{cor:recursive}
Recall that $\B_{w}(n,k) \cong \C_{w}(n,k+1)$ by
\Cref{thm:Bnk_iso_Cnk}.  Furthermore, $\A_{w}(n,k+1)$ is in bijection
with maximal chains of $\C_{w}(n,k+1)$ by \Cref{thm:chains_to_Ank}.
Therefore, there is a natural interpretation of commutation
equivalence for maximal chains of $\B_{w}(n,k)$ and $\C_{w}(n,k+1)$
via the bijection with $\A_{w}(n,k+1)$.  Since packet flips are
defined on $\A_{w}(n,k+1)$, they also have a natural interpretation on
maximal chains of $\B_{w}(n,k)$ and $\C_w(n,k+1)$.  Therefore, the
equivalence classes of maximal chains in $\B_w(n,k)$ can be partially
ordered via lex-to-antilex packet flips, and this poset is isomorphic
to $\B_w(n,k+1)$.  This parallels the recursive nature of the higher
Bruhat orders as described by Manin-Schechtman in
\Cref{thm:MS}\ref{item:bij}.
\end{remark}

\subsection{Acyclicity of $G_{R}$}\label{sub:G_R}

In this section, we will prove \cref{lem:GRacyclic}.  For any $R \in
\C_{w}(n,k+1),$ recall the reversal graph $G_{R}$ has vertices in
$\Inv_{k}(w)$ and edges given by quasi-inversion relations and
lex/antilex packet orders.  The proof methods will be based on several
operations that will allow us to apply induction on certain subgraphs
of $G_R$ and then translate acyclicity on these subgraphs to
acyclicity of $G_R$.  For any $G_R$, we first partition its vertices
into the subsets
\begin{equation}
    \begin{split}\label{eq:partition}
        V_0(G_R) &= \{X \in \Inv_k(w): n\notin X \text{ and } X \cup \{n \} \not\in R\},\\
        V_1(G_R) &= \{X \in \Inv_k(w): n\in X\},\\
        V_2(G_R) &= \{X \in \Inv_k(w): n\notin X \text{ and } X \cup \{n \} \in R\}.
    \end{split}
\end{equation}
Note, that $V_{1}(G_{R})$ depends only on $w,k,n$ and does not depend on
$R$.

\begin{example}\label{ex:partition}
We illustrate this definition by referring back to $w=(6,4,5,2,3,1)\in \S_6$ and $R=\{[1,2,5,6],[1,3,5,6]\}\in \C_w(6,4)$ from \cref{ex:GR example} and \cref{fig:Gr example}. In this case, we have that $V_0(G_R)=\{[1,2,4],[1,3,4]\}$, $V_2(G_R)=\{[1,2,5],[1,3,5]\}$, and $V_1(G_R)$ consists of all remaining vertices. This partition is illustrated in \cref{fig:Gr example2}.
\begin{figure}[h!]
\begin{tikzpicture}
      \draw[fill=gray!20, draw=black, dashed]
    (-1.25,-0.5) -- (0.25,-0.5) -- (2.25,1.5) -- (0.75,1.5) -- cycle;
      \draw[fill=gray!20, draw=black, dashed]
    (-2.25,0.5) -- (-0.75,0.5) -- (1.25,2.5) -- (-0.25,2.5) -- cycle;
      \draw[fill=gray!20, draw=black, dashed]
    (2.35,-0.5) -- (2.35, 2.5) -- (9.75, 2.5) -- (9.75, -0.5) -- cycle;
   
        \node (134) at (0,0) {$[1,3,4]$};
        \node (135) at (-1,1) {$[1,3,5]$};
        \node (124) at (1,1) {$[1,2,4]$};
        \node (125) at (0,2) {$[1,2,5]$};
        \node (136) at (3,0) {$[1,3,6]$};
        \node (126) at (3,1) {$[1,2,6]$};
        \node (156) at (5,0) {$[1,5,6]$};
        \node (146) at (5,1) {$[1,4,6]$};
        \node (256) at (8,0) {$[2,5,6]$};
        \node (246) at (7,1) {$[2,4,6]$};
        \node (356) at (9,1) {$[3,5,6]$};
        \node (346) at (8,2) {$[3,4,6]$};
        \node at (-0.5,-0.75) {$V_0(G_R)$};
        \node at (6, -0.75) {$V_1(G_R)$};
        \node at (0.5,2.75) {$V_2(G_R)$};
        \draw[->] (134) edge node {} (135);
        \draw[->] (134) edge node {} (124);
        \draw[->] (135) edge node {} (125);
        \draw[->] (124) edge node {} (125);
        \draw[->] (136) edge node {} (126);
        \draw[->] (156) edge node {} (146);
        \draw[->] (256) edge node {} (246);
        \draw[->] (246) edge node {} (346);
        \draw[->] (256) edge node {} (356);
        \draw[->] (356) edge node {} (346);
        \draw[<-,color=red] (125) edge node {} (126);
        \draw[<-,color=red] (126) edge node {} (156);
        \draw[<-,color=red] (156) edge node {} (256);
        \draw[<-,color=red] (135) edge node {} (136);
        \draw[<-,color=red] (136) edge node {} (156);
        \draw[<-,color=red] (156) edge node {} (356);
        \draw[->,color=blue] (124) edge node {} (126);
        \draw[->,color=blue] (126) edge node {} (146);
        \draw[->,color=blue] (146) edge node {} (246);
        \draw[->,color=blue] (134) edge node {} (136);
        \draw[->,color=blue] (136) edge node {} (146);
        \draw[->,color=blue] (146) edge node {} (346);
    \end{tikzpicture}
    \caption{Partitioning the vertices in $G_R$ according to \eqref{eq:partition}.}
    \label{fig:Gr example2}
\end{figure}
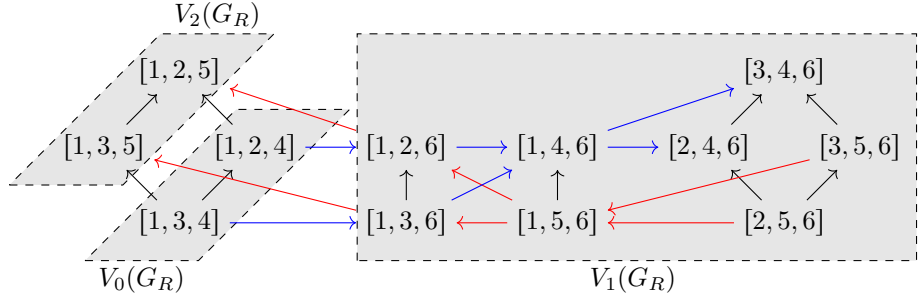
\end{example}

For any directed graph $G=(V,E)$ and subset of vertices $V'\subseteq
V$, the \emph{subgraph induced by $V'$} is the graph on vertex set
$V'$ with all directed edges in $E$ that are between vertices in
$V'$. We will study the structure of the subgraphs of $G_R$ induced by
$V_0(G_R),V_1(G_R)$, or $V_2(G_R)$. We begin with several definitions
analogous to deletion and contraction in graphs and matroids.

\begin{definition}\label{def:deletion_contraction}
 For any $j\in [n]$, define the \defn{flatten function} $\fl_j:[n]\setminus
\{j\}\to [n-1]$ to be the unique order-preserving bijection
\[\fl_j(i)=\begin{cases}
    i & \text{if } i<j \\
    i-1 & \text{if } i>j.
\end{cases}\]
\begin{enumerate}[label=(\alph*)]
    \item For any $w\in \S_n$, define $\Fl_j(w)$ to be the permutation
    in $\S_{n-1}$ with window
    notation \[\Fl_j(w)=(\fl_j(w(1)),\ldots,\widehat{j},\ldots,\fl_j(w(n)))\]
obtained from the window notation for $w$ by removing $j$ wherever
it occurs and shifting the values above $j$ down by 1.

    \item For any $X=[x_1,\ldots,x_\ell]\subseteq [n]$, define $\Fl_j(X)$ by
    \[\Fl_j(X)=\begin{cases}
        [\fl_j(x_1),\ldots,\fl_j(x_\ell)] & \text{ if $j\notin X$}\\
        [\fl_j(x_1),\ldots,\widehat{j}, \ldots,\fl_j(x_\ell)] & \text{ if $j\in X$}. 
    \end{cases}\]
    We emphasize that the first case results in an element in $\binom{[n-1]}{\ell}$, while the second case results in an element in $\binom{[n-1]}{\ell-1}$. 
    \item For any $R\subseteq \binom{[n]}{\ell}$, the \defn{deletion} of $R$ with respect to $j$ is 
    \[R\setminus j =\{\Fl_j(X): X\in R, j\notin X\},\]
    and the \defn{contraction} of $R$ with respect to $j$ is
    \[R/j=\{\Fl_j(X):X\in R,j\in X\}.\]
    We emphasize that the deletion of $R\subseteq \binom{[n]}{\ell}$ is a
subset of $\binom{[n-1]}{\ell}$, while the contraction of $R$ is a subset
of $\binom{[n-1]}{\ell-1}$. 
\end{enumerate}
\end{definition}

\begin{example}\label{ex:deletion_contraction}
Consider $w=(6,4,5,2,3,1)\in \S_6$ from \cref{ex:GR example}. Then $\Fl_3(w)=(5,3,4,2,1)$. For $R=\{[1,2,5,6],[1,3,5,6]\}\in \C_w(6,4)$, its deletion with respect to $3$ is $R\backslash 3=\{[1,2,4,5]\}$, and its contraction with respect to $3$ is $R/3=\{[1,4,5]\}$. 
\end{example}

\begin{lemma}\label{lem:consequences.of.forget.maps} For any $j \in
[n]$, let $w'=\Fl_{j}(w)$.  Then, $\Inv_{k}(w) \setminus j =
\Inv_{k}(w')$ and $\Inv_{k}(w) / j \subseteq \Inv_{k-1}(w')$.
Furthermore, for any $R\subseteq \Inv_k(w)$, the restriction of
$\Fl_{j}$ to $R$ is a bijection with the disjoint union $R\setminus j
\uplus R/j.$
\end{lemma}

\begin{proof}
Observe the restrictions
\[
\Fl_{j}:\left\{X \in \binom{[n] }{\ell}: j \in X \right\} \to \binom{[n-1]
}{\ell-1}
\]
and 
\[
\Fl_{j}:\left\{X \in \binom{[n] }{\ell}: j \not \in X \right\} \to \binom{[n-1]
}{\ell}
\]
are bijections for all $\ell \geq 0$.  Therefore, it is
straightforward to check from the definitions that $\Inv_{k}(w)
\setminus j = \Inv_k(w')$ and $\Inv_{k}(w) / j \subseteq
\Inv_{k-1}(w')$ since $\fl_j$ is order-preserving on $[n]\setminus j$.
Furthermore, the restriction to $R$ can be partitioned into elements
containing $j$ or not containing $j$, this partition is preserved by
$\Fl_{j}$, and $\Fl_{j}$ restricts to a bijection on both parts.
\end{proof}

\begin{lemma}\label{lemma:F_j_properties}
    Let $j \in [n]$, and define $w' = \Fl_j(w)$. 
    \begin{enumerate}[label=(\alph*)]
    \item For any $X,Y\in \Inv_k(w)$ such that $j \not\in X \cup Y$, the quasi-inversion relation 
        $X <_{\P_w(n,k)} Y$ holds if and only if the quasi-inversion relation $\Fl_j(X) <_{\P_{w'}(n-1,k)} \Fl_j(Y)$ holds.
     \item For any $Z \in \binom{[n]}{k+1}$ such that $j \not\in Z$
	 and any $R \subseteq \binom{[n]}{k}$, the intersection $R \cap
	 P(Z)$ is a prefix (resp. suffix) of $P(Z)$ if and only if $(R
	 \setminus j) \cap P(\Fl_j(Z))$ is a prefix (resp. suffix) of
	 $P(\Fl_j(Z))$.
    \end{enumerate}
\end{lemma}

\begin{proof}
 If $X<_{\P_w(n,k)} Y$ is a quasi-inversion relation, then $Z = X\cup
Y$ is a $(k+1)$-quasi-inversion of $w$ that also does not contain $j$.
By \Cref{lem:consequences.of.forget.maps}, $\Fl_j$ induces a bijection
between the elements of $\Inv_{k}(w)$ that do not contain $j$ and the
elements in $\Inv_{k}(w')$ where $w'=\Fl_j(w)$.  Therefore, since
$\fl_{j}$ is order-preserving on $[n]\setminus \{j \}$, $\Fl_{j}(Z)$
is a $(k+1)$-quasi-inversion of $w'$ and $\Fl_{j}(X) \cup
\Fl_{j}(Y)=\Fl_{j}(Z)$.  If $X=Z_{i}$ and $Y=Z_{i+1}$, then
$\Fl_j(X)=\Fl_{j}(Z)_{i},$\  $\Fl_j(Y)=\Fl_{j}(Z)_{i+1},$ and
$\Fl_j(X)<_{\P_{w'}(n-1,k)} \Fl_j(Y)$ is a quasi-inversion relation in
$\P_{w'}(n-1,k)$ by the parity condition. Similarly, if
$X=Z_{i+1}$ and $Y=Z_{i}$, the same conclusion holds.  Conversely, if
$\Fl_j(X)<_{\P_{w'}(n-1,k)} \Fl_j(Y)$ is a quasi-inversion relation
and $j$ is not in either of $X$ or $Y$, then this argument can be
reversed to show $X<_{\P_w(n,k)} Y$ is a quasi-inversion
relation. This proves Part (a).

    For Part (b), observe that for any $i \in [k+1]$, $\Fl_j(Z)_i =
    \Fl_j(Z_i)$ since $j \not  \in Z$. Thus, $P(\Fl_j(Z))$ in lexicographic order is the same as $\Fl_j$ applied to $P(Z)$ in lexicographic order. It follows that $R \cap P(Z)$ is a prefix or suffix of $P(Z)$ if and only if $(R \setminus j) \cap P(\Fl_j(Z))$ is a prefix or suffix of $P(\Fl_j(Z))$.
\end{proof}

\begin{corollary}
    \label{cor:deletion_consistent}
    Suppose $R \in \C_w(n,k+1)$. For any $j \in [n]$, let $w' = \Fl_j(w)$. Then $R \setminus j \in \C_{w'}(n-1,k+1)$, and the
subgraph of $G_R$ induced by $\{X \in \Inv_{k}(w): j \not \in X \}$ is isomorphic as
a directed graph to the reversal graph $G_{R\setminus j}$ on $\Inv_{k}(w')$.  In particular, the
 subgraph of $G_R$ induced by $V_0(G_R)\cup V_2(G_R)$ is isomorphic as
 a directed graph to $G_{R\setminus n}$.
\end{corollary}
\begin{proof}

   Since $R$ is consistent, it is a lower order ideal of $\P_w(n,k+1)$
and satisfies the MSZ condition. By
\cref{lem:consequences.of.forget.maps}, there is a bijection between
$\{X \in \Inv_{k+1}(w) : j \not\in X\}$ and $\Inv_{k+1}(w) \setminus j
= \Inv_{k+1}(w')$. By \cref{lemma:F_j_properties}(a), the
subposet of $\P_w(n,k+1)$ on $\{X \in \Inv_{k+1}(w): j \not\in X\}$ given by
the transitive closure of the quasi-inversion relations restricted to
this set is isomorphic as a poset to $\P_{w'}(n-1,k+1)$. Thus, $R
\setminus j$ is a lower order ideal of $\P_{w'}(n-1,k+1)$. Furthermore,
since $R$ satisfies the MSZ condition, we know $R \setminus j$
satisfies the MSZ condition by \cref{lemma:F_j_properties}(b).  Hence,
$R \setminus j \in \C_{w'}(n-1,k+1)$.

Now let $H$ be the induced subgraph of $G_{R}$ on the vertices $V=\{X
\in \Inv_{k}(w): j \not \in X \}$. By
\Cref{lem:consequences.of.forget.maps}, $\Fl_j$ is a bijection between
$V$ and $\Inv_{k}(w')$, which are the vertices of $G_{R\setminus j}$.
By \Cref{lemma:F_j_properties}(a), $G_R$ contains the quasi-inversion
edge $X\to Y$ if and only if $G_{R\setminus j}$ contains the
quasi-inversion edge $\Fl_j(X)\to \Fl_j(Y)$. Additionally, $X\to Y$ is
a full lex edge (resp. antilex edge) precisely when they are adjacent
in a common packet $P(Z)$ where $Z\notin R$ (resp. $Z\in R$). By
construction, this occurs precisely when $\Fl_j(Z)\notin R\setminus j$
(resp. $\Fl_j(Z)\in R\setminus j$), and this corresponds to when
$\Fl_j(X)\to \Fl_j(Y)$ is a full lex (resp. antilex) edge in
$G_{R\setminus j}$. Hence, the bijection $\Fl_j$ is an isomorphism of
directed graphs between $H$ and $G_{R\setminus j}$.
\end{proof}

While the deletion of a consistent set is well-behaved, this is not always true for contraction. We give an example demonstrating some issues that arise involving the permanent posets.

\begin{example}\label{ex:contraction1}
    Consider $w=(4,5,3,1,2)\in \S_5$, where
$\Inv_3(w)=\{[1,3,4],[1,3,5],[2,3,4],[2,3,5]\}$ and
$\Inv_4(w)=\emptyset$, so the diamond poset

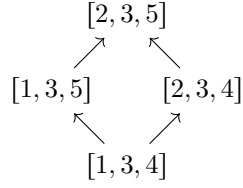
\begin{center}
\begin{tikzpicture}
        \node (134) at (0,0) {$[1,3,4]$};
        \node (135) at (-1,1) {$[1,3,5]$};
        \node (124) at (1,1) {$[2,3,4]$};
        \node (125) at (0,2) {$[2,3,5]$};
        \draw[->] (134) edge node {} (135);
        \draw[->] (134) edge node {} (124);
        \draw[->] (135) edge node {} (125);
        \draw[->] (124) edge node {} (125);
    \end{tikzpicture}
\end{center}
is $\P_{w}(5,3)$.  
As $[1,3,4]$ is minimal in $\P_w(5,3)$ and $\Inv_4(w)=\emptyset$, we
see that $R=\{[1,3,4]\}\in \C_w(5,3)$. Letting
$w'=\Fl_3(w)=(3,4,1,2)$, the corresponding permanent poset
$\P_{w'}(4,2)$ contains the quasi-inversion
relation
\[\Fl_3([1,3,4])=[1,3]>_{\P_{w'}(4,2)}[2,3]=\Fl_3([2,3,4]),
\]
so $R'=R/3=\{[1,3]\}$ is not a lower-order
ideal of $\P_{w'}(4,2)$. This example demonstrates that some relations
$X<_{\P_w(n,k)} Y$ may be reversed compared to $\Fl_j(X)>_{\P_{w'}(n-1,k-1)}
\Fl_j(Y)$. However, not all relations are reversed, as demonstrated by
$[1,3,4] <_{\P_w(5,3)} [1,3,5]$ and
\[
\Fl_3([1,3,4])=[1,3]<_{\P_{w'}(4,2)}[1,4]=\Fl_3([1,3,5]).
\]
In general,
for any quasi-inversion relation $X<_{\P_w(n,k)} Y$, the corresponding
relation between $\Fl_j(X)$ and $\Fl_j(Y)$ in $\P_{w'}(n-1,k-1)$ is
determined by whether $j$ appears left or right of the unique position
where $X$ and $Y$ differ.
\end{example}

The previous example demonstrates that the contraction of a consistent
set does not need to be consistent. However, several special cases of
contraction do work nicely such as contracting on $1$ or $n$. Below we
focus on contraction by $n$ because if $w(n)=n$, then
$\Inv_{k}(w)=\Inv_{k}(\Fl_{n}(w))$ for all $2\leq k\leq n$.  So, the
acyclicity of any $G_{R}$ follows directly from induction on $n$.  We
will also use a special case of contraction on two consecutive
elements for the induction argument in the proof of
\Cref{lem:GRacyclic}.  Because of the parity condition, the direction
of the quasi-inversion relations is flipped when contracting with
respect to $n$.

\begin{lemma}\label{lem:contract_n}
    Define $w'=\Fl_n(w)$. 
 \begin{enumerate}[label=(\alph*)] 
\item For
any $X,Y\in \Inv_k(w)$ such that $n\in X$ and $n\in Y$, we have
$X<_{\P_w(n,k)} Y$ is a quasi-inversion relation if and only if
$\Fl_n(X) >_{\P_{w'}(n-1,k-1)} \Fl_n(Y)$ is a quasi-inversion
relation.   
\item Let $R\subseteq\binom{[n]}{k}$, and let
$Z\in\binom{[n]}{k+1}$ with $n\in Z$. If $R\cap P(Z)$ is a prefix
(resp. suffix) of $P(Z)$, then $(R/n)\cap P(\Fl_n(Z))$ is a prefix
(resp. suffix) of $P(\Fl_n(Z))$.
    \end{enumerate}
\end{lemma}

\begin{proof}
    This follows from similar arguments to the ones in
\cref{lemma:F_j_properties}, with some adaptations. For (a), a
quasi-inversion relation $X<_{\P_w(n,k)} Y$ for
$\{X,Y\}=\{Z_i,Z_{i+1}\}$ comes from the parity of $k-i$ for some
quasi-inversion $Z$ of $w$. By definition, $\Fl_n(X),\Fl_n(Y)\in
\binom{[n-1]}{k-1}$ and $\Fl_n(Z)\in \binom{[n-1]}{k}$ are obtained by
removing the last element $n$ from $X,Y\in \binom{[n]}{k}$ and $Z\in
\binom{[n]}{k+1}$, so the relation between
$\{\Fl_n(X),\Fl_n(Y)\}=\{\Fl_n(Z)_i,\Fl_n(Z)_{i+1}\}$ is based on the
parity of $(k-1)-i$. As this is opposite of the parity of $k-i$, we
conclude $\Fl_n(X) >_{\P_{w'}(n-1,k-1)} \Fl_n(Y)$. Reversing this
argument shows that (a) holds.

For (b), the statement follows by definition of prefix/suffix and the
contraction operation on $R$. The proof is similar to
 \cref{lemma:F_j_properties}(b), but deletion is replaced with
contraction. Note that the element $Z_{k+1}\in P(Z)$ does not contain
$n$, so it is removed from $P(Z)$ when contraction is performed. Consequently,
the reverse implication of (b) does not hold in general.
\end{proof}

\begin{lemma}\label{lem:reduction_remove2}
Let $R\in \C_w(n,k+1)$ where $n=k+2$. Suppose $X\to Y$ is an edge in $G_R$ such that
\begin{enumerate}[label=(\roman*)]
\item $X$ and $Y$ do not contain $n$ and 
\item $X$ and $Y$ both
contain some consecutive elements $j,j+1\in [n-2]$.
\end{enumerate}
If $R'=(R/(j+1))/j\subseteq \Inv_{k-1}(w')$ is the double contraction,
$w'=\Fl_{j}\circ \Fl_{j+1}(w)\in \S_{n-2}$, $X'=\Fl_{j}\circ
\Fl_{j+1}(X)$, and $Y'=\Fl_{j}\circ \Fl_{j+1}(Y)$, then the following
hold.
\begin{enumerate}[label=(\alph*)]
    \item We have that $R'\in \C_{w'}(n-2,k-1)$.
    \item The edge $X\to Y$ is a quasi-inversion, full lex, or full antilex edge in $G_R$ if and only if $X'\to Y'$ is a quasi-inversion, full lex, or full antilex edge in $G_{R'}$, respectively. 
    \item For each $i\in \{0,2\}$, we have that $X\in V_i(G_R)$  if
    and only if $X'\in V_i(G_{R'})$.  The same holds for $Y$ and $Y'$.  
\end{enumerate}
\end{lemma}

\begin{proof}
    Throughout, we let $Z=[1,2,\ldots,n-1,n]$, and let
$Z'=\Fl_{j}\circ \Fl_{j+1}(Z)$ be the double flatten set. As $X\to Y$
is an edge in $G_R$, neither element contains $n$, and $n=k+2$ by
hypothesis, it must be that $X\cup Y=Z_n=Z \setminus \{n \}$ is either
an inversion or a quasi-inversion for $w$. Additionally, $X$ and $Y$
differ at exactly one position $h$ and both contain $j$ and $j+1$, so
$\{X,Y\}=\{Z_{h,n},Z_{h+1,n}\}$, where $\{j,j+1\}\cap
\{h,h+1\}=\emptyset$.

To prove part (c), assume (a) holds.  If $X=Z_{h,n}$, then $X\in
V_0(G_R)$ if and only if $Z_{h} \not \in R$ if and only if
$\Fl_{j}\circ \Fl_{j+1}(Z_{h})\not \in R'$ if and only if $X'\in V_0(G_{R'})$.
Similarly, $X\in V_2(G_R)$ if and only if $X'\in V_2(G_{R'})$.  A
similar argument applies if $X=Z_{h+1,n}$, and we can then apply the
same arguments with $Y$ and $Y'$ in place of $X$ and $X'$ to establish
the corresponding claims for $Y$ and $Y'$.

To prove (a) and (b), we will use casework based on whether $X\cup Y$
is a quasi-inversion or a $(k+1)$-inversion.  In each case, we
consider subcases based on where $n$ appears in the window notation of
$w$ and the resulting permutation $w'$ after $j$ and $j+1$ are
contracted. By hypothesis $j,j+1\in [n-2]$, so the elements $n-1$ and
$n$ are not contracted to obtain $w'$.

\bigskip

\noindent \textbf{Case 1: $X\cup Y$ is a quasi-inversion.}  Since
$X\cup Y$ is a quasi-inversion, the window notation of $w$ contains
\begin{equation}\label{eq:double_quasi}
(n-1,n-2,\dots,h+2,h,h+1,h-1,\ldots,2,1). \end{equation} 
\begin{itemize}
\item If $w^{-1}(n)=1$, then $\Inv_{k+1}(w)=\{Z_h,Z_{h+1}\}$ and the
unique relation in $\P_w(n,k+1)$ is determined by the parity of
$k+1-h$.  Depending on the relative order of $h$ and $j$, we have that
$\Inv_{k-1}(w')$ is either $\{Z_h',Z_{h+1}'\}$ or
$\{Z_{h-2}',Z_{h-1}'\}$. The parities of $(k+1)-h$, $(k-1)-h$, and
$(k-1)-(h-2)$ are the same. Therefore, $\Fl_{j} \circ
\Fl_{j+1}:\P_w(n,k+1) \to \P_{w'}(n-2,k-1)$ is an order-preserving
bijection on two-element posets, so it maps lower-order ideals of
$\P_w(n,k+1)$ to lower-order ideals of $\P_{w'}(n-2,k-1)$.
\item If $w^{-1}(n)>1$, then $\Inv_{k+1}(w)$ and $\Inv_{k-1}(w')$ have
at most one element since $n=k+2$ and $w$ contains the pattern in
\eqref{eq:double_quasi}.  Since $R' \subseteq \Inv_{k-1}(w')$, it must
be the case that $R'$ is a lower-order ideal of $\P_{w'}(n-2,k-1)$. 

\end{itemize}
 In both cases, $R'$ trivially satisfies the MSZ condition, so we
conclude (a) holds.

Depending on the relative order of $h$ and $j$, we also see that
$\{X',Y'\}$ is either $\{Z_{h,n-2}',Z_{h+1,n-2}'\}$ or
$\{Z_{h-2,n-2}',Z_{h-1,n-2}'\}$.  By \eqref{eq:double_quasi} and the
order-preserving flattening map, $X' \cup Y'$ is a quasi-inversion of
$w'$.  Furthermore, the parity of $k-h$ determines the quasi-inversion
edge $X\to Y$ in $G_R$, and the same parity of $(k-2)-h$ or
$(k-2)-(h-2)$ will determine the quasi-inversion edge $X'\to
Y'$ in $G_{R'}$. Hence, (b) follows.

\bigskip

\noindent
\textbf{Case 2: $X\cup Y$ is a $(k+1)$-inversion}. 
Since $X\cup Y \in \Inv_{k+1}(w)$, the window notation of  $w$
contains 
\begin{equation}\label{eq:double_full}
(n-1,n-2,\dots,h+2,h+1,h,h-1,\ldots,2,1). 
\end{equation}

\begin{itemize}
\item If $w^{-1}(n)=1$, then $\Inv_{k+1}(w)=P(Z)$ and
$\Inv_{k-1}(w')=P(Z')$. As $R$ is a prefix or suffix of $P(Z)$, we
conclude that $R'$ is a prefix or suffix of $P(Z')$ so the MSZ
condition holds.  By \Cref{lem:poset.packet.antichain}, $P(Z')$ is an
antichain of $\P_{w'}(n-2,k-1)$, so any subset is a lower order ideal
including $R'$.  

\item If $w^{-1}(n)=2$, then $\Inv_{k+1}(w)=\{Z_n,Z_{n-1}\}$. Since
$j,j+1 \in [n-2]$, we have that $\Inv_{k-1}(w')=\{Z_{n-2}',Z_{n-3}'\}$
and $Z'$ is a quasi-inversion of $w'$ by \eqref{eq:double_full}.  As
in Case 1 with two element inversion sets, $R$ being a lower-order ideal
of $\P_w(n,k+1)$ implies that $R'$ is a lower-order ideal of
$\P_{w'}(n-2,k-1)$ by the parity condition.  Since
$\Inv_{k}(w')=\emptyset$, the MSZ condition trivially holds.

\item If $w^{-1}(n)>2$, then $\Inv_{k+1}(w)=\{Z_n\}$, so $R'$ is
either $\emptyset$ or $\{Z'_{n-2}\}$. The largest value $n-2$ of $w'$
cannot occur first, since the original values $n-1,n$ are not
contracted. If it occurs second, then
$\Inv_{k-1}(w')=\{Z'_{n-2},Z'_{n-3}\}$, implying the quasi-inversion
relation $Z'_{n-2}<_{\P_{w'}(n-2,k-1)}Z'_{n-3}$ holds. Otherwise,
$\Inv_{k-1}(w')=\{Z'_{n-2}\}$. In either case, both possibilities for
$R'$ are lower order ideals. Since $\Inv_k(w')=\emptyset$, the MSZ
condition is also satisfied.

\end{itemize}
In all three cases, we conclude (a) holds.

Observe that $Z_n\not \in R$ or $Z_n\in R$ determines if $X\to Y$ is a
full lex or antilex edge. This also determines $Z_{n-2}'\not \in R'$ or 
$Z_{n-2}'\in R'$, which in turn determines if $X'\to Y'$ is a full lex
or full antilex edge. Hence, (b) holds as well.
\end{proof}

\begin{remark}
The various hypotheses in \cref{lem:reduction_remove2} cannot be removed in general. We give some examples below.
\begin{itemize}
    \item Consider $w=(5,3,2,1,4)$, $R=\Inv_4(w)=\{[1,2,3,5]\}$, and
$j=1$, so $w'=(3,1,2)=\Fl_{1}\circ \Fl_{2}(w)$.  Note that $[1,2,5]
\textcolor{red}{\to} [1,2,3]$ is an antilex edge in $G_{R}$, and both elements
contain $1$ and $2$. However, one of these contains $n=5$, violating
condition (i) in \cref{lem:reduction_remove2}. Observe that
$\P_{w'}(3,2)$ contains the relation $[2,3]<_{\P_{w'}(3,2)} [1,3]$, so
$R'=(R/2)/1=\{[1,3]\}$ is not a lower-order ideal of $\P_{w'}(3,2)$
and $R'\notin \C_{w'}(3,2)$.

\item Consider $w=(5,3,4,2,1)$, $R=\Inv_4(w)=\{[1,2,3,5],[1,2,4,5]\}$,
and $j=2$, so $w'=(3,2,1)=\Fl_2\circ \Fl_3(w)$.  Note that $[1,2,3] \to [1,2,4]$ is a quasi-inversion edge in
$G_{R}$, but $[1,2,4]$ does not contain both $j=2$ and $j+1=3$, violating condition (ii) in \cref{lem:reduction_remove2}.  Observe that $R'=(R/3)/2=\{[1,3]\}$ is not a prefix or suffix of $P([1,2,3])=\Inv_2(w')$, so $R'\notin \C_{w'}(3,2)$.  Note that if we instead chose $j=1$, then the conditions in \cref{lem:reduction_remove2} are satisfied, and $R'=(R/2)/1$ is consistent for $w'=\Fl_1\circ \Fl_2(w)$.

\item Consider $w=(7,6,2,5,4,3,1)$, $R=P([1,3,4,5,6,7])=\Inv_5(w)$,
and $j=4$, so $w'=(5,4,2,3,1)=\Fl_4\circ \Fl_5(w)$. The edge
$[3,4,5,6]\textcolor{red}{\to} [1,4,5,6]$ in $G_R$ satisfies the
conditions (i) and (ii) in \cref{lem:reduction_remove2}, but this is an example where $n=7$ and $k=4$, so $n>k+2$. We find that $R'=(R/5)/4=P([1,3,4,5])$. However, notice that $[2,4,5]<_{\P_{w'}(5,3)} [3,4,5]$, so $R'$ is not a lower-order ideal of $\P_{w'}(5,3)$. Additionally, observe that $[1,2,4,5]\in \Inv_4(w')$, but $R'\cap P([1,2,4,5])=\{[1,4,5]\}$ is neither a prefix nor suffix of $P([1,2,4,5])$. This example demonstrates both the lower-order ideal and MSZ condition failing to hold.
\end{itemize}
 \end{remark}

We next consider edges between $V_0(G_R),V_1(G_R)$, and $V_2(G_R)$. In the graph $G_R$ from \cref{ex:partition} and \cref{fig:Gr example}, observe that all edges had the form $X\to Y$ with $X\in V_i(G_R)$, $Y\in V_j(G_R)$, and $i\leq j$. We will show that this holds in general using a combination of deletion and contraction, which allow us to reduce to a verification for finitely many permutations and consistent sets.

\begin{lemma}\label{lem:partition}
If $R\in \C_w(n,k+1)$ and $X\to Y$ is an edge in $G_R$, then $X\in
V_a(G_R)$ and $Y\in V_b(G_R)$ for some  $0\leq a\leq b\leq 2$.
\end{lemma}

\begin{proof}
    First, suppose that $n$ appears in $X$ or $Y$. If $n$ appears in both, then $X,Y\in V_1(G_R)$, and the result follows. If exactly one contains $n$, then there are three possibilities.
    \begin{itemize}
        \item If $X\to Y$ is a full lex edge, then this can only occur if $n\in Y$ and $X\cup \{n \}=X\cup Y\notin R$. Hence, $X\in V_0(G_R)$ and $Y\in V_1(G_R)$. 
        \item If $X\to Y$ is a full antilex edge, then this can only
	occur if $n\in X$ and $Y\cup \{n \} = X\cup Y \in R$. Hence, $X\in
	V_1(G_R)$ and $Y \in V_2(G_R)$. 
        \item If $X\to Y$ is a quasi-inversion edge, then this can
	only occur if $n\in Y$ and $X\cup \{n \} = X\cup Y\notin \Inv_{k+1}(w)$. This in turn implies that $X\cup \{n \}\notin R$, so $X\in V_0(G_R)$ and $Y\in V_1(G_R)$. 
    \end{itemize}
    In all three cases above, the result follows. 

    We now consider the case $X$ and $Y$ do not contain $n$, so these
    elements must be in $V_0(G_R) \cup V_2(G_R)$. In this
case, we have that $\{X,Y \}=\{Z_{i},Z_{i+1} \} \subseteq P(Z)$ for
some $1\leq i\leq k$ and $Z \in \binom{[n-1]}{k+1}$, where $Z$ is
either an inversion or a quasi-inversion for $w$ since $X\to Y$ is an
edge in $G_{R}$. As $Z$ does not
contain $n$ and has size $k+1$, this case only occurs when $n\geq
k+2$. 

If $n>k+2$, then there exists some $j\in [n-1]$ that is not an element
of $Z$. Defining $w'=\Fl_j(w)$ and $R'=R\setminus j$,
\cref{cor:deletion_consistent} implies that $R'\in \C_{w'}(n-1,k+1)$
and $\Fl_j(X)\to \Fl_j(Y)$ is an edge in $G_{R'}$.  Furthermore,
$X\in V_0(G_R)$ if and only if  $X\cup \{n\}\notin R$ if and only if  $\Fl_j(X\cup
\{n\})=\Fl_j(X)\cup \{n-1\}\notin R'$ if and only if $\Fl_j(X)\in V_0(G_{R'})$.
Similarly,  $X\in V_2(G_R)$ if and only if $\Fl_j(X)\in V_2(G_{R'})$.  The
same argument applies to $Y$ and $\Fl_j(Y)$. Hence, it suffices to
prove that if $\Fl_{j}(X) \in V_{a}(G_{R'})$ and $\Fl_{j}(Y) \in
V_{b}(G_{R'})$ then $a\leq b$.  Since $R \in \C_w(n,k+1)$ and $R'\in
C_{w'}(n-1,k+1)$, it suffices to consider when $n=k+2$ by iterating
this argument.

In the case where $n=k+2$, observe that when $k\geq 5$, the elements
$X$ and $Y$ must contain some consecutive elements $j,j+1\in
[n-2]$. Then \cref{lem:reduction_remove2} implies that we can instead
consider the edge $X'\to Y'$ in $G_{R'}$, where $X',Y'$, and $G_{R'}$
are defined as in that lemma. Repeating this argument, we reduce to
the case where $1\leq k\leq 4$ and $n=k+2$, so $3\leq n \leq 6$. The
lemma can be verified directly for this case by the reader, or by a
computer verification over all $w \in \S_n$ and $R \in \C_w(n,n-1)$
with $3\leq n\leq 6$.  See appendix for verification details.
\end{proof}

\begin{corollary}\label{cor:partition}
For any $R\in \C_w(n,k+1)$, define $G_0,G_1$, and  $G_2$ to be the subgraphs of $G_R$ induced by $V_0(G_R)$, $V_1(G_R)$, and $V_2(G_R)$, respectively. If $G_0,G_1$, and $G_2$ are all acyclic, then $G_R$ is acyclic. 
\end{corollary}

Finally, we turn our attention to the subgraph of $G_R$ induced by $V_1(G_R)$. For this, we will introduce an additional operation called truncation, which is a specific combination of deletion and contraction.

\begin{definition}\label{def:truncation}
    Let $w\in \S_n$, and let $m=w^{-1}(n)$.  Define
$t_w:\{w(m+1),\ldots,w(n)\}\to [n-m]$ to be the unique
order-preserving bijection. The \defn{truncation} of $w$ is the
permutation in $\S_{n-m}$ with window
notation \[\trunc{w}{n}=(t_w(w(m+1)),\ldots,t_w(w(n))).\] For any
element $X=[x_1,\ldots,x_{\ell-1},n]\subseteq
\{w(m),w(m+1),\ldots,w(n)\}$, define the \defn{truncation} of $X$ to
be \[T_w(X)=[t_w(x_1),t_w(x_2),\ldots,t_w(x_{\ell-1})].\] For any
$R\subseteq \Inv_{\ell+1}(w)$, the \defn{truncation} of $R$
is \[\trunc{R}{n}=\{T_w(X): X\in R,n\in X\}.\] 
\end{definition}

Observe that truncation of $R$, $w$, or $X$ is equivalent to first
performing deletion on the first $m-1$ elements in the window notation
of $w$, followed by a contraction on the largest element, which will
be the element $n-m+1$ in the first position.  Therefore,
\begin{equation}\label{eq:Twbijection}
T_w: \{X\in \Inv_{\ell+1}(w):n\in X\} \to \Inv_{\ell}(\trunc{w}{n})
\end{equation}
is a bijection and $\trunc{\Inv_{\ell+1}(w)
}{n}=\Inv_\ell(\trunc{w}{n})$ for any $\ell\geq 0$.

\begin{example}
    Consider $w=(2,5,6,3,4,1)\in \S_6$ where
$\Inv_{3}(w)=\{[1,3,5],[1,3,6],[1,4,5],[1,4,6]\}$.  Then,
$\trunc{w}{6}=(2,3,1)$ and
$\trunc{\Inv_{3}(w)}{6}=\{[1,2],[1,3]\}=\Inv_{2}(2,3,1)$.
\end{example}

We next will show that the subgraph of $G_R$ induced by $V_1(G_R)$ can be identified using truncation. Recall that the \defn{reverse} of a directed graph $G$ is the graph obtained by reversing the directions of all edges in $G$.

\begin{lemma}\label{lem:V1}
Let $m=w^{-1}(n)$, and $R\in \C_w(n,k+1)$.
If $R'=\Inv_{k}(\trunc{w}{n})\setminus (\trunc{R}{n})$, then $R'\in
\C_{\trunc{w}{n}}(n-m,k)$ and the subgraph of $G_R$ induced by
$V_1(G_R)$ is isomorphic as a directed graph to the reverse of
$G_{R'}$.
\end{lemma}

\begin{proof}
    Throughout, we will use the notation from \cref{def:truncation},
and let $w'=\trunc{w}{n}$. Additionally, we will use the fact that
truncation is equivalent to several deletions, followed by a
contraction on the largest element, and hence $T_{w}: \{X\in
\Inv_{\ell+1}(w):n\in X\} \to \Inv_\ell(w')$ is a bijection for all
$\ell\geq 0$.  Combining these facts with \cref{lemma:F_j_properties},
\cref{lem:contract_n}, and the assumption $R\in \C_w(n,k+1)$, we
conclude that $\trunc{R}{n}$ is an upper-order ideal of
$\P_{w'}(n-m,k)$ that satisfies the MSZ condition. Hence, its
complement $R'$ is a lower-order ideal that satisfies the MSZ
condition, so $R'\in \C_{w'}(n-m,k)$.

For the graph isomorphism, the map $T_w$ also gives a bijection
between $V_1(G_{R})=\{X \in \Inv_{k}(w): n \in X \}$ and $\Inv_{k-1}(w')$, which are the vertices in $G_{R'}$. Under this bijection, similar arguments using \cref{lemma:F_j_properties} and \cref{lem:contract_n} show that an edge $X\rightarrow Y$ in $G_R$ induced by ${\P_{w}(n,k)}$, $Z\notin R$, or $Z\in R$ respectively corresponds to an edge $T_w(X)\leftarrow T_w(Y)$ in $G_{R'}$ induced by $\P_{w'}(n-m,k-1)$, $T_w(Z)\in R'$, or $T_w(Z)\notin R'$. We conclude that the subgraph of $G_R$ induced by
$V_1(G_R)$ is isomorphic to the reverse of
$G_{R'}$.
\end{proof}

By \cref{cor:deletion_consistent}, the subgraph of $G_R$ induced by $V_0(G_R)\cup V_2(G_R)$ can be described in terms of deletion, and by \cref{lem:V1}, the subgraph induced by $V_1(G_R)$ can be described in terms of truncation. We illustrate the combined results of \cref{cor:deletion_consistent} and \cref{lem:V1} with an example.

\begin{example}\label{ex:deletioncontraction2}
We revisit $w$ from \cref{ex:partition}. Observe that $R\setminus 6=\Inv_4(\Fl_6(w))$ is a consistent subset of $\Inv_4(\Fl_6(w))$, and the graph $G_{R\setminus 6}$ is shown on the left in \cref{fig:deletioncontraction}. While $\trunc{R}{6}$ is not a consistent subset of $\Inv_3(\trunc{w}{6})$, its complement $\Inv_3(\trunc{w}{6})\setminus (\trunc{R}{6})=\{[1,2,4],[1,3,4]\}$ is, and the corresponding graph $G_{\Inv_3(\trunc{w}{6})\setminus (\trunc{R}{6})}$ is shown on the right in \cref{fig:deletioncontraction}. Partitioning $\Inv_3(w)$ as in \eqref{eq:partition} and comparing \cref{fig:Gr example2,fig:deletioncontraction}, we see that the subgraph of $G_R$ induced by $V_0(G_R)\cup V_2(G_R)$ is isomorphic to $G_{R\setminus 6}$, and the subgraph of $G_R$ induced by $V_1(G_R)$ is isomorphic to the reverse of $G_{\Inv_3(\trunc{w}{6})\setminus (\trunc{R}{6})}$.

    \begin{figure}[h!]
\centering 
    \begin{tikzpicture}
        \node (134) at (0,0) {$[1,3,4]$};
        \node (135) at (-1,1) {$[1,3,5]$};
        \node (124) at (1,1) {$[1,2,4]$};
        \node (125) at (0,2) {$[1,2,5]$};
        \node (13) at (3,0) {$[1,3]$};
        \node (12) at (3,1) {$[1,2]$};
        \node (15) at (5,0) {$[1,5]$};
        \node (14) at (5,1) {$[1,4]$};
        \node (25) at (8,0) {$[2,5]$};
        \node (24) at (7,1) {$[2,4]$};
        \node (35) at (9,1) {$[3,5]$};
        \node (34) at (8,2) {$[3,4]$};
        \draw[->] (134) edge node {} (135);
        \draw[->] (134) edge node {} (124);
        \draw[->] (135) edge node {} (125);
        \draw[->] (124) edge node {} (125);
        \draw[<-] (13) edge node {} (12);
        \draw[<-] (15) edge node {} (14);
        \draw[<-] (25) edge node {} (24);
        \draw[<-] (24) edge node {} (34);
        \draw[<-] (25) edge node {} (35);
        \draw[<-] (35) edge node {} (34);
        \draw[->,color=blue] (12) edge node {} (15);
        \draw[->,color=blue] (15) edge node {} (25);
        \draw[->,color=blue] (13) edge node {} (15);
        \draw[->,color=blue] (15) edge node {} (35);
        \draw[<-,color=red] (12) edge node {} (14);
        \draw[<-,color=red] (14) edge node {} (24);
        \draw[<-,color=red] (13) edge node {} (14);
        \draw[<-,color=red] (14) edge node {} (34);
    \end{tikzpicture}
    \caption{The graphs $G_{R\setminus 6}$ and $G_{\Inv_3(\trunc{w}{6})\setminus (\trunc{R}{6})}$ from \cref{ex:deletioncontraction2}. Quasi-inversion edges are shown in black, full lex edges are shown in blue, and full antilex edges are shown in red.}
    \label{fig:deletioncontraction}
\end{figure}
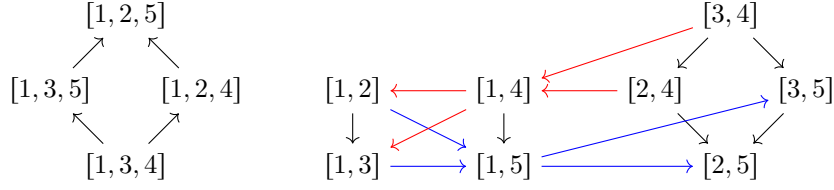
\end{example}

We conclude this section by proving \cref{lem:GRacyclic}, which states that $G_R$ is acyclic for all $R\in \C_w(n,k)$. Our approach uses induction with several key results from this section.

\begin{proof}[Proof of \cref{lem:GRacyclic}] Recall, the hypothesis
requires $k>1$ and $n>0$.  If $n<k$, then $\Inv_{k}(w)=\emptyset$ and
$\Inv_{k-1}(w)$ has at most one element so the equivalence claim can
be verified by definition.  Thus, we may assume $k\leq n$.  If $k=2$
and $R\in \C_w(n,2)$, then $G_R$ is an acyclic oriented complete graph
on $\Inv_1(w)=[n]$ by \Cref{rem:k=2graphs}.  Therefore, we may assume
by induction that $G_{R'}$ is acyclic for all permutations $w'\in
\S_{n'}$ and consistent sets $R'\in \C_{w'}(n',k')$ with $n'<n$ or
$k'<k$, where $3\leq k\leq n$.

Consider the partition of the vertices in $G_R$ from
\eqref{eq:partition}, and define their respective induced subgraphs
$G_0,G_1,$ and $G_2$. \cref{lem:V1} and the induction hypothesis imply
$G_1$ is acyclic. \cref{cor:deletion_consistent} and the induction
hypothesis imply that $G_0$ and $G_2$ are acyclic.  The result now
follows from \Cref{cor:partition} by induction.
\end{proof}

\section{Higher Bruhat Orders for Affine Permutations}\label{sec:affine}

In this section, we will generalize the definitions and some of the
results of the higher Bruhat orders given in \Cref{sec:non-longest} to
elements of the affine symmetric group.  Throughout this section
assume $k,n\in \mathbb{Z}$ are positive integers and $w \in
\affineS_{n}$, unless otherwise specified.   Since
$\affineS_{1}=\S_1$, we have already considered the $n=1$ case, so
assume throughout this section $n\geq 2$.

\subsection{Affine Higher Order Inversions/Quasi-inversions }

Recall the $k$-inversions for classical permutations are subsets
$[x_1,\dots , x_k] \in \binom{[n]}{k}$ such that $w^{-1}(x_{1})> \cdots >
w^{-1}(x_{k})$, where the square bracket notation means
$x_{1}<x_{2}<\cdots x_{k}$.  In this notation, $[x,y] \in
\binom{[n]}{2}$ is an inversion for $w$ whenever
$w^{-1}(x)>w^{-1}(y)$.  For an affine permutation $w \in
\affineS_{n}$, the inversions are defined similarly by pairs $(x,y)\in
[n] \times \mathbb{Z}$ such that $x<y$ and $w^{-1}(x)> w^{-1}(y)$ in
\eqref{eq:inversion.set.def}.  Furthermore, $(x,y) \in \Inv(w)$
implies $x$ and $y$ are not congruent modulo $n$ by the
$n$-periodicity of $w \in \affineS_{n}$.

In order to generalize the notion of $k$-inversions and
$k$-quasi-inversions to affine permutations accounting for
$n$-periodicity, let $\binom{\Z }{k}$ denote the $k$-subsets of $\Z$.
Define an equivalence relation $\sim_n$ on the elements of $\binom{\Z
}{k}$ where $ \{x_1< x_2 <\cdots< x_k \}\sim_n \{y_1< y_2< \cdots< y_k
\}$ if $k>1$ and there exists a single integer $m$ such that
$x_{i}=y_{i}+mn$ for all $1\leq i\leq k$. If $k=1$, then $\{x \}
\sim_n \{y \}$ if and only if $x = y$.  We use the square bracket
notation $[x_1,\dots , x_k]$ to denote the equivalence class of
$\{x_1< x_2< \cdots< x_k \} \in \binom{\Z }{k}$, listed in increasing
order as before.  Denote the equivalence classes of $k$-subsets of
$\Z$ with distinct elements modulo $n$ as
\[
\binom{\mathbb{Z}}{k}_n = \bigl\{ [x_1, \ldots, x_k] : x_i
\not\equiv x_j (\mathrm{mod\ } n)\text{ for $1 \le i < j \le k$} \bigr\}.
\]
Unless otherwise specified, the canonical representative of an
equivalence class will be chosen such that $1 \le x_1 \le n$ whenever
$k>1$.  Observe that for $k=0$, $\binom{\mathbb{Z}}{0}_n=\{\emptyset
\}$.  For $k=1$, $\binom{\Z}{1}_n$ is essentially the same as $\Z$
itself if we identify the equivalence class of a 1-subset $\{i \}$
with $i \in \Z$.  For all $0<k\leq n$, the size of
$\binom{\mathbb{Z}}{k}_n$ is infinite, and $\binom{\mathbb{Z}}{k}_n =
\emptyset$ for $k>n$.

\begin{definition}\label{def:congruent}
Two equivalence classes $X = [x_1, \ldots, x_k], Y = [y_1, \ldots,
y_k] \in \binom{\Z}{k}_n$ are \defn{congruent}, denoted $X \equiv Y$
if $x_i \equiv y_i \pmod{n}$ for all $1 \le i \le k$, assuming $X$ and
$Y$ are increasing ordered lists.  
\end{definition}

\begin{example}
If $n=3$, the equivalence class
\[
[2, 4]=\{\ldots, \{-4,-2 \}, \{-1,1 \}, \{2,4 \}, \{5,7 \}, \ldots\}
\in \binom{\Z}{2}_3
\]
so $[2,4]$ represents an infinite set of $2$-subsets of $\Z$. Also,
$[-4,-2] = [-1,1]=[2,4]=[5,7]$ as equivalence classes. They are all
congruent to $[2,7]$, but not equal. Note, $[2,4]$ and $[1,2]$ are not
congruent even though the set of congruence classes represented in
both are the same. The order matters.  The equivalence classes of $\{-4\}, \{-1 \},\{2
\},\{5 \}$ in $\binom{\Z}{1}_3$ are congruent but not equal for $n=3$.
\end{example}

Define the \defn{(affine) packet} for $X=[x_{1},x_{2},\dotsc ,x_{k}]
\in \binom{\Z}{k}_{n}$ as $P(X)=\{X_{1},\dotsc , X_{k} \}$ exactly as
in \cref{subsection:higher-bruhat-order-intro} provided $k\neq 2$.
However, in the case $k=2$, some care must be taken to make
well-defined statements using such packets since
$X=[x_{1},x_{2}]=[x_{1}+n,x_{2}+n]$ but $\{x_{1} \} \neq \{x_{1}+n
\}$.  The key thing to keep in mind is the $n$-periodicity of $w$ as a
total order on $\Z$, which implies $x_{1}<_{w}x_{2}$ if and only if
$x_{1}+mn<_{w}x_{2}+mn$ for all $m \in \Z$.  Therefore,
$[x_{1},x_{2}]$ is an inversion for $w$ if and only if
$x_{2}+mn<_{w}x_{1}+mn$ for any $m \in \Z$.

The following definitions are direct analogs of
\Cref{def:k.inversion.Sn} and \Cref{def:quasi.inversion}.  The key
difference is $\binom{[n]}{k}$ is replaced with 
$\binom{\Z}{k}_n$.

\begin{definition}\label{def:k.inversion.affine}
We say $X=[x_1,\ldots,x_k] \in
\binom{\Z}{k}_{n}$ is a $\boldsymbol{k}$-\defn{inversion} for $w$
provided $w^{-1}(x_1)>\cdots >w^{-1}(x_k)$.  The $\boldsymbol{k}$-\defn{inversion
set} of $w$ is
\begin{equation}\label{eq:affine.k.inversion.set.def}
\Inv_k(w)=\left\{[x_1,\ldots,x_k]\in \binom{\mathbb{Z}}{k}_n:
w^{-1}(x_1)>\ldots >w^{-1}(x_k) \right\}.
\end{equation}
Similarly, $X$ is a $\boldsymbol{k}$-\defn{quasi-inversion} of $w$ if all but exactly
one of the pairs $\{[x_i, x_j] : 1 \le i < j \le k\}$ is a 2-inversion
of $w$.  Let $\Qinv_{k}(w)$ be the set of $k$-quasi-inversions for $w$.
\end{definition}

Note that for $k>n$, we have that $\Inv_{k}(w)=\emptyset$, and
$\Inv_{0}(w)=\{\emptyset\}$.  Every $\{i\}
\in \binom{\Z}{1}_{n}$ is a 1-inversion of $w$, so $\Inv_{1}(w)=\binom{\Z}{1}_{n}$ is
infinite, while $|\Inv_{k}(w)|$ is finite for $k>1$.  For $k=2$,
\eqref{eq:affine.k.inversion.set.def} is equivalent to the inversion
set of $w$ as defined in \eqref{eq:inversion.set.def} assuming each
$[x,y] \in \Inv_{2}(w)$ is written in canonical form so $x \in [n]$.

Many of the statements from the classical permutation case in
\Cref{sec:non-longest} carry over easily to the affine permutation
case, but there are some caveats.  When no additional work is needed
to prove the affine analog, we will simply point out the notational
updates needed.  For example, to prove $[x_1,\ldots,x_k]\in
\binom{\mathbb{Z}}{k}_n$ is a $k$-inversion for $w$, it suffices to
show each adjacent $[x_{i},x_{i+1}] \in \Inv_{2}(w)$ as in
\Cref{lem:adjacent.pairs}.  The four possible intersection types for
$P(X)\cap \Inv_k(w)$ depending on size are characterized exactly as in
\Cref{lem:intersection_size} using the same proof simply by replacing
$\binom{[n]}{k}$ by $\binom{\mathbb{Z}}{k}_n$.  The properties of
$k$-quasi-inversions stated in \Cref{lem:types.of.packets} extend to
affine permutations in the same way.  However, for $k = 2$, the affine
generalization of \cref{lem:types.of.packets} says that the
2-quasi-inversions for $w$ are the elements in
$\binom{\mathbb{Z}}{2}_n \setminus \Inv_{2}(w)$, which is an infinite
set in the affine case.

\begin{example}\label{ex:m3m287}
    Consider the permutation $(-3, -2, 8, 7) \in \affineS_4$. It has
    inversion sets
    \begin{align*}
        \Inv_1(w) &\cong \Z\\
        \Inv_2(w) &= \left\{[1,3], [1, 4], [1, 7], [1,8], [2,3], [2,4], [2,7], [2,8], [3,4]\right\}\\
        \Inv_3(w) &= \left\{[1,3,4], [1,7,8], [2,3,4], [2,7,8]\right\}\\
        \Inv_4(w) &= \emptyset,
    \end{align*}
and quasi-inversion sets  
\begin{small}
\begin{align*}
\Qinv_{3}(w)&=\{[1,4,7],[1,3,8],[1,2,3],[1,6,7],[2,5,7],[1,2,7],[1,2,4],[1,6,8],[2,4,7],[2,5,8],[1,2,8],[2,3,8] \}\\
\Qinv_{4}(w)&=\{[1,2,3,4],[1,6,7,8],[2,5,7,8],[1,2,7,8] \}.
\end{align*}
\end{small}
\vspace{-1.25em}

\noindent 
For all $k>4$, $\Inv_{k}(w)=\Qinv_{k}(w)=\emptyset$.  Note that we
don't include $[7,8]$ in the list for $\Inv_2(w)$ since $[3,4]=[7,8]$
when $n=4$, but $[1,7,8]$ does appear in $\Inv_3(w)$ since $[1,7,8]
\neq [1,3,4]$.
\end{example}

In the affine case, there are additional properties of congruent
$k$-inversions that will be important in what follows. For example, if
$[x,y] \in \Inv_{2}(w)$, then switching the order of their congruence
classes results in $[y,x+bn]\not \in \Inv_{2}(w)$ for all integers $b$
such that $y<x+bn$ by $n$-periodicity of $w$.  The next lemma
generalizes this necessary condition to $k$-inversion sets.

\begin{lemma}\label{lem:congruence.order}
Let $X=[x_1,\ldots,x_k]$ and $Y=[y_1,\ldots,y_k]$.  If $X,Y \in
\Inv_{k}(w)$ and we have the equality of unordered
sets \begin{equation}\label{eq:sets.equal} \{x_1 (\textrm{mod }n),
\ldots, x_k (\textrm{mod }n)\} = \{y_1 (\textrm{mod }n), \ldots, y_k
(\textrm{mod }n)\},
\end{equation}
    then $X$ and $Y$ are congruent equivalence classes in $\binom{\Z
}{k}_{n}$.
\end{lemma}

\begin{proof}
The case $k=1$ is immediate, so assume $k>1$. Since the entries of $X$
and $Y$ are each distinct modulo $n$, \eqref{eq:sets.equal} determines
a unique permutation $\pi\in \S_k$ such that $x_i\equiv y_{\pi(i)}
\pmod n$ for all $i\in [k]$. Suppose $i<j$ and $\pi(i)>\pi(j)$. Set
$q=\pi(i)$ and $p=\pi(j)$, so $p<q$, $x_i\equiv y_q \pmod n$, and
$x_j\equiv y_p \pmod n$. By the affine analogue of
\Cref{lem:adjacent.pairs}, both $[x_i,x_j]$ and $[y_p,y_q]$ are in
$\Inv_2(w)$. Since $[y_p,y_q]=[x_j,x_i+mn]$ for some $m>0$, this
contradicts \Cref{lem:inversion.sets}. Hence $\pi$ has no inversions,
so $\pi=\id$, and $X$ and $Y$ are congruent.
\end{proof}

For $0\leq i\leq k$, let $v_i^{(n,k)}$ denote the vector
$(0,\ldots,0,n,\ldots,n)\in\mathbb{Z}^k$ consisting of $i$ copies of
$0$ followed by $(k-i)$ copies of $n$.  When $n,k$ are unambiguous, we
write $v_i$ for $v_i^{(n,k)}$.  If
$X=[x_1,\ldots,x_k]\in\binom{\Z}{k}_n$ and there exists an $i \in [k-1
]$ such that $x_{i+1}-x_i>n$, define the \defn{compression across the
$i$th gap} by
\[
X-v_i := [x_1,\ldots,x_i, x_{i+1}-n,\ldots,x_k-n]\in\binom{\Z}{k}_n.
\]
In the next lemma, we show $\Inv_{k}(w)$ is \defn{closed under compression} as in
\Cref{lem:inversion.sets} Part (c).   

\begin{lemma}\label{lem:affine_n_shift}
Assume $k>1$, then the following hold.  
\begin{enumerate}[label=(\alph*)]
\item 
If $X=[x_1,\ldots,x_k]\in\Inv_k(w)$ and $x_{i+1}-x_i>n$ for some $i\in[k-1]$,
then $X-v_i\in\Inv_k(w)$.

\item 
For $X=[x_1,\ldots,x_k]\in\Inv_k(w)$, let $[y_1,\ldots,y_k]$ be the unique element of $\binom{\Z}{k}_n$
such that $y_1=x_1$, $y_j\equiv x_j\pmod n$ for all $j\in[k]$, and $y_{j+1}-y_j<n$ for all $j\in[k-1]$.
Then $[y_1,\ldots,y_k]\in\Inv_k(w)$.  
\end{enumerate}
\end{lemma}

\begin{proof}
For Part (a), observe that $x_1<\cdots<x_i<x_{i+1}-n<\cdots<x_k-n$ and
the values are pairwise noncongruent since $X \in \Inv_{k}(w)
\subseteq \binom{\Z }{k}_{n}$ and $x_{i+1}-x_i>n$, so
$X-v_i=[x_1,\ldots,x_i,x_{i+1}-n,\ldots,x_k-n] \in \binom{\Z
}{k}_{n}$.  Observe $x_{i+1}-x_i>n$ and $[x_{i},x_{i+1}] \in
\Inv_{2}(w)$ imply $[x_{i},x_{i+1}-n]\in \Inv_{2}(w)$ by
\Cref{lem:inversion.sets}(c).  Furthermore,
$X=[x_1,\ldots,x_k]\in\Inv_k(w)$, implies $[x_{1},x_{2}],
[x_{2},x_{3}],\dotsc , [x_{i-1},x_{i}] \in \Inv_{2}(w)$ and
$[x_{i+1}-n,x_{i+2}-n]=[x_{i+1},x_{i+2}], \dotsc ,
[x_{k-1}-n,x_{k}-n]=[x_{k-1},x_{k}] \in \Inv_{2}(w)$.  Thus, $X-v_i
\in \Inv_{k}(w)$ since every adjacent pair in $X-v_i$ corresponds with
a 2-inversion of $w$.

To prove Part (b), start from $X$ and repeatedly apply the compression
operation from Part (a) whenever some gap exceeds $n$.  The terminal
element $[y_1,\ldots,y_k]$ is uniquely characterized by $y_1=x_1$,
$y_j\equiv x_j\pmod n$, and $y_{j+1}-y_j<n$ for all $j$. By (a), every
intermediate element stays inside $\Inv_k(w)$, so the terminal element
lies in $\Inv_k(w)$.
\end{proof}

We will call $[y_1, \ldots, y_k]$ the \defn{compressed representative}
for the congruence class of $[x_1, \ldots, x_k]$ if the conditions in
\Cref{lem:affine_n_shift} Part (b) are satisfied.  Note, the set of
all compressed representatives for $k$-inversions for $w$ is 
completely determined by the $k$-cliques of the affine inversion graph
$I_{w^{-1}}$ from \Cref{def:BB.graph}.

In order to characterize the distinct congruence classes in
$\Inv_{k}(w)$, we introduce another partial order where $X=[x_1,
\ldots, x_i,x_{i+1}, \ldots, x_k]$ and $X+v_{i}:=[x_1, \ldots,
x_i,x_{i+1}+n, \ldots, x_k+n]$ form a covering relation.  This will
allow us to generalize \Cref{lem:inversion.sets} Part (c) to
$k$-inversions in \Cref{cor:congruence.lemma}.

\begin{definition}\label{def:congruence.poset}
Recall $v_i=v_i^{(n,k)}$ from above.  If $k>1$, the \defn{congruence
poset} $\boldsymbol{\CR(n,k)}$ is the set $\binom{\Z}{k}_n$ with order
relations defined as the transitive closure of the \defn{congruence
relations} given by
\begin{itemize}
\item $X <_{\CR(n,k)} X + v_i$ \text{ if $k-i$ is odd, or}
\item $X + v_i <_{\CR(n,k)} X$ \text{ if $k-i$ is even}
\end{itemize}
for all $X \in \binom{\Z}{k}_n$ and $0 < i < k$.  For $k>1$,  note $X+v_{0}=X+v_{k}=X$
as equivalence classes in $\binom{\Z}{k}_n$, but this is not true if $k=1$.
For $k=1$, the \defn{congruence poset} $\boldsymbol{\CR(n,1)}$ is the set $\binom{\Z}{1}_n\cong \Z$ with
order relations defined as the transitive closure of the \defn{congruence relations} given by
$X=\{x \}<_{\CR(n,1)}\{x+n\}=X+v_{0}$.
\end{definition}

\begin{example}\label{ex:cong.poset}
For $k=3,n=4$, we have $v_{1}=v_1^{(4,3)}=(0,4,4)$ and
$v_{2}=v_2^{(4,3)}=(0,0,4)$.  The compressed representative of
$[1,7,8]$ and $[1,7,12]$ is $[1,3,4]$.  Since $[1,7,8]=[1,3,4]+v_{1}$
and $k-1=2$ is even, $[1,7,8]<_{\CR(4,3)}[1,3,4]$.  On the other hand, 
$[1,7,8]<_{\CR(4,3)}[1,7,12]$ since $[1,7,12]=[1,7,8]+v_{2}$ and
$k-2=1$ is odd.  Note that $[1,7,12]$ and $[1,3,4]$ are not comparable in
$\CR(4,3)$.  
\end{example}

\begin{theorem}\label{cor:congruence.lemma}
For $k>1$ and $X=[x_{1},\dots ,x_{k}] \in \Inv_k(w)$, let $I(X) = \{Z \in \Inv_{k}(w): X
\equiv Z \}$.  The set $I(X)$ is an interval in $\CR(n,k)$ isomorphic
to a product of chains $[m_{1}]\times [m_{2}] \times \cdots \times
[m_{k-1}]$ containing the compressed representative of $X$ where
$m_{i}$ is the number of inversions for $w$ congruent to
$[x_{i},x_{i+1}]$.
\end{theorem}

\begin{proof}
We will first show that certain coordinates parameterize the whole
congruence class of $X$ in $\binom{\Z }{k}_{n}$. Let
$Y=[y_1,\ldots,y_k]\in \Inv_{k}(w)$ be the compressed representative
of $X$ from \Cref{lem:affine_n_shift} Part (b).  Let $Z \in \binom{\Z
}{k}_{n}$ be any element congruent to $X$ and $Y$.  Since $k>1$, we
may shift the values in $Z$ by a global multiple of $n$ so that its
representative $Z=[z_1,\ldots,z_k]$ satisfies $z_1=y_1$ and $z_j\equiv
y_j\pmod n$ for all $j$.  For each $i\in[k-1]$ define the \textit{gap
coordinate}
\[
a_i := \frac{(z_{i+1}-z_i)-(y_{i+1}-y_i)}{n} \in \Z.
\]
Each $a_{i}\in \Z_{\ge 0}$ since $z_j\equiv y_j\pmod n$,
$0<y_{i+1}-y_i<n$, and $z_{i+1}-z_i>0$.  Let $b_1:=0$ and
$b_j:=\sum_{i=1}^{j-1} a_i$ for $2\le j\le k$.  Then $z_j=y_j+nb_j$,
and therefore
\begin{equation}\label{eq:congruent.sum}
Z\;=\;Y+\sum_{i=1}^{k-1} a_i v_i.
\end{equation}
Conversely, any $(a_1,\ldots,a_{k-1})\in\Z_{\ge 0}^{k-1}$ determines a
distinct element congruent to $X$ by \eqref{eq:congruent.sum}.
Thus the congruence class of $X$ in $\binom{\Z}{k}_n$ is identified
with the gap coordinates in $\Z_{\ge 0}^{k-1}$.

To determine if $Z \in \Inv_{k}(w)$, it suffices to show each
$[z_{i},z_{i+1}] \in \Inv_{2}(w)$.  If $z_j=y_j+nb_j$ as in the
notation above, then $w^{-1}(z_j)=w^{-1}(y_j)+nb_j$ by $n$-periodicity
of $w$.  For each $i\in[k-1]$ we have
\begin{align*}
w^{-1}(z_i)>w^{-1}(z_{i+1})
&\iff w^{-1}(y_i)+nb_i > w^{-1}(y_{i+1})+nb_{i+1}\\
&\iff w^{-1}(y_i) > w^{-1}(y_{i+1}) + n a_i
\end{align*}
since $b_{i+1}=a_{i}+b_{i}$.  Thus $Z\in\Inv_k(w)$ if and only if
these inequalities hold for all $i \in [k-1]$, and each inequality
involves only $a_i$.

Similar to the Bj\"orner-Brenti graph from \Cref{def:BB.graph} and
\eqref{eq:BB.multiplicity}, define
\begin{equation}\label{eq:HB.multiplicity} 
m_i := \left|\left\lfloor \frac{w^{-1}(y_{i+1})-w^{-1}(y_{i})}{n}
\right\rfloor \right|\in\Z_{\ge 0}, \qquad \text{for all}\ 1\le i\le k-1.
\end{equation}
Then $w^{-1}(y_i)>w^{-1}(y_{i+1})+na_i$ holds if and only if
$[y_{i},y_{i+1}+na_{i}] \in \Inv_{2}(w)$ if and only if $0\le a_i<
m_i$. Therefore,
\[
I(X)
=\Bigl\{\,Y+\sum_{i=1}^{k-1} a_i v_i\;:\; 0\le a_i< m_i\ \text{for all }i\in[k-1]\,\Bigr\}.
\]

If $Z\mapsto(a_1,\ldots,a_{k-1})$ in gap coordinates, then adding $v_i$ to $Z$
increases $a_i$ by $1$ and fixes the other
coordinates.  By \Cref{def:congruence.poset}, $Z<Z+v_{i}$ when $k-i$
is odd corresponding with an increase in the $i$th coordinate, and
$Z<Z-v_{i}$ when $k-i$ is even corresponding with a decrease in the
$i$th coordinate.  Hence the induced subposet of $\CR(n,k)$ on $I(X)$
is an alternating product of finite chains on $m_i$ elements.  In
particular, $I(X)$ is isomorphic as an abstract poset to the product
of chains $[m_1]\times\cdots\times[m_{k-1}]$.  Moreover, if we set
\begin{equation}\label{eq:Xmin.max}
X_{\min}:=Y+\sum_{\substack{1\le i\le k-1\\ k-i\ \mathrm{even}}} (m_i-1) v_i,
\qquad
X_{\max}:=Y+\sum_{\substack{1\le i\le k-1\\ k-i\ \mathrm{odd}}} (m_i-1) v_i,
\end{equation}
then the gap coordinate description shows $I(X)$ is exactly the
interval $[X_{\min},X_{\max}]$ inside $\CR(n,k)$.  
\end{proof}

\begin{remark}
In addition to the compressed representative, each nonempty congruence
class $I(X)$ in $\Inv_{k}(w)$ has an \defn{expanded representative}
where each gap coordinate is maximized, provided $k>1$. The expanded
representative $Z=[z_{1},\dots ,z_{k}]=Y+\sum_{i=1}^{k-1} (m_i-1) v_i$
has the property that $0<w^{-1}(z_{i})-w^{-1}(z_{i+1})<n$.
\end{remark}

\subsection{Affine Permanent Posets}\label{sub:permanent.affine}

As in the classical permutation case, we define a permanent poset for
every affine permutation $w \in \affineS_{n}$ on $\Inv_{k}(w)$
generalizing the requirements on inversion sets of affine permutations
characterized by \Cref{lem:inversion.sets}.  One difference between
the classical permutation and affine permutation cases comes from the
third requirement in \Cref{lem:inversion.sets} on the inversions that
are congruent. For example, by \Cref{cor:characterize.ref.order}, we
know that in any reflection order for $w$, each congruent pair
$[x,y],[x,y+n] \in \Inv_{2}(w)$ will have $[x,y]$ before
$[x,y+n]$. This necessary condition on reflection orders is
generalized by the congruence relations in
\Cref{def:congruence.poset}.  Note the similarity in the parity
condition between congruence relations and quasi-inversion relations
in \cref{def:permanent.poset.}.  This little miracle allows many of
the proofs from \Cref{sec:non-longest} to carry over to the affine
case.

\begin{definition}\label{def:affine-permanent-poset}
Let $\boldsymbol{\P_w(n,k)}$ denote the \defn{permanent poset} on the set
$\Inv_{k}(w)$ with order relations given by the transitive closure of
the affine analog of the quasi-inversion relations of
\cref{def:permanent.poset.} for all $(k+1)$-quasi-inversions for $w$ and
the congruence relations of \Cref{def:congruence.poset} for all
congruent pairs of distinct elements $X, X + v_i \in \Inv_{k}(w)$.
\end{definition}

\begin{example}\label{ex:m3m287.perm} Recall the $k$-inversions and
$k$-quasi-inversions for $w=(-3, -2, 8, 7) \in \affineS_4$ from
\Cref{ex:m3m287}.  The Hasse diagrams for the permanent posets
$\P_w(4,2)$ and $\P_w(4,3)$ are shown in
\Cref{fig:affine-permanent-poset-example}.  Since $ \Inv_4(w) =
\emptyset$, $\P_w(4,4)=\P_w(4,5)=\P_w(4,6)=\cdots$ are all posets on
the empty set.
\end{example}

    \begin{figure}[!ht]
        \centering
\begin{tikzpicture}
        \node (24) at (0,0) {$[2,4]$};
        \node (23) at (2,0) {$[2,3]$};
        \node (28) at (0,2) {$[2,8]$};
        \node (14) at (0,1) {$[1,4]$};
        \node (27) at (2,2) {$[2,7]$};
        \node (13) at (2,1) {$[1,3]$};
        \node (17) at (2,3) {$[1,7]$};
        \node (18) at (0,3) {$[1,8]$};
        \node (34) at (-1.5,1) {$[3,4]$};
        
        \draw (23) edge node {} (13);
        \draw (23) edge node {} (28);
        \draw (27) edge node {} (13);
        \draw (27) edge node {} (17);
        \draw (13) edge node {} (18);
        \draw (14) edge node {} (17);
        \draw (28) edge node {} (18);
        \draw (14) edge node {} (28);
        \draw (24) edge node {} (14);
        \draw (24) edge node {} (27);
    \end{tikzpicture}
    \qquad \qquad \qquad     \qquad \qquad 
    \begin{tikzpicture}
        \node (134) at (0,2) {$[1,3,4]$};
        \node (278) at (0,1) {$[2,7,8]$};
        \node (178) at (0,0) {$[1,7,8]$};
        \node (234) at (0,3) {$[2,3,4]$};
        
        \draw (178) edge node {} (278);
        \draw (278) edge node {} (134);
        \draw (134) edge node {} (234);
    \end{tikzpicture}
        \caption{The Hasse diagrams of the permanent posets
	$\P_{w}(4,2)$ on $\Inv_{2}(w)$ (left)  and
	$\P_w(4,3)$ on $\Inv_{3}(w)$ (right) for the affine permutation $w = (-3,-2,8,7)$.}
    \label{fig:affine-permanent-poset-example}
    \end{figure}
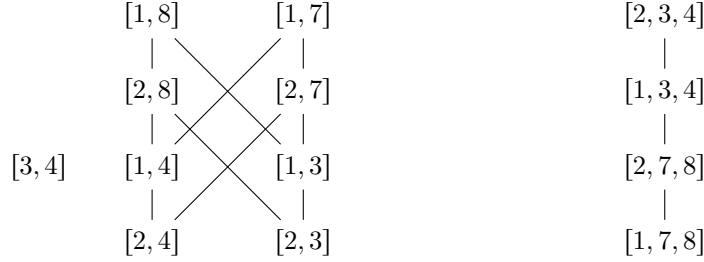

\begin{remark}\label{rem:easy.pnks.affine}
For $k=1$, the permanent poset $\P_w(n,1)$ on $\Inv_{1}(w) \cong \Z$
is generated by the relations $\{x \}<_{\P_w(n,1)}\{y \}$ whenever
$x<y$ and $[x,y]$ is a quasi-inversion of $w$ or $y=x+n$.  By
$n$-periodicity of $w$, $[x_{1},x_{2}]$ is a quasi-inversion for $w$
if and only if $x_{1}+mn<_{w}x_{2}+mn$ for all $m \in \Z$.  Thus, the
quasi-inversion relations in $\P_w(n,1)$ are well-defined.  Since this
set of generating relations is isomorphic to a subset of the natural
total order on $\Z$, they generate a relation over $\Inv_{1}(w)$ that
is reflexive, anti-symmetric, and transitive, so $\P_w(n,1)$ is indeed
a poset.  Furthermore, when $k>n$, we know that $\Inv_k(w)=\emptyset$,
so $\P_w(n,k)$ is the empty poset.
\end{remark}

As we did for permutations in \Cref{lem:poset.packet.antichain}, we
will show that $\P_w(n,k)$ is a poset for $w\in \affineS_n$ and $2\leq
k\leq n$. Our approach will be similar, but we must make several
technical modifications to account for the fact that each element in
$\Inv_k(w)$ can be represented in multiple ways if $k>1$.

Let $\mathbb{Z}^{2\times k}$ be the abelian group of $2\times k$
matrices with entries in $\mathbb{Z}$ under addition. Let $N\in \mathbb{Z}^{2\times
k}$ be the matrix whose entries are all $n$, and let $(N)$ be the
cyclic subgroup of $\mathbb{Z}^{2\times k}$ generated by $N$.
Consider the map from $\Inv_k(w)$ to the quotient group
$\mathbb{Z}^{2\times k}/(N)$ defined by 
\begin{equation}\label{eq:affine_inversion_matrix}
    \wt{M}(Y) = \begin{bmatrix}
    w^{-1}(y_1) & w^{-1}(y_2) & \ldots & w^{-1}(y_{k-1}) & w^{-1}(y_k) \\
    y_1 & y_2 & \ldots &  y_{k-1} &
y_k \end{bmatrix}
\end{equation}
for all $Y=[y_1,\ldots,y_k]\in \Inv_k(w)$.

\begin{lemma}\label{lem:inversion_matrix}
 The map $\wt{M}:\Inv_k(w) \to \mathbb{Z}^{2\times k}/(N)$ is independent of the choice of representative for elements in $\Inv_k(w)$, and hence is a well-defined function.
\end{lemma}

\begin{proof}
For $k=1$, the statement is clear, so suppose $k >1$.  Alternate
expressions for $Y\in \Inv_k(w)$ have the form
$[y_1+mn,y_2+mn,\ldots,y_k+mn]$ where $m\in \mathbb{Z}$. The affine
permutation $w^{-1}\in \affineS_n$ is $n$-periodic by definition, so
$w^{-1}(y_i+mn)=w^{-1}(y_i)+mn$. Consequently, this choice of
representative for $Y$ results in \[\wt{M}(Y)=\begin{bmatrix}
    w^{-1}(y_1)+mn & w^{-1}(y_2)+mn & \ldots & w^{-1}(y_{k-1})+mn & w^{-1}(y_k)+mn \\
    y_1+mn & y_2+mn & \ldots & y_{k-1}+mn &
y_k+mn \end{bmatrix}.\]
This is equivalent to \eqref{eq:affine_inversion_matrix} as an element in $\mathbb{Z}^{2\times k}/(N)$. 
\end{proof}

Recall from \eqref{eq:Fi_Gi} that the homomorphisms
$F_j:\mathbb{Z}^{2\times k} \to \mathbb{Z}$ and
$G_j:\mathbb{Z}^{2\times k} \to \mathbb{Z}$ were defined by
\[
F_j(A)=(-1)^{k-j}(A_{1,j}-A_{2,j-1})\quad \text{ and }\quad
G_j(A)=(-1)^{k-j}(A_{2,j}-A_{1,j-1})
\]
for $2\leq j\leq k$.  Observe that both $F_{j}$ and $G_{j}$ map
elements in $(N)$ to $0$, so they descend to the quotient.  Let
\begin{align}\label{eq:FG.tilde}
\begin{split}
\wt{F}_j :\ &\mathbb{Z}^{2\times k}/(N) \to \mathbb{Z},\\
\wt{G}_j :\ &\mathbb{Z}^{2\times k}/(N) \to \mathbb{Z}
\end{split}
\end{align}
denote the group homomorphisms on the quotient derived from $F_{j}$
and $G_{j}$.  Then, we have the following analog of
\Cref{lem:diagonal_differences}.

\begin{lemma}\label{lem:affine_diagonal_differences}
Suppose $2 \le k \le n$. If $X <_{\P_w(n,k)} Y$, then $\wt{F}_j\circ
\wt{M}(X)\leq \wt{F}_j\circ \wt{M}(Y)$ and $\wt{G}_j\circ
\wt{M}(X)\leq \wt{G}_j\circ \wt{M}(Y)$ for any $j\in [k]\setminus
\{1\}$.  Furthermore, there exists $j\in [k]\setminus \{1\}$ such that
the inequality is strict for both $\wt{F}_j$ and $\wt{G}_j$.
\end{lemma}

\begin{proof}
    It suffices to establish the claim for the generating relations in
$\P_w(n,k)$. The case of quasi-inversion relations for $X_{i}, X_{i+1}
\in \Inv_{k}(w)$ follows from the proof of
\cref{lem:diagonal_differences}, as the calculations there can be
viewed as choosing representative matrices in $\mathbb{Z}^{2\times k}$
to calculate and compare $\wt{F}_j\circ \wt{M}(X_i), \wt{F}_j\circ
\wt{M}(X_{i+1})$, $\wt{G}_j\circ \wt{M}(X_i)$, and $\wt{G}_j\circ
\wt{M}(X_{i+1})$.
    
    For the case of congruence relations, assume both $X,X+v_{i}\in
\Inv_k(w)$ for some $i\in [k-1]$. Without loss of generality, we may
assume both are denoted by their canonical representatives so their
smallest value is equal.  Then,
\begin{equation}\label{eq:congruence_matrix_difference}
\wt{M}(X+v_i)-\wt{M}(X)=
\begin{bmatrix}
0 & \cdots & 0 & n & \cdots & n\\
0 & \cdots & 0 & n & \cdots & n
\end{bmatrix} \in \mathbb{Z}^{2\times k}/(N), 
\end{equation}
where the first nonzero column is $i+1$, and
\[
\wt{F}_j\circ \wt{M}(X+v_i)-\wt{F}_j\circ \wt{M}(X) = \begin{cases}
(-1)^{k-(i+1)} n& j=i+1 \\
0& j \neq i+1.
\end{cases}
\]
When $k-i$ is odd, we have that $X<_{\P_w(n,k)} X+v_i$ and
$\wt{F}_{i+1}\circ \wt{M}(X+v_i)-\wt{F}_{i+1}\circ \wt{M}(X)=n>0$.  In
the case where $k-i$ is even, we have that $X+v_i<_{\P_w(n,k)} X$ and
$\wt{F}_{i+1}\circ \wt{M}(X+v_i)-\wt{F}_{i+1}\circ \wt{M}(X)=-n<0$.

Combined, we see that $\wt{F}_2\circ \wt{M},\ldots ,\wt{F}_k\circ
\wt{M}$ are monotone increasing functions with respect to the 
generating relations $X<_{\P_w(n,k)} Y$, and at least one of these
strictly increases.  A similar argument establishes the results for
$\wt{G}_2\circ \wt{M},\ldots ,\wt{G}_k\circ \wt{M}$.
\end{proof}

\begin{lemma}\label{lem:affine_diagonal_packets}
    Suppose $X\in \Inv_{k+1}(w)$ and $2\leq j\leq k$. If $k-j$ is
even, then
\begin{equation}\label{eq:affine_FG_on_packets} \begin{split}
  \wt{F}_j\circ \wt{M}(X_1)=\ldots=\wt{F}_j\circ \wt{M}(X_{j-1})<\wt{F}_j\circ \wt{M}(X_j)<\wt{F}_j\circ \wt{M}(X_{j+1})=\ldots = \wt{F}_j\circ \wt{M}(X_{k+1}),\\
    \wt{G}_j\circ \wt{M}(X_1)=\ldots=\wt{G}_j\circ
\wt{M}(X_{j-1})>\wt{G}_j\circ \wt{M}(X_j)>\wt{G}_j\circ
\wt{M}(X_{j+1})=\ldots = \wt{G}_j\circ
\wt{M}(X_{k+1}).  \end{split}
\end{equation}
If $k-j$ is odd, then the above inequalities are reversed.
\end{lemma}

\begin{proof}
    This follows from the same argument as in
\cref{lem:diagonal_packets}. The calculations there for each
$F_{j}\circ M$ can be viewed as choosing representative matrices in
$\mathbb{Z}^{2\times k}$ with which to calculate $\wt{F}_j\circ
\wt{M}$. An analogous argument can also be used for each
$\wt{G}_j\circ \wt{M}$.
\end{proof}

\begin{theorem}\label{lem:affine.poset.packet.antichain}
For any positive integers $k,n\in \mathbb{Z}$ and $w\in \affineS_n$,
we have that $\P_{w}(n,k)$ is a poset.  Furthermore, if $X\in
\Inv_{k+1}(w)$, then the elements in $P(X)$ form an antichain in
$\P_w(n,k)$.
\end{theorem}

\begin{proof}
For $k=1$ and $k>n$, the statement follows from
\Cref{rem:easy.pnks.affine}.  For $2\leq k\leq n$, the statement
follows from the same argument as in
\cref{lem:poset.packet.antichain}, replacing the functions $F_{j}\circ
M$ and $G_{j}\circ M$ with $\wt{F}_{j}\circ \wt{M}$ and $\wt{G}_{j}\circ
\wt{M}$, respectively. References to \cref{lem:diagonal_differences}
and \cref{lem:diagonal_packets} should be replaced with references to
\cref{lem:affine_diagonal_differences} and
\cref{lem:affine_diagonal_packets}, respectively.
\end{proof}

\begin{corollary}\label{cor:w.invariance.of.perm.poset}
For $v,w \in \affineS_{n}$ and $k>1$, we have $\P_{v}(n,k)=
\P_{w}(n,k)$ if and only if $\Inv_{k}(v)=\Inv_{k}(w)$.
\end{corollary}

\begin{proof}
Since $\P_{w}(n,k)$ is a poset on $\Inv_{k}(w)$, it is clear that
$\P_{v}(n,k)= \P_{w}(n,k)$ implies $\Inv_{k}(v)=\Inv_{k}(w)$.
Conversely, we know $X \in \Qinv_{k+1}(w)$ if and only if $P(X) \cap
\Inv_{k}(w) =\{X_{i},X_{i+1} \} \neq P(X)$ for some $i$ by
\Cref{lem:types.of.packets} provided $k>1$.  Hence,
$\Inv_{k}(v)=\Inv_{k}(w)$ implies $\Qinv_{k+1}(v)=\Qinv_{k+1}(w)$, so
$\P_{v}(n,k)$ and $\P_{w}(n,k)$ have exactly the same quasi-inversion
relations.  Similarly, if $Y, Y+v_{i} \in \Inv_{k}(v)=\Inv_{k}(w)$,
then the corresponding congruence relation holds in both $\P_{v}(n,k)$
and $\P_{w}(n,k)$.  Since the permanent posets are generated by the
quasi-inversion relations and the congruence relations, the result
holds.
\end{proof}

\subsection{Affine Higher Bruhat Orders, Admissible Orders, and
Consistent Sets}\label{sub:afffine.higher}

With the affine analogs of $k$-subsets, $k$-inversion sets, and the
permanent posets $\P_{w}(n,k)$ in place, it is now straightforward to
give the affine analogs of the classical concepts from
Manin-Schechtman's original work for all $w \in \affineS_{n}$ as we did
for permutations.  The special cases of $k=1,2,3$ will be discussed
further below in connection with intervals $[\mathrm{id},w]$,
reflection orders, and commutation classes of reduced expressions.  

We first define the affine analogs of admissible orders for $w\in
\affineS_n$ and their reversal sets. These definitions directly
parallel those in \cref{def:admissible.BCL}.

\begin{definition}
    The $\boldsymbol{k}$-\defn{admissible orders}, denoted
$\boldsymbol{\A_w(n,k)}$, are the total orders $\rho$ on $\Inv_{k}(w)$
that are linear extensions of $\P_w(n,k)$ satisfying the
\defn{full-packet condition}:
\begin{quote}
$\rho|_{P(X)}$ is the full-packet in lex order
$(X_{k+1},\ldots,X_{2},X_1)$ or in antilex order
$(X_1,X_2,\ldots,X_{k+1})$ for every $X\in \Inv_{k+1}(w)$.
\end{quote}
The \defn{reversal
set} of an admissible order $\rho \in \A_w(n,k)$ is the set of
all $X\in \Inv_{k+1}(w)$ such that $P(X)$ appears in antilex order
in $\rho$.
\end{definition}

For $\rho \in \A_w(n,k)$ and $k>1$, if the packet of $X\in
\Inv_{k+1}(w)$ forms a saturated chain in $\rho$, then it is
\defn{flippable} as in \Cref{def:flippable}. The terminology for
flippable packets extends unchanged to the affine case as well.  Thus,
$\A_{w}(n,k)$ is closed under packet flip as shown in
\Cref{lem:packet.flips.Sn}, substituting
\Cref{lem:affine.poset.packet.antichain} for
\Cref{lem:poset.packet.antichain}.  Similarly, two elements $X,Y \in
\Inv_{k}(w)$ \defn{commute with respect to} $\boldsymbol{w}$ if they
do not belong to a common packet $P(Z)$ for any $Z \in \binom{\Z
}{k+1}_{n}$ and they are incomparable in $\P_{w}(n,k)$, exactly as in
\Cref{def:commutes.Sn}.  Then, for all $k>1$, the same proof as in
\Cref{lem:commutations} shows that $\A_{w}(n,k)$ is closed under
swapping adjacent elements that commute with respect to $w$.  Using
these definitions, we now define affine higher Bruhat orders for
elements in $\affineS_n$, paralleling \Cref{def:higher.B.Sn}.

\begin{definition}\label{def:affine.higher.B} 
    The \defn{affine higher Bruhat order}, denoted $\boldsymbol{\B_w(n,k)}$, is the poset on the
commutation equivalence classes of the admissible orders $\{[\rho]:
\rho \in \A_{w}(n,k) \}$ with relations $[\rho]\leq_{\B_w(n,k)}[\sigma]$
if $\sigma$ is obtained from $\rho$ by a sequence of commutations with
respect to $w$ and lex-to-antilex packet flips.   
\end{definition}

As in \Cref{def:classical.GR}, we associate a reversal graph to each
subset of the $k$-inversions of $w$.  For $k>1$ and $w \in \S_{n}$,
this is a direct generalization of the classical permutation case.  We
will discuss the differences for $k=1$ below.

\begin{definition}\label{def:affine.GR}
For any $R\subseteq \Inv_{k+1}(w)$, the affine analog of the
\defn{reversal graph} $G_R$ is defined to be the graph with vertex set
$\Inv_{k}(w)$ and directed edges as follows:
\begin{itemize}
\item $X \to Y$ if $X<_{\P_w(n,k)}Y$ is a quasi-inversion or
congruence relation.
\item $X_i\color{red}\to\color{black} X_{i+1}$ if $P(X)$ is a
full-packet for $w$ and $X\in R$, and
\item $X_{i+1}\color{blue}\to\color{black} X_i$ if $P(X)$ is a
full-packet for $w$ and $X\notin R$.
\end{itemize}
In the case $G_{R}$ is acyclic, let $(\Inv_{k}(w), \leq_{R})$ denote
the \defn{heap poset} defined as the transitive closure of relations
determined by the edges of $G_{R}$.
\end{definition}

Recall that in the classical case, any two vertices connected by an
edge in $G_{R}$ differ in exactly one position.  Here the same holds
for all except the congruence edges, where two vertices differ by a
vector of the form $(0,\dotsc , 0,n,\dotsc, n)$.

Finally, we define the affine analog of Ziegler's consistent sets,
paralleling \Cref{def:consistent.Z} and \Cref{def:consistent.Sn}. We
then show that the properties of consistent sets are necessary
conditions for reversal sets of admissible orders.

\begin{definition}\label{def:consistent.affine}
    A subset $R \subseteq \Inv_{k}(w)$ is \defn{consistent} with
respect to $w$ if it is a lower order ideal of $\P_{w}(n,k)$ that
satisfies the \defn{MSZ condition}, namely $P(X) \cap R$ is a
prefix or a suffix of $P(X)$ for all $X \in
\Inv_{k+1}(w)$. The $\boldsymbol{k}$-\defn{consistent set partial
order}, denoted $\boldsymbol{\C_w(n,k)}$, is again the partial order
on the $k$-consistent sets for $w$ ordered by single step inclusion.
\end{definition}

\begin{example}\label{ex:m3m287.admissible}
Continuing with \Cref{ex:m3m287} and \Cref{ex:m3m287.perm}, let $w =
(-3,-2,8,7)$.  From the Hasse diagram of $\P_{w}(4,2)$ shown in
\Cref{fig:affine-permanent-poset-example} and by verifying the MSZ
condition for each 3-inversion given in \Cref{ex:m3m287}, one can
verify that
\[
  \rho = ([2,3],[2,4],[3,4],[1,4],[1,3],[2,8],[2,7],[1,8],[1,7])
\]
is an admissible order in $\A_w(4,2)$.  Note, this is exactly the
reflection order for the reduced word $232124134$ for $w$.  The
reversal set of $\rho$ is $\Rev(\rho) = \{[1,7,8], [2,7,8],
[1,3,4]\}$. Observe that $\Rev(\rho) \subseteq\Inv_3(w)$ is a lower
order ideal of $\P_w(4,3)$ which satisfies the MSZ condition so it is
a 3-consistent set for $w$. Additionally, applying the braid move
$232124134\to 323124134$ corresponds to a lex-to-antilex packet flip
for $[2,3,4]$ on the first three positions of $\rho$, which gives rise
to another admissible order in $\A_w(n,2)$.  The pair $[2,7],[1,8]$
commute with respect to $w$ so swapping these two adjacent elements in
$\rho$ results in another admissible order in the equivalence class
$[\rho]$.  Note, $[2,3]$ and $[1,8]$ do not commute with respect to
$w$ since $[2,3]<_{\P_{w}(4,2)}[1,8]$, even though they are not both
contained in any packet.  The reversal graph $G_{\Rev(\rho )}$ has 9
vertices and 24 edges as shown in \Cref{fig:m3m287dot.pdf}.  It is
acyclic.  Its transitive reduction is shown in
\Cref{fig:Adm3m287dot.pdf_reduction}, which is equivalent to the Hasse
diagram of a poset with edges pointing in the increasing direction.
Observe that $\rho$ is a linear extension of this poset and every
admissible order in $\A_w(4,2)$ with the same reversal set is also a
linear extension of this poset.
\end{example}

\begin{figure}
\begin{center}
\begin{tikzpicture}
        \node[rectangle,draw=green!50!white,fill=green!25!white] (23) at (0,0) {$[2,3]$};
        \node[rectangle,draw=yellow!75!white,fill=yellow!50!white] (24) at (2,0) {$[2,4]$};
        \node[rectangle,draw=teal!50!white,fill=teal!25!white] (34) at (4,0) {$[3,4]$};
        \node[rectangle,draw=orange!50!white,fill=orange!25!white] (14) at (6,0) {$[1,4]$};
        \node[rectangle,draw=pink,fill=pink!75!white] (13) at (8,-1) {$[1,3]$};
        \node[rectangle,draw=yellow!75!white,fill=yellow!50!white] (28) at (8,1) {$[2,8]$};
        \node[rectangle,draw=green!50!white,fill=green!25!white] (27) at (10,1) {$[2,7]$};
        \node[rectangle,draw=orange!50!white,fill=orange!25!white] (18) at (10,-1) {$[1,8]$};
        \node[rectangle,draw=pink,fill=pink!75!white] (17) at (12,0) {$[1,7]$};
        \draw[->,color=red,thick] (23) edge node {} (24);
        \draw[->,color=red,thick] (24) edge node {} (34);
        \draw[->,color=blue,thick] (34) edge node {} (14);
        \draw[->,color=blue,thick] (14) edge node {} (13);
        \draw[->,color=black,thick] (14) edge node {} (28);
        \draw[->,color=blue,thick] (28) edge node {} (18);
        \draw[->,color=black,thick] (13) edge node {} (27);
        \draw[->,color=blue,thick] (28) edge node {} (27);
        \draw[->,color=black,thick] (13) edge node {} (18);
        \draw[->,color=black,thick] (27) edge node {} (17);
        \draw[->,color=blue,thick] (18) edge node {} (17);
        \draw[->,color=black,thick] (23) to [out=15,in=165] (28);
        \draw[->,color=black,thick] (23) to [out=30,in=150] (27);
        \draw[->,color=black,thick] (23) to [out=345,in=180] (13);
        \draw[->,color=black,thick] (23) to [out=315,in=240] (17);
        \draw[->,color=black,thick] (24) to [out=30,in=165] (27);
        \draw[->,color=black,thick] (24) to [out=15,in=180] (28);
        \draw[->,color=black,thick] (24) to [out=345,in=195] (14);
        \draw[->,color=black,thick] (24) to [out=330,in=210] (18);
        \draw[->,color=blue,thick] (34) to [out=330,in=195] (18);
        \draw[->,color=black,thick] (34) to (28);
        \draw[->,color=black,thick] (14) to (18);
        \draw[->,color=black,thick] (14) to (17);
        \draw[->,color=black,thick] (13) to (17);
    \end{tikzpicture}
\end{center}
\caption{The reversal graph $G_{R}$ on $\Inv_{2}(-3,-2,8,7)$ with $R =
\{[1,3,4], [2,7,8], [1,7,8]\}$.  When two vertices are congruent, they
are boxed by the same color.}  \label{fig:m3m287dot.pdf}
\end{figure}
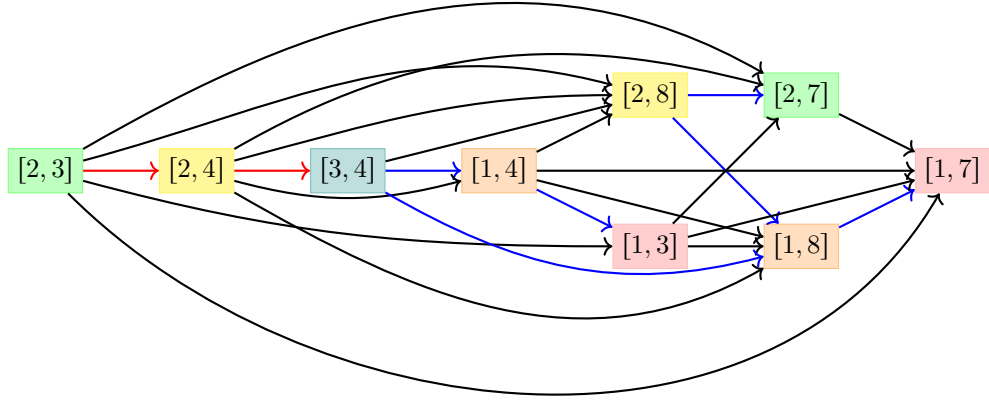

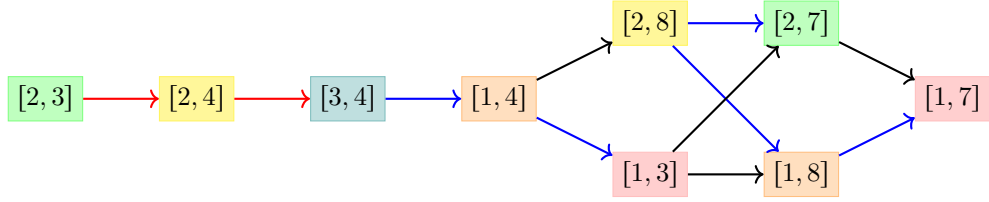
\begin{figure}
\begin{center}
    \begin{tikzpicture}
        \node[rectangle,draw=green!50!white,fill=green!25!white] (23) at (0,0) {$[2,3]$};
        \node[rectangle,draw=yellow!75!white,fill=yellow!50!white] (24) at (2,0) {$[2,4]$};
        \node[rectangle,draw=teal!50!white,fill=teal!25!white] (34) at (4,0) {$[3,4]$};
        \node[rectangle,draw=orange!50!white,fill=orange!25!white] (14) at (6,0) {$[1,4]$};
        \node[rectangle,draw=pink,fill=pink!75!white] (13) at (8,-1) {$[1,3]$};
        \node[rectangle,draw=yellow!75!white,fill=yellow!50!white] (28) at (8,1) {$[2,8]$};
        \node[rectangle,draw=green!50!white,fill=green!25!white] (27) at (10,1) {$[2,7]$};
        \node[rectangle,draw=orange!50!white,fill=orange!25!white] (18) at (10,-1) {$[1,8]$};
        \node[rectangle,draw=pink,fill=pink!75!white] (17) at (12,0) {$[1,7]$};
        \draw[->,color=red,thick] (23) edge node {} (24);
        \draw[->,color=red,thick] (24) edge node {} (34);
        \draw[->,color=blue,thick] (34) edge node {} (14);
        \draw[->,color=blue,thick] (14) edge node {} (13);
        \draw[->,color=black,thick] (14) edge node {} (28);
        \draw[->,color=blue,thick] (28) edge node {} (18);
        \draw[->,color=black,thick] (13) edge node {} (27);
        \draw[->,color=blue,thick] (28) edge node {} (27);
        \draw[->,color=black,thick] (13) edge node {} (18);
        \draw[->,color=black,thick] (27) edge node {} (17);
        \draw[->,color=blue,thick] (18) edge node {} (17);
    \end{tikzpicture}
\end{center}
\caption{The transitive reduction of $G_{R}$ on $\Inv_{2}(-3,-2,8,7)$
with $R = \{[1,3,4], [2,7,8], [1,7,8]\}$.}
\label{fig:Adm3m287dot.pdf_reduction}
\end{figure}

\begin{lemma}\label{lem:reversal_is_consistent.affine}
For any $\rho\in \A_w(n,k)$, the reversal set $\Rev(\rho) \in \C_w(n,k+1)$. Furthermore, $\Rev(\rho)$ is constant on the equivalence class $[\rho]$, so $\Rev:\B_w(n,k) \to \C_w(n,k+1)$ is a well-defined order-preserving map.
\end{lemma}   
\begin{proof}
By definition, $\Rev(\rho)$ is a subset of $\Inv_{k+1}(w)$.  Since the
MSZ condition does not involve any congruence relations, $\Rev(\rho)$
satisfies the MSZ condition by the same proof as in
\cref{lem:reversal_is_consistent.Sn}.  Therefore, to prove
$\Rev(\rho)$ is consistent it suffices to show that $\Rev(\rho)$ is a
lower order ideal of the permanent poset.

When $k=1$, the quasi-inversion part of the lower order ideal
claim follows as in \cref{lem:reversal_is_consistent.Sn}. For a
congruence relation $[x,y]<_{\P_w(n,2)}[x,y+n]$, membership of
$[x,y+n]$ in $\Rev(\rho)$ implies 
\[
\{y\}<_{\rho}\{y+n\}<_{\rho}\{x\},
\]
so $[x,y]\in\Rev(\rho)$. Thus the lower order ideal claim holds
for $k=1$, and for the remaining argument we may assume $k>1$.

Suppose $X,Y \in \Inv_{k+1}(w)$ are related by a generating relation
$X<_{\P_w(n,k+1)} Y$. If this relation is a quasi-inversion relation,
then the proof proceeds as in \cref{lem:reversal_is_consistent.Sn}.
Otherwise, $X$ and $Y$ are related by a congruence relation, so either
$Y = X + v_i^{(n,k+1)}$ where $(k+1)-i$ is odd or $Y = X -
v_i^{(n,k+1)}$ where $(k+1)-i$ is even. Consider the first case, where
$Y = X+v_i^{(n,k+1)}$ and $(k+1)-i$ is odd. If $i=k$, then
$Y_{k+1}=X_{k+1}$, and if $i=1$, then $X_{1}=Y_{1}$. When these do not
hold, we respectively have $Y_{k+1} = X_{k+1} + v_i^{(n,k)}$ and $Y_1=
X_1 + v_{i-1}^{(n,k)}$, so $Y_{k+1} <_{\P_w(n,k)} X_{k+1}$ and $X_1
<_{\P_w(n,k)} Y_1$ by the parity condition of the congruence
poset. Since $\rho$ is admissible, it is a linear extension of
$\P_w(n,k)$, and since $Y \in \Rev(\rho)$, we have $Y_1 <_\rho
Y_{k+1}$. Therefore, it follows that
\begin{equation}\label{eq:nesting.dueto.congruence.requirement}
X_1 \leq_\rho Y_1 <_\rho Y_{k+1} \leq_\rho X_{k+1},
\end{equation}
which implies $X \in \Rev(\rho)$ since the restriction of $\rho$ to
$P(X)$ must be the lex or antilex order. A similar nesting argument
applies in the second case where $Y = X - v_i^{(n,k+1)}$ and $(k+1)-i$
is even. Hence, $\Rev(\rho)$ is a lower order ideal of $\P_w(n,k+1)$.

The proof of \Cref{lem:commutations} extends to the affine case
verbatim to prove that reversal sets are constant on commutation
classes of admissible orders.  Hence, the map $\rho \mapsto
\Rev(\rho)$ descends to a well-defined map $[\rho] \mapsto
\Rev([\rho])$. Extending the proof of \cref{lem:packet.flips.Sn} to
the affine case verbatim shows that $\Rev$ is order-preserving.
\end{proof}

\begin{remark}\label{rem:nesting.congrence}
The nesting property in
\eqref{eq:nesting.dueto.congruence.requirement} is closely related to
the nesting property for quasi-inversions mentioned in
\Cref{rem:nested}.  In both cases, the parity condition is used in an
essential way for the lower order ideal argument.
\end{remark}

The following six statements are generalizations of
\Cref{lem:GR_to_admissible}, \Cref{thm:equiv_class.and.linear.exts},
\Cref{cor:reversals_unique}, \Cref{cor:B.equiv.acyclic}, 
\Cref{lem:not_consistent_cycle.Sn}, and \Cref{lem:suffix_set_consistent} to the affine permutation case.
For all $k>1$, $\Inv_{k}(w)$ is a finite set and the proofs of those
statements carry over directly using the affine analogs of
$\Inv_{k}(w)$, $\A_{w}(n,k)$, $\B_{w}(n,k)$, and $\P_{w}(n,k)$, though we will
give a proof of the affine generalization of \Cref{lem:suffix_set_consistent}
to specifically account for the inclusion of congruence relations.  See
\Cref{sub:firstAHB} below for more details on the $k=1$ case, from
which the proofs follow.

\begin{lemma}\label{lem:GR_to_admissible.affine}
Let $R\subseteq \Inv_{k+1}(w)$.  If $\rho\in \A_w(n,k)$ and
$\Rev(\rho)=R$, then $G_R$ is acyclic. Conversely, if $G_R$ is
acyclic, then
\[
    \{\rho \in \A_w(n,k):\Rev(\rho)=R\} =
\mathscr{L}(\Inv_k(w),\le_R).
\]
\end{lemma}

\begin{theorem}\label{thm:equiv_class.and.linear.exts.affine}
If $\rho \in \A_{w}(n,k)$ and $R=\Rev(\rho)$, then $G_{R}$ is acyclic
and its commutation equivalence class $[\rho]\in \B_w(n,k)$ is given
by
\begin{equation}\label{eq:rev_inverse.affine}
    [\rho]=\{\sigma \in \A_{w}(n,k): \Rev (\sigma )=R\}
          =\mathscr{L}(\Inv_k(w),\leq_R).
\end{equation}
\end{theorem}

\begin{corollary}\label{cor:reversals_unique.affine}
Let $\rho,\sigma \in \A_{w}(n,k)$. Then, $\Rev(\rho)=\Rev(\sigma)$ if
and only if $[\rho]=[\sigma]$.
\end{corollary}

\begin{corollary}\label{cor:B.equiv.acyclic.affine}
The reversal map induces a bijection
\[
Rev: \B_{w}(n,k) \to \{R \subseteq \Inv_{k+1}(w): G_{R} \text{ is
acyclic} \}, 
\]
where the inverse sends $R$ to the commutation class given by linear
extensions 
\[
    \mathscr{L}(\Inv_k(w),\le_R).
\]
\end{corollary}

\begin{corollary}\label{lem:not_consistent_cycle}
Assume $k>1$.    If $R\subseteq \Inv_k(w)$ is not consistent, then $G_R$ contains a cycle.
\end{corollary}

\begin{lemma}\label{lem:suffix_set_consistent.affine}
Assume $k>1$.  For every consistent set $R\in \C_w(n,k)$, the
corresponding suffix set
\[
S(R) = \{Y \in \Inv_{k+1}(w): Y_1\in R\} \in
\C_w(n,k+1).
\]
If $G_{S(R)}$ is acyclic, then $R$ is a lower order ideal in the
corresponding heap poset $(\Inv_{k}(w), \leq_{S(R)})$, and every lower
order ideal of $(\Inv_{k}(w), \leq_{S(R)})$ is in $\C_w(n,k)$.
\end{lemma}

\begin{proof}
Since $k>1$, the packet and second packet of any element in $\binom{\Z
}{k+2}_{n}$ is well-defined.  Therefore, arguments in the proof of
\Cref{lem:suffix_set_consistent}(a) extend directly to the affine
case after making the necessary updates to the notation. It remains to
check that if $X<_{\P_w(n,k+1)}Y=X\pm v_{i}$ is a generating relation in $\P_w(n,k+1)$,
then $Y_{1} \in R$ implies $X_{1}\in R$. This is trivial in the case $Y_1=X_1$. Otherwise,
the parity conditions in \Cref{def:congruence.poset} with $(k+1)-i=k-(i-1)$
 imply $X_1<_{\P_w(n,k)} Y_1=X_1\pm v_{i-1}$.  Thus, $R\in \C_w(n,k)$ implies $S(R)\in \C_w(n,k+1)$.  The
second part follows from the affine generalization of
\Cref{lem:suffix_set_consistent} Parts (b), (c), and (d).  
\end{proof}

The next lemma is the combined analog of
\cref{lem:admissible_to_consistent_set} and \Cref{thm:chains_to_Ank}
for affine permutations. Its proof also follows verbatim using the
affine permutation definitions for $\Inv_k(w)$ and $\P_w(n,k)$ with
the case of $k=1$ omitted.  Note, we have not shown that $\C_w(n,k)$
has a unique minimal element and a unique maximal element as in
\Cref{thm:Cnk_min_max}.  Hence, the map below is not yet known to be a
bijection.

\begin{proposition}\label{lem:affine_admissible_to_consistent_set}
Assume $k>1$, and suppose $\rho=(\rho_1,\rho_2,\ldots,\rho_{|\Inv_k(w)|})\in \A_w(n,k)$. 
Then for any $0\leq i\leq
|\Inv_k(w)|$,
\[
\{\rho_1,\rho_2,\ldots,\rho_i\}\in \C_w(n,k).
\]
Therefore, the successive prefix map
\[\mathcal{R}:
\A_w(n,k) \to \{\text{maximal chains of } \C_w(n,k) \}
\]
as defined in \cref{def:consistentchaintoadmissibleseq} extends to
affine permutations and is an injection for $w \in \affineS_{n}$.  Its
image is all maximal chains in $\C_w(n,k)$ beginning with $\emptyset$
and ending with $\Inv_k(w)$.
\end{proposition}

While the definitions and many lemmas carry over nicely from the
classical permutation case, the affine permutations do add some
complexity to the theory.  In particular, we conjecture that if $R \in
\C_w(n,k)$, then $G_R$ is acyclic. This is open for $k\geq 4$. See
\Cref{conj:affineG_R.advanced} for further discussion.  In the
subsections below, we verify the conjecture for all $k\leq 3$.  We
conclude here with two special cases.  In these two special cases, we
give a different proof of acyclicity from the previous approach.

\begin{lemma}\label{lem:empty_full_acyclic}
Both $\emptyset,\ \Inv_{k+1}(w) \in \C_w(n,k+1)$, and the reversal
graphs $G_\emptyset$ and $G_{\Inv_{k+1}(w)}$ on $\Inv_{k}(w)$ are
acyclic.
\end{lemma}

\begin{proof}
By the same reasoning as in \Cref{ex:empty.full}, both $\emptyset$ and
$\Inv_{k+1}(w)$ are in $\C_w(n,k+1)$ for all $k\geq 1$. Acyclicity of
 $G_{\emptyset}$ and $G_{\Inv_{k+1}(w)}$ when $k=1$ will be 
established later in  \cref{cor:G-R.k=2.acyclic}, where an arbitrary 
$R\in \C_w(n,2)$ is considered. For acyclicity of
 $G_{\emptyset}$ and $G_{\Inv_{k+1}(w)}$ when $k>1$, we will use the functions
$\wt{M}, \wt{F}_{j}, \wt{G}_{j}$ defined in \eqref{eq:affine_inversion_matrix} 
and \eqref{eq:FG.tilde} restricted to $\Inv_{k}(w)$.

For $G_{\emptyset}$, define $\wt{L}: \Inv_{k}(w) \to \mathbb{Z}$ by
\begin{equation}\label{eq:lex.function}  
\wt{L}=(\wt{G}_k+\wt{F}_{k-1}+\wt{G}_{k-2}+\wt{F}_{k-3}+\ldots)\circ \wt{M},
\end{equation}
where the alternating sum terminates at $\wt{F}_{2}$ or
$\wt{G}_{2}$ depending on the parity of $k$.  By
\cref{lem:affine_diagonal_differences} and
\cref{lem:affine_diagonal_packets}, $\wt{L}(X)<\wt{L}(Y)$
for all permanent poset and full lex edges $X \to Y$ in $G_\emptyset$. 
As $G_\emptyset$ contains no full antilex edges, we conclude that
$\wt{L}$ is strictly increasing along all directed walks, so
there cannot be a directed cycle in $G_\emptyset$.

For $G_{\Inv_{k+1}(w)}$, similarly define the function
 $\wt{A}:\Inv_{k}(w) \to \Z$ by
\begin{equation}\label{eq:antilex.function}
\wt{A}=(\wt{F}_{k}+\wt{G}_{k-1}+\wt{F}_{k-2}+\wt{G}_{k-3}+\cdots)\circ \wt{M},
\end{equation}
where the alternating sum terminates at $\wt{F}_{2}$ or
$\wt{G}_{2}$ depending on the parity of $k$.  Again by
\cref{lem:affine_diagonal_differences} and
\cref{lem:affine_diagonal_packets}, $\wt{A}$ is strictly
increasing  along all permanent poset edges and full antilex
edges in $G_{\Inv_{k+1}(w)}$. As $G_{\Inv_{k+1}(w)}$ contains no full lex
edges, we similarly conclude $G_{\Inv_{k+1}(w)}$ is acyclic.
\end{proof}

\begin{corollary}\label{cor:empty_full_acyclic}
There exist $\rho, \sigma \in \A_{w}(n,k)$ such that $\Rev(\rho) =
\emptyset$ and $\Rev(\sigma) = \Inv_{k+1}(w)$.  Such $\rho$ can be
found by taking any linear extension of the heap poset
$(\Inv_{k}(w),\leq_{\emptyset})$, and similarly for $\sigma$ using
$(\Inv_{k}(w),\leq_{\Inv_{k+1}(w)})$.
\end{corollary}

\begin{proof}
    This follows from the argument in \cref{lem:GR_to_admissible} applied to $G_{\emptyset}$ and $G_{\Inv_{k+1}(w)}$, where we use acyclicity from \cref{lem:empty_full_acyclic}. 
\end{proof}

\begin{corollary}\label{cor:w.invariance}
Let $v,w \in \affineS_{n}$.  If $k>1$ and $\Inv_{k}(v)=\Inv_{k}(w)$,
then $\A_{v}(n,k')= \A_{w}(n,k')$, \ $\B_{v}(n,k')= \B_{w}(n,k')$,
$\C_{v}(n,k')= \C_{w}(n,k')$, $\Inv_{k'}(v)=\Inv_{k'}(w)$, and
$\P_{v}(n,k')= \P_{w}(n,k')$ for all $k'\geq k$.
\end{corollary}

\begin{proof}
For $k>1$, we know $\Inv_{k}(w)$ contains the full-packet $P(X)$ if
and only if $X \in \Inv_{k+1}(w)$.  Therefore,
$\Inv_{k}(v)=\Inv_{k}(w)$ implies $\Inv_{k'}(v)=\Inv_{k'}(w)$ for all
$k'\geq k$.  By \Cref{cor:w.invariance.of.perm.poset}, we know
$\P_{w}(n,k)$ only depends on $\Inv_{k}(w)$ for $k>1$.  By definition,
$\A_{w}(n,k), \B_{w}(n,k),$ and $\C_{w}(n,k)$ only depend on
$\P_{w}(n,k)$ and $\Inv_{k+1}(w)$.  Hence, the equalities all hold by
induction on $k'\geq k$.  
\end{proof}

\subsection{The First Affine Higher Bruhat Orders}\label{sub:firstAHB}

The admissible orders $\A_{w}(n,1)$, the higher Bruhat orders
$\B_{w}(n,1)$, and the consistent sets $\C_{w}(n,2)$ have natural
interpretations in terms of affine permutations, which we will
describe next. However, the definition of $\C_{w}(n,1)$ and the
full-packet condition for $\A_{w}(n,1)$ need to be handled carefully
because packets of 2-inversions are not well-defined.  We will remark
on this issue below.  Before we begin identifying the explicit
correspondences, we need to emphasize a key difference between the
classical permutation case and the affine case.

\begin{remark}
For $k=1$ and $w \in \S_{n}$, we are abusing the notation for
$\P_{w}(n,1)$, $\A_w(n,1)$, $\B_{w}(n,1)$, $\C_{w}(n,1)$, and $G_{R}$
as defined in \Cref{sec:non-longest}.  For classical permutations, we
originally defined $\Inv_{1}(w)=\binom{[n] }{1}$.  However, for all
affine permutations, we defined $\Inv_{1}(w)=\binom{\Z }{1}_{n}$,
which is infinite.  This abuse of notation is analogous to the use of
the finite one-line notation for a permutation $(w(1),\dots ,
w(n))\in \S_n$ becoming the window notation for an affine permutation
considered as a bijection on $\Z$.  For $k>1$, all of the definitions
carry over simply by identifying $\Inv_{k}(w) \subseteq \binom{[n]
}{k}$ with the naturally corresponding subset of $\binom{\Z }{k}_{n}$
whenever $w \in \S_{n}$.  Throughout this section, we have assumed $w$
is an affine permutation unless otherwise noted so we use the affine
variation of all definitions.
\end{remark}

By definition, an admissible order $\rho \in \A_{w}(n,1)$ is a linear
extension of $\P_{w}(n,1)$ on $\binom{\Z }{1}_{n} \cong \Z$ such that
for every $X=[x_{1},x_{2}] \in \Inv_{2}(w)$ we must have $P(X)$ listed
in lex order or antilex order in $\rho$.  However, $[x_{1},x_{2}]$
represents all subsets of the form $\{x_{1}+mn,x_{2}+mn\}$ for $m \in
\Z$.  Therefore, the full-packet condition for admissible orders in
$\A_{w}(n,1)$ requires a well-defined labeling on the packet of each
2-inversion of $w$ as ``lex'' or ``antilex'' using any one of its
equivalent representations.  So, the full-packet condition for admissible
orders in $\A_{w}(n,1)$ becomes
\begin{align}\label{eq:packet.condition.1}
&\{x_{1}+mn\}<_{\rho}\{x_{2}+mn \} \text{ for all  }m \in \Z \text{ and
} X\not \in \Rev(\rho ), \text{ or }\\
&\{x_{1}+mn\}>_{\rho}\{x_{2}+mn \} \text{ for all } m \in \Z \text{
and } X \in \Rev(\rho )\label{eq:packet.condition.2}
\end{align}
for all $[x_{1},x_{2}] \in \Inv_{2}(w)$.  Therefore, the packet of
$X=[x_1,x_2]\in\Inv_2(w)$ is \defn{flippable} in $\rho$ when $\{x_1+mn\}$ and
$\{x_2+mn\}$ are adjacent in $\rho$ for every $m\in\Z$, and a \defn{packet
flip} simultaneously reverses all these pairs. Similarly, for any
2-quasi-inversion $[y_{1},y_{2}]\in \Qinv_{2}(w)$, the corresponding
quasi-inversion relation implies
\begin{equation}\label{eq:quasiinversion.condition}
\{y_{1}+mn\}<_{\rho}\{y_{2}+mn \} \text{ for all } m \in \Z, 
\end{equation}
since $\rho$ is a linear extension of $\P_{w}(n,1)$ by definition and
\Cref{rem:easy.pnks.affine}.  For any $i \in \Z$, the congruence
relation $\{i\}<_{\P_{w}(n,1)} \{i+n \}$ implies
\begin{equation}\label{eq:congruence.condition}
\{i\}<_{\rho}\{i+n \}.
\end{equation}

As with the full-packet condition in $\A_{w}(n,1)$, there is a similar issue
with well-defined packet labels for $\C_{w}(n,1)$.  Recall $\C_w(n,1)$
is defined to be the lower order ideals of $\P_{w}(n,1)$ that satisfy
the MSZ condition ordered by single step inclusion.  For $R \in
\C_w(n,1)$, the MSZ condition becomes:
\begin{center}
$P(X) \cap R$ is a \textit{well-defined} prefix or a suffix of $P(X)$
for all $X \in \Inv_{2}(w)$.
\end{center}
The only way for $P([x,y]) \cap R$ to be well-defined when $w \neq
\id$ is for $R$ to be the empty set or all of $\Inv_{1}(w)$.
Therefore, we adopt the convention that
$\C_w(n,1)=\{\emptyset,\Inv_{1}(w)\}$, and these two elements are
incomparable under the single step inclusion order.  However, we note
that $\C_w(n,1)$ does not play a role in \Cref{thm:nonlongest} or its
conjectural generalization to affine permutations so there are some
choices in how to define it.

\begin{lemma}\label{lem:An1.weak.order.affine}
The admissible orders in $\A_{w}(n,1)$ are in bijection with the
affine permutations in the interval $[\mathrm{id},w]$ in the weak
order on $\affineS_{n}$.  The map sends $\rho \in \A_{w}(n,1)$ to the
unique affine permutation in $[\mathrm{id},w]$ with inversion set
given by $\Rev(\rho)$.  The inverse map sends $u \in [\mathrm{id},w]$
to $(\ldots , \{u(-1) \},\{u(0) \}, \{u(1) \}, \{u(2) \},\ldots)\in
\A_{w}(n,1)$.  
\end{lemma}

\begin{proof}
This proof is similar to \Cref{lem:An1.weak.order}, but we will spell
out the details since affine packets of 2-subsets need to be handled
carefully.  Let $\rho \in \A_{w}(n,1)$.  Together the full-packet
condition \eqref{eq:packet.condition.1},
\eqref{eq:packet.condition.2}, the quasi-inversion condition
\eqref{eq:quasiinversion.condition}, and the congruence condition
\eqref{eq:congruence.condition} imply that for every pair of integers
$i<j$, we have a well-defined label on $[i,j]$ as reversed or not in
$\rho$.  Therefore, we cannot have a chain $\{i \} <_{\rho } \{j \}
<_{\rho} \{j+n\} <_{\rho } \{i+n \}$ without having different labels
on $[i,j]=[i+n,j+n]$.  Thus, the well-defined labeling implies that
between $\{i \}$ and $\{i+n \}$ in $\rho$ there must be at most
one representative from each of the other congruence classes modulo $n$.

On the other hand, we claim there is at least one representative from
each of the other congruence classes modulo $n$ between $\{i \}$ and
$\{i+n \}$ in $\rho$.  To see this, fix $j\not\equiv i\pmod n$. Since
$\Inv_2(w)$ is finite, the differences $y-x$ for $[x,y]\in\Inv_2(w)$
are bounded.  The quasi-inversion relations therefore give
$\{j+mn\}<_{\rho}\{i\}$ for all sufficiently negative $m$ and
$\{i\}<_{\rho}\{j+mn\}$ for all sufficiently positive $m$.  Let $m_0$
be the least integer such that $\{i\}<_{\rho}\{j+m_0n\}$. Then
$\{j+(m_0-1)n\}<_{\rho}\{i\}<_{\rho}\{j+m_0n\}$, and the well-defined
property of inversion labelings gives $ \{j+m_0n\}<_{\rho}\{i+n\}$.
Thus there is exactly one representative of each other congruence
class between $\{i\}$ and $\{i+n\}$. Hence, the well-defined labeling
on all equivalence classes $[i,j] \in \binom{\Z }{2}_{n}$ implies
$\rho$ is $n$-periodic in the same sense as an affine permutation.

Since $\rho$ is $n$-periodic, some unique saturated chain $\{u(1)
\}<_{\rho } \{u(2) \}<_{\rho}\cdots <_{\rho } \{u(n) \}$ has the
property that $\sum_{j=1}^{n} u(j) =\binom{n+1 }{2}$.  Therefore,
$\rho$ maps uniquely to the affine permutation $u\in \affineS_{n}$
with window $(u(1),\dotsc ,u(n))$.  By construction $\Inv_{2}(u) =
\Rev (\rho) \subseteq \Inv_{2}(w)$, so $u \in [\mathrm{id},w]$ in weak
order.

Conversely, for each $u \in [\mathrm{id},w]$, we know $\Inv_{2}(u)
\subseteq\Inv_{2}(w)$ and $u^{-1}(j)<u^{-1}(i)$ for $i<j$ if and only
if $[i,j] \in \Inv_{2}(u)$.  Hence, $(\dots , \{u(-1) \},\{u(0) \},
\{u(1) \},\{u(2) \},\dots )$ is a linear extension of $\P_{w}(n,1)$
satisfying the full-packet condition, and hence is admissible for $w$.
Thus, 
\[
u \mapsto (\dots , \{u(-1) \},\{u(0) \}, \{u(1) \},\{u(2)
\},\dots )
\]
is a bijection from $[\mathrm{id},w]$ to $\A_{w}(n,1)$
such that $\Inv_{2}(u)=\Rev(\dots , \{u(-1) \},\{u(0) \}, \{u(1) 
\},\dots )$.
\end{proof}

\begin{corollary}\label{cor:higher.bruhat.k=1}
The affine higher Bruhat order $\B_{w}(n,1)$ is isomorphic to the
interval $[\mathrm{id},w]$ in weak order on $\affineS_{n}$ as posets.
\end{corollary}

\begin{proof}
For every pair of non-congruent integers $i<j$, we know $[i,j]$ is
either an inversion or a quasi-inversion for $w$. Furthermore, if
$i<j$ are congruent, then $\{i \}<_{\P_{w}(n,1)} \{j \}$.  Hence, no
two single element sets $\{i \},\{j \}\in \binom{\mathbb{Z}}{1}_n$
commute with respect to $w$, and all equivalence classes in
$\B_{w}(n,1)$ have size one.  Thus, $\A_{w}(n,1)$ and $\B_{w}(n,1)$
 are the same as sets.  Thus, the bijection in
\Cref{lem:An1.weak.order.affine} extends to a bijection
$[\mathrm{id},w] \to \B_{w}(n,1)$.  

To see the bijection $[\mathrm{id},w] \to \B_{w}(n,1)$ is a poset
isomorphism, recall the covering relations in weak order are of the
form $u \lessdot us_i$ provided $u(i)<u(i+1)$.  Applying $s_{i}$ on
the right is analogous to a well-defined lex-to-antilex packet flip on
the corresponding admissible order, and vice versa, which are the
corresponding covering relations in $\B_{w}(n,1)$.
\end{proof}

\begin{lemma}\label{lem:affine_Cwn2_bijection}
The consistent sets in $\C_w(n,2)$ are exactly the inversion sets of
affine permutations in the interval $[\mathrm{id},w]$ in the weak
Bruhat order on $\affineS_n$.  
\end{lemma}

\begin{proof}
By definition, the elements of $\C_w(n,2)$ are lower order ideals of
$\P_w(n,2)$ satisfying the MSZ condition. Therefore, each $R\in
\C_w(n,2)$ is a subset of $\Inv_2(w)$ satisfying the following
conditions for each triple of integers $x<y<z$.  The justification
in each case follows the condition. 
\begin{enumerate}[label=(\alph*)]
\item $[x,z] \in R$ implies $[x,y] \in R$ or $[y,z] \in R$.   
\begin{itemize}
\item If $x,y,z$ are distinct modulo $n$ and $|P([x,y,z])\cap
\Inv_{2}(w)|=3$, then $[x,y,z]\in \Inv_3(w)$.  If in addition $[x,z]
\in R$, then either $[x,y] \in R$ or $[y,z] \in R$ since
$P([x,y,z])\cap R$ must be a prefix or suffix of $P([x,y,z])$ by the MSZ
condition.  
\item If $x,y,z$ are distinct modulo $n$ and $|P([x,y,z])\cap
\Inv_{2}(w)|=2$, then $P([x,y,z])\cap \Inv_{2}(w)$ must equal
$\{[x,y],[x,z] \}$ or $\{[y,z],[x,z] \}$ by \cref{lem:inversion.sets}
(a) and (b).  If $[x,z] \in R$, then either $[x,y] \in R$ or $[y,z]
\in R$ since $R$ is a lower order ideal of $\P_{w}(n,2)$, and $[x,z]$
would be greater than either $[x,y]$ or $[y,z]$ via a quasi-inversion
relation.  
\item If $x,y,z$ are distinct modulo $n$ and $|P([x,y,z])\cap
\Inv_{2}(w)|\leq 1$, then $[x,z] \not \in \Inv_{2}(w)$ by
\cref{lem:inversion.sets} Part (a), so $[x,z] \not \in R \subseteq
\Inv_{2}(w)$.
\item If $[x,z] \in R \subseteq \Inv_{2}(w)$, then $x,z$ are not
congruent by \cref{lem:inversion.sets} Part (c), so it remains to
consider the case where $y$ is congruent to one of $x$ or $z$.  If
$x<y$ are congruent, then $[x,z] \in \Inv_{2}(w)$ implies $[y,z] \in
\Inv_{2}(w)$ by \cref{lem:inversion.sets} Part (c), and the congruence
relations imply $[y,z]<_{\P_{w}(n,2)}[x,z]$. Hence, $[y,z] \in R$
since $R$ is a lower order ideal of $\P_{w}(n,2)$.  Similarly, if
$y,z$ are congruent, then $[x,y]<_{\P_{w}(n,2)}[x,z]$ so $[x,y] \in R$
by the lower order ideal condition.
\end{itemize}
\bigskip
\item $[x,y] \in R$ and $[y,z] \in R$ implies $[x,z] \in R$. 
\begin{itemize}
\item If $[x,y] \in \Inv_{2}(w)$ and $[y,z] \in \Inv_{2}(w)$, then
$[x,z] \in \Inv_{2}(w)$ by \cref{lem:inversion.sets}(b) and $[x,y,z]
\in \Inv_{3}(w)$.  Since $R$ satisfies the MSZ condition, $R \cap
P([x,y,z])$ is a prefix or suffix of $P([x,y,z])$.  Hence, $[x,y] \in R$
and $[y,z] \in R$ implies $[x,z] \in R$.
\end{itemize}
\bigskip
\item $[x,y]\in R$ implies $x,y$ are not congruent and $[x,y-en] \in
R$ for all $e \in \Z_{\geq 0}$ such that $x<y-en<y$.  
\begin{itemize}
\item If $[x,y] \in R \subseteq \Inv_{2}(w)$, then $x,y$ are not
congruent by $n$-periodicity of $w$.
\item By \cref{lem:inversion.sets}, $[x,y-en] \in \Inv_{2}(w)$ for all
$e \in \Z_{\geq 0}$ such that $x<y-en<y$. By congruence relations
$[x,y-en]<_{\P_{w}(n,2)} \dots <_{\P_{w}(n,2)} [x,y-n]<_{\P_{w}(n,2)}
[x,y]$ whenever $x< y-en<y$.  Hence $[x,y]\in R$ implies each such
$[x,y-en] \in R$ since $R$ is a lower order ideal in $\P_{w}(n,2)$.
\end{itemize}
\end{enumerate}
Conditions (a), (b), and (c) above are exactly the requirements $R$
must satisfy in order to be the inversion set of some affine
permutation $u\in \affineS_{n}$ by \cref{lem:inversion.sets}.
Furthermore, since $R \subseteq \Inv_{2}(w)$, there exists some $u \in
[\id,w]$ such that $R=\Inv_{2}(u)$ by the characterization of the weak
Bruhat order as containment of inversion sets.

Conversely, each element $u\in [\id,w]$ is determined by its
2-inversion set $\Inv_{2}(u)$, where $\Inv_{2}(u) \subseteq \Inv_2(w)$.
The MSZ condition holds for $\Inv_{2}(u)$ by \cref{lem:inversion.sets}
Parts (a) and (b).  \cref{lem:inversion.sets} Part(a) implies that
$\Inv_{2}(u)$ is a lower order ideal in the poset generated by the
quasi-inversion relations.  \cref{lem:inversion.sets} Part (c)
implies that $\Inv_{2}(u)$ is a lower order ideal in the poset on $\Inv_2(w)$
generated by the congruence relations.  Therefore, $\Inv_{2}(u)$ must
be a lower order ideal in $\P_w(n,2)$ since it is generated by these
quasi-inversion and congruence relations. Thus, $\Inv_{2}(u)$ satisfies
the conditions for consistent sets from \Cref{def:consistent.affine}, so 
the inversion sets for affine permutations in $[\id,w]$ are precisely the 
consistent sets in $\C_w(n,2)$. 
\end{proof}

\begin{corollary}\label{cor:affine_Cwn2_isomorphism}
The map $\Inv_{2}: [\mathrm{id},w] \to \C_w(n,2)$ is a poset
isomorphism.  Consequently,
\begin{equation}\label{eq:Cwn2_isom}
[\mathrm{id},w] \cong \B_{w}(n,1) \cong \C_w(n,2).
\end{equation}
All three are ranked posets, each has a unique minimal element
corresponding to the identity element with inversion set $\emptyset$,
and each has a unique maximal element corresponding to $w$ with
inversion set $\Inv_{2}(w)$.
\end{corollary}

\begin{proof}
The partial orders on both $[\id,w]$ and $\C_w(n,2)$ are generated by
single step inclusion.  Therefore, the bijection $[\id,w] \to
\C_w(n,2)$ mapping $v \to \Inv_{2}(v)$ and the inverse map are 
order-preserving by construction. Hence, this is an isomorphism of posets.  
By \Cref{cor:higher.bruhat.k=1}, $[\id,w] \to \B_w(n,1)$ is an
isomorphism of posets, so by transitivity all three posets are
isomorphic.  The statement of minimal and maximal elements now follows
from the isomorphism with the interval $[\id,w]$.
\end{proof}

Recall from \Cref{rem:k=2graphs} that for $w \in \S_{n}$ and each $R
\in \C_w(n,2)$, the reversal graph $G_{R}$ is a distinct acyclic
orientation of the complete graph on $[n]\cong
\Inv_{1}(w)$. Furthermore, the poset determined by $G_{R}$ is the
total order on $[n]$ determined by the unique permutation $u\in \S_n$
with $\Inv_{2}(u)=R$. One can think of $G_{R}$ as the
\textit{inversion graph} for $u$.  A similar phenomenon occurs for
affine permutations and affine reversal graphs, but the vertex set is
infinite.

\begin{corollary}\label{cor:G-R.k=2.acyclic}
For $R \in \C_w(n,2)$, the transitive closure of the affine reversal
graph $G_{R}$ is an acyclic orientation of the complete graph on
$\Inv_{1}(w)=\binom{\Z }{1}_{n}\cong \Z$, and the poset determined by
$G_{R}$ is the total order determined by the unique affine permutation
in $[\mathrm{id},w]$ with inversion set $R$.
\end{corollary}

\begin{proof}
By \cref{lem:affine_Cwn2_bijection}, if $R\in \C_w(n,2)$, then there
exists a unique $u\in [\id,w]$ such that $R=\Inv_2(u)$.  We claim the
transitive closure of $G_{R}$ is exactly the directed graph on
$\Inv_{1}(w)$ with edges $\{u(i) \} \to \{u(j) \}$ for all $i<j$.
Since $u$ is a bijection from $\Z$ to $\Z$, we can view it as a total 
order on $\mathbb{Z}$, so the statement will follow from this claim.

To verify the claim using \Cref{def:affine.GR}, let $x<y$ be integers.  Note that
if $x\equiv y \pmod n$ then there is a sequence of congruence edges
$\{x \} \to \{x +n\} \to \{x +2n\} \to \dots \to \{y \}$ in $G_{R}$ 
and $u^{-1}(x)<u^{-1}(y)$.  If $x \not \equiv y\pmod n$ then $[x,y]
\in \binom{\Z }{2}_{n}$ is either an inversion or a quasi-inversion
for $w$.  If $[x,y] \in \binom{\Z }{2}_{n} \setminus R$, there are
edges in $G_{R}$ of the form $\{x +mn\} \to\{y+mn \} $
and $u^{-1}(x+mn)<u^{-1}(y+mn)$ for all $m \in \Z$.  If $[x,y] \in R$,
there are edges in $G_{R}$ of the form $\{y +mn\}\to\{x+mn \} $ and
 $u^{-1}(y+mn)<u^{-1}(x+mn)$ for all $m \in \Z$. Thus, the claim holds.
\end{proof}

\begin{remark}\label{rem:compare.bb.GR}
For $w \in \affineS_{n}$, the information in the reversal graph
$G_{\Inv_2(w)}$ is equivalent to the inversion multigraph $I_{w^{-1}}$
from \Cref{def:BB.graph}, but the two graphs are not equal.  One could
think of $G_{\Inv_2(w)}$ as a \textit{blow-up} of $I_{w^{-1}}$ where
each vertex $x\in [n]$ of $I_{w^{-1}}$ is replaced by a chain of
congruent vertices
\[\ldots \to\{x-2n\}\to \{x-n\}\to \{x\}\to \{x+n\}\to \{x+2n\}\to \ldots. \]
The edges correspond as follows. Fix distinct $x,y\in [n]$. For each
$y'\equiv y \pmod n$ with $[x,y']\in \Inv_2(w)$, there is a
full-packet edge $\{y'\}\to \{x\}$ in $G_{\Inv_2(w)}$. Collapsing each
congruence chain modulo $n$ sends this edge to an edge $y\to x$ in
$I_{w^{-1}}$. Conversely, every edge $y\to x$ in $I_{w^{-1}}$ arises
uniquely in this way from some $[x,y']\in \Inv_2(w)$ written in
canonical notation with $x \in [n]$ and $y'\equiv y\pmod n$, as the
multiplicity of $y\to x$ from \cref{def:BB.graph} can be alternatively
calculated using
\[
m_{x,y}
=
\left|\{\,y'\equiv y \pmod n : [x,y']\in \Inv_2(w)\,\}
\right|.
\]
The quasi-inversion edges in $G_{\Inv_2(w)}$ involving the congruence
classes of $x$ and $y$ can be easily distinguished from full-packet
edges since they correspond with non-inversions for $w$, and these
remaining edges do not have an image in $I_{w^{-1}}$.  See
\cref{fig:GInv2} and compare with $I_{w^{-1}}$ in
\cref{fig:inversion_graph_example} for an example.
\end{remark}

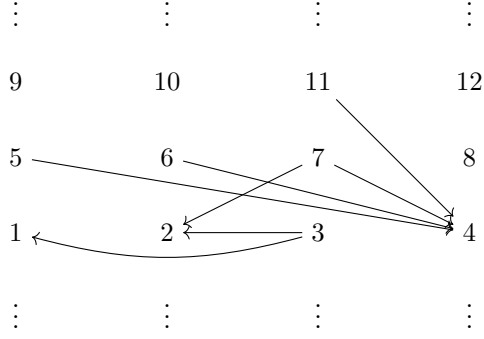
\begin{figure}
\begin{tikzpicture}
        \node (1) at (0,0) {$1$};
	\node (5) at (0,1) {$5$};
	\node (9) at (0,2) {$9$};
	\node at (0,3) {$\vdots$};
	\node at (0,-1) {$\vdots$};
        \node (2) at (2,0) {$2$};
	\node (6) at (2,1) {$6$};
	\node (10) at (2,2) {$10$};
	\node at (2,3) {$\vdots$};
	\node at (2,-1) {$\vdots$};
        \node (3) at (4,0) {$3$};
	\node (7) at (4,1) {$7$};
	\node (11) at (4,2) {$11$};
	\node at (4,3) {$\vdots$};
	\node at (4,-1) {$\vdots$};
        \node (4) at (6,0) {$4$};
	\node (8) at (6,1) {$8$};
	\node (12) at (6,2) {$12$};
	\node at (6,3) {$\vdots$};
	\node at (6,-1) {$\vdots$};
        \draw[->] (5) to (4);
        \draw[->] (6) to (4);
        \draw[->] (7) to (4);
        \draw[->] (3) to (2);
        \draw[->] (3) to [out=195,in=345] (1);
        \draw[->] (11) to (4);
        \draw[->] (7) to (2);
\end{tikzpicture}
\caption{A portion of $G_{\Inv_2(w)}$ for $w=(1,7,2,0)\in \affineS_4$, where edges corresponding to $[x,y]\in \Inv_2(w)$ with $x\in [n]$ are shown. Note that all other edges involving $x\in [n]$ point from the smaller element to the larger one, so the shown edges can be easily identified in $G_{\Inv_2(w)}$.}
\label{fig:GInv2}
\end{figure}

\subsection{The Second Affine  Higher Bruhat Orders}\label{sub:secondAHB}

We will next consider our constructions in the case of $k=2$ and prove
\cref{thm:affine}.  \cref{lem:inversion.sets} and
\cref{cor:characterize.ref.order} will be used repeatedly throughout
this section.  We start by relating the admissible orders in
$\A_{w}(n,2)$ with reflection orders for $w$, as in
\Cref{lem:An2.weak.order}.

\begin{lemma}\label{cor:characterize.ref.order.new} 
 A total order $\rho = (\rho_{1},\rho_{2},\dotsc , \rho_{\ell})$ of
$\Inv_{2}(w)$ is in $\A_{w}(n,2)$ if and only if
\begin{enumerate}[label=(\alph*)]
\item for each $[x,y,z] \in \binom{\mathbb{Z}}{3}_n$, 
$\rho|_{P([x,y,z])\cap \Inv_2(w)}$ is either the restriction of the
lex order  to a prefix of $P([x,y,z])$ or the restriction of antilex
order to a suffix of $P([x,y,z])$, and 
\item for each pair $[x,y],[x,y+n] \in \Inv_2(w)$, the pair
$[x,y]<_{\rho}[x,y+n]$.  
\end{enumerate}
\end{lemma}

\begin{proof}
Assume $\rho \in \A_{w}(n,2)$.  Then $\rho$ is a linear extension of
$\P_{w}(n,2)$ that satisfies the full-packet condition.  The
congruence relations in $\P_{w}(n,2)$ guarantee that (b) holds. To see
(a) holds, consider $X=[x,y,z] \in \binom{\mathbb{Z}}{3}_n$.  If $X$
is a 3-quasi-inversion, then $P(X) \cap \Inv_{2}(w)=\{X_{i},X_{i+1}
\}$ for $i\in \{1,2\}$ by \Cref{lem:types.of.packets}. In either case,
$\rho$ restricted to $\{X_{i},X_{i+1} \}$ is a prefix in lex order or
a suffix in antilex order by the corresponding quasi-inversion
relation.  If $X\in \Inv_{3}(w)$, then $\rho|_{P(X)}$ is lex order or
antilex order on $P(X)$ by the full-packet condition. Next, observe
that if $[x,z] \in P([x,y,z])\cap \Inv_{2}(w)$, then at least one of
$[x,y]$ or $[y,z]$ is also in $\Inv_{2}(w)$ by
\cref{lem:inversion.sets}. This implies $[x,y,z]$ is either a
3-quasi-inversion or 3-inversion of $w$, and these cases have already
been covered. Hence, the remaining cases correspond to when
$P([x,y,z])\cap \Inv_{2}(w)$ is a prefix or suffix of $P([x,y,z])$ of
size 0 or 1, and these cases automatically satisfy the condition in
(a).

Conversely, if $\rho$ satisfies (a), then it satisfies the full-packet
condition and it is a linear extension of the poset
generated by the quasi-inversion relations.  If $\rho$ satisfies (b),
then it is a linear extension of the poset on $\Inv_{2}(w)$ generated
by the congruence relations.  Together, this implies $\rho \in
\A_{w}(n,2)$ by definition.
\end{proof}

\begin{lemma}\label{lem:Bwn2_and_graph}
The reflection orders for $w$ are exactly the admissible orders for
$w$ in $\A_{w}(n,2)$.  Furthermore, there are natural bijections
between maximal chains of $\C_w(n,2)$, reduced words for $w$, and the
admissible orders in $\A_w(n,2)$. Under the bijection between reduced
words and admissible orders, two reduced words
\begin{enumerate}[label=(\alph*)]
            \item are commutation equivalent if and only if their corresponding two admissible orders are commutation equivalent, and
            \item differ by a (directed) braid relation $i(i+1)i\to (i+1)i(i+1)$ if and only if their corresponding admissible orders differ by a lex-to-antilex packet flip. 
    \end{enumerate}

\end{lemma}

\begin{proof}
The proof that reflection orders and 2-admissible orders for $w$ are
the same follows directly by comparing
\Cref{cor:characterize.ref.order} and
\Cref{cor:characterize.ref.order.new}.  To prove the second statement,
observe that $\C_w(n,2)$ is isomorphic to the weak order interval
$[\id, w]$ by \cref{cor:affine_Cwn2_isomorphism}.  It is well-known
that there is a natural bijection between maximal chains of $[\id,w]$
and reduced words for $w$ since the covering relations in weak order
are of the form $w<ws_{i}$. The latter are in bijection with the
reflection orders, which we know are admissible orders by the first
statement. The correspondences in (a) and (b) now follow from the
arguments in \cref{lem:Bn2.weak.order} provided $n\geq 3$.  In the
case $n=2$, each affine permutation has a unique reduced word so the
final sentence holds trivially.   
\end{proof}

\begin{remark}
    While not explicitly stated in \cref{lem:Bwn2_and_graph}, one can show that the bijection between maximal chains of $\C_w(n,2)$ and admissible orders in $\A_w(n,2)$ is given by mapping
    any maximal chain
\[C=(\emptyset = R_{0}\lessdot_{1} R_1 \lessdot_{1} R_2 \lessdot_{1} \dots \lessdot_{1} R_{\ell(w)})\] in $\C_w(n,2)$
to
\[
\rho=(\rho_{1},\rho_{2},\dots, \rho_{\ell(w)})
\]
where $R_{i}\setminus R_{i-1}=\{\rho_i\}$ and $\ell(w)=|\Inv_{2}(w)|$. This map generalizes the
inverse successive prefix map from \cref{thm:chains_to_Ank} to affine
permutations in the case of $k=2$.
\end{remark}

We next turn our attention to $\B_w(n,2)$ and $\C_w(n,3)$. Paralleling
our approach in \cref{sec:non-longest}, we will state the following
theorem involving acyclicity of $G_R$ when $R\in \C_w(n,3)$, but we
postpone its proof until \cref{sec:affineGR} in order to state the
main results of this section.  Note, the corollary below is an affine
analog of \Cref{thm:Bnk_iso_Cnk} for the second higher Bruhat order.

\begin{theorem}\label{thm:affine_GR}
    For any $w\in \affineS_n$ and $R\subseteq \Inv_3(w)$, $G_R$ is acyclic if and only if $R\in \C_w(n,3)$. 
\end{theorem}

\begin{corollary}\label{thm:affine_Bn2_Cn3}
     The map $\Rev_{n,2,w}: \B_w(n,2) \to
\C_w(n,3)$ is a poset isomorphism with inverse given
by
\begin{equation}\label{eq:affine_rev_inverse}
    \Rev_{n,2,w}^{-1}(R)=\mathscr{L}(\Inv_2(w),\leq_R).
\end{equation}
\end{corollary}

\begin{proof}
    This follows by adapting the arguments from
\cref{lem:GR_to_admissible}, \cref{cor:reversals_unique}, and
\cref{thm:Bnk_iso_Cnk} for $k=2$ to the affine
case. \cref{thm:affine_GR} is used in place of \cref{lem:GRacyclic}.
\end{proof}

Next we show that $\B_w(n,2)\cong \C_w(n,3)$ has a unique minimal
element and a unique maximal element. Our methods in
\cref{sec:non-longest} for $w\in \S_n$ involved using the suffix set
of $R\in \C_w(n,k)$, which we showed is a consistent set $S(R)\in
\C_w(n,k+1)$. However, when $R\in \C_w(n,3)$, \cref{thm:affine_GR}
does not apply to $S(R)\subseteq \Inv_4(w)$, so we use a different
approach for $w\in \affineS_n$.  We begin with a notion of deletion
that extends the one in \cref{sub:G_R} to the affine setting.

\begin{definition}\label{def:finite_deletion}
    Let $A\subseteq \mathbb{Z}$ be a nonempty finite subset.  Define
$\bar{A}=\mathbb{Z}\setminus A$, and let $\fl_{\bar{A}}:A\to [|A|]$ be
the unique order-preserving bijection from $A$ to $\{1,2,\dots , |A|
\}$.  

\begin{enumerate}[label=(\alph*)]
\item For $w\in \affineS_n$, define $\Fl_{\bar{A}}(w)$ to be the permutation
in $\S_{|A|}$ obtained by restricting the one-line notation of $w$ to
the elements in $A$ and applying $\fl_{\bar{A}}$ to each entry.
\item An $\ell$-inversion $X\in \Inv_\ell(w)$ is \defn{compatible} with $A$
if it has a unique representative $[x_{1},\dots , x_{\ell}]$ such that
each $x_{i} \in A$. In this case, define 
\begin{equation}\label{eq:compatible.flatten}
\Fl_{\bar{A}}(X) = [\fl_{\bar{A}}(x_{1}),\dots , \fl_{\bar{A}}(x_{\ell})].
\end{equation}

\item For any $R\subseteq \Inv_\ell(w)$, the \defn{deletion of $R$ with
respect to $\bar{A}$} is \[R\setminus
\bar{A}=\{\Fl_{\bar{A}}(X):X\in R \text{ is compatible with
$A$}\}\]
\item For any $\rho\in \A_w(n,\ell)$, the \defn{deletion of $\rho$ with
respect to $\bar{A}$}, denoted $\rho\setminus \bar{A}$, is
the total order obtained by restricting $\rho$ to the elements
compatible with $A$ and then applying $\Fl_{\bar{A}}$ to these
elements.
\end{enumerate}
\end{definition}

\begin{example}\label{example:extra.quasi.invs}
Let $w=(-3,8,3,2) \in \affineS_{4}$.  If $A=\{1,2,3,8 \}$, then
$w'=\Fl_{\bar{A}}(w)=(4,3,2,1) \in \S_4$, and $\Fl_{\bar{A}}$ gives a
bijection from the $\ell$-inversions ($\ell$-quasi-inversions) for
$w$ that are compatible with $A$ to the $\ell$-inversions
($\ell$-quasi-inversions) for $w'$ for all $\ell\geq 0$. However, if
$B=\{1,2,3,4,8 \}$, then $\Fl_{\bar{B}}(w)=(4,5,3,2,1) \in \S_5$,
which has $[1,2,3,4,5]$ as a 5-quasi-inversion while $w=(-3,8,3,2)$
has no 5-quasi-inversions.
\end{example}

We will primarily use the above notion of deletion when $A$ contains
distinct values modulo $n$. In this case, the deletion of a consistent
set or admissible order is well-behaved, allowing us to translate
certain claims involving the affine higher Bruhat orders to
corresponding ones for permutations in $\S_n$. The next lemma
formalizes this.

\begin{lemma}\label{lem:finite_deletion}
Let $A$ be a nonempty set of integers of size $m$ that are
distinct modulo $n$. Define $w'=\Fl_{\bar{A}}(w)$.

\begin{enumerate}[label=(\alph*)]
\item If $\rho \in \A_{w}(n,k)$, then $\rho':=\rho\setminus \bar{A}\in
\A_{w'}(m,k)$ and $\Rev(\rho')=\Rev(\rho)\setminus \bar{A}$.  
\item If $R \in \C_w(n,k+1)$, then $R':=R\setminus \bar{A}\in
\C_{w'}(m,k+1)$ as a subset of $\binom{[m] }{k+1}$. Furthermore, the induced subgraph of $G_R$ on the
$k$-inversions of $w$ compatible with $A$ is isomorphic to $G_{R'}$ on
$\Inv_{k}(w')$.
\end{enumerate}
\end{lemma}

\begin{proof}
It is straightforward to show that $\Fl_{\bar{A}}$ is a bijection
between the inversions in $\Inv_{\ell}(w)$ compatible with $A$ and the
inversions in $\Inv_{\ell}(w')$ for all $\ell\geq 1$, where we
consider $\Inv_{\ell}(w') \subseteq \binom{[m] }{\ell}$ as classical
subsets.  Furthermore, we have that $\Inv_{\ell}(w)\setminus
\bar{A}=\Inv_{\ell}(w')\subseteq \binom{[m] }{\ell}$.  Similarly, since
$A$ contains only noncongruent values modulo $n$, $\Fl_{\bar{A}}$ is a
bijection between the quasi-inversions in $\Qinv_{\ell}(w)$ compatible
with $A$ and the quasi-inversions in $\Qinv_{\ell}(w')$, and
$\Qinv_{\ell}(w)\setminus \bar{A}=\Qinv_{\ell}(w')$ for all $\ell\geq
1$.

For (a), we first show that $\rho'$ is a linear extension of $\P_{w'}(m,k)$. 
Every generating relation $X'<_{\P_{w'}(m,k)} Y'$ must be a quasi-inversion relation since $w'\in \S_m$ is 
a classical permutation. Let $X=\Fl_{\bar{A}}^{-1}(X')$ and $Y=\Fl_{\bar{A}}^{-1}(Y')$, which
are elements of $\Inv_k(w)$ compatible with $A$. As $\fl_{\bar{A}}:A\to [m]$
is an order-preserving bijection and quasi-inversions map to
quasi-inversions, we see that $X<_{\P_w(n,k)} Y$ is a quasi-inversion relation. Since $\rho$
is an admissible order of $\Inv_k(w)$, it must be that $X<_{\rho} Y$, so by the construction of the deletion,
we conclude $X'<_{\rho'} Y'$. As $X'<_{\P_{w'}(m,k)} Y'$ was an arbitrary generating
relation, we conclude $\rho'$ is a linear extension of $\P_{w'}(m,k)$. 

We next consider the restriction of $\rho'$ to the packet of any $X'\in \Inv_{k+1}(w')$.
Let $X=\Fl_{\bar{A}}^{-1}(X')$, which is an element of  $\Inv_{k+1}(w)$ compatible with $A$. 
Observe that $\Fl_{\bar{A}}(X_i)=X_i'$
for all $i\in [k+1]$. From this and the construction of the deletion, we see that if $\rho$ restricts to the lex (resp. antilex)
order on $P(X)$, then $\rho'$ restricts to the lex (resp. antilex) order on $P(X')$.
As $\rho$ is an admissible order, we conclude that
$\rho'\in \A_{w'}(m,k)$ and $\Rev(\rho')=\Rev(\rho)\setminus \bar{A}$. 

The claims in (b) are done similarly. We find the same correspondence between generating relations
$X'<_{\P_{w'}(m,k+1)} Y'$ and pulled-back $X<_{\P_w(n,k+1)} Y$ with $X=\Fl_{\bar{A}}^{-1}(X')$ and $Y=\Fl_{\bar{A}}^{-1}(Y')$,
 so $R\in \C_w(n,k+1)$ being a lower-order ideal of $\P_w(n,k+1)$ implies
 $R'$ is a lower-order ideal of $\P_{w'}(m,k+1)$. The correspondence between $P(X')\subseteq \Inv_{k+1}(w')$
 and $P(X)\subseteq \Inv_{k+1}(w)$ for  $X=\Fl_{\bar{A}}^{-1}(X')$ 
shows that if $R\cap P(X)$ is  a prefix (resp. suffix) of $P(X)$, then $R'\cap P(X')$ is a prefix (resp. suffix) of $P(X')$.
 Hence, $R'\in \C_{w'}(m,k+1)$.

Finally, $\Fl_{\overline{A}}$ gives a bijection between the vertices of
$G_R$ compatible with $A$ and the vertices of $G_{R'}$.
Observe both sets are finite, even if $k=1$.  
Since the
elements of $A$ are distinct modulo $n$, no congruence edge of $G_R$
has both endpoints compatible with $A$. The permanent poset argument
above shows that quasi-inversion edges are preserved under this
bijection. The correspondence between $X'\in R'$ (resp. $X'\notin R'$)
and $X=\Fl_{\bar A}^{-1}(X')\in R$ (resp.
$X=\Fl_{\bar A}^{-1}(X')\notin R$) shows that full antilex
(resp. full lex) edges are also preserved. Hence, 
$\Fl_{\overline A}$ is a graph isomorphism.
\end{proof}

Recall from \Cref{rem:nested} and \Cref{rem:nesting.congrence} that
the quasi-inversion relations and the congruence relations are related
to nested packets in admissible orders.  We formalize this in a
definition below in order to identify flippable packets in $[\rho] \in
\B_w(n,2)$.

\begin{definition}
    Let $\rho\in \A_w(n,k)$, and $X,Y\in \Inv_{k+1}(w)$ be distinct. The
packet $P(Y)$ is \defn{lex-nested} within $P(X)$ 
if
\begin{equation}\label{eq:lex-nested} X_{k+1}\leq _{\rho} Y_{k+1}<_{\rho}
Y_1\leq _{\rho} X_1.
\end{equation}
When there are no packets lex-nested within $P(X)$ and $\rho|_{P(X)}$
is in lex order, we call $P(X)$ \defn{lex-unnested} in $\rho$. We
similarly define \defn{antilex-nested} and \defn{antilex-unnested} by
replacing \eqref{eq:lex-nested}
with \begin{equation}\label{eq:antilex-nested} X_1\leq_{\rho} Y_1
<_{\rho} Y_{k+1} \leq_{\rho} X_{k+1}.  \end{equation}
\end{definition}

\begin{lemma} \label{lem:flippable_Bn2_Cn3}
For any $[\rho] \in \B_w(n,2)$ with $\Rev(\rho) \neq \Inv_3(w)$, there
exists $X \in \Inv_3(w) \setminus \Rev(\rho)$ such that $P(X)$ is
lex-to-antilex flippable in $[\rho]$.  Similarly, if $\Rev(\rho)\neq
\emptyset$ there exists an $X\in \Rev(\rho)$ such that $P(X)$ is
antilex-to-lex flippable in $[\rho]$.
\end{lemma}

\begin{proof}
    Suppose $\Rev(\rho)\neq \Inv_3(w)$, so there exists
$X=[x_1,x_2,x_3]\in \Inv_3(w)\setminus \Rev(\rho)$. Then $\rho$
restricts to the lex order on $P(X)$, so $[x_1,x_2]<_{\rho} [x_1,x_3]
<_{\rho} [x_2,x_3]$. We can assume without loss of generality that
$P(X)$ is lex-unnested in $\rho$, as otherwise, we can iteratively
replace $X$ with another $3$-inversion whose packet is lex-nested
within $P(X)$ in $\rho$. We will show that any elements between
$X_{3}=[x_1,x_2]$ and $X_{1}=[x_2,x_3]$ in $\rho$ other than
$X_{2}=[x_{1},x_{3}]$ can be commuted out of that interval so that
$P(X)$ is flippable in $[\rho]$.

We first show that any $Y$ satisfying $[x_1,x_2]<_{\rho} Y<_{\rho}
[x_2,x_3]$ cannot be congruent to an element in $P(X)$ without
contradicting the fact that $P(X)$ is lex-unnested in $\rho$.  Note
that $Y$ cannot have the form $[x_1,x_2-mn]$ or $[x_2,x_3+mn]$ 
for a positive integer $m$, as $[x_1,x_2]<_{\rho} Y<_{\rho} [x_2,x_3]$ 
would contradict the relations in $\P_w(n,2)$. The remaining four cases
 are detailed below. Throughout, assume $m$ is a positive integer.
\begin{itemize}

\item If $Y=[x_1,x_2+mn]=[x_1-mn,x_2]$, then $[x_1-mn,x_2,x_3]\in
\Inv_3(w)$ by the affine analog of \Cref{lem:adjacent.pairs}.
Furthermore, $[x_1,x_2]<_{\rho} Y<_{\rho} [x_2,x_3]$ implies that
$[x_1-mn,x_2,x_3]$ is lex-nested within $P(X)$ in $\rho$.

\item If $Y=[x_1,x_3-mn]=[x_1+mn,x_3]$, then the
congruence relations imply $[x_1,x_2]<_{\rho} Y<_{\rho}
[x_1,x_3]$. Consider the packet on $x_{1},x_{2},x_{3}-mn$. If
$x_3-mn<x_2$, then $[x_{3}-mn,x_{2}] \not \in \Inv_{2}(w)$ by
\cref{lem:inversion.sets}(c) since $[x_{2},x_{3}]\in \Inv_{2}(w)$.
Hence, $[x_1,x_3-mn,x_2]$ would have to be a quasi-inversion for $w$,
but the order of $[x_1,x_2]<_{\rho}Y$ contradicts the 
corresponding quasi-inversion relation.  Alternatively, if
$x_2<x_3-mn$, then $[x_{2},x_{3}-mn] \in \Inv_{2}(w)$ by
\cref{lem:inversion.sets}(c).  Since $[x_{1},x_{2}]$ and
$[x_{2},x_{3}-mn]$ are in $\Inv_{2}(w)$, we know $[x_1,x_2,x_3-mn]\in
\Inv_3(w)$.  Since $[x_1,x_2]<_{\rho} [x_2,x_3-mn]<_{\rho} [x_2,x_3]$,
we conclude that $P([x_1,x_2,x_3-mn])$ is lex-nested within $P(X)$ in
$\rho$.

\item If $Y=[x_1,x_3+mn]=[x_1-mn,x_3]$, then the relations in
$\P_w(n,2)$ imply that $X_2=[x_1,x_3]<_{\rho} Y <_{\rho}
[x_2,x_3]$. However, $[x_1-mn,x_3]<_{\rho} [x_2,x_3]$ implies that
$[x_1-mn,x_2,x_3]=[x_1,x_2+mn,x_3+mn]$ cannot be a quasi-inversion, so
this must be a $3$-inversion in lex order in $\rho$.  Hence,
$P([x_{1},x_{2}+mn,x_{3}+mn])$ is lex-nested within $P(X)$ in $\rho$.

\item If $Y=[x_2,x_3-mn]$, then $[x_1,x_2,x_3-mn] \in \Inv_3(w)$. From
$[x_1,x_2]<_{\rho} Y<_{\rho}[x_{2},x_{3}]$, we conclude $P([x_1,x_2,x_3-mn])$
is lex-nested within $P(X)$ in $\rho$.
\end{itemize}

Second, assume $[x_1,x_2]\lessdot_{\rho} Y<_{\rho} [x_1,x_3]$ and $Y$ is
not congruent to any element in $P(X)$. We claim that $Y$ commutes with $[x_1,x_2]$. 
Otherwise, $Y$ shares an element with $[x_1,x_2]$ and contains a second element
that is not congruent modulo $n$ to any element in
$\{x_1,x_2,x_3\}$. Consequently, we have that $X_1,X_2,X_3,Y\in P^2(Z)$ for some
$Z=[z_1,z_2,z_3,z_4]\in \binom{\mathbb{Z}}{4}_n$. Setting
$A=\{z_1,z_2,z_3,z_4\}$ and applying the deletion from
\cref{def:finite_deletion} results in permutation $w'\in \S_4$ and
admissible order $\rho'\in \A_{w'}(4,2)$ by
\cref{lem:finite_deletion}. Using a computer verification over all $w'
\in \S_{4}$ and all $\rho' \in \A_{w'}(4,2)$ such that some
lex-unnested $X' \in \Inv_{3}(w')\setminus \Rev (\rho')$ exists, we
find that if $X_3'\lessdot_{\rho'}Y'<_{\rho'}X_2'$, then $Y'$ 
and $X_3'$ commute with respect to $w'$. 
See appendix for verification details. Consequently, it cannot be that $Y$ 
shares an element with $[x_1,x_2]$, and hence $Y$ and $[x_1,x_2]$
commute with respect to $w$. Iterating this argument, we see that each $Y$ between
$[x_1,x_2]$ and $[x_1,x_3]$ in $\rho$ commutes with $[x_1,x_2]$.

Next, assume $[x_1,x_3]<_{\rho} Y\lessdot_{\rho} [x_2,x_3]$ and $Y$ is not 
congruent to any element in $P(X)$. Using a similar argument, we conclude
that $Y$ must commute with $[x_2,x_3]$. A deletion argument and computer verification
 are again used, where we verify that if $X_2'<_{\rho'}Y'\lessdot_{\rho'}X_1'$, then $Y'$ and 
  $X_1'$ commute with respect to $w'$. Iterating, we see that any $Y$ between $[x_1,x_3]$
 and $[x_2,x_3]$ in $\rho$ commutes with $[x_2,x_3]$. Combined, we conclude that 
$P(X)$ can be commuted into a saturated chain, so $P(X)$ is
flippable in $[\rho]$.

  Finally, the same argument is adapted for the analogous result when
$\Rev(\rho)\neq \emptyset$.  For this, we choose $X$ so that $P(X)$ is
antilex-unnested in $\rho$. We then rule out congruence relations and
use deletion to reduce to a corresponding check in $\S_4$.
\end{proof}

\begin{remark}\label{rem:nonnested.does.not.always.imply.flippable}
\cref{lem:flippable_Bn2_Cn3} shows that all lex-unnested and
antilex-unnested packets for $\rho\in \A_w(n,2)$ are flippable in $[\rho]\in \B_w(n,2)$.
However, a packet being unnested is no longer sufficient for being
flippable for $k\geq 3$.  For example, consider $n=5$, $k=3$,
$w=(5,4,3,2,1) \in \S_{5}$, and the admissible order
\[\rho = (123,124,125,134,135,145,234,235,245,345) \in
\A_{w}(5,3)
\]
with the empty reversal set.  Then $P(1245)$ is not flippable in
$[\rho]$ even though it is lex-unnested.  In this case, $P(1234)$ is
lex-unnested and flippable in $[\rho]$.  
\end{remark}

\begin{lemma}\label{cor:affine_min_max}
    The poset $\B_w(n,2)$ is a ranked poset of rank $|\Inv_3(w)|$ with a unique minimal element with reversal set $\emptyset$ and unique maximal element with reversal set $\Inv_3(w)$.
\end{lemma}

\begin{proof}
    Let $[\rho]\in \B_w(n,2)$. \cref{lem:flippable_Bn2_Cn3} implies
that if $\Rev(\rho)\neq \Inv_3(w)$ (resp. $\Rev(\rho)\neq \emptyset$),
then $[\rho]$ cannot be maximal (resp. minimal). Hence, a maximal
element in $\B_w(n,2)$ must have reversal set $\Inv_3(w)$ and a
minimal element must have reversal set $\emptyset$. Uniqueness of
these elements follows from the isomorphism $\B_w(n,2) \cong
\C_w(n,3)$ in \cref{thm:affine_Bn2_Cn3}.  The fact that
$\B_w(n,2)\cong \C_w(n,3)$ are ranked by size follows from the same
argument as in \cref{thm:Cnk_min_max}.
\end{proof}

We next prove the correspondence between maximal chains in $\C_w(n,3)$
and admissible orders in $\A_w(n,3)$. We then conclude with a proof of
\cref{thm:affine}, from which \cref{cor:cGgraph} follows.

\begin{lemma}\label{lem:affine_chains_to_admissible_orders}
 There is a natural bijection between maximal chains of
$\C_w(n,3)$ and the admissible orders in $\A_w(n,3)$.
\end{lemma}

\begin{proof}
We now use the same map $\mathcal{R}$
from \cref{def:consistentchaintoadmissibleseq}.   \cref{lem:affine_admissible_to_consistent_set} shows
 $\mathcal{R}$ is a well-defined injective function. For surjectivity, we first note that
\cref{thm:affine_Bn2_Cn3,cor:affine_min_max}
show that maximal chains in $\C_w(n,3)$ all have the form
\[C=(\emptyset = R_{0}\lessdot_{1} R_1 \lessdot_{1} R_2 \lessdot_{1} \dots \lessdot_{1} R_{|\Inv_{3}(w)|}).\]
Now apply the surjectivity argument from \cref{thm:chains_to_Ank} in the special case of $k=3$, where
references to $\P_w(n,3)$ involve the generalization for affine permutations.
\end{proof}

\begin{proof}[Proof of \cref{thm:affine}] 
Part (a) follows from \cref{lem:An1.weak.order.affine} and \cref{cor:affine_Cwn2_isomorphism}, and the latter also implies part (b). The bijections in (c) are established in \cref{lem:Bwn2_and_graph}, and the correspondence established between braid relations and packet flips implies (d). Part (e) follows from combining \cref{thm:affine_Bn2_Cn3} and \cref{cor:affine_min_max}. Part (f) then follows by combining \cref{lem:Bwn2_and_graph} and \cref{cor:reversals_unique.affine}. Finally, (g) is given in \cref{lem:affine_chains_to_admissible_orders}.
\end{proof}

\begin{proof}[Proof of \cref{cor:cGgraph}]
By \cref{thm:affine}, we know $\cG(w)$ is isomorphic to the Hasse
diagram of $\C_w(n,3)$ as a directed graph, and the isomorphism sends
a commutation class of reduced words to the unique reversal set of the
corresponding reflection orders.  Thus, we will use the Hasse diagram of 
$\C_w(n,3)$ to prove these results about $\cG(w)$. 

For Part (a), \cref{thm:affine} shows that
$\C_w(n,3)$ is a ranked poset with a unique minimal element
 and a unique maximal element. This implies $\cG(w)$ is connected, as 
there must be a path from any vertex to the unique minimal or maximal
 element. Additionally, $\cG(w)$ is bipartite by using the even vs odd rank
 as the partition of the vertices. In general, if $\cG(w)$ is a path, then it is 
a tree. If the Hasse diagram of $\C_w(n,3)$ is a tree that is not a path,
then $\cG(w)$ has a vertex of degree at least three. At least two branches 
would have to merge in order to reach the unique minimal or maximal element
in $\C_w(n,3)$, producing an
undirected cycle. Thus, it must be that if $\cG(w)$ is a tree, then it is a path. 
We conclude Part (a) holds.

For Part (b), observe that the shortest path in the Hasse diagram of
$\C_w(n,3)$ from the minimal element $\emptyset$ to the maximal
element $\Inv_{3}(w)$ has length $|\Inv_{3}(w)|$ since the order
relation is single step inclusion.  Thus, the diameter of $\cG(w)$ is
at least $|\Inv_{3}(w)|$.  On the other hand, if $R,S \in \C_w(n,3)$,
then there is a path in the undirected Hasse diagram $\C_w(n,3)$ from
$R$ to $S$ through $\emptyset$ of length $|R|+|S|$ and a path of
length $2\cdot |\Inv_{3}(w)|-|R|-|S|$ through $\Inv_{3}(w)$.  One of these
two paths must have length less than or equal to
$|\Inv_{3}(w)|$. Thus, the diameter of the undirected Hasse diagram of
$\C_w(n,3)$ is $|\Inv_{3}(w)|$. In particular, the diameter is zero if and only if
$\Inv_3(w)=\emptyset$, which is precisely when $w$ is $321$-avoiding. 
We conclude Part (b) holds.  
\end{proof}

\subsection{Acyclicity of $G_R$ when $R\in \C_w(n,3)$}\label{sec:affineGR}

In this section, we will prove \cref{thm:affine_GR}. Throughout, fix $w\in
\affineS_n$, $ k=2$, and $R\in \C_w(n,3)$. In this case, $G_{R}$ is a graph
on $\Inv_{2}(w)$ that we will show is acyclic. 

When $w\in \S_n$, our approach in \Cref{sub:G_R} was to partition the
vertices of $G_R$ into three subsets $V_0(G_R)$, $V_1(G_R)$, and
$V_2(G_R)$, where $V_1(G_R)$ consisted of inversions containing
$n$. However, those methods and results do not immediately translate
to affine permutations, and examples like the graph in
\cref{fig:m3m287dot.pdf} show that we do not expect direct
generalizations of results like \cref{lem:partition} to hold. Hence,
we will need a different approach when $w\in \affineS_n$ and $R\in
\C_w(n,3)$.  We start by analyzing the congruence edges in any
potential cycle in $G_R$.  By a cycle, we will always mean a directed
cycle in $G_{R}$.

\begin{lemma}\label{lem:shorten1}
    Suppose $X,Y\in \Inv_2(w)$ are not congruent and $X\to Y\to Y+v_1\to \ldots \to Y+mv_1$ is a path in $G_R$ for some $m>0$. Then $X\to Y+mv_1$ is an edge in $G_R$.
\end{lemma}

\begin{proof}
Express $X=[x_1,x_2]$, so $X\to Y$ implies that $Y$ has either the
form $Y=[y,x_2]$ or $Y=[x_1,y]$. Consider when $Y=[y,x_2]$. We will
prove the result based on the type of edge between $X$ and $Y$, which
is not a congruence edge by hypothesis.
\begin{enumerate}[label=(\alph*)]

\item If $X\to Y$ is a full lex edge, then $x_1<y<x_2$ and
$[x_1,y,x_2]\in \Inv_3(w)\setminus R$. If $x_1+mn>y$, then $[x_1,y]\in
\Inv_2(w)$ implies $[y,x_1+mn]\notin \Inv_2(w)$ by
\cref{lem:inversion.sets}. This implies $[y,x_1+mn,x_2+mn]$ is a
quasi-inversion, and the corresponding edge in $G_R$ is $X\to
Y+mv_1=[y,x_{2}+mn]$ as claimed. If $x_1+mn<y$, then $[x_1,y]\in
\Inv_2(w)$ implies $[x_1+mn,y]\in \Inv_2(w)$ by
\cref{lem:inversion.sets}. Since $Y=[y,x_2]\in \Inv_2(w)$ and
$Y+mv_1=[y, x_2+mn]\in \Inv_2(w)$, we conclude
$[x_1,y,x_2+mn],[x_1+mn,y,x_2+mn]\in \Inv_3(w)$ by
\Cref{lem:adjacent.pairs}. We then have the congruence
relations \[[x_1,y,x_2]<_{\P_w(n,3)} [x_1,y,x_2+mn]<_{\P_w(n,3)}
[x_1,y-mn,x_2]=[x_1+mn,y,x_2+mn]. \] As $[x_1,y,x_2]\notin R$,
consistency implies $[x_1+mn,y,x_2+mn]\notin R$, so $X\to Y+mv_1$ is a
full lex edge in $G_R$.

\item If $X\to Y$ is a full antilex edge, then $[y,x_1,x_2]\in R$. We
now consider $[y-mn,x_1,x_2]=[y,x_1+mn,x_2+mn]$. If this is a
quasi-inversion, then the corresponding edge in $G_R$ has the form
$X\to Y+mv_1$. If this is a $3$-inversion, then
$[y,x_1+mn,x_2+mn]<_{\P_w(n,3)} [y,x_1,x_2]$, so $[y,x_1,x_2]\in R$
implies $[y,x_1+mn,x_2+mn]\in R$. We again have the edge $X\to
Y+mv_1$.

\item If $X\to Y$ is a quasi-inversion edge, then this must come from
the quasi-inversion $[y,x_1,x_2]$ and $[y,x_1]\notin \Inv_2(w)$.  The
latter implies $[y-mn,x_1]\notin \Inv_2(w)$ by
\cref{lem:inversion.sets}. Since $Y+mv_1\in \Inv_2(w)$, we conclude
that $[y-mn,x_1,x_2]=[y,x_1+mn,x_2+mn]$ is also a quasi-inversion, and
the corresponding edge in $G_R$ has the form $X\to Y+mv_1$.
\end{enumerate}
The case of $Y=[x_1,y]$ is verified similarly. 
\end{proof}

\begin{corollary}\label{cor:shorten1}
If there exists a cycle $C$ in $G_R$, then there exists a cycle on a subset of the vertices in $C$ that contains no congruence edges. 
\end{corollary}

\begin{proof}
    A cycle in $G_R$ cannot consist of only congruence edges, as this
    would contradict \cref{lem:affine.poset.packet.antichain}. Hence,
    it must contain at least one edge between non-congruent elements
    $X\to Y$. Whenever this is followed by congruence edges $Y\to
    Y+v_1\to \ldots Y+mv_1$, \cref{lem:shorten1} implies that there is
    a noncongruence edge $X\to Y+mv_1$ in $G_R$. Hence, we can iteratively replace sequences of congruence edges with a single non-congruence edge. This process eventually terminates since the number of vertices in the cycle decreases with each application of \cref{lem:shorten1}.
\end{proof}

For any potential cycle $C$ in $G_R$, we will use $|C|$ to denote its
length, which is its number of edges or equivalently its number of
vertices.  We next rule out cycles of short length.

\begin{lemma}\label{lem:no_short_cycles}
    The graph $G_R$ cannot contain a cycle $C$ with $|C|\leq 3$.
\end{lemma}

\begin{proof}
    It is clear that there are no cycles with $|C|=2$, as the
congruence relations, quasi-inversion relations, and full-packet
relations that generate edges in $G_R$ cannot create a 2-cycle by
construction. Now suppose there is a cycle $C$ with $|C|=3$.  Observe
that any two vertices connected by an edge in $G_{R}$ differ in
exactly one position since $k=2$.  Then this either has the form
$[a,x]\to[a,y]\to[a,z]\to [a,x]$ or $[x,a]\to[y,a]\to [z,a]\to
[x,a]$. If two of the elements in $\{x,y,z\}$ are congruent modulo
$n$, then \cref{cor:shorten1} implies that there exists a cycle of
length $2$, which cannot be true. Hence, the four elements
$\{a,x,y,z\}$ are all distinct modulo $n$ by
\cref{lem:inversion.sets}. Then we can view the elements of $C$ as
$2$-subsets of some $Z=[z_1,z_2,z_3,z_4]\in \binom{\Z}{4}_n$. Let
$A=\{z_1,z_2,z_3,z_4\}$, and construct $w'=\Fl_{\bar{A}}(w)\in \S_4$
and $R'=R\setminus \bar{A}$ as in \cref{def:finite_deletion}. Then
\cref{lem:finite_deletion} shows that a cycle $C$ of length $3$ in
$G_R$ implies the existence of a cycle of length $3$ in
$G_{R'}$. However, such a cycle cannot exist in $G_{R'}$ by
\cref{lem:GRacyclic}, as $w'\in \S_4$ and $R'\in \C_{w'}(4,3)$.
\end{proof} 

\cref{cor:shorten1} shows that if there is a cycle in $G_R$, there must be one with no congruence edges. The next results show the existence of cycles with additional properties.

\begin{lemma}\label{lem:shorten2}
    If there exists a cycle $C$ in $G_R$, then there exists a cycle $C'$ in which the following property holds: \[\text{for all $i\in [n]$, all elements of $\Inv_2(w)$ containing $i$ in $C'$ are consecutive.}\]
	Furthermore, when $C$ does not satisfy this property, $C'$ can be chosen so that $|C'|< |C|$. 
\end{lemma}

\begin{proof}
    Suppose $C$ is a cycle in $G_R$ that does not satisfy the property
above. By \cref{lem:shorten1}, we can assume without loss of
generality that there are no congruence edges in $C$. For each $i\in
[n]$, let $C_i$ be the induced subgraph of $C$ on the vertices that
contain a value congruent to $i$ modulo $n$, and define $p_i(C)$ to be
the number of connected components of $C_i$. By assumption, $p_z(C)>1$
for some $z\in [n]$. Then let $X,Y\in C$ be elements containing a
value congruent to $z$ that do not appear in the same path in
$C_z$. Without loss of generality, we will assume that the elements
$X',Y'$ respectively preceding $X,Y$ in $C$ do not contain $z$.

First, suppose that $X$ and $Y$ are congruent. Then there exists a path
from one to the other in $G_R$, which we assume without loss of
generality is $X\to \ldots \to Y$. By \cref{lem:shorten1}, there exists an edge 
$X'\to Y$. Combined with the path from $Y$ to $X'$ in $C$, this
results in a cycle $C'$ where $|C'|<|C|$. Furthermore, we have that
 $p_z(C')<p_z(C)$ and $p_i(C')\leq p_i(C)$ for all $i\neq z$.

Second, suppose that $X$ and $Y$ are not congruent, but
$z$ appears in the same position of $X$ and
$Y$. As $X,Y\in \Inv_2(w)$ only differ in one position, there is a
directed edge between $X$ and $Y$, which we assume
without loss of generality is $X\to Y$. Then we
can combine this with the path from $Y$ to $X$ in $C$ to obtain a
cycle $C'$ with $|C'|<|C|$. We again have that
 $p_z(C')<p_z(C)$ and $p_i(C')\leq p_i(C)$ for all $i\neq z$.

Next, suppose that
$X$ and $Y$ are not congruent and 
that $z$ does not appear in the same position in $X$
and $Y$. Without loss of generality, we will assume $X=[x,z]$ and
$Y=[z,y]$. Then $[x,z,y]\in \Inv_3(w)$, so there is a directed path in
$G_{R}$ between $X$ and $Y$ oriented according to whether or not
$[x,z,y]$ is an element of $R$. We will detail the case where $[x,z,y]\notin R$, 
as $[x,z,y]\in R$ is similar. In the case $[x,z,y]\notin R$, the path 
between $X$ and $Y$ has the form $X=[x,z]\to [x,y]\to
[z,y]=Y$.  

If $[x,y]$ already lies on the
path from $Y$ to $X$ in $C$, then there must be some path in $C$ from $Y$ to $[x,y]$. 
Then we can form a cycle $C'$ using the path from $Y$ to $[x,y]$ along with the edge $[x,y]\to Y$. 
This cycle omits a path of elements containing $X$, so we know $|C'|<|C|$ and 
$p_z(C')<p_z(C)$. Furthermore, its construction comes from 
removing a portion of $C$, so it is clear that $p_i(C')\leq p_i(C)$ for
all $i\in [n]$.

We may therefore assume that $[x,y]$ does not already appear on the path from $Y$ to $X$ in $C$. 
Combining $X=[x,z]\to [x,y]\to
[z,y]=Y$ with the path from $Y$ to $X$ in $C$
produces a cycle $C'$ with $|C'|\leq |C|$. Furthermore, we have that $p_i(C')\leq p_i(C)$ for
all $i\in [n]$. Now consider the element
$X'$ preceding $X$ in both $C$ and $C'$.  As $C$ contains no
congruence edges and $X'$ does not contain $z$, $X'$ has the form
$X'=[x,x']$. We consider the relationship between this element and $[x,y]$. 
\begin{itemize}
\item Suppose $X'=[x,x']$ and $[x,y]$ are congruent. Then,
they are comparable in $\P_{w}(n,2)$, so
there is some path between 
these two elements in $G_R$. If this path of congruence edges
 has the form $[x,y]\to \ldots \to X'$, then there is a cycle of the form
$X\to [x,y]\to \ldots \to X'\to X$. \cref{lem:shorten1} then implies that 
$X\to X'$ is an edge, resulting in a cycle of length 2 that contradicts 
\cref{lem:no_short_cycles}. We conclude the path of congruence edges must have the form
$X'\to \ldots \to [x,y]$. Letting $X''$ be the element preceding $X'$ in $C$, \cref{lem:shorten1}
shows that the path $X''\to X'\to X\to [x,y]\to  Y$ in $C'$ can be replaced with $X''\to [x,y]\to Y$.
\item Suppose $X'=[x,x']$ and $[x,y]$ are not congruent.
Then, there exists a $3$-inversion or $3$-quasi-inversion with both in
its packet,
so there is an edge between them
 in $G_R$.  Since $X'\to X\to [x,y]$ is a path, 
\cref{lem:no_short_cycles} implies that this edge must have direction $X'\to
[x,y]$. Hence, we can replace $[x,x']\to[x,z]\to[x,y]$ in $C'$ with $[x,x']\to
[x,y]$.
\end{itemize}
 In both cases, this results in a cycle $C''$ where $|C''|<|C'|\leq |C|$. Furthermore,
 we have that $p_z(C'')<p_z(C')\leq p_z(C)$ and $p_i(C'')\leq p_i(C')\leq 
p_i(C)$ for all $i\neq z$. 
    
    Combined, we see that when $p_z(C)>1$, there is a cycle $D$ in
$G_{R}$ where $|D|<|C|$, $p_z(D)<p_z(C)$, and $p_i(D)\leq p_i(C)$ for
all $i\neq z$. We can now iterate this argument as needed until $D$
satisfies $p_i(D)\leq 1$ for all $i\in [n]$, so all elements of
$\Inv_2(w)$ in $D$ that contain each $i\in [n]$ are consecutive.
\end{proof}

\begin{lemma}\label{lem:shorten3}
    Suppose there exists a cycle $C$ in $G_R$ that does not contain
    congruence edges and $|C|\geq 4$. Suppose $C$ contains a path
    $X\to Y\to Z$ where all three elements contain a value congruent to  $i\in [n]$. Then there exists a cycle $C'$ in $G_R$ using a proper subset of the elements in $C$.
\end{lemma}

\begin{proof}
    For the edges $X\to Y$ and $Y\to Z$ to exist, each pair must contain $i$ in the same position, so $X$ and $Z$ contain $i$ in the same position. If $X$ and $Z$ are not congruent, then there exists an edge between $X$ and $Z$. By \cref{lem:no_short_cycles}, this edge must be $X\to Z$, so replacing $X\to Y\to Z$ in $C$ with $X\to Z$ results in a shorter cycle $C'$. If $X$ and $Z$ are congruent, then there is some path of congruence edges between them. If this path is $Z\to \ldots \to X$, then \cref{lem:shorten1} implies that $Y\to X$ is an edge in $G_R$, which results in a cycle of length $2$ that contradicts \cref{lem:no_short_cycles}. Consequently, the path must be $X\to \ldots \to Z$. Let $W\to X$ be the edge preceding $X\to Y\to Z$ in $C$. Observe that $W$ and $X$ cannot be congruent since $C$ contains no congruence edges, so $W$ and $Z$ also are not congruent. Using \cref{lem:shorten1}, the path $W\to X\to Y\to Z$ in $C$ can be replaced with $W\to Z$, again resulting in a shorter cycle $C'$.
\end{proof}

\begin{corollary}\label{lem:shorten4}
    Suppose there exists a cycle $C$ in $G_R$ where $|C|\geq 4$. Then
    there exists a cycle $C'$ with $|C'|\leq |C|$ satisfying the
    following properties.

    \begin{enumerate}[label=(\alph*)]
        \item $C'$ contains no congruence edges. 
        \item For all $i\in [n]$, either $C'$ does not contain any elements with $i$, or the elements containing $i$ consist of two consecutive elements in $C'$.
        \item As we progress along $C'$, the position that changes
	value alternates. In particular, the number of elements in
	$C'$ must be even.
     \item For any path on $1\leq \ell \leq |C'|-1$ vertices contained
     in $C'$, the entries that appear in these vertices, each
     represented by some $[x,y] \in \binom{\Z }{2}_{n}$, form $\ell+1$ distinct congruence classes modulo $n$. 
    \end{enumerate}
    Furthermore, if $C$ itself did not satisfy the above properties, then $|C'|<|C|$. 
\end{corollary}

\begin{proof}
As $G_R$ contains some cycle $C$, we consider a minimal length cycle
$C'$. If $C'$ does not satisfy (a), then applying 
\cref{lem:shorten1} to some congruence edge in $C'$ results in a shorter cycle, which would
contradict that $C'$ has minimal length. If $C'$ does not
 satisfy (b), then there are several cases.
\begin{itemize}
	\item If $C'$ contains some $i\in [n]$ that does not appear in a single path, then 
\cref{lem:shorten2} implies that there is a shorter cycle, which would 
contradict that $C'$
has minimal length. 
	\item If $C'$ contains some $i\in [n]$ in three or more consecutive inversions, then
\cref{lem:shorten3} implies there is a shorter cycle, resulting in the same contradiction. 
	\item If $C'$ contains $i\in [n]$ in only one inversion,
then the element $j\in [n]$ that appears with $i$ must remain
the same for three consecutive elements in $C'$, so \cref{lem:shorten3} again supplies the contradiction.
\end{itemize} 
Hence, $C'$ must satisfy properties (a) and (b), and from these
arguments, we also see that if $C$ does not satisfy these two
properties, then $|C'|<|C|$. Furthermore, since each edge of $C'$
changes exactly one position, two consecutive edges changing the same
position would leave the other residue class unchanged in three
consecutive vertices, contradicting (b). Thus the position that
changes alternates, and $|C'|$ is even, which proves Property (c).

For (d), assume $C'$ is a cycle in $G_{R}$ with $m=|C'|\geq 4$
vertices satisfying Properties (a), (b), and (c).  Consider a path with
 $\ell<m$ vertices in $C'$ starting at $X^{(1)}$.  Index all of the vertices in the
cycle as
\[X^{(1)}\to X^{(2)}\to \ldots \to X^{(m)}\to X^{(1)}.\]
If the path is the single vertex $X^{(1)}=[x_{1},y_{1}]\in \Inv_2(w)$,
then $x_{1}$ and $y_{1}$ cannot be congruent modulo $n$ by
\cref{lem:inversion.sets}(c).  By induction on $\ell$, assume entries
in the path $X^{(1)}\to \ldots X^{(\ell)}$ in $C'$ form $\ell+1$
distinct congruence classes, where $1\leq \ell\leq |C'|-2$. By (a), we
know that $X^{(\ell)}\to X^{(\ell+1)}$ cannot be a congruence edge. So
these vertices can be represented by elements in $\binom{\Z }{2}_{n}$
so that they share one entry $x$, and $X^{(\ell+1)}$ contains another
entry $y$ that is not equivalent modulo $n$ to either entry in
$X^{(\ell)}$.  By (b), only $X^{(\ell+1)}$ and $X^{(\ell+2)}$ contain
an entry congruent to $y$ modulo $n$.  If $\ell+2\leq |C'|$, we know
that $X^{(\ell+2)}\neq X^{(1)}$.  Consequently, no elements in
$\{X^{(1)},\ldots,X^{(\ell)}\}$ contain an entry congruent to $y$
modulo $n$.  We conclude that the entries among the subsets 
$\{X^{(1)},\ldots,X^{(\ell)},X^{(\ell+1)}\}$ represent $\ell+2$ distinct
congruence classes modulo $n$.
\end{proof}

It now suffices to show that there cannot exist cycles that satisfy
the properties in \cref{lem:shorten4}. Our approach will involve a
reduction to checking cycles of short length and then directly ruling
out such cycles.  Define the \defn{length of a path} in a directed
graph $G$ to be the number of edges. We say that a vertex $v$
 is \defn{reachable} from vertex $u$ if a directed path from $u$ to $v$ exists 
in $G$, allowing a path of length $0$ when $u=v$. In this case, the \defn{distance} 
from $u$ to $v$ is the minimum length over all directed paths from
$u$ to $v$. For a nonempty finite directed graph $G$,
its \defn{reachability diameter} is the maximum of these distances
over all reachable pairs. We define the reachability diameter of
the empty graph to be $0$.

\begin{lemma}\label{lem:shorten5}
    For any $w\in \S_n$ and $R\in \C_{w}(n,3)$, if there exists a path
of length $4$ in $G_{R}$ between $X,Y\in \Inv_2(w)$, then there exists
a path in $G_{R}$ between $X$ and $Y$ with length strictly less than
$4$.  Hence, the reachability diameter of $G_{R}$ is
at most 3 as a graph on $\Inv_{2}(w)$.
\end{lemma}

We note that the global bound of reachability diameter at most 3 is optimal.  For example,
if $w=(3,2,5,4,1)$, $R=\{[1,2,3] \}$, the shortest path from 
 $[2,3]$ to $[4,5]$ is $[2,3]\textcolor{red}{\to } [1,3]\to [1,5]
\textcolor{blue}{\to} [4,5]$.

\begin{proof}
For $n\leq 6$, we have computer verified the claim using breadth-first search
 on every connected component of $G_{R}$ for every $w\in
\S_n$ and $R\in \C_{w}(n,3)$.  This is a finite verification since we
only consider permutations in $\S_n$, not affine permutations.  See
appendix for verification details.

If $n>6$, let $P$ be a path of length 4 or more in $G_{R}$.  Each
vertex of $P$ is of the form $[x,y] \in \binom{[n] }{2}$. Each edge
changes exactly one value. Therefore, the first 5 vertices in $P$
contain at most 6 distinct values in $[n]$.  Let $A$ be the set of
values appearing in the first 5 vertices of $P$.  By
\cref{lem:finite_deletion}, $R'=R\setminus \bar{A}$ is consistent for
$w'=\Fl_{\bar{A}}(w) \in S_{m}$ with $m\leq 6$, and $G_{R'}$ is
isomorphic to the induced subgraph of $G_{R}$ on vertices compatible
with $A$.  Hence, by the computer verification, there exists a
strictly shorter path $P'$ in $G_{R'}$ that lifts to a shorter path
from the first to the fifth element of $P$ under the graph
isomorphism.  Applying this reduction recursively, one can reduce the
length of any path in $G_{R}$ to a path of length at most 3 from the
first to last elements.
\end{proof}

\begin{lemma}\label{lem:shorten6}
    If there exists a cycle $C$ in $G_R$, then there exists a cycle $C'$ where $|C'|= 4$.
\end{lemma}

\begin{proof}
    We can assume without loss of generality that $C$ is a minimal
length cycle in the induced subgraph of $G_{R}$ on vertices in $C$ and
it satisfies the properties in \cref{lem:shorten4}. Property (c)
implies that $C$ must contain an even number of elements, so it
suffices to consider when $|C|\geq 6$. Then $C$ contains a path
$X^{(1)}\to X^{(2)}\to \ldots \to X^{(5)} \to X^{(6)}$ in $G_R$. We
can assume that all six vertices are distinct, since $C$ is a minimal
length cycle in the induced subgraph.  \cref{lem:shorten4}(d) implies that
the entries of $X^{(1)}\to X^{(2)}\to \ldots \to X^{(5)}$ must form six
distinct congruence classes modulo $n$.  Because each distinct
congruence class appears in at most 2 consecutive vertices of this
path, we can view $X^{(1)}, X^{(2)}, \ldots ,X^{(5)}$ as $2$-subsets
of the canonical representative of some $Z=[z_1,z_2,\ldots,z_6]\in
\binom{\mathbb{Z}}{6}_n$. Letting $A=\{z_1,z_2,\ldots,z_6\}$,
construct $w'=\Fl_{\bar{A}}(w)\in \S_6$ and $R'=R\setminus \bar{A}$ as
in \cref{def:finite_deletion}. By \cref{lem:finite_deletion}, $R' \in
\C_{w'}(6,3)$, and the induced subgraph of $G_R$ on elements compatible
with $A$ is isomorphic to $G_{R'}$ on $\Inv_{2}(w')$. By \cref{lem:shorten5}, there
exists a path of length at most 3 in $G_{R'}$ from
$\Fl_{\bar{A}}(X^{(1)})$ to $\Fl_{\bar{A}}(X^{(5)})$, and this shorter
path lifts to a shorter path in $G_R$ from $X^{(1)}$ to
$X^{(5)}$ using the graph isomorphism. Iterating this argument as needed, we see that if $|C|\geq
6$, then there exists a shorter cycle $C'$ in $G_{R}$ with $|C'|\leq
5$.  Any cycle with exactly five elements does not satisfy the
properties in \cref{lem:shorten4}. Hence, when this occurs, we can
apply \cref{lem:shorten4} again to obtain a cycle $C''$ where
$|C''|\leq 4$. From \cref{lem:no_short_cycles}, we conclude there
exists a $C''$ in $G_{R}$ such that $|C''|=4$.
\end{proof}

\begin{proof}[Proof of \cref{thm:affine_GR}]
    \cref{lem:not_consistent_cycle} shows that if $R\subseteq
    \Inv_3(w)$ is not consistent, then $G_R$ contains a cycle. It
    remains to show that if $R\in \C_w(n,3)$, then $G_R$ is
    acyclic. By \cref{lem:shorten6}, it suffices to show that there
    are no cycles $C$ with $|C|=4$. Additionally, it suffices to
    consider cycles satisfying the properties in
    \cref{lem:shorten4}. By \cref{lem:empty_full_acyclic}, $C$
 cannot consist of only antilex and permanent poset edges,
as otherwise, this cycle would also appear in $G_{\Inv_3(w)}$. Similarly,
$C$ cannot consist of only lex edges, as this cycle would also appear in $G_{\emptyset}$. 
 Hence, there must be some lex edge in $C$ that is
    followed by a non-lex edge, which must be an antilex or
    quasi-inversion edge, as \cref{lem:shorten4} allows us to assume
    there are no congruence edges in $C$.   A contradiction is found
    in each case below.  

\bigskip 

\noindent \textbf{Case 1: The lex edge changes the second position and is followed by an antilex edge.}  By \cref{lem:shorten4}(b) and (c), the cycle must then have the form
    \begin{equation}\label{eq:lex_then_antilex}
        [x_2,x_3]\to [x_2,x_4]\to [x_1,x_4]\to [x_1,x_3']\to [x_2,x_3],
    \end{equation}
    where $x_1<x_2<x_3<x_4$ and $x_3'\equiv x_3\pmod n$. In this case,
    we see that $[x_2,x_3,x_4],[x_1,x_2,x_4]\in \Inv_3(w)$, so
    $[x_1,x_2,x_3,x_4]\in \Inv_4(w)$. Note that we can assume
    $x_3'\neq x_3$, because if $x_3=x_3'$, we can apply the same deletion argument from \cref{lem:no_short_cycles} to show that there is a cycle in $G_{R'}$ for some $w'\in \S_4$ and $R'\in \C_{w'}(4,3)$, which would contradict \cref{lem:GRacyclic}.
    \begin{enumerate}[label=(\alph*)]
        \item Consider when $x_3'=x_3+mn$ for some $m>0$. Consistency of $R$ with $[x_2,x_3,x_4]\notin R$ and $[x_1,x_2,x_4]\in R$ implies $[x_1,x_2,x_3]\in R$. However, this implies that there is a cycle $[x_2,x_3]\to [x_1,x_3] \to \ldots \to [x_1,x_3+mn]\to [x_2,x_3]$ in $G_R$. \cref{lem:shorten1} implies that $[x_2,x_3]\to [x_1,x_3+mn]$ is an edge in $G_R$, resulting in a cycle of length $2$ in $G_R$, which is not possible by \cref{lem:no_short_cycles}.
        \item Consider when $x_3'=x_3-mn$ for some $m>0$. Since $[x_1,x_3-mn],[x_3,x_4]\in \Inv_2(w)$, we know that $x_1<x_3-mn<x_4$. As $[x_1,x_4]\to [x_1,x_3-mn]$ cannot be a quasi-inversion edge, it must be an antilex edge, so $[x_1,x_3-mn,x_4]=[x_1+mn,x_3,x_4+mn]\in R$. We now consider two subcases based on the order of $x_2$ and $x_1+mn$.
    \begin{itemize}
        \item If $x_2<x_1+mn$, then $[x_2,x_1+mn,x_3,x_4+mn]\notin \Inv_4(w)$, as $[x_1,x_2]\in \Inv_2(w)$ implies $[x_2,x_1+mn]\notin \Inv_2(w)$. From $[x_2,x_3,x_4],[x_1+mn,x_3,x_4+mn]\in \Inv_3(w)$ we conclude  $[x_2,x_1+mn,x_3,x_4+mn]$ is a quasi-inversion, and the parity condition implies $[x_2,x_3,x_4+mn]<_{\P_w(n,3)} [x_{1}+mn,x_3,x_4+mn]$. However, $R$ is consistent, so $[x_1+mn,x_3,x_4+mn]\in R$ implies that $[x_2,x_3,x_4+mn]\in R$. Combined with $[x_2,x_3,x_4]\notin R$, this contradicts consistency of $R$. 
        \item If $x_2>x_1+mn$, then $[x_1,x_2]\in \Inv_2(w)$ implies $[x_1+mn,x_2]=[x_1,x_2-mn]\in \Inv_2(w)$ by \cref{lem:inversion.sets}. Combined with $[x_1,x_2,x_3,x_4]\in \Inv_4(w)$ and $[x_3,x_4+mn]\in \Inv_2(w)$, we conclude that $[x_1+mn,x_2,x_3,x_4+mn]\in \Inv_4(w)$. The edge $[x_2,x_3]\to [x_2,x_4]$ implies $[x_2,x_3,x_4]\notin R$, which in turn implies $[x_2,x_3,x_4+mn]\notin R$. The edge $[x_1,x_3-mn]\to [x_2,x_3]$ implies that $[x_1+mn,x_2,x_3]\notin R$. However, this combined with $[x_1+mn,x_3,x_4+mn]\in R$ contradicts the MSZ condition on $P([x_1+mn,x_2,x_3,x_4+mn])$.
    \end{itemize}
    \end{enumerate}

\noindent \textbf{Case 2: The lex edge changes  the second
position and is followed by a quasi-inversion edge.} As in Case 1, the
cycle must have the form described in \cref{eq:lex_then_antilex}.  In
particular, the hypotheses in this case imply that $[x_2,x_3,x_4]\in
\Inv_3(w) \setminus R$ and $[x_1,x_2]\notin \Inv_2(w)$. 
    \begin{enumerate}[label=(\alph*)]
        \item Consider when $x_3'=x_3+mn$ for some $m>0$. Since
	$[x_1,x_2]\notin \Inv_2(w)$, we have that $[x_1-mn,x_2]\notin
	\Inv_2(w)$ by \cref{lem:inversion.sets}, so $[x_1-mn,x_2,x_3]$
	must be a quasi-inversion. However, the direction of
	$[x_1-mn,x_3]=[x_1,x_3+mn]\to [x_2,x_3]$ contradicts the
	corresponding quasi-inversion relation from the parity
	condition since $x_{1}<x_{2}$. 
        \item Consider when $x_3'=x_3-mn$ for some $m>0$. We consider cases based on the relative order of $x_1+mn$ and $x_2$. 
    \begin{itemize}
        \item Suppose $x_1+mn<x_2$. Then the edge
	$[x_1+mn,x_3]=[x_1,x_3-mn]\to [x_2,x_3]$ implies that
	$[x_1+mn,x_2,x_3]$ cannot be a quasi-inversion. Hence,
	$[x_1+mn,x_2,x_3]\in \Inv_3(w)$, and since $[x_2,x_3,x_4]\in
	\Inv_3(w)$, we see that $[x_1+mn,x_2,x_3,x_4]\in
	\Inv_4(w)$. The known edges imply
	$[x_1+mn,x_2,x_3],[x_2,x_3,x_4]\in \Inv_3(w)\setminus R$, so
	the MSZ condition implies $[x_1+mn,x_3,x_4]\in
	\Inv_3(w)\setminus R$. However, this implies that there is a
	cycle $[x_1+mn,x_4]\to \ldots  \to [x_1,x_4]\to [x_1,x_3-mn]\to
	[x_1+mn,x_4]$. \cref{lem:shorten1} implies that $[x_1,x_3-mn]\to [x_1,x_4]$ is an edge
	in $G_R$, resulting in a cycle of length $2$ in $G_R$, which cannot exist
	by \cref{lem:no_short_cycles}.
	
        \item Suppose $x_1+mn>x_2$ and $[x_2,x_1+mn]\in\Inv_2(w)$. As $[x_1+mn,x_3]=[x_1,x_3-mn]\in \Inv_2(w)$ and $[x_2,x_3,x_4]\in \Inv_3(w)$, we conclude that $[x_2,x_1+mn,x_3,x_4]\in \Inv_4(w)$. The edge $[x_1+mn,x_3]=[x_1,x_3-mn]\to [x_2,x_3]$ implies that $[x_2,x_1+mn,x_3]\in R$, and combined with $[x_2,x_3,x_4]\in \Inv_3(w)\setminus  R$, the MSZ condition implies that $[x_1+mn,x_3,x_4]\in \Inv_3(w)\setminus R$. We then obtain the same cycle as in the preceding case.
        \item Suppose $x_1+mn>x_2$ and $[x_2,x_1+mn]\notin \Inv_2(w)$. In this case, $[x_1+mn,x_3]=[x_1,x_3-mn]\in \Inv_2(w)$ with $[x_2,x_3,x_4]\in \Inv_3(w)$ imply that $[x_1+mn,x_3,x_4]\in \Inv_3(w)$. Since $[x_1,x_4]\to [x_1,x_3-mn]$ cannot be a quasi-inversion edge, we also find that $[x_1,x_3-mn,x_4]\in R$. As $[x_1+mn,x_3,x_4]=[x_1,x_3-mn,x_4-mn]<_{\P_w(n,3)} [x_1,x_3-mn,x_4]$, we conclude $[x_1+mn,x_3,x_4]\in R$. Furthermore, $[x_2,x_1+mn,x_4]$ must be a quasi-inversion with relation $[x_1+mn,x_4]<_{\P_w(n,2)}[x_2,x_4]$, and combined with $[x_2,x_3,x_4]\in \Inv_3(w)\setminus R$, we find a cycle $[x_2,x_4]\to [x_3,x_4]\to [x_1+mn,x_4]\to [x_2,x_4]$ of length $3$, which contradicts \cref{lem:no_short_cycles}.
    \end{itemize}
    \end{enumerate}

\noindent \textbf{Case 3: The lex edge changes  the first
position and is followed by an antilex edge.} The cycle has the
form
\begin{equation}\label{eq:lex.antilex.edge}
[x_1,x_4]\to [x_2,x_4]\to [x_2,x_3]\to [x_1',x_3]\to
[x_1,x_4].
\end{equation}
where $x_1'\equiv x_1\pmod n$ and $x_1'\neq x_1$. Then
$[x_1,x_2,x_4],[x_2,x_3,x_4]\in \Inv_3(w)$ imply $[x_1,x_2,x_3,x_4]\in
\Inv_4(w)$. Furthermore, as $[x_1,x_2,x_4]\notin R$ and
$[x_2,x_3,x_4]\in R$, we have that $[x_1,x_2,x_3]\notin R$ by
consistency of $R$.
    \begin{enumerate}[label=(\alph*)]
        \item Suppose $x_1'=x_1+mn$ for some $m>0$.  We consider cases based on the relative order of $x_1+mn$ and $x_2$. 
        \begin{itemize}
            \item If $x_1+mn<x_2$, then $[x_1,x_2]\in \Inv_2(w)$
	    implies $[x_1+mn,x_2]\in \Inv_2(w)$ by
	    \cref{lem:inversion.sets}, which in turn implies
	    $[x_1+mn,x_2,x_3,x_4]\in \Inv_4(w)$. However, the edges in
	    \eqref{eq:lex.antilex.edge} indicate $[x_1+mn,x_2,x_3]\in
	    R$, and combined with $[x_1,x_2,x_3]\notin R$, this
	    contradicts consistency of $R$.

            \item If $x_1+mn>x_2$, then $[x_1,x_2]\in \Inv_2(w)$ implies $[x_2,x_1+mn]\notin \Inv_2(w)$. This then implies $[x_2,x_1+mn,x_3]$ is a quasi-inversion, which contradicts the direction of $[x_2,x_3]\to [x_1+mn,x_3]$.
        \end{itemize}
        \item Suppose $x_1'=x_1-mn$ for some $m>0$. In this case,
	$[x_1,x_3,x_4]\in \Inv_3(w)$ and $[x_1-mn,x_3]\in \Inv_2(w)$
	imply $[x_1-mn,x_3,x_4]\in \Inv_3(w)$. Notice that if
	$[x_1,x_3,x_4]\in R$, then there exists a cycle of the form
	$[x_1,x_4]\to [x_1,x_3]\to \ldots  \to [x_1-mn,x_3]\to [x_1,x_4]$. Then 
	\cref{lem:shorten1} implies $[x_1,x_4]\to [x_1-mn,x_3]$ is an edge, resulting in a 
	cycle of length $2$ that 
	would contradict \cref{lem:no_short_cycles}. Hence, it
	suffices to consider when $[x_1,x_3,x_4]\in \Inv_3(w)\setminus
	R$, and in this case, we conclude $[x_1-mn,x_3,x_4]\in
	\Inv_3(w)\setminus R$ since $R$ is a lower order ideal of
	$\P_w(n,3)$. We now consider subcases for $[x_1-mn,x_2,x_4]$,
	where $[x_{2},x_{4}],[x_{1}-mn,x_{4}] \in \Inv_{2}(w)$.
        \begin{itemize}
            \item If $[x_1-mn,x_2,x_4]$ is a quasi-inversion or a
	    $3$-inversion in $R$, then combined with $[x_2,x_3,x_4]\in
	    R$ and $[x_1-mn,x_3,x_4]\in\Inv_3(w)\setminus R$,
	    we find there exists a cycle of the form
	    $[x_1-mn,x_4]\to [x_3,x_4]\to [x_2,x_4] \to [x_1-mn,x_4]$, which again would contradict \cref{lem:no_short_cycles}.

           \item If $[x_1-mn,x_2,x_4]\in \Inv_3(w)\setminus R$, then
	    we also have that $[x_1-mn,x_2,x_3,x_4]\in
	    \Inv_4(w)$. However, the edge $[x_2,x_3]\to [x_1-mn,x_3]$
	    implies $[x_1-mn,x_2,x_3]\in R$.  By hypothesis,
	    $[x_1-mn,x_2,x_4]\in \Inv_3(w)\setminus R$ and
	    $[x_2,x_3,x_4]\in R$. Combined, these contradict the MSZ
	    condition on $R$.
        \end{itemize}
    \end{enumerate}

\noindent \textbf{Case 4: The lex edge changes the first position and
is followed by a quasi-inversion edge.} Then the cycle can be
expressed as
\begin{equation}\label{eq:lex.quasi.edge}
[x_1,x_3]\to [x_2,x_3]\to [x_2,x_4]\to [x_1',x_4]\to [x_1,x_3],
\end{equation}
where $x_1'\equiv x_1\pmod n$ and $x_1'\neq x_1$. In particular, we see that $[x_1,x_2,x_3]\in \Inv_3(w)\setminus R$ and $[x_3,x_4]\in \binom{\mathbb{Z}}{2}_n\setminus \Inv_2(w)$ where $x_3<x_4$. 
    \begin{enumerate}[label=(\alph*)]
        \item Suppose $x_1'=x_1+mn$ for some $m>0$. We consider cases based on the relative order of $x_1+mn$ and $x_2$.
        \begin{itemize}
            \item Suppose $x_1+mn<x_2$. Then $[x_1,x_2]\in \Inv_2(w)$
	    implies $[x_1+mn,x_2]\in \Inv_2(w)$, so the edge
$[x_2,x_4]\to [x_1+mn,x_4]$ comes from $[x_1+mn,x_2,x_4]\in R$. However, we analogously use $[x_1+mn,x_2],[x_2,x_3]\in \Inv_2(w)$ to conclude $[x_1+mn,x_2,x_3]\in \Inv_3(w)$. As $[x_1+mn,x_2,x_3,x_4]$ cannot be a $4$-inversion due to $[x_3,x_4]\notin \Inv_2(w)$, it must be a quasi-inversion, so $[x_1+mn,x_2,x_4]\in R$ implies $[x_1+mn,x_2,x_3]\in R$. However, we originally had that $[x_1,x_2,x_3]\in \Inv_3(w)\setminus R$, and since $[x_1,x_2,x_3]<_{\P_w(n,3)}[x_1,x_2-mn,x_3-mn]=[x_1+mn,x_2,x_3]$, this contradicts consistency of $R$. 
            \item Suppose $x_1+mn>x_2$. Then $[x_2,x_4]\to [x_1+mn,x_4]$ cannot be a quasi-inversion edge, implying $[x_2,x_1+mn,x_4]\in \Inv_3(w)$. However, this implies $[x_2,x_1+mn],[x_1,x_2]\in \Inv_2(w)$, which is a contradiction. 
        \end{itemize}
        \item Suppose $x_1'=x_1-mn$ for some $m>0$. As $[x_3,x_4]\notin \Inv_2(w)$, we know $[x_3,x_4+mn]\notin \Inv_2(w)$. Then $[x_1,x_3,x_4+mn]$ is a quasi-inversion, but this contradicts the direction of $[x_1-mn,x_4]\to [x_1,x_3]$. \qedhere 
    \end{enumerate}
\end{proof}

\subsection{Characterizing Totally Ordered $\C_w(n,k)$ for Small $k$}

In \cite[Theorem 6.12]{tenner}, Tenner characterized the permutations
$w\in \S_n$ such that the commutation class graph $\cG(w)$ as an
undirected graph is a path. These are the permutations that are
$4321$-avoiding where every pair of 3-inversions has two values in
common. \cref{thm:affine} shows that $\cG(w)$ being a path also
corresponds to when $\B_w(n,2)\cong \C_w(n,3)$ is a total order. In
this case, $\P_w(n,3)$ must also be a total order.

In this section, we will focus on characterizing when $\C_w(n,k)$ is a total order for small $k$. In the process, we will extend and strengthen Tenner's path graph characterization to $w\in \affineS_n$. We begin by showing $\C_w(n,k)$ is a total order if and only if $\P_w(n,k)$ is a total order. 

\begin{lemma}\label{lem:nonempty_notchain}
    Let $k\geq 2$. If $\Inv_{k+1}(w)\neq \emptyset$, then $\P_w(n,k)$ and $\C_w(n,k)$ are not total orders. 
\end{lemma}

\begin{proof}
    Suppose $k\geq 2$ and $\Inv_{k+1}(w)\neq \emptyset$. Then
    \cref{lem:affine.poset.packet.antichain} implies that $P(X)$ for
    any $X\in \Inv_{k+1}(w)$ is an antichain in $\P_w(n,k)$, so
    $\P_w(n,k)$ is not a total order since $k\geq 2$. Additionally,
    \cref{cor:empty_full_acyclic} implies that there exist
    $\rho,\sigma\in \A_w(n,k)$ such that $\Rev(\rho)=\emptyset$ and
    $\Rev(\sigma)=\Inv_{k+1}(w)$. As $\Inv_{k+1}(w)\neq \emptyset$,
    these admissible orders are distinct. Let $i$ be the first index
    where $\rho_i\neq \sigma_i$. By
    \cref{lem:affine_admissible_to_consistent_set},
    $R=\{\rho_1,\ldots,\rho_i\}$ and $R'= 
    \{\sigma_1,\ldots,\sigma_i\}$ are elements in $\C_w(n,k)$. As $R$
    and $R'$ differ by interchanging $\rho_i$ and $\sigma_i$, neither
    is a subset of the other. Consequently, these elements cannot be
    related by the single step inclusion order defining $\C_w(n,k)$. Hence, $\C_w(n,k)$ is also not a total order. 
\end{proof}

\begin{corollary}\label{lem:4321_totalorder}
    Let $k\geq 2$. Then $\C_w(n,k)$ is a total order if and only if $\P_w(n,k)$ is a total order. 
\end{corollary}

\begin{proof}
By \cref{lem:nonempty_notchain}, either hypothesis requires that $\Inv_{k+1}(w)=\emptyset$. When this occurs, the MSZ condition for subsets of $\Inv_k(w)$ is trivial. Consequently, $\C_w(n,k)$ reduces to the lattice of lower-order ideals of $\P_w(n,k)$. Hence, each is a total order if and only if the other is a total order. 
\end{proof}

We now consider the case of $k=2$ in more detail. \cref{lem:4321_totalorder} and results from \cref{sub:secondAHB} allow us to obtain the following characterization. 

\begin{proposition}\label{thm:k=2_chain}
The following are equivalent for $w \in \affineS_{n}$:
\begin{enumerate}[label=(\alph*)]
\item $\P_w(n,2)$ is a total order,
\item $\C_{w}(n,2)$ is a total order,
\item there exists a unique element in $\A_{w}(n,2)$, 
\item $w$ has exactly one reduced word, and
\item $\Inv_3(w)=\emptyset$ and every 2-inversion can be expressed with a common value in some fixed position.
\end{enumerate}
\end{proposition}

\begin{proof}
    The equivalence of (a) and (b) follows from
\cref{lem:4321_totalorder}. The equivalence of (b), (c), and a unique
reduced word follows from the bijections in
\cref{lem:Bwn2_and_graph}. Note that $w$ has only one reduced word if
and only if that reduced word can be written in the form
${i}(i+1)\cdots (i+j)$ or the reverse using the relation
$s_{i}=s_{i+n}$.  Any nonempty reduced word that cannot be written in
such a form must have either a commutation or braid relation, so a
unique reduced word can only occur in the two cases listed
above. Hence, we have the equivalence of (a), (b), (c), and (d). To
conclude, we will show that (e) implies (a), and (c) implies (e).
    
    Suppose (e) holds. If every element of $\Inv_2(w)$ can be
    expressed in the form $[a,x]$ for fixed $a\in \mathbb{Z}$, then
    all of these are comparable in $\P_w(n,2)$ since $\Inv_{3}(w)=\emptyset$. The same is true when every element of $\Inv_2(w)$ can be expressed in the form $[x,a]$ for fixed $a\in \mathbb{Z}$. From this, we see that (e) implies (a).

    Finally, assume (c) holds, and let $\rho\in \A_w(n,2)$ be this unique admissible order. By \cref{cor:empty_full_acyclic}, the uniqueness requires that $\Inv_3(w)=\emptyset$, so $\rho$ is the unique linear extension of $(\Inv_{2}(w),\leq_{\emptyset})=\P_w(n,2)$. If $\rho$ contains consecutive elements of the form
    \begin{equation}\label{eq:k=2_chain}
        [b,c]<_{\rho} [a,c] <_{\rho} [a,d],
    \end{equation}
    then this can only occur when $a<b<c<d$ and
    $w^{-1}(c)<w^{-1}(d)<w^{-1}(a)<w^{-1}(b)$. However, this shows
    that $[b,d]\in \Inv_2(w)$ and $[b,c]<_{\P_w(n,2)}
    [b,d]<_{\P_w(n,2)}[a,d]$.
The element $[b,d]$ must appear in \eqref{eq:k=2_chain}, since $\rho$
is a linear extension of $\P_w(n,2)$. Thus, $[b,d]=[a,c]$ as
equivalence classes, $a$ and $b$ are congruent, and $[b,c]=[a,c+mn]$
for some $m$.
A similar argument applies when there are consecutive elements of the form $[b,c]<_{\rho} [b,d]<_{\rho} [a,d]$. From this, we conclude that elements in $\rho$ can all be expressed in the form $[b,c]$ with $c$ fixed, or $[b,c]$ with $b$ fixed. Hence, (e) holds. 
\end{proof}

We now turn our attention to $\C_w(n,3)$ and $\P_w(n,3)$. We first show that a direct extension of the property in \cite[Theorem 6.12]{tenner} is sufficient for $\P_w(n,3)$ and $\C_w(n,3)$ to be a total order.

\begin{lemma}\label{lem:2common_to_path}
Suppose $\Inv_4(w)=\emptyset$ and any pair of elements in $\Inv_3(w)$ can be expressed with 2 values in common. Then $\P_w(n,3)$ is a total order. 
\end{lemma}

\begin{proof}
Let $X,Y\in \Inv_3(w)$, which by assumption can be expressed with 2
values in common. First suppose that $X$ and $Y$ are congruent. If
$Y=X+cv_2$ or $Y=X+cv_1$ for some nonzero integer $c$, then it is
clear that $X$ and $Y$ are related in $\P_w(n,3)$. If
$Y=X+c_1v_1+c_2v_2$ for nonzero integers $c_1$ and $c_2$, then
$X+c_1v_1, X+c_2v_2\in \Inv_3(w)$ by \Cref{cor:congruence.lemma}. Since $c_1$ and $c_2$ are nonzero, it is straightforward to show that $X+c_1v_1$ and $X+c_2v_2$ cannot be expressed with $2$ values in common, which contradicts the hypotheses. Hence, it must be that this last case cannot occur, and $X$ and $Y$ are comparable in $\P_w(n,3)$.

Alternatively, suppose $X$ and $Y$ are not congruent. Since they can
be expressed with two values in common, they belong to the packet of
some $Z\in \binom{\mathbb{Z}}{4}_n$. As $X,Y\in \Inv_3(w)$ and
$\Inv_4(w)=\emptyset$, we conclude $Z\in \Qinv_4(w)$. Hence, $X$ and
$Y$ are related by a generating relation in $\P_w(n,3)$ by \Cref{def:affine-permanent-poset}. Combined, we
conclude that any pair of elements in $\Inv_3(w)$ is comparable in
$\P_w(n,3)$, so $\P_w(n,3)$ is a total order.
\end{proof}

A converse to \cref{lem:2common_to_path} in the affine case is not immediate due to the multiple types of relations in $\P_w(n,3)$ and congruence modulo $n$. However, a direct analysis of these relations allows us to characterize $\Inv_3(w)$ when $\P_w(n,3)$ is a total order. 

\begin{lemma}\label{lem:path_to_2common}
Suppose $\P_w(n,3)$ is a total order. Then there exists an $i \in [3]$ such that
\[
|\{X_{i}: X\in \Inv_{3}(w) \}|\leq 1.
\]
Equivalently, every
3-inversion can be represented using 2 fixed common values in the same two
positions $[3]\setminus \{i\}$.  
\end{lemma}

\begin{proof}
Assume $\P_w(n,3)$ is a total order.  By expressing $X+c_{1}v_{1}$ as
$X-(c_{1}n,0,0)$, we can categorize the generating relations in
$\P_w(n,3)$ into three types depending on if the first, second, or
third position changes.  Because the quasi-inversion and congruence
relations are a bit different, we can more concretely express them as
follows:
\begin{enumerate}
	\item quasi-inversion relations changing the first position and congruence relations involving $v_1$,
	\item quasi-inversion relations changing the second position, or
	\item quasi-inversion relations changing the third position and congruence relations involving $v_2$.
\end{enumerate}
If the total order $\P_w(n,3)$ involves only relations of one type,
then the result is clear. Hence, we assume by contradiction that $\P_w(n,3)$ contains relations of different types. Then there is some saturated chain $X\lessdot_{\P_w(n,3)} Y \lessdot_{\P_w(n,3)} Z$ in $\P_w(n,3)$ where the type of relation between $X$ and $Y$ differs from the one between $Y$ and $Z$.

We now consider all combinations of the relations above for the
saturated chain $X\lessdot_{\P_w(n,3)} Y \lessdot_{\P_w(n,3)} Z$.  Let
$X=[x_{1},x_{2},x_{3}]$.  Since $X\lessdot_{\P_w(n,3)} Y$, we may
express $Y$ using two values from $X$.  Similarly, we may express $Z$
using two values from $Y$.   In each case below, we show there is some
$Y'\in \Inv_3(w)\setminus \{X,Y,Z \}$ that must be in the interval
between $X$ and $Z$ in $\P_w(n,3)$.  We leave it to the reader to
check that each such $Y'$ is above $X$ and below $Z$ in $\P_w(n,3)$
via generating relations.

\begin{itemize}
	\item Suppose $Y=[y,x_2,x_3]$ and $Z=[y,z,x_3]$.  Then $Y$ and 
	$Z$ must form a quasi-inversion.  
	Combined with $X,Y, Z\in \Inv_3(w)$, we find
	$x_1<y<z<x_2<x_3$ and 
	\[w^{-1}(x_1)>w^{-1}(x_2)>w^{-1}(z)>w^{-1}(x_3),\]
so $Y'=[x_1,z,x_3]\in \Inv_3(w)$.

	\item Suppose $Y=[y,x_2,x_3]$ and $Z=[y,x_2,z]$.  Then,  $[x_1,x_2],[x_2,z]\in \Inv_2(w)$ imply
	$Y'=[x_1,x_2,z]\in \Inv_3(w)$.

	\item Suppose $Y=[x_1,y,x_3]$ and $Z=[z,y,x_3]$. In this case,
	 $Y$ and $Z$ are congruent or form a quasi-inversion.  
	In either case, combined with $X,Y,Z \in \Inv_3(w)$, we find  $x_1<z<y<x_2<x_3$
	and 
	\[w^{-1}(z)>w^{-1}(x_1)>w^{-1}(x_2)>w^{-1}(x_3),\]
so $Y'=[z,x_2,x_3]\in \Inv_3(w)$.

	\item Suppose $Y=[x_1,y,x_3]$ and $Z=[x_1,y,z]$. Then $X$ and
	$Y$ must form a quasi-inversion, and combined with $X,Z\in
	\Inv_3(w)$, we find  $x_1<y<x_2<x_3<z$ and \[w^{-1}(x_1)>w^{-1}(x_2)>w^{-1}(y)>w^{-1}(z),\]
so $Y'=[x_1,x_2,z]\in \Inv_3(w)$.

	\item Suppose $Y=[x_1,x_2,y]$ and $Z=[z,x_2,y]$.  Then, $[z,x_2],[x_2,x_3]\in \Inv_2(w)$ imply
	$Y'=[z,x_2,x_3]\in \Inv_3(w)$.

	\item Suppose $Y=[x_1,x_2,y]$ and $Z=[x_1,z,y]$.  In this
	case, either $X$ and $Y$ are congruent or form a
	quasi-inversion.   Combined with $X,Y,Z \in \Inv_3(w)$, we find
	$x_1<z<x_2<x_3<y$ and 
	\[w^{-1}(x_1)>w^{-1}(z)>w^{-1}(y) > w^{-1}(x_3),\]
	so $Y'=[x_1,z,x_3]\in \Inv_3(w)$.
\end{itemize}
Since $\P_w(n,3)$ is a total order, $Y$ and $Y'$ are related in $\P_w(n,3)$. This contradicts that $X\lessdot_{\P_w(n,3)} Y \lessdot_{\P_w(n,3)} Z$ is a saturated chain. 
\end{proof}

\begin{proposition}\label{thm:path_iff_2common}
The following are equivalent for $w \in \affineS_{n}$:
\begin{enumerate}[label=(\alph*)]
	\item $\C_w(n,3)$ is a total order,
	\item $\P_w(n,3)$ is a total order,
    \item there exists an $i \in [3]$ such that
$|\{X_{i}: X\in \Inv_{3}(w) \}|\leq 1$, 
	\item $\Inv_4(w)=\emptyset$ and every pair of elements in
	$\Inv_3(w)$ can be expressed with 2 values in common, and
	\item $\cG(w)$ is a path as an undirected graph.
\end{enumerate}
\end{proposition}

\begin{proof}
	The equivalence of (a) and (b) is \cref{lem:4321_totalorder}.
\cref{lem:path_to_2common} shows that (b) implies (c). For (c) implies
(d), note that when $\Inv_4(w)$ is nonempty, the packet of any $X\in
\Inv_4(w)$ provides a counterexample to (c). Hence, when (c) holds,
$\Inv_4(w)$ must be empty, and the second condition in (d) is then
immediate.  \cref{lem:2common_to_path} shows that (d) implies (b).  By
\cref{thm:affine}, $\cG(w)$ is a path if and only if $\C_w(n,3)$ is a
total order proving (e) is equivalent to (a).
\end{proof}

\subsection{The Ultimate Affine  Higher Bruhat Orders}\label{sub:ultimateAHB}

Here we consider the case when $k=n$ and $k>n$.  Note, for $k\geq n$,
$w$ has no $(k+1)$-inversions due to $n$-periodicity.  However, $w$
can have any number of $n$-inversions.

One may identify the affine permutations with the alcoves in the
affine braid arrangement determined by hyperplanes with equations of
the form $x_{i}-x_{j}=c$ for all $c \in \mathbb{Z}$.  See \cite[Ch
4]{Hum} for an introduction using reflection groups and root systems.
The affine permutations for two adjacent alcoves differ by a simple
transposition.  Reduced words for $w$ correspond with alcove walks of
minimal length from the identity alcove to the alcove identified with
$w$.  In that context, the $n$-inversions of $w$ can then be
identified with the weights in the weight lattice in $\mathbb{R}^{n}$
inside the convex hull of the union of reduced alcove walks from the
identity alcove to the alcove indexed by $w$.  We note however that
the natural partial ordering on the weight lattice in $\mathbb{R}^{n}$
is not the same as the permanent poset $\P_{w}(n,n)$.

\begin{theorem}\label{thm:ultimateAHB}
    For $w\in \affineS_n$, the following hold.
\begin{enumerate}[label=(\alph*)]
\item If $\Inv_{n}(w)$ is nonempty, the permanent poset $\P_{w}(n,n)$ on $\Inv_{n}(w)$ is isomorphic
to a product of chains $[m_{1}]\times [m_{2}] \times \cdots \times
[m_{n-1}]$ where $m_{i}$ is the number of inversions for $w$ congruent
to $[y_{i},y_{i+1}]$ and $Y=[y_{1},\dots , y_{n}]$ is the compressed
representative of the unique congruence class in $\Inv_{n}(w)$.

\item The heap poset $(\Inv_n(w),\leq_{\emptyset})$ is $\P_{w}(n,n)$.  

\item The $n$-admissible orders are linear extensions of the permanent
poset, so 
\[
\A_w(n,n) =  \mathscr{L}(\P_{w}(n,n))= \mathscr{L}(\Inv_n(w),\leq_{\emptyset}).
\]
\item The consistent sets in $\C_w(n,n)$ are the lower order ideals in
the permanent poset $\P_{w}(n,n)$ ordered by either single step
inclusion or equivalently by inclusion.
\item There is a natural bijection between
    maximal chains of $\C_w(n,n)$ and the admissible orders in
    $\A_w(n,n)$.
\item For all $k\geq n$, the higher Bruhat order $\B_w(n,k)\cong 
\C_w(n,k+1)$ is the one element poset on $\emptyset$.
\end{enumerate}
\end{theorem}

Note the lower order ideals of a product of chains generalize the
notion of partitions inside a rectangle and plane partitions inside a
box.  See \cite{Early.2019} for history and combinatorial results on
lower order ideals in a product of chains.

\begin{proof}
For Part (a), assume $\Inv_{n}(w)$ is nonempty. Let $X \in \Inv_{n}(w)$.  By
\Cref{lem:congruence.order}, we know $\Inv_{n}(w) = I(X) = \{Z \in
\Inv_{n}(w): X \equiv Z \}$.  Thus, $\P_{w}(n,n)$ is the congruence
poset $\CR(n,n)$ restricted to $\Inv_{n}(w)$, which is isomorphic to a
product of chains by \Cref{cor:congruence.lemma}.

To prove Parts (b), (c), (d), and (e) observe that
$\Inv_{n+1}(w)=\emptyset$.  The edges in the reversal graph
$G_{\emptyset}$ on $\Inv_{n}(w)$ are exactly the congruence edges used
to generate $\P_{w}(n,n)$.  It follows that the corresponding heap
poset $(\Inv_n(w),\leq_{\emptyset})=\P_{w}(n,n)$, proving Part (b).
The full-packet condition for $n$-admissible orders is trivially
satisfied by every linear extension of $\P_{w}(n,n)$, so Part (c)
holds by the definition of admissible orders.  Similarly, the MSZ
condition is trivially satisfied for every lower order ideal of
$\P_{w}(n,n)$, hence Part (d) follows from the definition of
consistent sets.  For Part (e), the bijection between maximal chains
of $\C_w(n,n)$ and the admissible orders in $\A_w(n,n)$ is just the
standard bijection between maximal chains of lower order ideals in a
poset and its set of linear extensions.

To prove Part (f), apply Part (c) and \Cref{lem:linear_extensions} to
observe that any two $n$-admissible orders for $w$ can be obtained
from one another by a sequence of commutations.  Thus, $\A_{w}(n,n)$
has only one commutation class corresponding with the empty reversal
set, so $\B_{w}(n,n)$ is a one element poset.  Since
$\Inv_{k}(w)=\emptyset$ for all $k>n$, we observe that the empty list
is the only $k$-admissible order for $w$, $\B_{w}(n,k)$ is the one
element poset on the empty list, and $\C_{w}(n,k)=\{\emptyset \}$.   
\end{proof}

\begin{corollary}\label{cor:n=3}
    \cref{thm:nonlongest} holds for all $w\in \affineS_n$ and positive
integers $k$ when $n\leq 3$.
\end{corollary}

\begin{proof}
This follows from \cref{thm:affine},
\Cref{lem:affine_Cwn2_bijection}, and \Cref{thm:ultimateAHB}.
\end{proof}

\section{Conclusion and Open Problems}\label{sec:open}

We have given a generalization of the classical admissible orders
$\A(n,k)$, higher Bruhat orders $\B(n,k)$, and consistent sets
$\C(n,k)$ for affine permutations $w\in \affineS_n$.  When $w\in
\S_n$, we extended both Ziegler's and Manin and Schechtman's classical
results, which are summarized in \cref{thm:nonlongest}. For arbitrary
$w\in \affineS_n$, \cref{thm:affine} gives analogous results for the
first and second higher Bruhat orders only, while
\Cref{thm:ultimateAHB} gives analogous statements for $k\geq n$.  We 
conjecture the analogy extends for all positive integers $k$.  We also hope
to further understand the structures defined here.

\subsection{Characterizing  Reversal Sets of Affine Admissible Orders} Our approach to
proving \cref{thm:nonlongest} relied on showing that $G_R$ is acyclic
for all $w\in \S_n$, $1\leq k\leq n$, and $R\in \C_w(n,k)$.  One key
step in the proof is providing a partition of the vertices of $G_{R}$
into sets $V_0(G_{R}), V_1(G_{R}), V_2(G_{R})$ depending on $R$ and
$n$.  The partition helps prove acyclicity of $G_{R}$ because edges
between these sets only go from $V_{i}(G_{R})$ to $V_{j}(G_{R})$ for
$i\leq j$.  The acyclicity of each $V_{i}(G_{R})$ is established by
induction using deletion and truncation to complete the proof.  

For
affine permutations, there is always some $z \in [n]$ such that if
$z+mn$ appears in $X=[x_{1},\dots , x_{k}] \in \Inv_{k}(w)$, then
it must be in the last position $x_{k}\equiv z$, for all $k\geq 1$.
Such a value $z$ is a source vertex of the affine inversion graph $I_{w^{-1}}$, with edges pointing down in the poset.
However, the straightforward generalization of $V_1(G_{R})$ using
$z(w)$ in place of $n$ leads to a set that is not necessarily convex
in the induced heap poset $(\Inv_{k}(w), \leq_{R})$.  For example,
if $k=2$, $n=3$, $w = (2, -2, 6)$, and $R=\emptyset$, then the Hasse
diagram of the heap poset $(\Inv_{2}(w), \leq_{R})$ is 
\[
[1,2] \to [1,3] \to [1,5]\to  [1,6]\to  [2,3]
\]
The inversions $[1,3], [2,3]$ both contain the maximal element
$z(w)=3$, but the interval between them contains $[1,5]$, so
$V_1(G_{R})$ is not convex in the heap poset. 
Consequently, paths can start in $V_1(G_R)$, leave $V_1(G_R)$, and return to
$V_1(G_{R})$ multiple times in the affine case.  Thus, a direct
generalization of our approach to the affine case will not
work.

\cref{conj:affineG_R.advanced} aims to characterize the
reversal sets of admissible orders and the acyclic $G_{R}$ graphs for
a given $w\in \affineS_n$ in another way. Once this is established,
our general approach for proving \cref{thm:nonlongest} in
\cref{sub:consistent.Sn} can be adapted for affine permutations.

\begin{conjecture}\label{conj:affineG_R.advanced}
   For $w\in \affineS_n$ and $R \subseteq \Inv_{k+1}(w)$,  the
   following are equivalent.
   \begin{enumerate}[label=(\alph*)]
\item There exists some $\rho \in \A_{w}(n,k)$ such that $R=\Rev
(\rho)$.
\item  $R \in \C_w(n,k+1)$.
\item The graph $G_R$ is an acyclic graph on $\Inv_{k}(w)$.
\item The preorder $(\Inv_k(w),\leq_R)$ is a poset, and 
\[
\{\rho \in \A_{w}(n,k): \Rev_{n,k,w}(\rho)=R
\} =\mathscr{L}(\Inv_k(w),\leq_R)
\]
is an element in $\B_w(n,k)$.
\end{enumerate}
\end{conjecture}

\Cref{lem:reversal_is_consistent.affine} shows every reversal set for
an admissible order is consistent, so (a) implies (b).
\Cref{lem:GR_to_admissible.affine}, \cref{thm:equiv_class.and.linear.exts.affine}, \cref{cor:reversals_unique.affine},
and \cref{cor:B.equiv.acyclic.affine} show that (a), (c), and (d) are equivalent.
Therefore, it only remains to prove (b) implies (a), (c), or (d) in
\Cref{conj:affineG_R.advanced}.

\Cref{conj:affineG_R.advanced} has been verified by computer for all
$3\leq k\leq n\leq 8$ and affine permutations $w \in \affineS_{n}$ up
to certain lengths, as indicated in \cref{tab:conjecture}.  While the
general case remains open, we showed acyclicity of $G_R$ when $R =
\emptyset$ and $R = \Inv_{k+1}(w)$ for all $1\leq k\leq n$ in
\Cref{lem:empty_full_acyclic}.  For all $k\geq n$, the conjecture
holds by \Cref{thm:ultimateAHB}.

\subsection{Geometric Interpretations} As noted in the
introduction, the classical higher Bruhat orders have connections to
geometric objects such as hyperplane arrangements, oriented matroids,
and zonotopal tilings. Hence, we pose the following problem.

\begin{problem}\label{prob:2}
    Find geometric interpretations of the admissible orders
$\A_w(n,k)$, higher Bruhat orders $\B_w(n,k)$, and the poset on
consistent sets $\C_w(n,k)$ for arbitrary $w\in \affineS_n$ and $k$.
\end{problem}

\subsection{Higher Inversion Sets} This paper relies on a careful
study of the $k$-inversion sets for classical and affine permutations.
These sets are natural generalizations of the $2$-inversion sets, but
 have not previously been extensively studied.  Recall that $2$-inversion
sets are characterized by \Cref{lem:inversion.sets}.  The next 
problems concern the $k$-inversion sets.

\begin{problem}\label{prob:3}
Let $1<k<n$, and let $R$ be a finite subset of $\binom{\Z }{k}_{n}$.
If $R=\Inv_{k}(w)$ for some $w \in \affineS_{n}$, then $R$ satisfies
certain conditions derived from \Cref{lem:adjacent.pairs},
\Cref{lem:intersection_size}, \Cref{lem:types.of.packets}, and
\Cref{cor:congruence.lemma}.  In general, these conditions are not 
sufficent to be the $k$-inversion set of a permutation. For example, 
when $n=5$ and $k=3$, we consider $R=\{[1,2,5],[1,3,4]\}$. Any permutation with 
these two inversions satisfies
$w^{-1}(1)>w^{-1}(2)>w^{-1}(5)$ and $w^{-1}(1)>w^{-1}(3)>w^{-1}(4)$. 
If $w^{-1}(2)>w^{-1}(3)$, then $[1,2,3]\in \Inv_3(w)$, and if
$w^{-1}(3)>w^{-1}(2)$, then $[1,3,5]\in \Inv_3(w)$. In either case, 
$\Inv_3(w)\neq R$. With this in mind, what additional properties must be 
added to obtain a characterization of all possible $\Inv_k(w)$ for $w\in \affineS_{n}$?
\end{problem}

\begin{problem}\label{prob:4}
If $w \in \S_n$, then the union of all $k$-inversion sets for $w$
forms a simplicial complex.  What nice properties does the inversion
simplicial complex for $w$ have for classical permutations?  In the
affine case we have a little wrinkle, as $\Inv_{1}(w)\cong \Z$ is
infinite and each of the finitely many 2-inversions $[x,y]$ covers an
infinite number of atoms $\{x+mn \}$ and $\{y+mn\}$. This structure is
not typical in algebraic combinatorics.  We could instead consider the
finite union of all $\Inv_{k}(w)$ for $k\geq 2$, but this is
technically not a simplicial complex unless $w = \id$.
\end{problem}

Enumerating Coxeter group elements by length is a classical
problem. In the case of $\affineS_n$ or $\S_n$, this corresponds to
enumerating possible $2$-inversion sets.  Using the latter view, we
pose several enumeration problems for general $k$-inversion sets.  The
following table gives the number of distinct $k$-inversion sets for
permutations in $\S_{n}$. The table read by rows for $1\leq k\leq n$ is \cite[A399980]{OEIS}.
The $k=2$ column is $n!$.  The $k=3$ and $k=4$ columns are
\cite[A399986,A399987]{OEIS}. For example, if $k=3$, $n=4$, there are 9
distinct $3$-inversion sets for all permutations in $\S_{4}$, namely
\[
\{\},\ \{[1, 2, 4]\},\ \{[2, 3, 4]\},\ \{[1, 2, 3]\},\ \{[1, 3, 4]\}, 
\]
\[
\{[1, 3, 4], [2, 3, 4]\},\ \{[1, 2, 4], [1, 3, 4]\},\ \{[1, 2, 3], [1, 2, 4]\}, 
\]
\[
\{[1, 2, 3], [1, 2, 4], [1, 3, 4], [2, 3, 4]\}.
\]

\begin{table}[h!]
\begin{tabular}{| c | c | c | c | c | c | c | c | c | c | c |c|}
\hline
$n,k$ & 1 & 2 & 3 & 4 & 5 & 6 & 7 & 8 & 9 & 10 & 11 \\ \hline 
1 &1 & 1 & 1& 1 & 1 & 1 & 1 & 1 & 1 & 1 & 1   \\ \hline 
2 & 1 & 2 & 1 & 1 & 1 & 1 & 1 & 1 & 1 & 1& 1  \\ \hline 
3 &1 & 6 & 2& 1 & 1 & 1 & 1 & 1 & 1 & 1 & 1  \\  \hline 
4 & 1 & 24 & 9 & 2& 1 & 1 & 1 & 1 & 1 & 1 & 1 \\  \hline 
5 & 1 & 120 & 52 & 11 & 2& 1 & 1 & 1 & 1 & 1  & 1 \\ \hline 
6 & 1 & 720 & 341 & 79 & 13 & 2& 1 & 1 & 1 & 1  & 1 \\ \hline 
7 & 1 & 5040 & 2540 & 633 & 108 & 15 & 2& 1 & 1 & 1  & 1 \\ \hline 
8 & 1 & 40320 & 21360 & 5700 & 967 & 141 & 17 & 2 & 1 & 1 & 1 \\ \hline 
9 & 1 & 362880 & 200536 & 57027 & 9704 & 1376 & 178 & 19 & 2& 1 & 1 \\ \hline 
10 & 1 & 3628800 & 2079837 & 625021 & 109342 & 14810 &1873 &219 & 21 & 2 & 1 \\ \hline  
11 & 1 & 39916800 & 23618672 & 7440523 & 1358486 & 178777 & 21375 & 2466 & 264 & 23 & 2\\ \hline  
\end{tabular}
\caption{The number of distinct $k$-inversion sets in $\S_n$, where $1\leq k,n\leq 11$.}
\label{tab:k-inversion-enumeration}
\end{table}

\begin{problem}
Is there a simple formula for the number of distinct $k$-inversion
sets over all permutations in $\S_n$ for $1\leq k\leq n$?  The table
suggests two additional problems related to this enumeration.
\begin{enumerate}[label=(\alph*)]
	\item The column corresponding to $3$-inversion sets
appears to align with the enumeration for certain Wilf equivalence
 classes of mesh patterns \cite[Table 12, Nr. 59 and 60]{hilmarsson}. Determine if there is a general
correspondence between $k$-inversion sets in $\S_n$ and appropriate Wilf-equivalence classes of mesh patterns.
	\item The diagonal corresponding to $k=n$ is constant when $n\geq 2$, the diagonal
corresponding to  $k=n-1$ is linear in $n$ when $n\geq 4$, and the diagonal where $k=n-2$ appears to be 
 quadratic in $n$ when $n\geq 6$. Determine if the diagonal corresponding to $k=n-d$ is eventually a polynomial
of degree $d$. 
\end{enumerate}
\end{problem}

\begin{problem}
As there are infinitely many permutations in $\affineS_n$, there are infinitely many possible 
$2$-inversion sets, so the previous problem is only posed for $\S_n$. However, there are finitely
many affine permutations of any given length $\ell$. For $k>2$ and 
$\ell\in \mathbb{N}$, determine how many 
distinct $k$-inversion sets in $\affineS_n$ have size $\ell$.
\end{problem}

\subsection{Enumeration of Higher Bruhat Orders} Enumeration of $\B_{w}(n,k)$ is a difficult
problem.  Even the exact enumeration of the number of commutation
classes of reduced words for $w_{0}\in \S_n$ is difficult
\cite[A006245]{OEIS}, corresponding to $|\B_{w_{0}}(n,2)|$ for $n\geq
1$.  See \cite{chau2024enumeratinghigherbruhatorders} for recent
progress and further history on this problem.  Since we have shown
that $\Rev: \B_{w}(n,k) \to \C_{w}(n,k+1)$ is injective for all
positive integers $k,n$ and $w \in \affineS_{n}$, the enumeration of
$\B_{w}(n,k)$ may be alternatively counted via consistent sets, namely
enumerating the lower order ideals $R$ of the corresponding permanent
poset $\P_{w}(n,k+1)$ satisfying the MSZ condition, and confirming
$G_{R}$ has at least one linear extension.  If $w$ is a classical
permutation, then we can skip the last confirmation by our main
theorem.  However, note $\P_{w_{0}}(n,k+1)$ is an antichain with
$\binom{n }{k+1}$ elements, so the number of lower order ideals is
$2^{\binom{n }{k+1}}$.  However,
\cite[Thm. 1.2]{chau2024enumeratinghigherbruhatorders} shows
$|\B_{w_{0}}(n,k)|=|\C_{w_{0}}(n,k+1)|=2^{\Theta(n^k)}$ as 
$n \to  \infty$ with $k\geq 2$ fixed,
while the lower order ideal condition alone gives the upper bound $2^{\binom{n }{k+1}}$,
so it must be the case that the MSZ restriction is significant.

\begin{problem}
Can one improve the bounds in
\cite[Thm. 1.2]{chau2024enumeratinghigherbruhatorders} using the MSZ
condition on lower order ideals of $\P_{w}(n,k+1)$?
\end{problem}

\begin{problem}
Suppose $w\in \affineS_n$ contains the permutation pattern $u\in \S_m$
with $m<n$ on values in a set $A \subseteq [n]$. Then for all $1 <
k\leq n$, the deletion map $R \mapsto R\setminus \bar{A}$ is an
order-preserving function from $\C_w(n,k)$ to $\C_u(m,k)$.  Is this
map surjective?  If so, then $|\C_w(n,k)|\geq |\C_u(m,k)|$, which
would generalize a theorem of Tenner \cite[Thm. 5.10]{tenner}.
\end{problem}

\section{Appendix}\label{sec:appendix}

Several lemmas in this paper require a finite computer verification.
The code can be found at \cite{code}.  Each verification is identified in
the code as follows.
\begin{enumerate}[label=(\alph*)]
\item \cref{lem:partition} claims every edge \(X\to Y\) in \(G_R\) has
      \(X\in V_a(G_R)\), \(Y\in V_b(G_R)\), and \(a\le b\). The proof reduces to checking over all $w \in \S_n$ and $R
\in \C_w(n,n-1)$ with $3\leq n\leq 6$. This is done in
``Partition edges test''.
\item The proof of \cref{lem:flippable_Bn2_Cn3} reduces to the following claims for $w\in \S_4$ and $\rho\in \A_w(4,2)$.
\begin{itemize}
	\item  If $P(X')$ is lex-unnested in $\rho$, then every element between $X_3'$ and $X_2'$ commutes with $X_3'$, and every element between $X_2'$ and $X_1'$ commutes with $X_1'$.
	\item If $P(X')$ is antilex-unnested in $\rho$, then every element between $X_1'$ and $X_2'$ commutes with $X_1'$, and every element between $X_2'$ and $X_3'$ commutes with $X_3'$.
\end{itemize}
These are respectively verified in ``Lex-unnested test'' and ``Antilex-unnested test''.
\item The proof of \cref{lem:shorten5} claims that for $w\in \S_n$ with $n\leq 6$, every directed path of length 4 in \(G_R\) between
      \(X,Y\in\operatorname{Inv}_2(w)\) can be replaced by a strictly shorter directed path. This is done in ``Diameter test.''
\item \cref{conj:affineG_R.advanced} involves acyclicity of $G_R$. The values in \cref{tab:conjecture} are tested in ``Conjecture test.'' Note that the claims in \cref{conjecture:main} then follow from our discussion of \cref{conj:affineG_R.advanced}.
\end{enumerate}

\section*{Acknowledgments}

We would like to thank Ben Elias for suggesting the problems in the
introduction to us and many helpful conversations along the way.  We
also thank Gaku Liu, Ricky Liu, Isabella Novik, Treanungkur Mal, Grace
O'Brien, Monica Vazirani, and Hailun Zhang for helpful conversations.
We are grateful to ChatGPT Pro 5.5 for editorial suggestions on this
manuscript.

\printbibliography

\end{document}